\documentclass{elsarticle}
\usepackage{amsmath,amssymb,amsfonts,latexsym, amsthm,array,shuffle}
\usepackage{wrapfig,xspace,hyperref}
\usepackage{rotating,caption}
\usepackage{tikz}
\usetikzlibrary{shadows}
\usetikzlibrary{decorations.text}
\usetikzlibrary{decorations.markings}
\usepackage{graphicx}
\usepackage[utf8]{inputenc}
\usepackage{appendix}
\usepackage{framed}

\newtheorem{defi}{Definition}[section]
\newtheorem{theo}[defi]{Theorem}
\newtheorem{lem}[defi]{Lemma}
\newtheorem{csq}[defi]{Consequence}
\newtheorem{rem}[defi]{Remark}

\newtheorem{example}[defi]{Example}

\newcommand{\SP}{\ensuremath{{\mathcal{SP}}}}
\newcommand{\M}{\ensuremath{{\mathcal M}}}
\newcommand{\U}{\ensuremath{{\mathcal U}}}
\newcommand{\A}[1]{\ensuremath{{\mathcal A}(#1)}}
\newcommand{\AC}[1]{\ensuremath{{\mathcal {AC}}(#1)}}
\def \O {\mathcal O}

\newcommand{\PhiP}[2]{\ensuremath{\chi^{(#1)}_{#2}}}
\newcommand{\Astar}{\ensuremath{A^{\star}}}

\newcommand{\Qplus}{\ensuremath{Q^{+}}}
\newcommand{\SHplus}{\ensuremath{S^{+}_H}}
\newcommand{\SVplus}{\ensuremath{S^{+}_V}}
\newcommand{\Qminus}{\ensuremath{Q^{-}}}
\newcommand{\SHminus}{\ensuremath{S^{-}_H}}
\newcommand{\SVminus}{\ensuremath{S^{-}_V}}

\newcommand{\F}[1]{\ensuremath{\mathfrak{P}^{(1)}_{(#1)}}}
\newcommand{\G}[1]{\ensuremath{\mathfrak{P}^{(3)}_{(#1)}}}

\newcommand{\quadrantell}{\ensuremath{h}} 

\newcommand{\fint}{\ensuremath{P}}
\newcommand{\fell}{\ensuremath{P^{(\quadrantell)}}}
\newcommand{\fmix}{\ensuremath{P^{\mathit mix}}}

\newcommand{\calE}{{\mathcal E}}
\newcommand{\calN}{{\mathcal N}}

\newcommand{\horse}[2]{\ensuremath{\mathcal (#1H#2)}}
\newcommand{\setH}{\ensuremath{\mathcal{H}}}

\newcommand{\epi}{oscillation\xspace}
\newcommand{\epis}{oscillations\xspace}

\newcommand{\snl}{strong numeral-led }

\newcommand{\fois}{pieces\xspace}

\usetikzlibrary{snakes,decorations,patterns,decorations,shapes.geometric,automata,positioning,decorations.pathreplacing}

\tikzstyle{simple}=[inner sep=2pt]
\tikzstyle{linear}=[circle, draw=none, fill=none, text centered, text=black,inner sep=0pt]
\tikzstyle{leaf}=[circle, draw=none, fill=black, text centered, text=white,inner sep=0, minimum size = 4pt]

\tikzstyle{boite} = [very thick]
\newenvironment{tikzDiagramme}[1]{\begin{tikzpicture}[scale=#1]}{\end{tikzpicture}}
\newenvironment{tikzArbre}[1]{\begin{tikzpicture}[scale=#1,inner
  sep=0]}{\end{tikzpicture}}

\newcommand{\inputState}[2]{
\draw[fill] (#1,#2) circle (2pt);
}
\newcommand{\outputState}[2]{
\draw (#1,#2) circle (1.5pt);
\draw (#1,#2) circle (3pt);
}

\newcommand{\patateTriangle}[4]{
\begin{scope}[shift={(#1,#2)}]
\draw [drop shadow,rounded corners,thick, fill=black!10] (0,0) -- (2,0) -- (0,2) -- cycle;
\draw (0.7,0.5) node {\footnotesize ${\mathcal A}(\xi_{#3},\xi_{#4})$};
\outputState{0.15}{1.65};
\outputState{1.65}{0.15};
\inputState{0.15}{0.15};
\end{scope}
}
\newcommand{\patateTriangleUnlabelled}[2]{
\begin{scope}[shift={(#1,#2)}]
\draw [thick, fill=black!10] (0,0) -- (2,0) -- (0,2) -- cycle;
\outputState{0.15}{1.65};
\outputState{1.65}{0.15};
\inputState{0.15}{0.15};
\end{scope}
}

\newcommand{\pin}[2]{
\fill (#1,#2) circle (.2);
}
\newcommand{\pinU}[2]{
\fill (#1,#2) circle (.2);
\draw [thick] (#1,#2) -- (#1,#2-1.3);
}
\newcommand{\pinD}[2]{
\fill (#1,#2) circle (.2);
\draw [thick] (#1,#2) -- (#1,#2+1.3);
}
\newcommand{\nameUnder}[3]{
\draw (#1-1,#2-1) node {#3};
}
\newcommand{\pinL}[2]{
\fill (#1,#2) circle (.2);
\draw [thick] (#1,#2) -- (#1+1.3,#2);
}
\newcommand{\pinR}[2]{
\fill (#1,#2) circle (.2);
\draw [thick] (#1,#2) -- (#1-1.3,#2);
}

\newcounter{indice}

\newcommand{\permutation}[1]{
\setcounter{indice}{0};
\foreach \i in {#1}
\addtocounter{indice}{1};

\addtocounter{indice}{1}
\draw [help lines] (1,1) grid (\theindice,\theindice);

\setcounter{indice}{1};

\foreach \i in { #1 } {
\draw (\theindice+.5,\i+.5) [fill] circle (.2);
\addtocounter{indice}{1};
}
\addtocounter{indice}{-1};

}
\newcommand{\permutationNomKnight}[2]{
\setcounter{indice}{0};
\foreach \i in {#1}
\addtocounter{indice}{1};

\addtocounter{indice}{1}
\draw [help lines] (1,1) grid (\theindice,\theindice);
\setcounter{indice}{1};

\foreach \i in { #1 } {
\draw (\theindice+.5,\i+.5) [fill] circle (.2);
\draw (\theindice+.5,0) node {\tiny \i};
\addtocounter{indice}{1};
}
\addtocounter{indice}{-1};
\foreach \xi/\yi/\xj/\yj in {#2} {
	\draw (\xi+.5,\yi+.5) -- (\xj+.5,\yj+.5);
}
}

\newcommand{\permutationNom}[1]{
\setcounter{indice}{0};
\foreach \i in {#1}
\addtocounter{indice}{1};

\addtocounter{indice}{1}
\draw [help lines] (1,1) grid (\theindice,\theindice);
\setcounter{indice}{1};

\foreach \i in { #1 } {
\draw (\theindice+.5,\i+.5) [fill] circle (.2);
\draw (\theindice+.5,-0.5) node {\tiny \i};
\addtocounter{indice}{1};
}
}

\newcommand{\permutationNomGros}[1]{
\setcounter{indice}{0};
\foreach \i in {#1}
\addtocounter{indice}{1};

\addtocounter{indice}{1}
\draw [help lines] (1,1) grid (\theindice,\theindice);
\setcounter{indice}{1};

\foreach \i in { #1 } {
\draw (\theindice+.5,\i+.5) [fill] circle (.2);
\draw (\theindice+.5,-0.5) node {\scriptsize \i};
\addtocounter{indice}{1};
}
}

\begin{document}

\title{An algorithm for deciding the finiteness of the
  number of simple permutations in permutation classes
  \footnote{This work was completed with the support of the ANR
   (project ANR BLAN-0204\_07  MAGNUM).}}
\author[fb]{Fr\'ed\'erique Bassino}
\address[fb]{Universit\'e Paris 13, Sorbonne Paris Cité,
LIPN, CNRS UMR 7030, F-93430, Villetaneuse, France.}
\author[mb]{Mathilde Bouvel\footnote{Corresponding author. \texttt{mathilde.bouvel@math.uzh.ch}, \texttt{+41 44 63 55 852 }.}$^,$}
\address[mb]{LaBRI UMR 5800, Univ. Bordeaux and
 CNRS, 351, cours de la Libération, 33405 Talence cedex, France; 
 and Institut f\"ur Mathematik, Uni. Z\"urich, Zurich, Switzerland. }
\author[ap]{Adeline Pierrot}
\address[ap]{LIAFA UMR 7089, Univ. Paris Diderot and
 CNRS, Case 7014, 75205 Paris cedex 13, France; 
 and Institut f\"ur Diskrete Mathematik und Geometrie, TU Wien, Vienna, Austria;
 and LRI UMR 8623, Univ. Paris Sud and CNRS, 91405 Orsay Cedex, France.}
\author[dr]{Dominique Rossin}
\address[dr]{LIX UMR 7161, Ecole Polytechnique and CNRS, 91128  Palaiseau, France.
}

\date{}

\begin{abstract}
  In this article, we describe an algorithm to determine whether a
  permutation class $\mathcal{C}$ given by a finite basis $B$ of
  excluded patterns contains a finite number of simple permutations.
  This is a continuation of the work initiated in [Brignall, Ru{\v{s}}kuc, Vatter, \emph{Simple permutations: decidability and unavoidable substructures}, 2008],
  and shares several aspects with it.
  Like in this article, the main difficulty is to decide whether $\mathcal{C}$ contains a finite number of proper pin-permutations,
  and this decision problem is solved using automata theory.
  Moreover, we use an encoding of proper pin-permutations by words over a finite alphabet, introduced by Brignall \emph{et al}.
  However, unlike in their article, our construction of automata is fully algorithmic and efficient.
  It is based on the study of pin-permutations in [Bassino, Bouvel, Rossin, \emph{Enumeration of pin-permutations}, 2011].
  The complexity of the overall algorithm is  ${\mathcal O}(n \log n + s^{2k} )$
  where $n$ denotes the sum of the sizes of permutations in the basis $B$,
  $s$ is the maximal size of a pin-permutation in $B$
  and $k$ is the number of pin-permutations in $B$.
\end{abstract}

\maketitle

\noindent\textit{Keywords: } Permutation class; finite basis; simple permutation; algorithm; automaton; pin-permutation.

\noindent\textit{Mathematics Subject Classification: } 05A05; 68R05; 05-04. 

\section{Introduction}

Since the definition of the pattern relation among permutations
by Knuth in the 70's~\cite{Knuth:ArtComputerProgramming:1:1973},
the study of permutation patterns and permutation classes in combinatorics
has been a quickly growing research field,
and is now well-established.
Most of the research done in this domain concerns {\em enumeration}
questions on permutation classes. Another line of research on
permutation classes has been emerging for about a decade: it is
interested in properties or results that are less precise but apply to
{\em families} of permutation classes.
Examples of such general results may regard
enumeration of permutation classes that fall into general frameworks,
properties of the corresponding generating functions,
growth rates of permutation classes,
order-theoretic properties of permutation classes\ldots
This second point of view is
not purely combinatorial but instead is intimately linked with
algorithms. Indeed, when stating general structural results on
families of permutation classes, it is natural to associate to an
existential theorem an algorithm that {\em tests} whether a class
given in input falls into the family of classes covered by the
theorem, and in this case to {\em compute} the result whose existence
is assessed by the theorem.

Certainly the best illustration of this paradigm that can be found in
the literature is the result of Albert and Atkinson~\cite{AA05},
stating that every permutation class containing a finite number of
simple permutations has a finite basis and an algebraic generating
function, and its developments by Brignall {\em et al.} in
\cite{BHV06b,BHV06a,BRV06}. A possible interpretation of this result
is that the simple permutations that are contained in a class somehow
determine how structured the class is. Indeed, the algebraicity of the
generating function is an echo of a deep structure of the class that
appears in the proof of the theorem of~\cite{AA05}: the permutations of the class (or
rather their decomposition trees) can be described by a context-free
grammar.
In this theorem, as well as in other results obtained in this
field, it appears
that {\em simple permutations} play a crucial role. They can be seen
as encapsulating most of the difficulties in the study of permutation
classes considered in their generality, both in algorithms and
combinatorics.

Our work is about these general results that can be obtained for large
families of permutation classes, and is resolutely turned towards
algorithmic considerations. It takes its root in the theorem of Albert
and Atkinson that we already mentioned, and follows its developments
in~\cite{BRV06} and~\cite{BBR09}.

In~\cite{BRV06}, Brignall, Ru{\v{s}}kuc and Vatter provide a criterion
on a finite basis $B$ for deciding whether a permutation class
$\mathcal{C} = Av(B)$ contains a finite number of simple
permutations. We have seen from~\cite{AA05} that this is a sufficient
condition for the class to be well-structured.
To this criterion, \cite{BRV06} associates a decision procedure testing
from a finite basis $B$ whether $\mathcal{C} = Av(B)$ contains a
finite number of simple permutations.
Both in the criterion and in the
procedure, the set of {\em proper pin-permutations} introduced in
\cite{BRV06} plays a crucial part. The procedure is based on the
construction of automata that accept languages of words on a finite
alphabet (that are called {\em pin words}) that encode such permutations
that do not belong to the class.
This procedure is however not fully algorithmic,
and its complexity is a double exponential,
as we explain in Subsection~\ref{ssec:procedure_of_BRV_2}.

Our goal is to solve the decision problem of~\cite{BRV06} with an actual algorithm,
whose complexity should be kept as low as possible.
For this purpose, we heavily rely on~\cite{BBR09} where
we perform a detailed study of the class of {\em
pin-permutations}, which contains the proper pin-permutations of~\cite{BRV06}.
These results allow us to precisely characterize the pin words
corresponding to any given pin-permutation, and to subsequently modify the automata
construction of~\cite{BRV06},
leading to our algorithm deciding
whether a permutation class given by a finite basis $B$ contains a
finite number of simple permutations. 
Figure~\ref{fig:algo} gives an overview of the general structure of our algorithm 
(the notations it uses will however be defined later, in Sections~\ref{sec:strict pw} and~\ref{sec:algos_finite_number_proper_pinperm}). 

\begin{figure}[ht]
\begin{framed}
\begin{enumerate}
 \item Check if $\mathcal C$ contains finitely many parallel alternations and wedge simple permutations.
 \item Check if $\mathcal C$ contains finitely many proper pin-permutations:
\begin{enumerate}
 \item Determine the set $PB$ of pin-permutations of $B$;
 \item For each $\pi$ in $PB$, build an automaton $\mathcal{A}_{\pi}$ recognizing the language $\overleftarrow{{\mathcal L}_{\pi}}$ 
 (or a variant of this language);
 \item From the automata $\mathcal{A}_{\pi}$, build an automaton $\mathcal{A}_{\mathcal{C}}$ recognizing the language
$\overleftarrow{\M \setminus \cup_{\pi \in B}{\mathcal L}_\pi}$;
 \item Check if the language recognized by $\mathcal{A}_{\mathcal{C}}$ is finite.
\end{enumerate}
\end{enumerate}
\end{framed}
\caption{Our algorithm testing if the number of simple permutations in $\mathcal{C} = Av(B)$ is finite.}
\label{fig:algo}
\end{figure}
As can be seen in Theorem~\ref{thm:main_result} (p.\pageref{thm:main_result}), the resulting algorithm is efficient:
it is polynomial w.r.t.~the sizes of the patterns in $B$ and simply
exponential w.r.t.~their number, which is a significant improvement
to the first decidability procedure of~\cite{BRV06}. 
Notice that we described in~\cite{BBPR10} an algorithm solving the same problem on
substitution-closed permutation classes, that is to say the classes of permutations whose bases
contain only simple permutations.
The complexity of our algorithm in this special framework is
${\mathcal O}(n \log n)$ where $n$ is the sum of the size of the patterns in $B$. 

The article is organized as follows.
Section~\ref{sec:preliminaries} starts with a reminder of previous definitions and results about permutation patterns, decomposition trees and pin-permutations. 
It also recalls from~\cite{BRV06} the characterization of classes with a finite number of simple permutations, where proper pin-permutations enter into play. 
Section~\ref{sec:strict pw} establishes our criterion for deciding whether a permutation class contains a finite number of proper pin-permutations: 
this is the condition tested by the second step of the algorithm of Figure~\ref{fig:algo}.
Stating this criterion requires that we review the encoding of pin-permutations by pin words used by \cite{BRV06}
and that we go further into the interpretation of the pattern order between pin-permutations in terms of words and languages.
In Section~\ref{sec:algos_finite_number_proper_pinperm},
we describe and compare two methods for testing whether a class contains finitely many proper pin-permutations. 
We start with the procedure of~\cite{BRV06}, and proceed with our method.
Then we outline in Subsection~\ref{ssec:idees_construction_automates} the most technical part of our algorithm: 
building an automaton $\mathcal{A}_{\pi}$ associated to every pin-permutation $\pi$ of the basis of the class.
Details and proofs for this step can be found in Appendices.
Finally, Section~\ref{sec:polynomial} describes and gives the complexity of our whole algorithm to decide, given a finite basis $B$, 
whether the class $\mathcal{C} = Av(B)$ contains a finite number of simple permutations. 
To conclude, we put this result in the context of previous and future research in
Section~\ref{sec:ccl}.


\section{Preliminaries on permutations}
\label{sec:preliminaries}
We recall in this section a few definitions and results about permutation classes, 
substitution decomposition and decomposition trees, pin
representations and pin-permutations. 
We also recall the characterization of classes with finitely many simple permutations. 
More details can be found in~\cite{AA05,BBR09, BHV06b,BRV06}.

\subsection{Permutation classes and simple permutations}

The topic of this paper is to answer algorithmically the question of 
whether a \emph{permutation class} contains finitely many \emph{simple permutations}, 
thereby ensuring that the generating function of the class is algebraic~\cite{AA05}. 
We naturally start by the definitions of this terminology. 

\smallskip

A permutation $\sigma \in S_n$ is a bijective function from
$\{1,\ldots ,n\}$ onto $\{1,\ldots ,n\}$. We represent a permutation $\sigma$ 
either by the word $\sigma_1 \sigma_2 \ldots \sigma_n$ where
$\sigma_i = \sigma(i)$ for every $i \in \{1, \ldots, n\}$, or by its {\it diagram}
consisting in the set of points at coordinates $(i,\sigma_i)$ drawn in the plane.
Figure~\ref{fig:definitions} (p.\pageref{fig:definitions}) shows for example
the diagram of $\sigma=4 \,7 \,2 \,6 \,3 \,1 \,5$.

A permutation $\pi = \pi_1 \pi_2 \ldots \pi_k$ is a {\it pattern} of a
permutation $\sigma = \sigma_1 \sigma_2 \ldots \sigma_n$ and we write
$\pi \leq \sigma$ if and only if there exist $1 \leq i_1 < i_2 <
\ldots < i_k \leq n$ such that $\pi$ is isomorphic to
$\sigma_{i_1}\ldots \sigma_{i_k}$ 
(see an example in Figure~\ref{fig:definitions}). 
We also say that $\sigma$ \emph{involves} or \emph{contains} $\pi$. 
If $\pi$ is not a pattern of $\sigma$ we say that $\sigma$ {\it avoids} $\pi$.

Let $B$ be a finite or infinite antichain of permutations -- {\it i.e.}, a set of
permutations that are pairwise incomparable for $\leq$.  The permutation
class of {\it basis} $B$ denoted $Av(B)$ is the set of all
permutations avoiding every element of $B$.

The reader familiar with the permutation patterns literature 
will notice that we do not adopt the (equivalent) point of view of defining 
permutation classes as downward closed sets for $\leq$. 
Indeed, in this article,
permutation classes are always given by their bases. 
We will further restrict our attention to classes having finite bases, 
since otherwise from \cite{AA05}, they contain infinitely many simple permutations. 

\smallskip

A \emph{block} (or {\em interval}) of a permutation $\sigma$ of size
$n$ is a subset $\{i,\ldots ,(i+\ell-1)\}$ of consecutive integers of
$\{1,\ldots ,n\}$ whose images under $\sigma$ also form an interval of
$\{1,\ldots ,n\}$. A permutation $\sigma$ is \emph{simple} when it is
of size at least $4$ and it contains no block, except the trivial
ones: those of size $1$ (the singletons) or of size $n$ ($\sigma$
itself). The only permutations of size smaller than $4$ that have only
trivial blocks are $1$, $12$ and $21$, nevertheless they are
\emph{not} considered to be simple in this article.

\subsection{Substitution decomposition and decomposition trees}

Let $\sigma$ be a permutation of $S_k$ and
$\pi_{1},\ldots, \pi_{k}$ be $k$ permutations of
$S_{\ell_1}$, $\ldots$, $S_{\ell_k}$ respectively.  The {\em
  substitution} $\sigma[\pi_{1}, \pi_{2} ,\ldots, \pi_{k}]$ of
$\pi_{1},\pi_{2} , \ldots, \pi_{k}$ in $\sigma$ (also called
\emph{inflation} in~\cite{AA05}) is defined as the permutation whose diagram
is obtained from the one of $\sigma$ by replacing each
point $\sigma_i$ by a block containing the diagram of
$\pi_{i}$. 
Alternatively, 
$\sigma[\pi_{1}, \pi_{2} ,\ldots, \pi_{k}]$ is the permutation of size $\sum \ell_i$ 
which is obtained as the concatenation $p_1 p_2 \ldots p_k$ of sequences $p_i$ of integers 
such that each $p_i$ is isomorphic to $\pi_{i}$ 
and all integers in $p_i$ are smaller than those in $p_j$
as soon as $\sigma_i < \sigma_j$. 
For example $ 1\, 3\, 2 [2\, 1, 1\, 3\, 2, 1] = 2\, 1\, \, 4\, 6\, 5\,  \, 3$.

Permutations can be decomposed using substitutions, as described in
Theorem~\ref{thm:decomp_perm} below. For this purpose, we now introduce
some definitions and notations. For any $k \geq 2$,
let $I_k$ be the permutation $1\ 2 \ldots k$ and $D_k$ be $k\ (k-1)
\ldots 1$. Denote by $\oplus$ and $\ominus$ respectively $I_k$ and
$D_k$. Notice that in inflations of the form $\oplus[\pi_1, \pi_2,
\ldots, \pi_k] = I_k[\pi_1, \pi_2, \ldots, \pi_k]$ or $\ominus[\pi_1,
\pi_2, \ldots, \pi_k] = D_k[\pi_1, \pi_2, \ldots, \pi_k]$, the integer
$k$ is determined without ambiguity by the number of permutations
$\pi_i$ of the inflation.

\begin{defi}
  A permutation $\sigma$ is \emph{$\oplus$-decomposable}
  (resp. \emph{$\ominus$-decomposable}) if it can be written as
  $\oplus[\pi_1,\pi_2 ,\ldots, \pi_k]$ (resp. $\ominus[\pi_1,\pi_2
  ,\ldots, \pi_k]$), for some $k \geq 2$.
  Otherwise, it is \emph{$\oplus$-indecomposable}
  (resp. \emph{$\ominus$-indecomposable}).
\end{defi}

\begin{theo} 
For any $n\geq 2$, every permutation $\sigma \in S_n$ can be uniquely decomposed as either:
\begin{itemize}
\item $\oplus[\pi_1,\pi_2,\ldots,\pi_k]$, with $k \geq 2$ and
  $\pi_1,\pi_2,\ldots,\pi_k$ $\oplus$-indecomposable,
\item $\ominus[\pi_1,\pi_2,\ldots,\pi_k]$, with $k \geq 2$ and
  $\pi_1,\pi_2,\ldots,\pi_k$ $\ominus$-indecomposable,
  \item $\alpha[\pi_1,\ldots,\pi_k]$ with $\alpha$ a simple permutation and $k = |\alpha|$ (so that $k \geq 4$).
\end{itemize}
\label{thm:decomp_perm}
\end{theo}
Theorem~\ref{thm:decomp_perm} appears in~\cite{AA05} under a form that is trivially equivalent. 
The reader can also refer to~\cite{Gallai} for a historical reference, or to~\cite{Heber01findingall} for a reference in an algorithmic context. 

\begin{rem}
  Any block of $\sigma = \alpha[\pi_1,\ldots,\pi_k]$ (with $\alpha$ a
  simple permutation) is either $\sigma$ itself, or is included in one
  of the $\pi_i$.
\label{fact:blockintopii}
\end{rem}

Theorem~\ref{thm:decomp_perm} can be applied recursively on each
$\pi_i$ leading to a complete decomposition where each permutation is
either $I_k,D_k$ (denoted by $\oplus, \ominus$ respectively) or a
simple permutation. This complete decomposition is called the
\emph{substitution decomposition} of a permutation. It is accounted for by
a tree, called the \emph{substitution decomposition tree},
where a substitution $\alpha[\pi_1,\ldots,\pi_k]$ is represented by a node
labeled $\alpha$ with $k$ ordered children representing the $\pi_i$.

\begin{defi}
\label{defn:deccompositionTrees}
The substitution decomposition tree $T$ of the permutation $\sigma$ is
the unique labeled ordered tree encoding the substitution
decomposition of $\sigma$, where each internal node is either labeled
by $\oplus,\ominus$ -- those nodes are called {\em linear} -- or by a
simple permutation $\alpha$ -- {\em prime} nodes. Each node labeled
by $\alpha$ has $|\alpha|$ children. 
See Figure~\ref{fig:decompTree} for an example. 
\end{defi}


Notice that in substitution decomposition trees, there are no edges
between two nodes labeled by $\oplus$, nor between two nodes labeled
by $\ominus$, since the $\pi_i$ are $\oplus$-indecomposable
(resp. $\ominus$-indecomposable) in the first (resp. second) item of
Theorem~\ref{thm:decomp_perm}.

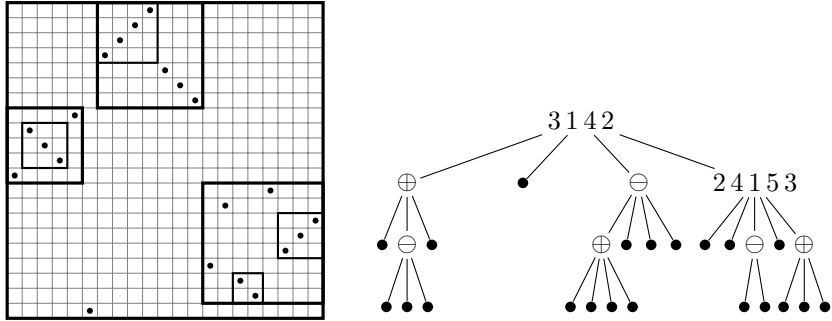
\begin{figure}[ht]
\begin{center}
\begin{tikzpicture}[level distance=15mm]
\begin{scope}[scale=.2,xshift=-700,yshift=-140]
\draw [help lines] (0,0) grid (21,21);
\pin{0.5}{9.5};
\pin{1.5}{12.5};
\pin{2.5}{11.5};
\pin{3.5}{10.5};
\pin{4.5}{13.5};
\pin{5.5}{0.5};
\pin{6.5}{17.5};
\pin{7.5}{18.5};
\pin{8.5}{19.5};
\pin{9.5}{20.5};
\pin{10.5}{16.5};
\pin{11.5}{15.5};
\pin{12.5}{14.5};
\pin{13.5}{3.5};
\pin{14.5}{7.5};
\pin{15.5}{2.5};
\pin{16.5}{1.5};
\pin{17.5}{8.5};
\pin{18.5}{4.5};
\pin{19.5}{5.5};
\pin{20.5}{6.5};

\draw (0,0) [very thick] rectangle (21,21);
\draw (0,9) [very thick] rectangle (5,14);
\draw (1,10) [thick] rectangle (4,13);
\draw (6,14) [very thick] rectangle (13,21);
\draw (6,17) [thick] rectangle (10,21);
\draw (13,1) [very thick] rectangle (21,9);
\draw (15,1) [thick] rectangle (17,3);
\draw (18,4) [thick] rectangle (21,7);
\end{scope} 
\begin{scope}[scale=0.55,xshift=140,yshift=85]
\tikzstyle{level 1}=[sibling distance=28mm]
\tikzstyle{level 2}=[sibling distance=6mm]
\tikzstyle{level 3}=[sibling distance=5mm]

 \node[simple]{$3\,1\,4\,2 $}
   child {node[linear] {$\oplus$}
     child {node[leaf] {}}
     child {node[linear] {$\ominus$}
         child {node[leaf] {}}
         child {node[leaf] {}}
         child {node[leaf] {}}
         }
     child {node[leaf] {}}
   }
   child {node[leaf] {}}
   child {node[linear] {$\ominus$}
     child {node[linear] {$\oplus$}
         child {node[leaf] {}}
         child {node[leaf] {}}
         child {node[leaf] {}}
         child {node[leaf] {}}
         }
     child {node[leaf] {}}
     child {node[leaf] {}}
     child {node[leaf] {}}
   }
   child {node[simple] {$2\,4\,1\,5\,3$}
     child {node[leaf] {}}
     child {node[leaf] {}}
     child {node[linear] {$\ominus$}
         child {node[leaf] {}}
         child {node[leaf] {}}
     }
     child {node[leaf] {}}
     child {node[linear] {$\oplus$}
         child {node[leaf] {}}
         child {node[leaf] {}}
         child {node[leaf] {}}
     }
   };
 \end{scope}
\end{tikzpicture}
\caption{The diagram and the substitution decomposition tree
  $T$ of the permutation $ \sigma = 10\,
  13\, 12\, 11\, 14\, 1\, 18\, 19\, 20\, 21\, 17\, 16\, 15\, 4\, 8\,
  3\, 2\, 9\, 5\, 6\, 7$. The internal nodes of $T$ correspond to the blocks of $\sigma$
  marked by rectangles. \label{fig:decompTree}}
\end{center}
\end{figure}


\begin{rem}
Permutations are bijectively characterized by their substitution
decomposition trees.
\label{thm:decompTree}
\end{rem}

In the sequel, when writing \emph{a child of a node $V$} we mean
the permutation corresponding to the subtree rooted at this child of
node $V$.

\subsection{Pin-permutations and pin representations}

In this article, the \emph{pin-permutations} (and their decomposition trees) play a central role. 
The remaining of this preliminary section recalls their definition 
and explains how they are related to our problem of testing whether a permutation class contains finitely many simple permutations. 


A \emph{pin} is a point in the plane. A pin $p$ {\em
  separates} -- horizontally or vertically -- the set of pins $P$ from
the set of pins $Q$ if and only if a horizontal -- resp. vertical --
line drawn across $p$ separates the plane into two parts, one
containing $P$ and the other one containing $Q$.
The \emph{bounding box} (also known as the \emph{rectangular hull}) 
of a set of points $P$ is the smallest axis-parallel
rectangle containing the set $P$. A \emph{pin sequence} is a
sequence $(p_1,\ldots,p_k)$ of pins in the plane such that no two
points are horizontally or vertically aligned
and for all $i \geq 2$, $p_i$
lies outside the bounding box of $\{p_1,\ldots ,p_{i-1}\}$ and
satisfies one of the following conditions:
\begin{itemize}
\item {\it separation condition}: $p_i$ separates $p_{i-1}$ from $\{p_1,\ldots,p_{i-2}\}$;
 \item {\it independence condition}: $p_i$ is independent from $\{p_1,\ldots,p_{i-1}\}$, {\it i.e.},
 it does not separate this set into two non-empty sets.
\end{itemize}

A pin sequence represents a permutation $\sigma$ if and only if it is
isomorphic to its diagram. We say that a permutation $\sigma$ is
a \emph{pin-permutation} if it can be represented by a pin sequence,
which is then called a \emph{pin representation} of $\sigma$ (see
Figure~\ref{fig:definitions}). Not all permutations are
pin-permutations (see for example the permutation $\sigma$ of
Figure~\ref{fig:definitions}).

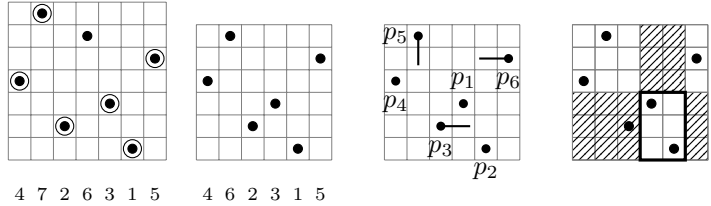
\begin{figure}[htbp]
\begin{center}
\begin{tikzpicture}
 \begin{scope}[scale=.3]
  \permutationNomGros{4,7,2,6,3,1,5}
\draw (1.5,4.5) circle (.4 cm);
\draw (2.5,7.5) circle (.4 cm);
\draw (3.5,2.5) circle (.4 cm);
\draw (5.5,3.5) circle (.4 cm);
\draw (6.5,1.5) circle (.4 cm);
\draw (7.5,5.5) circle (.4 cm);
 \end{scope}
\hspace{0.3 cm}
\begin{scope}[xshift=2.5cm,scale=.3]
\draw [help lines] (1,1) grid (7,7);
\draw (1.5,4.5)  [fill] circle (.2);
\draw (0.5,-0.5) node {\scriptsize 4};
\draw (2.5,6.5)  [fill] circle (.2);
\draw (1.5,-0.5) node {\scriptsize 6};
\draw (3.5,2.5)  [fill] circle (.2);
\draw (2.5,-0.5) node {\scriptsize 2};
\draw (4.5,3.5)  [fill] circle (.2);
\draw (3.5,-0.5) node {\scriptsize 3};
\draw (5.5,1.5)  [fill] circle (.2);
\draw (4.5,-0.5) node {\scriptsize 1};
\draw (6.5,5.5)  [fill] circle (.2);
\draw (5.5,-0.5) node {\scriptsize 5};
\end{scope}
\begin{scope}[xshift=5cm,scale=.3]
\draw (1,1) [help lines] grid +(6,6);
\pin{4.5}{3.5}
\nameUnder{4.5}{5.5}{$p_1$}
\pin{5.5}{1.5}
\nameUnder{5.5}{1.5}{$p_2$}
\pinL{3.5}{2.5}
\nameUnder{3.5}{2.5}{$p_3$}
\pin{1.5}{4.5}
\nameUnder{1.5}{4.5}{$p_4$}
\pinU{2.5}{6.5}
\nameUnder{1.5}{7.5}{$p_5$}
\pinR{6.5}{5.5}
\nameUnder{6.5}{5.5}{$p_6$}
\end{scope}
\begin{scope}[xshift=7.5cm,scale=.3]
\draw (1,1) [help lines] grid +(6,6);
\permutation{4,6,2,3,1,5}
\draw [gray, pattern=north east lines] (1,1) rectangle (4,4);
\draw [gray, pattern=north east lines] (4,4) rectangle (6,7);
\draw [gray, pattern=north east lines] (6,1) rectangle (7,4);
\draw[very thick] (4,1) rectangle (6,4);
\end{scope}
\end{tikzpicture}
\caption{The permutation $\sigma=4\,7\,2\,6\,3\,1\,5$, its pattern
  $\pi = 4\,6\,2\,3\,1\,5$, a pin representation $p$ of $\pi$,
  and the bounding box of $\{p_1,p_2\}$ with its sides shaded.}
\label{fig:definitions}
\end{center}
\end{figure}

Lemma 2.17 of~\cite{BBR09} is used several times in our proofs, and we state it here:
\begin{lem} \label{lem:[7]2.17}
Let $(p_1, \ldots, p_n)$ be a pin representation of $\sigma \in S_n$. Then for each
$i \in \{2, \ldots, n-1\}$, if there exists a point $x$ of $\sigma$ on the sides of the bounding box of $\{p_1, \ldots, p_i\}$,
then it is unique and $x = p_{i+1}$.
\end{lem}

A {\it proper} pin representation is a pin representation in which
every pin $p_i$, for $i\geq 3$, separates $p_{i-1}$ from $\{p_1,
\ldots, p_{i-2}\}$. A {\it proper} pin-permutation is a permutation
that admits a proper pin representation.

\begin{rem}\label{rem:simplepin}
  A pin representation of a simple pin-permutation is always proper as
  any independent pin $p_i$ with $i\geq 3$ creates a block
  corresponding to $\{p_1, \ldots, p_{i-1}\}$.
\end{rem}

%

\subsection{Characterization of classes with finitely many simple permutations} \label{ssec:procedure_of_BRV}

In~\cite{BRV06}, Brignall {\it et al.} provide a criterion characterizing when a class
contains a finite number of simple permutations. 
They show that it is equivalent to the class containing a finite number of permutations of three simpler kinds, which they define. 
Among the three new kinds of permutations that they introduce are the proper pin-permutations that we have already seen, but also 
the \emph{parallel alternations} and the \emph{wedge simple permutations}. 
The definition of these families of permutations is not crucial to our work, 
hence we refer the reader to~\cite{BRV06} for more details, and to Figure~\ref{fig:alternations} for examples.

\begin{theo}\cite{BHV06b,BRV06}\label{thm:brignall}
 A permutation class $Av(B)$ contains a finite number of simple
 permutations if and only if it contains:
\begin{itemize}
\item a finite number of wedge simple permutations, and
\item a finite number of parallel alternations, and
\item a finite number of proper pin-permutations.
\end{itemize}
\end{theo}

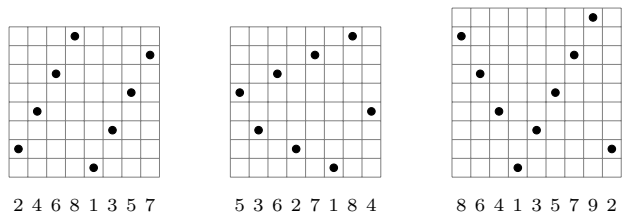
\begin{figure}[ht]
\begin{center}
\begin{tikzpicture}[scale=0.25]
\permutationNomGros{2,4,6,8,1,3,5,7}
\end{tikzpicture} \qquad 
\begin{tikzpicture}[scale=0.25]
\permutationNomGros{5,3,6,2,7,1,8,4}
\end{tikzpicture} \qquad 
\begin{tikzpicture}[scale=0.25]
\permutationNomGros{8,6,4,1,3,5,7,9,2}
\end{tikzpicture}
\end{center}
\caption{From left to right: a parallel alternation, and two wedge simple permutations (of type $1$ and $2$ respectively).}
\label{fig:alternations}
\end{figure}

Notice also that in Theorem~\ref{thm:brignall} above, the proper pin sequences of~\cite{BRV06} have been replaced by proper pin-permutations. 
But containing a finite number of proper pin-permutations is equivalent to containing a finite number of proper pin sequences. 
Indeed, the encoding of proper pin-permutations by proper pin sequences provides a finite-to-one correspondence. 
%
%

Whereas the exact definitions of the wedge simple permutations and the parallel
alternations have been omitted here, 
it is however essential for our purpose to be able to test whether a class given by a finite basis
contains a finite number of parallel alternations and wedge simple
permutations. 
Parallel alternations and wedge simple permutations, that can be of type $1$ or
$2$, are well characterized in~\cite{BRV06}.
This characterization leads to the following lemmas:

\begin{lem}\cite{BRV06}\label{lem:alternation}
The permutation class $Av(B)$ contains only finitely many parallel
alternations if and only if $B$ contains an element of every symmetry of
the class $Av(123, 2413,$ $3412)$.
\end{lem}

\begin{lem}\cite{BRV06}\label{lem:wedge1}
The permutation class $Av(B)$ contains only finitely many wedge simple
permutations of type $1$ if and only if $B$ contains an element of every
symmetry of the class $Av(1243, 1324, 1423, 1432, 2431, 3124, 4123,
4132, 4231, 4312)$.
\end{lem}

\begin{lem}\cite{BRV06}\label{lem:wedge2}
 The permutation class $Av(B)$ contains only finitely many wedge simple
 permutations of type $2$ if and only if $B$ contains an element of every
 symmetry of the class $Av(2134, 2143, 3124, 3142, 3241, 3412, 4123,
 4132, 4231, 4312)$.
\end{lem}

Using these lemmas together with a result of~\cite{AAAH01} we have:

\begin{lem}
\label{lem:complexity_nlogn}
Testing whether a finitely based class $Av(B)$ contains a finite number of wedge simple permutations and parallel alternations
can be done in $\mathcal{O}(n \log n)$ time, where $n = \sum_{\pi \in B} |\pi|$.
\end{lem}

\begin{proof}
From Lemmas~\ref{lem:alternation} to~\ref{lem:wedge2}, deciding if $Av(B)$ contains
a finite number of wedge simple permutations and parallel alternations is equivalent to checking
if elements of its basis $B$ involve patterns of size at most $4$.
From~\cite{AAAH01} checking whether a permutation $\pi$ involves a fixed set of patterns of size at most $4$
can be done in ${\mathcal O}(|\pi| \log |\pi|)$.
As we check for each permutation of $B$ the involvement of fixed sets of permutations of size at most $4$,
this leads to a ${\mathcal O}(n \log n)$ algorithm for deciding whether the number of parallel alternations
and of wedge simple permutations in the class is finite.
\end{proof}

  In~\cite{BRV06} Brignall {\it et al.} also proved that it is decidable whether
  $\mathcal{C} = Av(B)$ contains a finite number of proper
  pin-permutations. 
  Their proof heavily relies on an encoding of proper pin-permutations 
  by words over a finite alphabet (called \emph{pin words}), and on language theoretic arguments. 
  In the next section, we review and further develop the theory of pin words. 
  Then, in Section~\ref{sec:algos_finite_number_proper_pinperm} we will review the decision procedure of~\cite{BRV06}, 
  and explain how we could modify it into an efficient algorithm. 

\section{Characterization of classes with finitely many proper pin-permutations}\label{sec:strict pw}

Our goal in this section is to provide a criterion (that can be tested algorithmically, in the next sections) 
for a permutation class $\mathcal{C} = Av(B)$ given by its finite basis $B$ 
to contain finitely many proper pin-permutations. 
The encoding of pin-permutations by their pin words -- to be reviewed in Subsection~\ref{ssec:pinwords} -- has an essential property that can be used in establishing such a criterion: 
it allows to interpret the pattern order $\leq$ on pin-permutations as an order relation $\preccurlyeq$ on their pin words. 

This property is already at the core of~\cite{BRV06}, and we recall it below as Lemma~\ref{csq ordre}. 
In~\cite{BRV06}, it is used to derive a first criterion on $\mathcal{C}$ to contain a finite number of proper pin-permutations: 
\begin{quote}
$\mathcal{C}$ contains finitely many proper pin-permutations if and only if 
the language $\SP \setminus \bigcup \{\text{strict pin word }w \mid u \preccurlyeq w \}$ is finite, 
where the union is taken over all pin words $u$ that encode a permutation $\pi \in B$ 
and $\SP$ denote the language of all strict pin words (see Definition~\ref{def:strict_pin_words}). 
\end{quote}
This criterion may then be decided using automata theory, as explained in~\cite{BRV06} and reviewed in Subsection~\ref{ssec:procedure_of_BRV_2}. 

In what follows, we go further into the encoding of pin-permutations by words, 
and into the interpretation of the pattern order $\leq$ in terms of words and languages. 
This allows us to associate a language ${\mathcal L}_{\pi}$ to every pin-permutation $\pi$ 
in such a way that if $\pi \leq \sigma$ then ${\mathcal L}_{\sigma} \subseteq {\mathcal L}_{\pi}$. 
Subsequently, these languages ${\mathcal L}_{\pi}$ can be used to characterize
when $\mathcal{C}$ contains a finite number of proper pin-permutations -- see Theorem~\ref{thm:nbInfiniEtLangages}: 
\begin{quote}
$\mathcal{C}$ contains finitely many proper pin-permutations if and only if 
the language $\M \setminus \bigcup{\mathcal L}_\pi$ is finite, 
where the union is taken over all pin-permutations $\pi \in B$ 
and $\M$ denotes the set of words on the alphabet $\{L,R, U, D\}$ with no factor in $\{UU, UD, DU, DD, LL, LR, RL,RR\}$ (see p.\pageref{defi:m}).
\end{quote}
As we shall see in Section~\ref{sec:algos_finite_number_proper_pinperm}, this new criterion can be tested by an algorithm, 
far more efficiently than the first criterion above. 

\subsection{Pin words}
\label{ssec:pinwords}
Pin representations can be encoded on the alphabet
$\{1,2,3,4,U,D,L,R\}$  by words called {\it pin words}. Consider a pin
representation $(p_1,\ldots,p_n)$ and choose an origin $p_0$ in the
plane such that $(p_0,p_1,\ldots,p_n)$ is a pin sequence. Then every
pin $p_1,\ldots,p_n$ is encoded by a letter according to the following
rules:
\begin{itemize}
\item The letter associated with $p_i$ is $U$ -- resp. $D,L,R$ -- if
  $p_i$ separates $p_{i-1}$ and $\{p_0,p_1,\ldots,p_{i-2}\}$
  from the top -- resp. bottom, left, right.
\item The letter associated with $p_i$ is $1$ -- resp. $2,3,4$ -- if
  $p_i$ is independent from $\{p_0,p_1,\ldots,p_{i-1}\}$ and
  is located in the up-right -- resp. up-left, bottom-left,
  bottom-right -- corner of the bounding box of
  $\{p_0,p_1,\ldots,p_{i-1}\}$.
\end{itemize}
This encoding is summarized by Figure~\ref{fig:quadrant}.  The region
encoded by $1$ is called the first {\em quadrant} with respect to the box
\tikz\fill (0,0) rectangle (0.3,0.3);. The same goes for
$2,3,4$.  The letters $U,D,L,R$ are called {\em directions}, while
$1,2,3$ and $4$ are {\em numerals}.

\begin{figure}[htbp]
\hspace{.02\linewidth}
  \begin{minipage}[t]{.3\linewidth}
    \begin{center}
      \begin{tikzpicture}[scale=.7]
        \useasboundingbox (0,0) rectangle (4,3);
        \draw[help lines] (0,0) grid
        +(3,3); \fill (1,1) rectangle +(1,1); \draw (0.5,0.5) node
        {3}; \draw (1.5,0.5) node {D}; \draw (2.5,0.5) node {4}; \draw
        (2.5,1.5) node {R}; \draw (2.5,2.5) node {1}; \draw (1.5,2.5)
        node {U}; \draw (0.5,2.5) node {2}; \draw (0.5,1.5) node {L};
      \end{tikzpicture}
      \caption{Encoding \newline of pins by letters.}
      \label{fig:quadrant}
    \end{center}
  \end{minipage}
  \begin{minipage}[t]{.66\linewidth}
    \begin{center}
      \begin{tikzpicture}[scale=.37]
      \useasboundingbox (0,0) rectangle (6,6);
        \draw (2,2) [fill] circle (.2); \draw
        (4,4) [fill] circle (.2); \draw (1.5,1.5) node
        {\begin{small}$p_1$\end{small}}; \draw (4.5,4.5)
        node {\begin{small}$p_2$\end{small}}; \draw (1,1) node
        {\begin{small}$11$\end{small}}; \draw (1,3) node
        {\begin{small}$41$\end{small}}; \draw (1,5) node
        {\begin{small}$4R$\end{small}}; \draw (3,1) node
        {\begin{small}$21$\end{small}}; \draw (3,3) node
        {\begin{small}$31$\end{small}}; \draw (3,5) node
        {\begin{small}$3R$\end{small}}; \draw (5,1) node
        {\begin{small}$2U$\end{small}}; \draw (5,3) node
        {\begin{small}$3U$\end{small}}; \draw[thick] (0,2) -- (6,2);
        \draw[thick] (0,4) -- (6,4); \draw[thick] (2,0) -- (2,6);
        \draw[thick] (4,0) -- (4,6);
      \end{tikzpicture}
      \caption{The two letters in each cell indicate the first two
        letters of the pin word encoding $(p_1, \ldots, p_n)$ when $p_0$ is
        taken in this cell.}
      \label{fig:origine}
    \end{center}
  \end{minipage}
\hspace{.02\linewidth}
\end{figure}

\begin{example} \label{ex:encoding_p0}
$14L2UR$ (if $p_{0}$ is between $p_{3}$ and $p_{1}$) and 
$3DL2UR$ (if $p_0$ is horizontally between $p_1$ and $p_4$ and vertically between $p_2$ and $p_6$)
are pin words corresponding to the pin representation of $\pi = 4\,6\,2\,3\,1\,5$ shown in Figure~\ref{fig:definitions} (p.\pageref{fig:definitions}). 
\end{example}

Example~\ref{ex:encoding_p0} shows in particular that
several pin words encode the same pin representation,
depending on the choice of the origin $p_0$.
We may actually describe the number of these pin words:

\begin{rem}\label{rem:nb_pin_words}
Because of the choice of the origin $p_0$,
each pin-permutation of size greater than $1$ has at least $6$ pin words.
More precisely each pin representation $p$ is encoded by $6$
pin words if $p_3$ is a separating pin and $8$ pin words otherwise (see Figure~\ref{fig:origine}).
Indeed, once a pin representation $p$ is fixed,
the letters encoding $p_i$ for $i\geq 3$ in a pin word encoding $p$ are uniquely determined.
\end{rem}

Conversely, pin words indeed encode pin-permutations
since to each pin word corresponds a unique pin representation, hence a
unique permutation.

\smallskip

\begin{rem}\label{rem:UU,UD...}
The definition of pin sequences implies that pin words do not contain
any of the factors $UU, UD, DU, DD, LL, LR, RL$ and $RR$.
\end{rem}

\begin{defi}\label{def:strict_pin_words}
A {\em strict} (resp. {\em quasi-strict}) pin word is a pin word that
begins with a numeral (resp. two numerals) followed only by
directions. 
We denote by $\SP$ the set of all strict pin words. 
\end{defi}

\begin{rem}[Proper pin representations, strict and quasi-strict pin words] \label{rem:proper_strict}
~
Every pin word encoding a proper pin representation
is either strict or quasi-strict.
Conversely if a pin word is strict or quasi-strict,
then the pin representation it encodes is proper.
Finally a pin-permuta\-tion is proper if and only if it admits a strict pin word.
\end{rem}

\subsection{Pattern containment and piecewise factor relation}

Recall the definition of the partial order $\preccurlyeq $ on pin
words introduced in~\cite{BRV06}.

\begin{defi}
  Let $u$ and $w$ be two pin words.  We decompose $u$ in terms of its
  strong numeral-led factors as $u = u^{(1)} \ldots u^{(j)}$, {\em a
    strong numeral-led factor} being a strict pin word. We then write
  $u \preccurlyeq w$ if $w$ can be chopped into a sequence of factors
  $w=v^{(1)}w^{(1)} \ldots v^{(j)}w^{(j)}v^{(j+1)}$ such that for all
  $i \in \{1,\ldots, j\}$:
\begin{itemize}
\item if $w^{(i)}$ begins with a numeral then $w^{(i)} = u^{(i)}$, and
\item if $w^{(i)}$ begins with a direction, then $v^{(i)}$ is
  non-empty, the first letter of $w^{(i)}$ corresponds to a point lying
  in the quadrant -- w.r.t.~the origin of the encoding $w$ --
  specified by the first letter of $u^{(i)}$, and all other letters in
  $u^{(i)}$ and $w^{(i)}$ agree.
\end{itemize}
\label{def:preceq}
\label{defn:snl}
\end{defi}

\begin{example}
The strong numeral-led factor decomposition of $u=14L2UR$ is $u=1 \cdot 4L \cdot 2UR$.
Moreover, $u \preccurlyeq w = 2RU4LULURD4L$, because $w$ may be decomposed as $w= 2R\mathbf{U} \cdot \mathbf{4L} \cdot U\mathbf{LUR} \cdot D4L $,
where the factors $w^{(i)}$ satisfying the conditions of Definition~\ref{def:preceq} are emphasized by bold letters.
\end{example}

As we mentioned already, the essential property of this order 
is that it is closely related to the pattern containment order $\leq$ on permutations.
\begin{lem}\cite{BRV06}\label{csq ordre}
If the pin word $w$ encodes the permutation $\sigma$ and $\pi
\leq \sigma$ then there is a pin word $u$ encoding $\pi$ with
$u \preccurlyeq w$.  Conversely if $u \preccurlyeq w$ then the permutation
corresponding to $u$ is contained in the one corresponding to
$w$.
\end{lem}

The relation $u \preccurlyeq w$ on pin words is nearly a piecewise factor relation,
the factors being determined by the strong numeral-led factors of $u$. 
Our purpose in the following is to adapt the relation $\preccurlyeq$ into an actual piecewise factor relation. 
This is achieved in Theorem~\ref{thm:ordreMots}, using a further encoding of pin words that we introduced in~\cite{BBPR10} and recall hereafter. 

\medskip

Recall that $\SP$ denotes the set of strict pin words, and 
let $\SP_{\geq 2}= \SP \setminus \{1,2,3,4\}$ be the set of strict pin words of length at least $2$.  
Further denote by $\M$ \label{defi:m} (resp. $\M_{\geq 3}$) the set of words of length at
least $2$ (resp. at least $3$) over the alphabet $\{L,R, U, D\}$ such that ${L,R}$ is
followed by ${U,D}$ and conversely.
We define below a bijection that sends strict pin words to words of $\M$.
It consists of replacing the only numeral in a strict pin word by two directions.
Intuitively, given a numeral $q$ and a box \tikz\fill (0,0) rectangle (0.3,0.3);,
inserting two pins in the two directions prescribed by the bijection
ends up in a pin lying in quadrant $q$ with respect to the box \tikz\fill (0,0) rectangle (0.3,0.3);.

\begin{defi}\label{def:phi}
We define a bijection $\phi$ from $\SP_{\geq 2}$ to $\M_{\geq 3}$ as follows. For any
strict pin word $u \in \SP_{\geq 2}$ such that $u=u' u''$ with $|u'|=2$,
we set $\phi (u) = \varphi(u') u''$ where $\varphi$ is given by:
\begin{center}
\begin{tabular}{|c||c||c||c|}
\hline
$1R\mapsto RUR$ & $2R\mapsto LUR$   & $3R\mapsto LDR$   & $4R\mapsto RDR$ \\
 $1L\mapsto RUL$ & $2L\mapsto LUL$ & $3L\mapsto LDL$  & $4L\mapsto RDL$\\
 $1U\mapsto URU$ & $2U\mapsto ULU$ & $3U\mapsto DLU$ &  $4U\mapsto DRU$ \\
$1D\mapsto URD$ & $2D\mapsto ULD$ & $3D\mapsto DLD$ & $4D\mapsto DRD$ \\
\hline
\end{tabular}
\end{center}
\end{defi}

For any $n \geq 2$, the map $\phi$ is a bijection from the set
$\SP_{n}$ of strict pin words of length $n$ to the set ${\mathcal
  M}_{n+1}$ of words of $\M$ of length $n+1$.  Furthermore,
it satisfies, for any $u = u_1u_2\ldots \in \SP_{\geq 2}$, $u_{i}=\phi (u)_{i+1}$ for
any $i\geq 2$.

In the above table, we can notice that, for any $u \in \SP_{\geq 2}$, the first
two letters of $\phi (u)$ are sufficient to determine the first letter
of $u$ (which is a numeral). Thus it is natural to extend the
definition of $\phi$ to $\SP$ by setting for words of length $1$: $ \phi(1) =
\{UR,RU\}, \phi(2) = \{UL,LU\}, \phi(3) = \{DL,LD\}$ and $\phi(4) =
\{RD,DR\}$, and by defining consistently $\phi^{-1}(v) \in
\{1,2,3,4\}$ for any $v$ in $\{LU,LD,RU,RD,UL,UR,DL,DR\}$.

Lemma~\ref{lem:quadrant} below shows that for
each pin word $w$, we know in which quadrant (w.r.t.~the origin of the encoding) lies every pin of
the pin representation $p$ corresponding to $w$. More precisely for each $i \leq |w|$, knowing only $w_i$ and $w_{i-1}$, we can determine in which quadrant $p_i$ lies.

\begin{lem} \label{lem:quadrant}
  Let $w$ be a pin word and $p$ be the pin representation
  corresponding to $w$. For any $i\geq 2$, the numeral indicating 
  the quadrant in which $p_{i}$ lies with respect to 
  $\{p_0, \ldots, p_{i-2}\}$ is
$\begin{cases}
  w_i \textrm{ if } w_i \textrm{ is a numeral; }\\
  \phi^{-1}(w_{i-1}w_i) \textrm{ if } w_{i-1} \textrm{ and } w_i \textrm{ are directions; }\\
  \phi^{-1}(BC) \textrm{ otherwise, with } \phi(w_{i-1}w_i) = ABC. \end{cases}
$
\end{lem}

Notice that in the third case $w_{i-1}$ is a numeral and $w_i$ is a direction;
consequently, $ABC \in {\mathcal M}_{3}$ is given by the table of Definition~\ref{def:phi}.

\begin{proof}
Similar to the proof of Lemma 3.4 of~\cite{BBPR10}, adapted to the case where
$w$ is any pin word, \emph{i.e.}, is not necessarily strict.
\end{proof}

Lemma~\ref{lem:quadrant} is used in the proofs of Lemma~\ref{prop phi} and Theorem~\ref{thm:ordreMots}.
Their statement also requires that we extend some definitions from $\SP$ to $\M$.

\begin{rem} \label{rem:EtendrePourM}
Words of $\M$ may also be seen as encodings of pin sequences
(as in Subsection~\ref{ssec:pinwords}),
taking the origin $p_0$ to be a box instead of a point.
Moreover, the relation $u \preccurlyeq w$ can be extended to $w \in \M$,
and the map $\phi$ can be defined on words of $\M$ as the identity map
(although this extension of $\phi$ to the domain $\SP \cup \M$ is not a bijection anymore).
\end{rem}

By definition, strong numeral-led factors of any pin word $u$ are
strict pin words. Therefore we first study how the relation $u \preccurlyeq
w$ is mapped on $\phi(u), \phi(w)$ when $u$ is a strict pin word.

\begin{lem} \label{prop phi}
Let $u$ be a strict pin word and $w$ be a word of $\SP \cup \M$.
If $|u| \geq 2$ then $u \preccurlyeq w$ if and only if $\phi(u)$ is a factor of $\phi(w)$.
If $|u| = 1$ then $u \preccurlyeq w$ if and only $\phi(w)$ has a factor in $\phi(u)$.
\end{lem}

\begin{proof}
Note that if $|u| \geq 2$ then $\phi(u)$ is a word but if $|u|=1$ then $\phi(u)$ is a set of two words. 
The case where $|u| \geq 2$ and $w$ is a strict pin word corresponds exactly to Lemma~3.5 of~\cite{BBPR10}. 
Other cases are proved in a similar way, using, instead of Lemma~3.4 of~\cite{BBPR10}, 
its generalization provided by our current Lemma~\ref{lem:quadrant}.
\end{proof}

In the statement of Lemma~\ref{prop phi}, we have distinguished the cases $|u| \geq 2$ and $|u| = 1$
since $\phi(u)$ is a word or a set of two words in these respective cases.
However, to avoid such uselessly heavy statements,
we do not make this distinction in the sequel, and we write indifferently
``$\phi(u)$ is a factor of $w$'' or ``$w$ has a factor in $\phi(u)$'' meaning that
$\begin{cases}
\text{if } |u| = 1, w \text{ has a factor in } \phi(u) \\
\text{if } |u| \geq 2, \phi(u) \text{ is a factor of } w.
\end{cases}$

When the pin word $u$ is not strict, Lemma~\ref{prop phi} can be
extended formalizing the idea of piecewise factors mentioned at the
beginning of this section.

\begin{theo}\label{thm:ordreMots}
Let $u$ and $w$ be two pin words and $u = u^{(1)} \ldots u^{(j)}$ be the strong numeral-led factors decomposition of $u$.
Then $u \preccurlyeq w$ if and only if $w$ can be chopped into a sequence of factors $w=v^{(1)}w^{(1)} \ldots v^{(j)}w^{(j)}v^{(j+1)}$ such that for all $i \in \{1,\ldots, j\}$, $w^{(i)} \in \SP \cup \M$ and $\phi(w^{(i)})$ has a factor in $\phi(u^{(i)})$. 
\end{theo}

\begin{proof}
We prove that $u \preccurlyeq w$ if and only if $w$ can be chopped into a sequence of factors $w=v^{(1)}w^{(1)} \ldots v^{(j)}w^{(j)}v^{(j+1)}$ such that for all $i \in \{1,\ldots, j\}$, $w^{(i)} \in \SP \cup \M$ and $u^{(i)} \preccurlyeq w^{(i)}$.
Then the result follows using Lemma~\ref{prop phi}.

If $u \preccurlyeq w$, then $w$ can be chopped into $w=\bar{v}^{(1)}\bar{w}^{(1)} \ldots \bar{v}^{(j)}\bar{w}^{(j)}\bar{v}^{(j+1)}$ as in Definition~\ref{def:preceq}.
We set $w^{(i)}= \bar{w}^{(i)}$ if $\bar{w}^{(i)}$ begins with a numeral, and we take $w^{(i)}$ to be the suffix of $\bar{v}^{(i)}\bar{w}^{(i)}$ of length $|\bar{w}^{(i)}| + 1$ otherwise.
Then for all $i \in \{1,\ldots, j\}$, $w^{(i)} \in \SP \cup \M$ and we have $u^{(i)} \preccurlyeq w^{(i)}$.
Indeed, if $\bar{w}^{(i)}$ begins with a direction, 
by Lemma~\ref{lem:quadrant} the point corresponding to the first letter of $\bar{w}^{(i)}$ lies in the quadrant determined by the last letter of $\bar{v}^{(i)}$ and the first letter of $\bar{w}^{(i)}$
(w.r.t.~the origin of the encoding $w$ and also 
of the encoding $w^{(i)}$).

Conversely if $w$ can be chopped into $w=v^{(1)}w^{(1)} \ldots v^{(j)}w^{(j)}v^{(j+1)}$ such that for all $i \in \{1,\ldots, j\}$, $u^{(i)} \preccurlyeq w^{(i)}$ then from Definition~\ref{def:preceq} we can decompose $w^{(i)}$ as $y^{(i)} \bar{w}^{(i)}z^{(i)}$ and thanks to Lemma~\ref{lem:quadrant} it is sufficient to set $\bar{v}^{(i)} = z^{(i-1)} v^{(i)} y^{(i)}$ to have $w=\bar{v}^{(1)}\bar{w}^{(1)} \ldots \bar{v}^{(j)}\bar{w}^{(j)}\bar{v}^{(j+1)}$ as in Definition~\ref{def:preceq}.
\end{proof}

\subsection{Pattern containment and set inclusion}

Recall that our goal is to characterize when there are finitely many proper pin-permutations in a class. 
By Remark~\ref{rem:proper_strict}, for any proper pin-permutation $\sigma$,
there exists a \textbf{strict} pin word $w_{\sigma}$ that encodes $\sigma$. 
And from Lemma~\ref{csq ordre},
checking whether a permutation $\pi$ is a pattern of $\sigma$ is equivalent to
checking whether there exists a pin word $u$ corresponding to $\pi$ with $u \preccurlyeq w_{\sigma}$. 
Consequently, to study the proper pin-permutations not in $\mathcal{C} = Av(B)$, 
\emph{i.e.} those containing some pattern $\pi$ in $B$, 
it is enough to study the \textbf{strict} pin words containing some pin word $u$ encoding $\pi$, for $\pi$ in $B$. 
This is the reason why we introduce the languages ${\mathcal L}(u)$ and ${\mathcal L}_{\pi}$ below. 

\begin{defi}\label{def:lu}
Let $u$ be a pin word and $u = u^{(1)} \ldots u^{(j)}$ be its strong numeral-led factor
decomposition. We set
  $${\mathcal L}(u) = A^\star \phi(u^{(1)}) A^\star \phi(u^{(2)}) \ldots A^\star \phi(u^{(j)}) A^\star \text{ where } A =  \{U,D,L,R\}\text{.}$$

\noindent Let $\pi$ be a permutation, and $P(\pi)$ be the set of pin words that encode $\pi$. We set
  $${\mathcal L}_{\pi} = \cup_{u \in P(\pi)}{\mathcal L}(u)$$
\end{defi}

As we shall see in Lemma~\ref{lem:patterns}, the languages ${\mathcal L}_{\pi}$ 
allow to describe the strict pin words of proper pin-permutations that contain $\pi$ (or rather their image by $\phi$). 
Note however that not all words of ${\mathcal L}_{\pi}$ are in the image by $\phi$ of the strict pin words of proper pin-permutations containing $\pi$. 
For instance, there are words starting with $LLLLL \phi(u^{(1)})$ that belong to ${\mathcal L}_{\pi}$, 
and these are not even in ${\mathcal M}$ (\emph{i.e.}, are not the image of a strict pin word by $\phi$).
But Lemma~\ref{lem:patterns} proves that such trivial ``bad words'' not belonging to ${\mathcal M}$ are the only ones that we should exclude from ${\mathcal L}_{\pi}$.
Indeed ${\mathcal L}_{\pi} \cap {\mathcal M}$ is in one-to-one correspondence with strict pin words encoding proper pin-permutations that contain $\pi$, via $\phi^{-1}$. 

\medskip

These languages ${\mathcal L}_{\pi}$ further have the interesting property 
of somehow translating the pattern involvement between pin-permutations into set inclusion, as expressed by Theorem~\ref{thm:lienLangagesMotif}. 

Note that ${\mathcal L}_{\pi}$ is non-empty if and only if $P(\pi)$ is
non-empty or equivalently $\pi$ is a pin-permutation. 
So when $\pi$ is not a pin-permutation, the results of
Theorem~\ref{thm:lienLangagesMotif} and Lemma~\ref{lem:patterns}
follow easily from the following statement (see for instance Lemma 3.3 of~\cite{BBR09}):
if $\pi \leq \sigma$ and $\sigma$ is a pin-permutation, then $\pi$ is a pin-permutation.

\begin{theo} \label{thm:lienLangagesMotif}
Let $\pi$ and $\sigma$ be permutations, such that $\pi \leq \sigma$. 
Then ${\mathcal L}_{\sigma} \subseteq {\mathcal L}_{\pi}$. 
\end{theo}

In the following, we write $m =  v^{(1)} \phi(u^{(1)}) v^{(2)} \phi(u^{(2)}) \ldots v^{(j)} \phi(u^{(j)}) v^{(j+1)}$ for $m \in A^\star$, meaning that $m =  v^{(1)} w^{(1)} v^{(2)} w^{(2)} \ldots v^{(j)} w^{(j)} v^{(j+1)}$ with $w^{(i)} \in \phi(u^{(i)})$ if $u^{(i)}$ has length $1$ and $w^{(i)} = \phi(u^{(i)})$ otherwise.

\begin{proof}
Suppose that $\pi \leq \sigma$. 
If ${\mathcal L}_{\sigma}$ is empty, the statement trivially holds. 
Otherwise, let $m \in {\mathcal L}_{\sigma}$. 
We want to show that $m \in {\mathcal L}_{\pi}$. By definition of
${\mathcal L}_{\sigma}$, there exists $w \in P(\sigma)$ such that $m
= v^{(1)} \phi(w^{(1)}) v^{(2)} \ldots \phi(w^{(i)})
v^{(i+1)}$ where $w = w^{(1)} w^{(2)}\ldots w^{(i)}$ is the strong
numeral-led factor decomposition of $w$.  From Lemma~\ref{csq ordre}
there is $u \in P(\pi)$ such that $u \preccurlyeq w$. Let $u = u^{(1)}
\ldots u^{(j)}$ be the \snl factor decomposition of $u$. From
Theorem~\ref{thm:ordreMots}, $w=\bar{v}^{(1)}\bar{w}^{(1)} \ldots
\bar{v}^{(j)}\bar{w}^{(j)}\bar{v}^{(j+1)}$ where for all $k \in
\{1,\ldots, j\}$, $\bar{w}^{(k)} \in \SP \cup \M$ and
$\phi(\bar{w}^{(k)})$ has a factor in $\phi(u^{(k)})$. But $w^{(1)}
w^{(2)}\ldots w^{(i)}$ is the strong numeral-led factor decomposition
of $w=\bar{v}^{(1)}\bar{w}^{(1)} \ldots$
$\bar{v}^{(j)}\bar{w}^{(j)}\bar{v}^{(j+1)}$. Therefore the factors
$\bar{w}^{(1)}$, $\bar{w}^{(2)}$, $\ldots \bar{w}^{(j)}$ appear in
this order in $w^{(1)} w^{(2)} \dots w^{(i)}$, being non-overlapping
and each inside one $w^{(\ell)}$, since each $w^{(\ell)}$ begins with a
numeral and each $\bar{w}^{(k)}$ is in $\SP \cup \M$. Thus by
definition of $\phi$, 
the factors $\phi(\bar{w}^{(1)})$, $\phi(\bar{w}^{(2)})$, $\ldots$, $
\phi(\bar{w}^{(j)})$ appear in this order in $\phi(w^{(1)})
\phi(w^{(2)}) \ldots \phi(w^{(i)})$, being non-overlapping and each
inside one $\phi(w^{(\ell)})$. So $m = v^{(1)} \phi(w^{(1)}) v^{(2)}
\phi(w^{(2)}) \ldots v^{(i)} \phi(w^{(i)}) v^{(i+1)} \in A^\star
\phi(\bar{w}^{(1)}) A^\star \phi(\bar{w}^{(2)}) \ldots A^\star
\phi(\bar{w}^{(j)}) A^\star$. But for all $k \in \{1,\ldots, j\}$,
$\phi(\bar{w}^{(k)})$ has a factor in $\phi(u^{(k)})$, thus $m \in
    {\mathcal L}(u) = A^\star \phi(u^{(1)}) A^\star$ $ \phi(u^{(2)})
    \ldots A^\star \phi(u^{(j)}) A^\star$ and so $m \in {\mathcal
      L}_\pi$.
\end{proof}

\begin{rem}\label{rem:diese}
Although it is not necessary for our purpose, we would find interesting to have 
a stronger version of Theorem~\ref{thm:lienLangagesMotif} which would state 
the equivalence between $\pi \leq \sigma$ and ${\mathcal L}_{\sigma} \subseteq {\mathcal L}_{\pi}$, 
when $\pi$ and $\sigma$ are pin-permutations.
We do not know if this equivalence holds. 
However, we do know that, with a small modification of ${\mathcal L}(u)$ to allow for an extra symbol in $A$ 
(which plays the role of a separator), then we have, for all pin-permutations $\pi$ and $\sigma$, 
$\pi \leq \sigma$ if and only if ${\mathcal L}_{\sigma} \subseteq {\mathcal L}_{\pi}$.
\end{rem}

\subsection{Characterizing when a class has a finite number of proper pin-permutations}

We conclude Section~\ref{sec:strict pw} by putting together the above
definitions and results to answer to our original problem: 
providing a criterion that characterizes 
when a permutation class contains finitely many proper pin-permutations.

\begin{lem} \label{lem:patterns}
  Let $\sigma$ be a proper pin-permutation, $\pi$ be a permutation and
  $w$ be a strict pin word encoding $\sigma$. Then $\pi \leq \sigma$
  if and only if $\phi(w) \in {\mathcal L}_\pi$.
\end{lem}

\begin{proof}
Assume that $\pi \leq \sigma$, then from Theorem~\ref{thm:lienLangagesMotif} ${\mathcal L}_\sigma \subseteq {\mathcal L}_\pi$.
As $w$ is a strict pin word, $\phi(w) \in {\mathcal L}(w)$ thus $\phi(w) \in {\mathcal L}_\sigma$ and so $\phi(w) \in {\mathcal L}_\pi$.

Conversely, assume that $\phi(w) \in {\mathcal L}_\pi$. Then
there exists a pin word $u$ encoding $\pi$ such that
$\phi(w) \in \mathcal{L}(u)$.
Let us denote by $u = u^{(1)} \ldots u^{(j)}$ the strong numeral-led factor decomposition of $u$.
By definition of $\mathcal{L}(u)$, $\phi(w)$ can be decomposed into
$t^{(1)} \ldots t^{(j+1)}$, with $t^{(i)} \in A^\star\phi(u^{(i)}) \cap
\M$ for $i \in \{1, \ldots, j\}$.
By definition of $\phi$ and since $w$ is a strict pin word, there exists a strict pin word $t$ such that $w=t \, t^{(2)} \ldots t^{(j+1)}$.
Then $\phi(t) = t^{(1)}$ and $\phi(u^{(1)})$ is a factor of $\phi(t)$.
Furthermore, for $i \in \{2, \ldots, j\}$, $\phi(u^{(i)})$ is a factor of $\phi(t^{(i)})=t^{(i)}$ (this equality holds because $t^{(i)} \in \mathcal{M}$).
Consequently, from Theorem~\ref{thm:ordreMots}, $u \preccurlyeq w$. Finally from Lemma~\ref{csq ordre}, we conclude that $\pi \leq \sigma$.
\end{proof}

By $\phi^{-1}$, each word of $\M$ is turned into a strict pin word
and hence into a proper pin-permutation.
As a consequence of Lemma~\ref{lem:patterns}, $\mathcal{L}_\pi \cap \M$ is the image by $\phi$ of the language of strict pin words encoding proper pin-permutations $\sigma$
that contain $\pi$ as a pattern: $\mathcal{L}_\pi \cap \M = \{ \phi(w)
\mid \exists \sigma \text{ such that } \pi \leq \sigma \text{ and } w \in \SP \cap P(\sigma)\}$. 
With the same idea we have the following theorem, which provides the criterion announced at the beginning of Section~\ref{sec:strict pw}: 

\begin{theo} \label{thm:nbInfiniEtLangages}
A permutation class $Av(B)$ contains a finite number of proper
pin-permutations if and only if the set 
$\M \setminus \cup_{\pi \in B}{\mathcal L}_\pi$ is finite.
\end{theo}

\begin{proof}
Let $S_B$ be the set of strict pin words encoding permutations of
size at least $2$ in $Av(B)$. Then $\phi$ is a bijection from $S_B$ to
$\M_{\geq 3} \setminus \cup_{\pi \in B}{\mathcal L}_\pi$.
Indeed for any strict pin word $w$ of length at least $2$,
let $\sigma$ be the permutation encoded by $w$.
Then $\sigma$ is a proper pin-permutation and
Lemma~\ref{lem:patterns} implies that $\sigma \in Av(B)$ if and only
if $\phi(w) \notin \cup_{\pi \in B}{\mathcal L}_\pi$. We conclude the
proof observing that every proper pin-permutation $\sigma$ of size $n$
is associated to at least $1$ and (very loosely) at most $8^n$ strict pin words, as
any pin word encoding $\sigma$ is a word of length $n$ over an
$8$-letter alphabet.
\end{proof}

\section{Algorithm(s) testing if a class contains a finite number of proper pin-permutations}
\label{sec:algos_finite_number_proper_pinperm}

Section~\ref{sec:strict pw} (and specifically Theorem~\ref{thm:nbInfiniEtLangages}) 
provides us with a characterization of classes $Av(B)$
which contain finitely many proper pin-permutations: 
they are those such that $\M \setminus \cup_{\pi \in B}{\mathcal L}_\pi$ is finite. 
Our goal is now to find an algorithm, as efficient as possible, checking this condition. 

The general structure of this algorithm will be explained in Subsection~\ref{ssec:structure_algo}. 
It involves the use of some automata ${\mathcal A}_{\pi}$ 
recognizing ${\mathcal L}_\pi$ for any pin-permutation $\pi \in B$. 
The construction of these automata is rather technical. 
Subsection~\ref{ssec:idees_construction_automates} will give an overview of it, 
while the technical details are postponed to Appendices~\ref{sec:pinwords} and~\ref{sec:buildingAutomata}. 

Before we get to our algorithm, we first review the decision procedure of~\cite{BRV06}, 
which also answers the question of whether a class $Av(B)$ contains finitely many proper pin-permutations. 
This review will serve two purposes: one is to help the reader see the common aspects and the differences 
between this procedure and our algorithm; 
the other is to be able to compare their complexities.

\subsection{The decision procedure of Brignall, Ru{\v{s}}kuc and Vatter} \label{ssec:procedure_of_BRV_2}

As reviewed at the beginning of Section~\ref{sec:strict pw}, Brignall {\it et al.} provided in~\cite{BRV06} 
a first characterization of classes containing finitely many proper pin-permutations. 
Namely, they proved that $\mathcal{C} = Av(B)$ contains a finite number of proper pin-permutations 
if and only if the set $ \mathcal{L} = \SP \setminus \bigcup_{u \in P(B)} \{\text{strict pin word }w \mid u \preccurlyeq w \}$ is finite, 
where $P(B)$ denotes the set of pin words encoding a permutation of $B$. 
The main point of the procedure of~\cite{BRV06} (which will be similar in our algorithm) 
is then to decide the finiteness of the language $\mathcal{L}$ using automata theory. 

Given a pin word $u$, the authors of~\cite{BRV06} explain how to build an automaton $\mathcal{A}^{(u)}$
recognizing a language $\mathcal{L}^{(u)}$ such that $\SP \cap \mathcal{L}^{(u)} = \{\text{strict pin word }w \mid u \preccurlyeq w \}$. 
Then, they notice that $\SP$ is a recognizable language, and conclude -- with classical theorems of automata theory -- that 
it is decidable whether the language $ \mathcal{L}$ is finite, \emph{i.e.}, whether 
$\mathcal{C}$ contains a finite number of proper pin-permutations.

\medskip

The proof of~\cite{BRV06} is constructive and establishes that 
deciding whether $\mathcal{C}$ contains a finite number of proper pin-permutations 
\emph{may} be done algorithmically.
However the authors do not give an actual algorithm since many steps are not given explicitly. 
More precisely, if we turn into an actual algorithm the procedure of~\cite{BRV06}, 
the main steps would be:
\begin{enumerate}
\item Compute the set $P(B)$ of pin words encoding permutations of $B$;
\item For each $u \in P(B)$, build the automaton $\mathcal{A}^{(u)}$ recognizing $\mathcal{L}^{(u)}$;
\item Build an automaton $\mathcal{A}$ recognizing $\mathcal{L} = \SP \setminus \bigcup\limits_{u \in P(B)}\mathcal{L}^{(u)}$;
\item Test whether the language accepted by $\mathcal{A}$ is finite.
\end{enumerate}

In~\cite{BRV06}, the authors focus on the second step (which is indeed the main one), even though the complexity of building $\mathcal{A}^{(u)}$ is not analyzed.  
The first step is not addressed in~\cite{BRV06}, and the third (resp. fourth) step is solved applying an existential (resp. decidability) theorem of automata theory -- 
in particular, the complexity of the corresponding construction (resp. decision) is not studied. 
Analyzing the above four-step procedure, 
we prove in the following that it has a doubly exponential complexity
due to the resolution of a co-finiteness problem for a regular language given by a non-deterministic automaton. 
Let us first introduce some notations: denote by $n$ the sum of the sizes of permutations in the basis $B$, by
$s'$ (resp. $s$) the maximal size of a permutation (resp. \emph{pin}-permutation) in $B$ and by $k$ the number of pin-permutations in $B$.

\smallskip

Even though~\cite{BRV06} does not study the first step of the above procedure, 
there is a naive algorithm to solve it: 
for each permutation
$\pi$ in $B$, for each pin word $u$ of length $|\pi|$, check if the
permutation encoded by $u$ is $\pi$, and add $u$ to $P(B)$ in the affirmative.  
This is performed in ${\mathcal O}(n \cdot 8^{s'})$ time. 
In our work we explain how to replace this step by a
step solved in $\O(n)$ time.

For the second step,~\cite{BRV06} explains how to build automata $\mathcal{A}^{(u)}$ recognizing the languages $\mathcal{L}^{(u)}$. 
We will not detail the analysis of the complexity of building $\mathcal{A}^{(u)}$, 
but let us notice that these automata are non deterministic and have $\O(|u|)$ (and at least $|u|$) states.

About the third step,~\cite{BRV06} refers to automata theory without detail. 
The most direct way to achieve this step is 
to build by juxtaposition an intermediate automaton $\mathcal{A}^{P(B)}$ recognizing $\bigcup\limits_{u \in P(B)}\mathcal{L}^{(u)}$, to determinize this automaton in order to complement it, and then to compute the intersection with an automaton recognizing $\SP$.
But the determinization of an automaton is exponential w.r.t. the number of states of the automaton. 
Since the number of states of $\mathcal{A}^{P(B)}$ is $\O(\sum_{u \in 
P(B)} |u|)$ and is indeed at least $\sum_{u \in P(B)} |u|$,
the complexity of this algorithm for the third step is $\O(2^{\sum_{u \in P(B)} |u|})$. 
Moreover, $\sum_{u \in P(B)} |u| \leq k \cdot 8^s \cdot s$. 
Even if this bound may not be tight, there exist pin-permutations encoded by an exponential number of pin words. 
For instance, the identity of size $s$ is encoded by at least $2^s$ pin words,
since any word on the alphabet $\{1,3\}$ is suitable.
Therefore the complexity of the above algorithm for the third step is of order at least $\mathcal{O}(2^{k\cdot s \cdot 2^s})$,
which is doubly exponential w.r.t.~$s$.

Finally,~\cite{BRV06} refers to a classical algorithm for the fourth step. 
It consists in testing whether the automaton $\mathcal{A}$ obtained at the end of the third step 
contains a cycle that can be reached from an initial state and can lead to a final state. 
This is linear w.r.t.~the size of the automaton $\mathcal{A}$ 
(and we will detail why in Subsection~\ref{ssec:includes_looking_for_a_cycle_in_automaton}, as our algorithm ends with a similar step).

\subsection{A more efficient alternative} 
\label{ssec:structure_algo}

Reviewing the procedure of~\cite{BRV06} and analyzing its complexity, we have seen why this procedure is not efficient. 
The main issue is the determinization of the automaton $\mathcal{A}^{P(B)}$ recognizing $\bigcup\limits_{u \in P(B)}\mathcal{L}^{(u)}$ (so that it may be complemented). 
A secondary issue is that this automaton $\mathcal{A}^{P(B)}$ is built by juxtaposition of a large number of automata: 
one for each pin word in $P(B)$. 
In this subsection, we explain how to overcome these two issues. 

\smallskip

First, we do not start from the same characterization of classes ${\mathcal C} = Av(B)$ containing a finite number of proper pin-permutations. 
Instead of the one of~\cite{BRV06}, we take our alternative criterion provided by Theorem~\ref{thm:nbInfiniEtLangages}, 
and provide an algorithm testing whether the language $\M \setminus \cup_{\pi \in B}{\mathcal L}_\pi$ is finite. 

In this second characterization, the languages ${\mathcal L}_\pi$ 
play the same role as the languages $\mathcal{L}^{(u)}$ in the first one. 
There is however only one language for each $\pi \in B$, 
that somehow accounts for all the languages $\mathcal{L}^{(u)}$ 
for $u \in P(\pi)$. 
This solves the issue of the number of automata. 

Moreover, this saves us the trouble of computing $P(B)$, 
a step whose complexity was ${\mathcal O}(n \cdot 8^{s'})$ 
in the procedure of~\cite{BRV06}. 
Instead, since the definition of ${\mathcal L}_\pi$ relies on $P(\pi)$, 
we only need to compute a description of $P(\pi)$ for $\pi \in B$. 
The detailed study of pin words done in Appendix~\ref{sec:pinwords} shows that
this is possible from some quite simple tests on the decomposition tree of $\pi$, 
which will be detailed in Section~\ref{sec:polynomial}. 
These tests are performed in $\mathcal{O}(\pi)$, 
so that the total cost of computing the description of $P(\pi)$ for all $\pi \in B$
is $\mathcal{O}(n)$. 

\smallskip

To address the determinization issue, let us introduce a few notations. 
For any word $v = v_1 \ldots v_p$ denote by
$\overleftarrow{v} = v_p \ldots v_1$ the reverse of $v$, 
and for any language ${\mathcal L}$, denote by
$\overleftarrow{\mathcal{L}}$ the language 
$\{\overleftarrow{v} \mid v \in {\mathcal L}\}$. 
In practice, our algorithm does not test if $\M \setminus \cup_{\pi \in B}{\mathcal L}_\pi$ is finite, 
but rather whether its reverse language $\overleftarrow{\M \setminus \cup_{\pi \in B}{\mathcal  L}_\pi}$ is finite. 
Notice that by definition of $\M$, we have 
$ \overleftarrow{\M \setminus \cup_{\pi \in B}{\mathcal L}_\pi} = 
\overleftarrow{\M} \setminus \cup_{\pi \in B}\overleftarrow{{\mathcal L}_\pi} = 
\M \setminus \cup_{\pi \in B}\overleftarrow{{\mathcal L}_\pi}$. 
As in~\cite{BRV06}, to test the finiteness of this language, 
we will build an automaton ${\mathcal A}_{\mathcal{C}}$ accepting $\M \setminus \cup_{\pi \in B}\overleftarrow{{\mathcal L}_\pi}$, 
and then test whether ${\mathcal A}_{\mathcal{C}}$ contains a cycle -- see Subsection~\ref{ssec:includes_looking_for_a_cycle_in_automaton}. 
Taking reverse languages is the trick that allows us to build a \emph{deterministic} automaton ${\mathcal A}_{\mathcal{C}}$, 
thus avoiding the determinization 
which causes the complexity blow-up in the procedure of~\cite{BRV06}. 
The connection between reverse and determinism is certainly unclear at the moment, 
but it is also hard to explain at this stage of the presentation of the algorithm. 
It follows from some details of Appendix~\ref{sec:pinwords}, 
and we will only explain this choice at the beginning of Appendix~\ref{sec:buildingAutomata}. 

We start with building, for each pin-permutation\footnote{Recall that 
when $\pi$ is not a pin-permutation, both $P(\pi)$ and 
$\overleftarrow{{\mathcal L}_\pi}$ are empty. } $\pi \in B$,
a \emph{deterministic} automaton ${\mathcal A}_{\pi}$ 
which accepts $\overleftarrow{{\mathcal L}_\pi}$. 
An overview of the construction of these automata ${\mathcal A}_{\pi}$ 
and of its complexity
is given in the next subsection, while the details -- that are quite technical -- 
are postponed to Appendix~\ref{sec:buildingAutomata}. 

From deterministic automata $\mathcal A_{\pi}$ recognizing the languages $\overleftarrow{{\mathcal L}_{\pi}}$, 
we can obtain a deterministic automaton accepting words of $\cup_{\pi \in B}\overleftarrow{{\mathcal L}_\pi}$ 
\emph{i.e.}, (up to intersection with $\M$) words whose reverses encode proper pin-permutations containing some pattern $\pi \in B$. 
To preserve determinism, the automaton accepting the union 
is not simply built by juxtaposition of the $\mathcal A_{\pi}$. 
Instead, we do the Cartesian product of the automata $\mathcal A_{\pi}$ 
to compute a \emph{deterministic} automaton accepting the union 
$\cup_{\pi \in B}\overleftarrow{{\mathcal L}_\pi}$.  
This deterministic automaton can then be complemented
in linear time, in order to build the automaton ${\mathcal A}_{\mathcal{C}}$ recognizing 
$\M \setminus \cup_{\pi \in B}\overleftarrow{{\mathcal L}_\pi}$. 
Recall that the same operation on non-deterministic automata 
would be exponential in the worst case.
The construction of ${\mathcal A}_{\mathcal{C}}$ will be detailed in Section~\ref{sec:polynomial}, 
and its complexity analyzed. 
It is performed in ${\mathcal O}(s^{2k} )$ where $s$ and $k$ are, as before, 
the maximal size of a pin-permutation in $B$ and the number of pin-permutations in $B$. 
The complexity of testing with this method 
whether $\mathcal{C} = Av(B)$ contains finitely many proper pin-permutations 
is then also ${\mathcal O}(s^{2k} )$.
Writing that ${\mathcal O}(s^{2k} ) = {\mathcal O}(2^{k \cdot 2\log s})$ enables us to measure the complexity improvement
w.r.t.~the complexity $\mathcal{O}(2^{k\cdot s \cdot 2^s})$ of the procedure of~\cite{BRV06}: 
the complexity gain is doubly exponential w.r.t.~$s$.

\subsection{Construction of the automata ${\mathcal A}_{\pi}$}
\label{ssec:idees_construction_automates}

The most difficult part of the algorithm outlined above 
is the construction of the deterministic automata $\mathcal A_{\pi}$
accepting the languages $\overleftarrow{{\mathcal L}_\pi}$, 
for every pin-permutation $\pi \in B$. 
We give below the general idea of the method that is used for this construction, 
together with some results about the complexity of building $\mathcal A_{\pi}$. 
Detailed statements and proofs are provided in Appendices~\ref{sec:pinwords} and~\ref{sec:buildingAutomata}. 

\smallskip

From the definition of ${\mathcal L}_{\pi}$ given p.\pageref{def:lu}
we have:\label{expression L_pi}
\begin{align*}
\overleftarrow{{\mathcal L}_{\pi}}& = \bigcup_{ u \in P(\pi) \atop u = u^{(1)}u^{(2)}\ldots u^{(j)}} A^\star
\overleftarrow{\phi(u^{(j)})} A^\star \ldots A^\star
\overleftarrow{\phi(u^{(2)})} A^\star \overleftarrow{\phi(u^{(1)})}
A^\star
\end{align*}
where $A = \{U,D,L,R\}$, $\phi$ is the map introduced in
Definition~\ref{def:phi} (p.\pageref{def:phi}) and for every pin word $u$, by $u=u^{(1)}u^{(2)}\ldots u^{(j)}$ we mean 
that $u^{(1)}u^{(2)}\ldots u^{(j)}$ is the strong
numeral-led factor decomposition of $u$.
Therefore, we see that a description of $\overleftarrow{{\mathcal L}_{\pi}}$ will follow 
as soon as we are able to describe the set $P(\pi)$ of pin words of $\pi$, 
and more precisely their  strong numeral-led factor decompositions.

In our previous work~\cite{BBR09}, we have given a recursive characterization of the decomposition trees of pin-permutations. 
We recall it below as Equation~\eqref{eq:pin_perm_trees}. 
We will then follow this characterization to recursively describe $P(\pi)$ for any pin-permutation $\pi$, 
and to subsequently give a recursive algorithm to build the automata $\mathcal A_{\pi}$ recognizing $\overleftarrow{{\mathcal L}_{\pi}}$. 
We also present an alternative construction of $\mathcal A_{\pi}$
whose complexity is optimized; but instead of
$\overleftarrow{{\mathcal L}_{\pi}}$, the automaton recognizes a
language $\mathcal{L}'_{\pi}$ such that $\mathcal{L}'_{\pi}\cap {\mathcal M} =
\overleftarrow{\mathcal{L}_{\pi}} \cap {\mathcal M}$, 
which turns out to be sufficient for our purpose (see Subsection~\ref{ssec:includes_looking_for_a_cycle_in_automaton}). 

\smallskip

The recursive characterization of the decomposition trees of pin-permutations provided in~\cite{BBR09} is as follows. 
It involves oscillations and quasi-oscillations, two special kinds of pin-permutations whose definitions are technical and given in Appendix~\ref{sec:oscillations}. 
The set ${\mathcal S}$ of substitution decomposition trees of pin-permutations 
is recursively characterized by:
\begin{align} \label{eq:pin_perm_trees} &{\mathcal S} =  \tikz[sibling distance=7mm,level
  distance=15pt,inner sep=1pt] \draw (0,0) node[leaf]{};  + \begin{tikzpicture}[baseline=(X.base),sibling distance=12pt,level
  distance=15pt,inner sep=1pt] \node[linear] (X){$\oplus$} child {node {$\calE^{+}$}} child
  {node (xx1) {$\calE^{+}$}} child [missing] child {node (xx2) {$\calE^{+}$}};
  \draw [dotted] (xx1) -- (xx2);
  \end{tikzpicture}
  +
  \begin{tikzpicture}[level distance=15pt,sibling
  distance=9pt,baseline=(X.base),inner sep=0] \node[simple,inner sep=0] (X)
  {$\oplus$} child {node (xx1){$\calE^{+}$}} child [missing] child [sibling distance=0pt] {node (xx2){} edge from parent [draw=none] }child  {node[thin, shape=isosceles
    triangle, solid, draw, shape border rotate=90,anchor=apex, minimum
    height=5mm,inner sep=0pt,isosceles triangle apex angle=90]   {$\calN^{+}$}}
  child [sibling distance=0pt] {node (xx3){} edge from parent [draw=none] } child [missing] child {node (xx4) {$\calE^{+}$}};
  \draw [dotted] (xx1) -- (xx2);
  \draw [dotted] (xx3) -- (xx4);
  \end{tikzpicture}
   + \begin{tikzpicture}[baseline=(X.base),sibling distance=18pt,level
  distance=15pt,inner sep=1pt] \node[linear] (X){$\ominus$} child {node {$\calE^{-}$}} child
  {node (xx1) {$\calE^{-}$}} child [missing] child {node (xx2) {$\calE^{-}$}};
  \draw [dotted] (xx1) -- (xx2);
  \end{tikzpicture}
  +
  \begin{tikzpicture}[level distance=15pt,sibling
  distance=9pt,baseline=(X.base),inner sep=0] \node[simple,inner sep=0] (X)
  {$\ominus$} child {node (xx1){$\calE^{-}$}} child [missing] child [sibling distance=0pt] {node (xx2){} edge from parent [draw=none] }child  {node[thin, shape=isosceles
    triangle, solid, draw, shape border rotate=90,anchor=apex, minimum
    height=5mm,inner sep=0pt,isosceles triangle apex angle=90]   {$\calN^{-}$}}
  child [sibling distance=0pt] {node (xx3){} edge from parent [draw=none] } child [missing] child {node (xx4) {$\calE^{-}$}};
  \draw [dotted] (xx1) -- (xx2);
  \draw [dotted] (xx3) -- (xx4);
  \end{tikzpicture}
  \nonumber \\ &
+
\begin{tikzpicture}[baseline=(X.base),sibling distance=15pt,level
  distance=15pt,inner sep=1pt] \node[simple, inner sep=1pt] (X){$\alpha$} child {[fill] circle (2pt)} child
  {[fill] circle (2pt) node (xx1) {}} child [missing] child {[fill] circle (2pt) node (xx2) {}};
  \draw [dotted] (xx1) -- (xx2);
  \end{tikzpicture}
+ \begin{tikzpicture}[level distance=17pt,sibling
  distance=6pt,baseline=(X.base),inner sep=0pt] \node[simple,inner sep=1pt] (X)
  {$\alpha$} child {[fill] circle
    (2pt) node (xx1){}} child [missing] child [sibling distance=0pt] {node (xx2){} edge from parent [draw=none] }child[dashed]  {node[thin, shape=isosceles
    triangle, solid, draw, shape border rotate=90,anchor=apex, minimum
    height=5mm,inner sep=0pt,isosceles triangle apex angle=110]   {${\mathcal S} \setminus \{ \tikz
      \fill[draw] (0,0) circle (2pt); \}$}}
  child [sibling distance=0pt] {node (xx3){} edge from parent [draw=none] } child [missing] child {[fill] circle(2pt) node (xx4) {}};
  \draw [dotted] (xx1) -- (xx2);
  \draw [dotted] (xx3) -- (xx4);
  \end{tikzpicture}+ \begin{tikzpicture}[level distance=17pt,sibling
  distance=6pt,baseline=(X.base),inner sep=0] \node[simple,inner sep=0] (X)
  {$\beta^{+}$} child {[fill] circle (2pt)} child [missing] child {[fill] circle
    (2pt) node (xx1){}} child [missing] child [sibling distance=0pt] {node (xx2){} edge from parent [draw=none] }child[thick, dotted]  {node[thin, shape=isosceles
    triangle, solid, draw, shape border rotate=90,anchor=apex, minimum
    height=5mm,inner sep=0pt,isosceles triangle apex angle=110]   {${\mathcal S} \setminus \{ \tikz
      \fill[draw] (0,0) circle (2pt); \}$}}
  child [sibling distance=0pt] {node (xx3){} edge from parent [draw=none] } child [missing] child[dash pattern=on 3pt off 2pt on 1pt off
  2pt] {node (xx4) {\scriptsize{$12$}}} child [missing] child {[fill] circle (2pt)};
  \draw [dotted] (xx1) -- (xx2);
  \draw [dotted] (xx3) -- (xx4);
  \end{tikzpicture}
+ \begin{tikzpicture}[level distance=17pt,sibling
  distance=6pt,baseline=(X.base),inner sep=0] \node[simple,inner sep=0] (X)
  {$\beta^{-}$} child {[fill] circle (2pt)} child [missing] child {[fill] circle
    (2pt) node (xx1){}} child [missing] child [sibling distance=0pt] {node (xx2){} edge from parent [draw=none] }child[thick, dotted]  {node[thin, shape=isosceles
    triangle, solid, draw, shape border rotate=90,anchor=apex, minimum
    height=5mm,inner sep=0pt,isosceles triangle apex angle=110,inner sep=0]   {${\mathcal S} \setminus \{ \tikz
      \fill[draw] (0,0) circle (2pt); \}$}}
  child [sibling distance=0pt] {node (xx3){} edge from parent [draw=none] } child [missing] child[dash pattern=on 3pt off 2pt on 1pt off
  2pt] {node (xx4) {\scriptsize{$21$}}} child [missing] child {[fill] circle (2pt)};
  \draw [dotted] (xx1) -- (xx2);
  \draw [dotted] (xx3) -- (xx4);
  \end{tikzpicture}
\tag{$\star$}
\end{align}
where $\calE^{+}$ (resp. $\calE^{-}$) is the set of decomposition
trees of increasing (resp. decreasing) oscillations, $\calN^{+}$
(resp. $\calN^{-}$) is the set of decomposition trees of pin-permutations
that are not increasing (resp. decreasing) oscillations and whose root
is not $\oplus$ (resp. $\ominus$), $\alpha$ is any simple pin-permutation
and $\beta^{+}$ (resp. $\beta^{-}$) is any increasing (resp. decreasing)
quasi-oscillation. 
In Appendix~\ref{sec:oscillations}, 
in addition to recalling the definitions of increasing and decreasing (quasi-)oscillations, we present some of their properties. 
For the moment, let us only mention that some special pairs of points in quasi-oscillations can be identified, 
that are called \emph{auxiliary} and \emph{main substitution points}. 
Also, in every simple pin-permutation, we can identify special points, 
that are called \emph{active points} whose definition will also be given in the appendix -- see p.\pageref{def:active}.
In Equation~\eqref{eq:pin_perm_trees} above, edges written \tikz \draw[dashed] (0,0)--(0.5,0);
(resp. \tikz \draw[dash pattern=on 3pt off 2pt on 1pt off 2pt]
(0,0)--(0.5,0);, \tikz \draw[thick, dotted] (0,0)--(0.5,0);) correspond to an active point 
of $\alpha$ (resp. to a pair formed by an auxiliary point and a main substitution point of $\beta^+$ or $\beta^-$).
In this equation the only terms that are recursive are those containing a subtree labeled by 
$\calN^{+}, \calN^{-}$ or ${\mathcal S} \setminus \{ \tikz \fill[draw] (0,0) circle (2pt); \}$. 

Our characterization of $P(\pi)$ and construction of $\mathcal A_{\pi}$ are divided into several cases,
depending on which term of Equation~\eqref{eq:pin_perm_trees} $\pi$ belongs to. 
First, we study the non-recursive cases, 
then the recursive cases with a linear root and
finally the recursive cases with a prime root.
Notice that the cases with root $\ominus$ (resp. $\beta^-$) 
are up to symmetry\footnote{This notion of symmetry is formalized by Remark~\ref{rem:ominus=oplus_transposed} in the appendix.} 
identical to those with root $\oplus$ (resp. $\beta^+$). 
We will therefore only consider the former in our analysis to follow. 

\medskip

\paragraph*{Permutation of size $1$}

Notice first that the permutation $\pi=1 = \tikz \fill[draw] (0,0) circle (2pt);$ (whose decomposition tree
is a leaf) has exactly four pin words -- namely, $P(\pi) = \{1,2,3,4\}$.

\hspace*{-0.05\textwidth}\begin{minipage}{.65\textwidth}
Then
$$\overleftarrow{{\mathcal
    L}_{\pi}}=\{A^\star\overleftarrow{\phi(w)}A^\star \mid w \in
P(\pi)\} = A^\star{\mathcal M}_2 A^\star$$ where $${\mathcal
  M}_2={\mathcal M} \cap A^2= \{UR, UL, DR, DL, RU, RD, LU, LD\}.$$
The language $\overleftarrow{{\mathcal L}_{\pi}}$ is recognized
by the automaton ${\mathcal A}_{\pi}$ of Figure~\ref{fig:automate1}.
\end{minipage}
\quad
\begin{minipage}{.3\textwidth}
\begin{center}
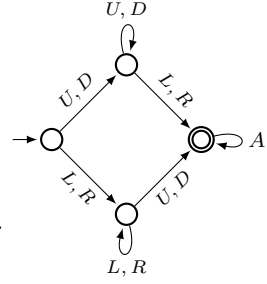

   \begin{tikzpicture}[>=latex, node distance=40pt, shorten >=1pt,
 on grid, every state/.style={draw,thick,minimum size=5pt},/tikz/initial text=]
     {\footnotesize
       \node [state, initial] (q_0) {};
       \node [state] (q_1) [above right=of q_0] {};
       \node [state] (q_2) [below right=of q_0] {};
       \node [state, accepting] (q_3) [above right=of q_2] {};
       \path[->] (q_0) edge node[above,sloped] {\scriptsize $U,D$} (q_1);
       \path[->] (q_0) edge node[below,sloped] {\scriptsize $L,R$} (q_2);
       \path[->] (q_1) edge node[above,sloped] {\scriptsize $L,R$} (q_3);
       \path[->] (q_2) edge node[below,sloped] {\scriptsize $U,D$} (q_3);
       \path[->] (q_3) edge [loop right] node {$A$} (q_3);
       \path[->] (q_2) edge [loop below] node {\scriptsize $L, R$} (q_2);
       \path[->] (q_1) edge [loop above] node {\scriptsize $U, D$} (q_1); }
\end{tikzpicture}
  \captionof{figure}{The automaton ${\mathcal A}_{\pi}$ when $\pi=1$.} \label{fig:automate1}
  \end{center}
\end{minipage}

\paragraph*{Simple permutations}
Let $\pi$ be a simple pin-permutation.
Theorem~\ref{thm:nbpinwords} in Appendix~\ref{sec:pinwords} (p.\pageref{thm:nbpinwords}) shows that the number of pin words of $\pi$ is at most 48. 
This of course does not describe the set $P(\pi)$ of pin words encoding $\pi$ explicitly,
but Algorithm 2 of~\cite{BBPR10} explains how to compute $P(\pi)$ in this case.
We will get back to this computation in Section~\ref{sec:polynomial}.

The construction of an automaton $\mathcal{A}_{\pi}$ accepting $\overleftarrow{{\mathcal L}_\pi}$ is then explained 
in Remark~\ref{rem:ascquadratique} in Appendix~\ref{sec:buildingAutomata}. 
The time and space complexity of the construction of $\mathcal{A}_{\pi}$ is
quadratic w.r.t.~$|\pi|$, as soon as the pin words of $\pi$ are given. 
In Remark~\ref{rem:simple_linear} in Appendix~\ref{sec:buildingAutomata}, 
we explain how to improve the complexity of the construction of $\mathcal{A}_{\pi}$, 
so that it is linear in time and space. 
The automaton $\mathcal{A}_{\pi}$ so obtained however does not accept $\overleftarrow{{\mathcal L}_\pi}$ but 
a language $\mathcal{L}'_{\pi}$ such that $\mathcal{L'}_{\pi} \cap \mathcal{M} = \overleftarrow{{\mathcal L}_{\pi}}\cap \mathcal{M}$. 

\paragraph*{Non recursive case with a linear root}
W.l.o.g., we consider $\pi = \oplus[\xi_1, \ldots, \xi_r]$ 
where $\xi_i$ are increasing oscillations. 
The set of pin words of $\pi$ is expressed using the shuffle product $\shuffle$ (defined p.\pageref{def:shuffle}) 
in Theorem~\ref{thm:pinwords_cas_lineaire_non_recursif} of Appendix~\ref{sec:pinwords}. 
Namely, denoting $P^{(1)}(\xi_k)$ (resp. $P^{(3)}(\xi_k)$) the set of pin words that encode $\xi_k$ 
and whose origin lies in quadrant $1$ (resp. $3$) with respect to the points of $\xi_k$, we have: 
{\small $$P(\pi) = \bigcup_{1\leq i \leq r-1} P(\oplus[\xi_i,\xi_{i+1}]) \cdot
\Big( (P^{(1)}(\xi_{i-1}) , \ldots , P^{(1)}(\xi_1) ) \shuffle
(P^{(3)}(\xi_{i+2}) , \ldots , P^{(3)}(\xi_j) ) \Big)\text{.}
$$}
Together with Lemmas~\ref{lem:f1f3} and \ref{lem:fdouble}, this provides an explicit expression of $P(\pi)$.

The automaton $\mathcal{A}_{\pi}$ accepting $\overleftarrow{{\mathcal L}_\pi}$ 
(resp.~a language $\mathcal{L}'_{\pi}$ such that $\mathcal{L'}_{\pi} \cap \mathcal{M} = \overleftarrow{{\mathcal L}_{\pi}}\cap \mathcal{M}$) 
is built by assembling smaller automata, which correspond to the different languages that appear in the shuffle product expression of $P(\pi)$. 
Subsection~\ref{ssec:non-rec_linear} gives the details of this construction, and Theorem~\ref{thm:5-9} (resp.~Remark~\ref{rem:versionopt1}) proves its correctness. 
Lemma~\ref{lem:complexite_linear} shows that it is achieved in time and space ${\mathcal O}(|\pi|^4)$ (resp. $\O(|\pi|^2)$). 
The automaton $\mathcal{A}_{\pi}$ in this case is shown in Figure~\ref{fig:automate_cas_oplus_non_recursif} (p.\pageref{fig:automate_cas_oplus_non_recursif}). 

\paragraph*{Recursive case with a linear root}
W.l.o.g., 
we consider a permutation $\pi = \oplus[\xi_1, \ldots , \xi_{\ell},$ $ \rho, \xi_{\ell+2}, \ldots ,\xi_r]$, 
where all $\xi_i$ are increasing oscillations, but $\rho$ is not. 
By induction, we may assume that we have an explicit description of $P(\rho)$, 
and an automaton $\mathcal{A}_\rho$ which accepts $\overleftarrow{{\mathcal L}_\rho}$ 
(resp. a language $\mathcal{L}'_{\rho}$ such that $\mathcal{L'}_{\rho} \cap \mathcal{M} = \overleftarrow{{\mathcal L}_{\rho}}\cap \mathcal{M}$). 
In Appendices~\ref{sec:pinwords} and~\ref{sec:buildingAutomata}, the decomposition tree of $\rho$ is denoted $T_{i_0}$. 

The set $P(\pi)$ of pin words of $\pi$ (which of course depends on $P(\rho)$) is given by Theorem~\ref{thm:linearRoot} in Appendix~\ref{sec:pinwords}. 
It always contains 
\[
P_0= P(\rho)\cdot \, \big( P^{(1)}(\xi_{\ell}), \ldots , P^{(1)}(\xi_1) \big) \shuffle \big(P^{(3)}(\xi_{\ell +2}), \ldots , P^{(3)}(\xi_r) \big)
\]
but it may contain more words, if $\pi$ satisfies additional conditions that are shown in the middle two columns of Figure~\ref{fig:H} (p.\pageref{fig:H}). 
In all these possible cases (considered up to symmetry), Theorem~\ref{thm:linearRoot} describes explicitly the complete set of pin words of $\pi$. 

If $P(\pi) =P_0$, like in the previous case the automaton $\mathcal{A}_{\pi}$ accepting $\overleftarrow{{\mathcal L}_\pi}$ is built by assembling
$\mathcal{A}_\rho$ with small automata corresponding to the different languages in the shuffle product defining $P_0$. 
The automaton $\mathcal{A}_{\pi}$ so obtained is shown in Figure~\ref{fig:automate_cas_oplus_recursif} (p.\pageref{fig:automate_cas_oplus_recursif}).

Otherwise, $P_0 \varsubsetneq P(\pi)$, so that the automaton of Figure~\ref{fig:automate_cas_oplus_recursif} accepts some 
but not all words of $\overleftarrow{{\mathcal L}_\pi}$. 
It is however possible to modify it by adding some transitions, 
so that the resulting automaton $\mathcal{A}_{\pi}$ accepts exactly $\overleftarrow{{\mathcal L}_\pi}$.
These modifications of the automaton are shown in the last column of Figure~\ref{fig:H} (p.\pageref{fig:H}). 

These constructions are explained in Subsection~\ref{ssec:rec_linear}. 
The proof that they are correct are however omitted even in the appendix, as they are very similar to some other proofs 
that are detailed there. 
The complexity of these constructions are given in Lemmas~\ref{lem:complexite_linear_rec} and~\ref{lem:complexite_linear_rec_2}: 
it is done in time and space
$\O\left((|\pi|-|\rho|)^2\right)$ plus the additional complexity
due to the construction of $\mathcal{A}_{\rho}$. 
The construction and its complexity are not modified (except for the recursive part)
for building an automaton $\mathcal{A}_{\pi}$ accepting a language 
$\mathcal{L}'_{\pi}$ such that $\mathcal{L'}_{\pi} \cap \mathcal{M} = \overleftarrow{{\mathcal L}_{\pi}}\cap \mathcal{M}$ 
instead of $\overleftarrow{{\mathcal L}_\pi}$. 

\paragraph*{Recursive case with a prime root}
We start with the case where $\pi = \alpha[1,\ldots, 1, \rho, $ $1, \ldots, 1]$ 
with $\alpha$ a simple pin-permutation and $\rho \neq 1$. 
The only other possibility is that $\pi$ is a permutation whose decomposition tree has a root labeled by a 
quasi-oscillation with two children that are not leaves. This special case will be considered in the next paragraph. 
 
By induction, we may assume that we have an explicit description of $P(\rho)$, 
and an automaton $\mathcal{A}_\rho$ which accepts $\overleftarrow{{\mathcal L}_\rho}$ 
(resp. a language $\mathcal{L}'_{\rho}$ such that $\mathcal{L'}_{\rho} \cap \mathcal{M} = \overleftarrow{{\mathcal L}_{\rho}}\cap \mathcal{M}$). 
We denote by $x$ the point of $\alpha$ expanded by $\rho$. 
Let us also record here that the decomposition tree of $\rho$ will be denoted $T$ in Appendices~\ref{sec:pinwords} and~\ref{sec:buildingAutomata}.

As in Definition~\ref{def:Qxalpha} (p.\pageref{def:Qxalpha}), 
let $Q_x(\alpha)$ denote the set of strict pin words 
obtained by deleting the first letter of a quasi-strict pin word 
encoding a pin representation of $\alpha$ starting in $x$. 
The set $P(\pi)$ of pin words of $\pi$ is given by Theorem~\ref{thm:conditionc} (p.\pageref{thm:conditionc}) in Appendix~\ref{sec:pinwords}. 
As in the previous case, it always contains a ground set of words, in this case $P(\rho)\cdot Q_x(\alpha)$, 
but it contains more words in case $\pi$ satisfies some further condition -- denoted $(\mathcal{C})$ in Appendices~\ref{sec:pinwords} and~\ref{sec:buildingAutomata}. 
Although it becomes quite technical, it is possible to make explicit the description of $P(\pi)$ in Theorem~\ref{thm:conditionc}, 
using Lemma~\ref{lem:w_at_least_3} and Remarks~\ref{rem:compute_Q_x(alpha)}, \ref{rem:w_explicit} and~\ref{rem:w_2}.

When $P(\pi) = P(\rho)\cdot Q_x(\alpha)$, the automaton $\mathcal{A}_\pi$ which accepts $\overleftarrow{{\mathcal L}_\pi}$ 
is easily constructed from $\mathcal{A}_\rho$ and $Q_x(\alpha)$. 
As before, when $P(\rho)\cdot Q_x(\alpha) \varsubsetneq P(\pi)$, $\mathcal{A}_\pi$ has the same general structure as this automaton, 
with some new transitions added. This is shown on Figure~\ref{fig:automate_primitif_rec} (p.\pageref{fig:automate_primitif_rec}). 
Theorem~\ref{lem:A_pi_primitif_rec} proves the correctness of this construction. 
This however only holds for $|\rho| \neq 2$. 
In the special case $|\rho| =2$, the construction of $\mathcal{A}_\pi$ is not recursive anymore, 
and is easily solved (see p.\pageref{page:automaton_conditionC_T=2}). 

The complexity of the construction of $\mathcal{A}_\pi$ is discussed in Lemma~\ref{lem:complexite_primitif_rec}. 
Except for the special case $|\rho| =2$, it is $\mathcal{O}(|\pi|-|\rho|)$ in time and space, in addition to the complexity of computing $\mathcal{A}_\rho$. 
This holds both for the automaton $\mathcal{A}_{\pi}$ accepting $\overleftarrow{{\mathcal L}_\pi}$
and for its variant (whose construction is unchanged except for the recursive part) which accepts a language 
$\mathcal{L}'_{\pi}$ such that $\mathcal{L'}_{\pi} \cap \mathcal{M} = \overleftarrow{{\mathcal L}_{\pi}}\cap \mathcal{M}$. 
For the special case $|\rho| =2$, the complexity of building $\mathcal{A}_{\pi}$ accepting $\overleftarrow{{\mathcal L}_\pi}$ is $\mathcal{O}(|\pi|^3)$, 
and it drops to $\mathcal{O}(|\pi|^2)$ for the variant accepting $\mathcal{L}'_{\pi}$. 

\paragraph*{The special case of increasing quasi-oscillations}
W.l.o.g., the only remaining case in Equation~\eqref{eq:pin_perm_trees} is that of a permutation 
$\pi = \beta^{+}[1, \ldots, 1 ,\rho, 1, \ldots, 1, 12,1,$ $ \ldots, 1]$ 
where $\beta^+$ is an increasing quasi-oscillation, 
the permutation $12$ expands an auxiliary point of $\beta^{+}$ 
and $\rho$ (of size at least $2$) expands the corresponding main substitution point of $\beta^{+}$. 
Again, in Appendices~\ref{sec:pinwords} and~\ref{sec:buildingAutomata}, the decomposition tree of $\rho$ is denoted $T$. 
And by induction, we may assume that we have an explicit description of $P(\rho)$, 
and an automaton $\mathcal{A}_\rho$ which accepts $\overleftarrow{{\mathcal L}_\rho}$ 
(resp. a language $\mathcal{L}'_{\rho}$ such that $\mathcal{L'}_{\rho} \cap \mathcal{M} = \overleftarrow{{\mathcal L}_{\rho}}\cap \mathcal{M}$). 

This case is very much constrained, and the set of pin words of $\pi$ is completely determined 
by Theorem~\ref{thm:primeRoot_cas_special}(p.\pageref{thm:primeRoot_cas_special}): $P(\pi) = P(\rho)  \cdot w$, for some word $w$ 
which is uniquely determined, and that Remark~\ref{rem:w_cas_2_non_feuilles} shows explicitly. 
From this, is it not hard to build from $\mathcal{A}_\rho$ an automaton $\mathcal{A}_\pi$ which accepts $\overleftarrow{{\mathcal L}_\pi}$ 
(resp. a language $\mathcal{L}'_{\pi}$ such that $\mathcal{L'}_{\pi} \cap \mathcal{M} = \overleftarrow{{\mathcal L}_{\pi}}\cap \mathcal{M}$). 
This is explained in Paragraph~\ref{sec:tcl} (p.\pageref{sec:tcl}), and is performed in $\O\big(|\pi|-|\rho|\big)$ time and space 
in addition to the time and space complexity of the construction of $\mathcal{A}_{\rho}$.

\medskip

The automaton $\mathcal{A}_\pi$ associated with a pin-permutation $\pi$ is then build recursively, 
by first determining which shape of tree in Equation~\eqref{eq:pin_perm_trees} is matched by the 
decomposition tree of $\pi$, and then applying the corresponding construction. 
From the complexities of these constructions, 
it is not hard to evaluate the overall complexity of building $\mathcal{A}_\pi$. 
Namely: 

\begin{theo}\label{thm:complexity_new_statement}
Given $\pi$ a pin-permutation, the above recursive construction allows to build 
an automaton $\mathcal{A}_\pi$ which accepts $\overleftarrow{{\mathcal L}_\pi}$ 
(resp. a language $\mathcal{L}'_{\pi}$ such that $\mathcal{L'}_{\pi} \cap \mathcal{M} = \overleftarrow{{\mathcal L}_{\pi}}\cap \mathcal{M}$)
in time and space complexity $\mathcal{O}(|\pi|^4)$ (resp. $\mathcal{O}(|\pi|^2)$).
\end{theo}

The proof of Theorem~\ref{thm:complexity_new_statement} is however postponed to the appendix 
-- see Theorem~\ref{thm:complexity-A_pi} (p.\pageref{thm:complexity-A_pi}). 
Indeed, to prove this theorem, we need to be careful about some details of the construction 
of automata, that are only explained in the appendix. 
The main difficulty is to ensure that some special states in the automata can be ``marked'' 
in some of the above constructions without increasing the complexity. 
The marking of these special states is needed to build the additional transitions in the cases 
$\pi = \oplus[\xi_1, \ldots , \xi_{\ell}, \rho, \xi_{\ell+2}, \ldots ,\xi_r]$ 
and $\pi = \alpha[1,\ldots, 1, \rho, 1,\ldots, 1]$, 
since these transitions are actually pointing towards these ``marked'' states. 
Subsection~\ref{subsection:marquage} explains how to mark these special states along the construction of~$\mathcal{A}_\pi$.

\section{A polynomial algorithm deciding whether a class contains a finite number of simple permutations}
\label{sec:polynomial}

In this section, we prove the main result of this article: 
%
\begin{theo}\label{thm:main_result}
Given a finite set of permutations $B$, we describe an algorithm that determines whether
the permutation class ${\mathcal C} = Av(B)$ contains a finite number of simple permutations.
Denoting $n = \sum_{\pi \in B} |\pi|$, 
$p = \prod |\pi|$ where the product is taken over all pin-permutations in $B$, 
$k$ the number of pin-permutations in $B$ 
and $s$ the maximal size of a pin-permutation of $B$,
the complexity of the algorithm is ${\mathcal O}(n \log n + s^{2k})$ or more precisely ${\mathcal O}(n \log n + p^{2})$.
\end{theo}

The complexity which is achieved in Theorem~\ref{thm:main_result} makes use of the 
optimized variant of the construction of the automata $\mathcal{A}_{\pi}$. 
Notice that with the non-optimized construction of the automata $\mathcal{A}_{\pi}$, 
although we would have an algorithm whose details are a little bit simpler to describe, its complexity 
would be significantly worst than with the optimized variant, 
namely ${\mathcal O}(n \log n + p^4) = {\mathcal O}(n \log n + s^{4k})$.

The algorithm announced in Theorem~\ref{thm:main_result} can be decomposed into several steps 
and is described in the rest of this section. 

\subsection{Finitely many parallel alternations and wedge simple permutations  in ${\mathcal C}$?}

Following~\cite{BRV06} (see Theorem~\ref{thm:brignall}
p.\pageref{thm:brignall}) we first check whether $\mathcal C$ contains
finitely many parallel alternations and wedge simple permutations.
From Lemmas~\ref{lem:alternation}, \ref{lem:wedge1} and \ref{lem:wedge2}
(p.\pageref{lem:wedge2}) this problem is equivalent to testing if permutations
of $B$ contain some patterns of size at most $4$. Using a result of~\cite{AAAH01}, this can be done in ${\mathcal O}(n \log n)$ time
(see Lemma~\ref{lem:complexity_nlogn} p.\pageref{lem:complexity_nlogn}).

\subsection{Finding pin-permutations in the basis}\label{ssec:finding_pin-perm}

The next step is to determine the subset $PB \subseteq B$ of pin-permutations of $B$.
To do so we use the characterization of the class of
pin-permutations by their decomposition trees established in~\cite{BBR09}, 
and recalled in Equation~\eqref{eq:pin_perm_trees} (p.\pageref{eq:pin_perm_trees}).

More precisely, for each $\pi \in B$, we proceed as follows.

\medskip

$\bullet$ First we compute its decomposition tree ${T}_{\pi}$.

This is achieved in linear time w.r.t.~$|\pi|$, computing first the skeleton
of ${T}_{\pi}$ following~\cite{BCMR05bis} or~\cite{BXHC05}, and next
the labels of linear and prime nodes as explained in~\cite[\S 2.2]{bcmr11}.

\medskip

$\bullet$ Second we add some information on the decomposition tree.

This information will be useful in later steps of our algorithm  
to check whether $\pi$ is a pin-permutation, 
and next (in the affirmative) to determine which construction of the automaton $\mathcal{A}_{\pi}$ 
(see Subsection~\ref{ssec:idees_construction_automates} or details in Appendix~\ref{sec:buildingAutomata}) applies to $\pi$.

\smallskip

- For  each prime node $N$, we record whether the simple
permutation $\alpha$ labeling $N$ is an 
increasing or decreasing
\epi or quasi-\epi.

This may be recorded
by performing a linear time depth-first traversal of
${T}_{\pi}$, and checking each node when it is reached.  As there are
$4$ \epis of each size that are explicitly described as
$2\,4\,1\,6\,3\,8\,5\ldots$ (see Figure~\ref{fig:episgrands} p.\pageref{fig:episgrands})
or one of its symmetries, checking if a simple permutation $\alpha$ is
of this form can be done in linear time w.r.t.~$|\alpha|$.  The same
kind of explicit description also holds for quasi-oscillations, and in
addition we can record which children correspond to the auxiliary
and main substitution points.

\smallskip

- For each node $N$, we record whether the subtree rooted at $N$ encodes an increasing or decreasing \epi.

This may be recorded easily, along the same depth-first traversal of ${T}_{\pi}$ as above. 
Indeed \epis of size greater than $3$ are simple permutations, and
increasing (resp. decreasing) \epis of smaller sizes are $1$, $21$, $231$ and $312$ (resp. $1$, $12$, $132$ and $213$).
So it is sufficient to check whether $N$ is a leaf,
or a prime node labeled by an increasing (resp. decreasing) \epi
all of whose children are leaves,
or a linear node with exactly two children satisfying extra constraints:
they are either both leaves, or one is a leaf and the second one is a linear node with exactly two children that are both leaves.
In this later case the oscillation is increasing (resp. decreasing) if $N$ is labeled $\ominus$ (resp. $\oplus$).

\smallskip

These computations are performed in linear time w.r.t.~$|\alpha|$ for
any prime node labeled by $|\alpha|$, and in constant time for any
linear node.  Hence,
as the sum of the sizes of the labels of all internal nodes is linear w.r.t.~$|\pi|$,
the overall complexity of this step is
linear w.r.t.~$|\pi|$.

\medskip

$\bullet$ Finally we determine whether $\pi$ is a pin-permutation or not.

To do so, we recursively check starting with the root whether its decomposition tree is of the shape described in~\cite{BBR09}
(see Equation~\eqref{eq:pin_perm_trees} p.\pageref{eq:pin_perm_trees}).

\smallskip

- If the root is linear, with the additional information stored we can check  whether
all its children are increasing (resp. decreasing) \epis in linear time w.r.t.~the number of children.
If exactly one child is not an increasing (resp. decreasing) \epi,
we check recursively whether the subtree rooted at this child is the decomposition tree of a pin-permutation.

\smallskip

- If the root is prime, we first check whether its label $\alpha$ is a {\em pin}-permutation.
More precisely, with Algorithm 2 of~\cite{BBPR10} we compute the set of pin words of $\alpha$ and test its emptiness.
By Lemma~4.1 of~\cite{BBPR10}, this is done in linear time w.r.t.~$|\alpha|$. Then we check whether all the children of the root are leaves.

\indent \indent - If exactly one child is not a leaf, we furthermore have to check whether the point $x$ it expands is an active point of $\alpha$. 
With some precisions given in the appendix, this can be done in $\mathcal{O}(|\alpha|)$ time. 
Namely, from Remark~\ref{rem:tailleQx} (p.\pageref{rem:tailleQx})
we just have to test the emptiness of $Q_x(\alpha)$,
which is computed in linear time w.r.t.~$|\alpha|$
(see Remark~\ref{rem:compute_Q_x(alpha)} p.\pageref{rem:compute_Q_x(alpha)}).
Then we check recursively whether the subtree rooted at~$x$ is the decomposition tree of a pin-permutation.

\indent \indent - If exactly two children are not leaves, with the additional information stored we can check in constant time whether
$\alpha$ is an increasing (resp. decreasing) quasi-\epi,
if the two children that are not leaves expand the auxiliary and main substitution points,
and if the one expanding the auxiliary point is the permutation $12$ (resp. $21$).
Then we check recursively whether the subtree rooted at the main substitution point is the decomposition tree of a pin-permutation.

\medskip

As the complexity of each step is linear w.r.t.~the number of children
(which is also the size of the label for a prime node), deciding
whether a permutation $\pi$ is a pin-permutation or not can be done in
linear time w.r.t.~$|\pi|$. The overall determination of $PB$ is therefore
linear in $n = \sum_{\pi \in B} |\pi|$.

\medskip

Moreover, in addition to computing $PB$, the above procedure produces additional results,
that we also record as they are useful in the next step.
Namely, for every permutation $\pi$ of $PB$, we record its decomposition tree $T_{\pi}$, together with the additional information computed on its nodes; and we also record
the set of pin words that encode each simple permutation $\alpha$ labeling a prime node $N$ of $T_{\pi}$ and the set $Q_x(\alpha)$ when $N$ has exactly one non-trivial child.
Notice that the knowledge of these is sufficient to characterize the set of pin words that encode $\pi$ thanks to results of Appendix~\ref{sec:pinwords} 
outlined in Subsection~\ref{ssec:idees_construction_automates}. 

\subsection{Finitely many proper pin-permutations in ${\mathcal C}$?}
\label{ssec:includes_looking_for_a_cycle_in_automaton}

From Theorem~\ref{thm:nbInfiniEtLangages}
(p.\pageref{thm:nbInfiniEtLangages}) it is enough to check whether $\M
\setminus \cup_{\pi \in B}{\mathcal L}_\pi$ is finite. This can be
easily decided with a deterministic automaton $\mathcal{A}_{\mathcal{C}}$ recognizing $\overleftarrow{\M \setminus \cup_{\pi \in B}{\mathcal L}_\pi}$.
From the previous step of the procedure, we know the set $PB$ of pin-permutations of $B$
and some additional results described above. 
First notice that $\cup_{\pi \in B}{\mathcal L}_\pi = \cup_{\pi \in PB}{\mathcal L}_\pi$
as ${\mathcal L}_\pi$ is empty when $\pi$ is not a pin-permutation (see p.\pageref{def:lu}).
We build the automaton $\mathcal{A}_{\mathcal{C}}$ as follows.

\smallskip

$\bullet$ First for each pin-permutation $\pi \in PB$, we construct
$\mathcal{A}_{\pi}$ -- which is deterministic and complete --
recognizing a language $\mathcal{L}'_{\pi}$ such that $\mathcal{L'}_{\pi} \cap \mathcal{M} = \overleftarrow{{\mathcal L}_{\pi}}\cap \mathcal{M}$. 
For this optimized variant, the construction is performed in time and space at most ${\mathcal O}(|\pi|^2)$ as 
presented in Subsection~\ref{ssec:idees_construction_automates} and 
described in details in Appendix~\ref{sec:buildingAutomata} 
(see Theorem~\ref{thm:complexity_new_statement} p.\pageref{thm:complexity_new_statement} 
or Theorem~\ref{thm:complexity-A_pi} p.\pageref{thm:complexity-A_pi}).
Notice that the construction of $\mathcal{A}_{\pi}$ depends on the shape of the decomposition tree $T_{\pi}$ of $\pi$.
But thanks to the additional information stored in $T_{\pi}$, 
we can determine which tree shape matches $T_{\pi}$ 
in linear time w.r.t.~the number of children of the root of $T_{\pi}$, 
and the same holds at each recursive step of the construction.

\smallskip

$\bullet$ Then we build a deterministic automaton ${\mathcal A}_1$
recognizing $\bigcup_{\pi \in PB} {\mathcal L}'_{\pi}$,
where ${\mathcal L}'_{\pi}$ is defined as in the first item. 
The automaton ${\mathcal A}_1$ is obtained
performing the deterministic union
(as a Cartesian product, see~\cite{HU79} for details)
of all the automata $\mathcal{A}_{\pi}$.
This is done in time and space
${\mathcal O}(\prod\limits_{\pi \in PB} |\mathcal{A}_{\pi}|) = {\mathcal O}(\prod\limits_{\pi \in PB} |\pi|^2)$.

\smallskip

$\bullet$ Then we build the automaton
${\mathcal A}_2$ which is the deterministic intersection (again as a Cartesian product) between
${\mathcal A}_1$ and the automaton ${\mathcal A(\M)}$ given in Figure~\ref{fig:AM}
in time and space $ 
\mathcal{O}(|{\mathcal A}_1|.|{\mathcal A} (\M))|) = {\mathcal O}(\prod\limits_{\pi \in PB} |\pi|^2)$.

\begin{figure}[ht]
\begin{center}
\begin{tikzpicture}[node distance=40pt,inner sep=2pt,shorten >= 2pt,every state/.style={draw,thick,minimum size=3pt},/tikz/initial text=]
{\footnotesize
\node[state,initial] at (0,0) (q_00) {};
\node[state] at (1,0.7) (q_01) {}; 
\node[state,accepting] (q_03) [right =of q_01] {};
\node[state] at (1,-0.7) (q_02) {}; 
\node[state,accepting] (q_04) [right =of q_02] {};
\path[->] (q_00) edge node[above,sloped] {\scriptsize $U,D$} (q_01);
\path[->] (q_01) edge node[above,sloped] {\scriptsize $L,R$} (q_03);
\path[->] (q_00) edge node[below,sloped] {\scriptsize $L,R$} (q_02);
\path[->] (q_02) edge node[below,sloped] {\scriptsize $U,D$} (q_04);
\path[->] (q_03) edge [bend left=30] node[right] {\scriptsize $U,D$} (q_04);
\path[->] (q_04) edge [bend left=30] node[left] {\scriptsize $L,R$} (q_03);
}
\end{tikzpicture}
\caption{A deterministic automaton ${\mathcal A}(\M)$ recognizing the set $\M$ of words of length at least $2$ without any factor in {\footnotesize $\{UU,UD,DU,DD,RR,RL,LR,LL\}$}. } \label{fig:automates_stricts_pinwords} \label{fig:AM}
\end{center}
\end{figure}
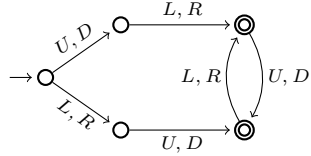
\noindent The automaton ${\mathcal A}_2$ recognizes $\Big( \bigcup_{\pi \in PB} {\mathcal L}'_{\pi} \Big) \cap \M = \Big( \bigcup_{\pi \in PB} \overleftarrow{{\mathcal L}_{\pi}} \Big) \cap \M$.
By Lemma~\ref{lem:patterns} (p.\pageref{lem:patterns}) this is the language of words $\overleftarrow{\phi(w)}$
for all strict pin words $w$ encoding permutations having a pattern in $PB$, \emph{i.e.} that are not in ${\mathcal C}$.
Notice that by Remark~\ref{rem:proper_strict} (p.\pageref{rem:proper_strict}) such permutations are necessarily proper pin-permutations.

\smallskip

$\bullet$ Next we complement ${\mathcal A}_2$ 
to build a deterministic automaton
${\mathcal A}_3$ recognizing $\Astar \setminus \Big( \big(
\bigcup_{\pi \in PB} \overleftarrow{{\mathcal L}_{\pi}} \big) \cap \M
\Big)$.  As ${\mathcal A}_2$ is deterministic, its complement is
obtained in linear time w.r.t.~its size,
by completing it and then turning every final (resp. non-final) state into a non-final (resp. final) state. 
Moreover the size of ${\mathcal A}_3$ is the same as that of the
automaton obtained completing ${\mathcal A}_2$, \emph{i.e.}, ${\mathcal
O}(\prod\limits_{\pi \in PB} |\pi|^2)$

$\bullet$ Finally we compute the deterministic intersection between ${\mathcal A}_3$ and the automaton ${\mathcal A(\M)}$ to obtain the automaton ${\mathcal A}_{\mathcal{C}}$.
This is done in time and space $\mathcal{O}(|{\mathcal A}_3|.|{\mathcal A} (\M))|) = {\mathcal O}(\prod\limits_{\pi \in PB} |\pi|^2)$.
The automaton ${\mathcal A}_{\mathcal{C}}$ built in this way recognizes $\M \setminus \Big( \bigcup_{\pi \in PB} \overleftarrow{{\mathcal L}_{\pi}} \Big) = 
\overleftarrow{\M \setminus \Big( \bigcup_{\pi \in B} {\mathcal L}_{\pi} \Big)}$.
This is the language of all words $\overleftarrow{\phi(w)}$ where $w$ is a strict pin word encoding a permutation of $\mathcal C$ (that is necessarily a proper pin-permutation, as above).

Then, by Theorem~\ref{thm:nbInfiniEtLangages} (p.\pageref{thm:nbInfiniEtLangages}),
checking whether the
permutation class $\mathcal{C}$ contains a finite number of proper
pin-permutations is equivalent to checking whether the language
recognized by ${\mathcal A}_{\mathcal{C}}$ is finite {\it i.e.},
whether ${\mathcal A}_{\mathcal{C}}$ does not contain any cycle that is
accessible and co-accessible ({\it i.e.}, a cycle that can be reached from an
initial state and from which a final state can be reached).
The automaton ${\mathcal A}_{\mathcal{C}}$ is not necessarily accessible and co-accessible.
Its accessible part is made of all states that can be reached in a traversal of the automaton from the initial state;
its co-accessible part is obtained similarly by a traversal from the set of final states taking the edges of the automaton backwards.
Before looking for a cycle, we make ${\mathcal A}_{\mathcal{C}}$ accessible and co-accessible
by keeping only its accessible and co-accessible part, yielding a smaller automaton ${\mathcal A}'_{\mathcal{C}}$.
The complexity of this double reduction of the size of the automaton is linear in time w.r.t.~the size of ${\mathcal A}_{\mathcal{C}}$.
Moreover the size of ${\mathcal A}'_{\mathcal{C}}$ is smaller
than or equal to the one of ${\mathcal A}_{\mathcal{C}}$, \emph{i.e.}, ${\mathcal O}(\prod\limits_{\pi \in PB} |\pi|^2)$.
Finally we test whether
${\mathcal A}'_{\mathcal{C}}$ does not contain any cycle.  This can be done
 in $\mathcal{O}(|{\mathcal A}'_{\mathcal{C}}|)$ time with a
depth-first traversal of ${\mathcal A}'_{\mathcal{C}}$.

Let $s$ be the maximal size of a pin-permutation of $B$ and $k$ the number of pin-permutations in $B$,
then ${\mathcal O}(\prod\limits_{\pi \in PB} |\pi|^2) = {\mathcal O}(s^{2k})$.
Hence putting all these steps together leads to an algorithm whose
complexity is ${\mathcal O}(s^{2k})$ to check whether there are finitely many proper pin-permutations in $\mathcal{C}$,
when the set $PB$ of pin-permutations of $B$, their decomposition trees and the set of pin words of each simple permutation appearing in these trees are given. 

\section{Conclusion}
\label{sec:ccl}

The work reported here follows the line opened by~\cite{AA05} and
continued by~\cite{BRV06}. In~\cite{AA05}, the main theorem provides
(in particular) a sufficient condition for a permutation class
$\mathcal{C}$ to have an algebraic generating function: namely, that
$\mathcal{C}$ contains a finite number of simple permutations. Then,
\cite{BRV06} introduces new objects (most importantly,
pin-permutations) to provide a decision procedure testing this
sufficient condition, for classes with a finite and explicit
basis. Making use of the detailed study of pin-permutations in
\cite{BBR09}, we have described in the above an algorithm
testing this condition. The analysis of its complexity shows that it is efficient.

Because an algebraic generating function is a witness of the combinatorial structure of a permutation class,
we may interpret our result as giving an efficient algorithm testing a sufficient condition for a permutation class to be well-structured.
We believe that more could and should be done on the algorithmization of finding structure in permutation classes.
In particular, we plan to provide efficient algorithms that do not only {\em test} that there is an underlying structure in a permutation class,
but that also {\em compute} this structure.
We set in the sequel the main steps towards the achievement of this project.

\medskip

As discussed in~\cite{AA05}, the proof of the main theorem therein is constructive.
Namely, given the basis $B$ of a class $\mathcal{C}$,
and the set $\mathcal{S}_{\mathcal{C}}$ of simple permutations in $\mathcal{C}$
(assuming that both are finite), the proof of the main theorem of~\cite{AA05}
describes how to compute (a polynomial system satisfied by) the generating function of $\mathcal{C}$, proving thereby that it is algebraic. 
The main step is actually
to compute a (possibly ambiguous) context-free grammar of trees for the
permutations of $\mathcal{C}$, or rather their decomposition trees.

Such a context-free grammar of trees almost captures the combinatorial structure of a permutation class.
The only reason why it does not completely is because the grammar may be ambiguous, and thus may generate several times the same permutation in the class.
On the contrary, unambiguous context-free grammars of trees fall exactly in the context of the {\em combinatorial specifications} of~\cite{Flajolet}, and
describing a permutation class by such a combinatorial specification is undoubtedly demonstrating the structure of the class.
Consequently, we aim at describing an algorithm to compute this combinatorial specification, assuming we are given the finite basis $B$ characterizing the class $\mathcal{C}$.
There would be four main steps in such an algorithm.

First, we should ensure that $\mathcal{C}$ falls into the set of permutation classes we can handle, \emph{i.e.}, ensure that $\mathcal{C}$ contains a finite number of simple permutations.
The present work gives an algorithm for this first step.

Second, when finite, we should compute the set $\mathcal{S}_{\mathcal{C}}$ of simple permutations in $\mathcal{C}$. 
A naive method to do so can be immediately deduced from the results of~\cite{AA05}, but it
is of highly exponential complexity. An algorithm for this second step has subsequently been described in~\cite{PR11}, and its complexity analyzed.
It should be noticed that the complexity of this algorithm also depends on the size of its output, namely on $|\mathcal{S}_{\mathcal{C}}|$ and on $\max \{|\pi| : \pi \in \mathcal{S}_{\mathcal{C}}\}$.

Third, from $B$ and $\mathcal{S}_{\mathcal{C}}$, we should turn the constructive proof of~\cite{AA05} into an actual algorithm,
that would compute the (possibly ambiguous) context-free grammar of trees describing the decomposition trees of the permutations of $\mathcal{C}$.

Finally, we should transform this (possibly ambiguous) context-free grammar into an unambiguous combinatorial specification for $\mathcal{C}$.
We have described in the extended abstract~\cite{fpsac2012} an algorithm for these last two steps, whose complexity is still to analyze.

\smallskip

Combining these four steps will provide an algorithm  to obtain from a basis $B$ of excluded patterns a combinatorial specification for the permutation class $\mathcal{C} = Av(B)$.
We are not only convinced of the importance of this result from a theoretical point of view,
but also (and maybe more importantly) we are confident that it will be of practical use to the permutation patterns community.
Indeed, from a combinatorial specification, it is of course possible with the
methodology of~\cite{Flajolet} to immediately deduce a system of
equations for the generating function of $\mathcal{C}$.
But other algorithmic developments can be considered. In particular,
this opens the way to obtaining systematically uniform random
samplers of permutations in a class, or to the automatic evaluation of
the Stanley-Wilf growth rate of a class.

\bigskip

\paragraph*{Acknowledgments} 
We are very grateful to the anonymous referee for providing both specific comments and global suggestions on the organization of our paper. 
These helped us improve the presentation of our work, on both the large and small scales. 
We would also like to thank Joseph Kung for his availability and efficiency as an editor. 

\bigskip

\begin{appendices}

\section*{Appendices}

\section{Simple pin-permutations, oscillations and quasi-oscillations}
\label{sec:oscillations}

This appendix groups together some technical definitions and results about subsets of pin-permutations: 
the simple ones, the oscillations, and the quasi-oscillations. 
The last two play an important role in the characterization 
of substitution decomposition trees associated with pin-permutations 
(see Equation~\eqref{eq:pin_perm_trees} p.\pageref{eq:pin_perm_trees}).

\subsection{Simple pin-permutations, active knights and active points}

Let $\sigma$ be a simple pin-permutation. 
We have seen from Remark~\ref{rem:simplepin} (p.\pageref{rem:simplepin}) that all pin representations 
of $\sigma$ are proper. 
This implies in particular (see~\cite[Lemma 4.3]{BBR09} for an immediate proof) 
that the first two points in every pin representation of $\sigma$ are in \emph{knight position}, 
\emph{i.e.}, form one of the following configurations in the diagram of $\sigma$: 
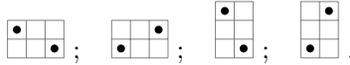
\[
\begin{tikzpicture}[scale=.25] \draw [help lines] (0,0) grid (3,2);
\fill (0.5,1.5) circle (6pt); \fill (2.5,0.5) circle (6pt);
\end{tikzpicture} \ ; \quad 
\begin{tikzpicture}[scale=.25] \draw [help lines] (0,0) grid (3,2);
\fill (0.5,0.5) circle (6pt); \fill (2.5,1.5) circle (6pt);
\end{tikzpicture} \ ; \quad 
\begin{tikzpicture}[scale=.25] \draw [help lines] (0,0) grid (2,3);
\fill (0.5,2.5) circle (6pt); \fill (1.5,0.5) circle (6pt);
\end{tikzpicture} \ ; \quad 
\begin{tikzpicture}[scale=.25] \draw [help lines] (0,0) grid (2,3);
\fill (0.5,0.5) circle (6pt); \fill (1.5,2.5) circle (6pt);
\end{tikzpicture} \ . 
\]

We define an {\em active knight} of $\sigma$ to be a pair of points of $\sigma$ in knight position 
which is the possible start of a pin representation of $\sigma$. 
The definition of active knights may be extended to pin-permutations $\sigma$ that are not necessarily simple, 
as pairs of points of $\sigma$ that are the possible start of a pin representation of $\sigma$. 
In addition to the four configurations shown above, such active ``knights'' of non-simple pin-permutations 
may form a configuration \begin{tikzpicture}[scale=.2] \draw [help lines] (0,0) grid (2,2);
\fill (0.5,0.5) circle (6pt); \fill (1.5,1.5) circle (6pt);
\end{tikzpicture} or \begin{tikzpicture}[scale=.2] \draw [help lines] (0,0) grid (2,2);
\fill (0.5,1.5) circle (6pt); \fill (1.5,0.5) circle (6pt);
\end{tikzpicture}
in the diagram of $\sigma$. 
We also define an {\em active point} of $\sigma$ to be a point of $\sigma$ belonging to an active knight of $\sigma$. \label{def:active} 
The active knights (and hence the active points) of any simple pin-permutation may be described (see~\cite[Lemma 4.6]{BBR09}). 
This is however not needed for us in this work, except in the case of oscillations, which we will review in the following. 

\subsection{Oscillations}
\label{ssec:oscillations}

Following~\cite{BRV06}, let us consider the infinite oscillating
sequence defined by $\omega = 3\ 1\ 5\ 2\ 7\ 4\ 9\ 6\ \ldots
(2k+1)\ (2k-2)\ \ldots$. The leftmost part of Figure~\ref{fig:episgrands} shows the
diagram of a prefix of $\omega$.

\begin{defi}
  An {\em increasing \epi} of size $n\geq 4$ is a simple permutation
  of size $n$ that is contained as a pattern in $\omega$. For smaller
  sizes the increasing \epis are $1$, $21$, $231$ and $312$. A {\em
    decreasing \epi} is the reverse\footnote{The reverse of
    $\sigma=\sigma_1\sigma_2\ldots \sigma_n$ is
    $\overleftarrow{\sigma}=\sigma_n\ldots \sigma_2\sigma_1$.} of an increasing
  \epi. 
\label{defn:epi}
\end{defi}

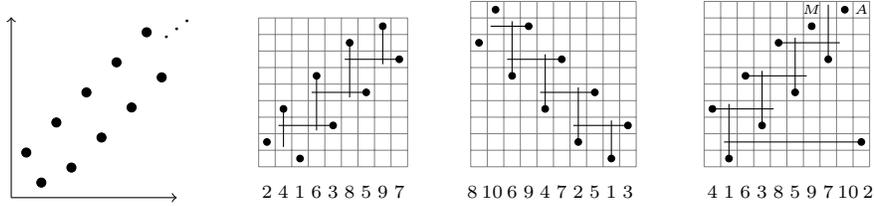
\begin{figure}[ht]
\begin{center}
\begin{tikzpicture}[scale=0.2]
\draw (12,0) node {~};
\draw[->] (0,0) -- (0,12);
\draw[->] (0,0) -- (11,0);
\draw (1,3) [fill] circle (.3);
\draw (2,1) [fill] circle (.3);
\draw (3,5) [fill] circle (.3);
\draw (4,2) [fill] circle (.3);
\draw (5,7) [fill] circle (.3);
\draw (6,4) [fill] circle (.3);
\draw (7,9) [fill] circle (.3);
\draw (8,6) [fill] circle (.3);
\draw (9,11) [fill] circle (.3);
\draw (10,8) [fill] circle (.3);
\draw (11.5,10) node {\begin{rotate}{75} $\ddots$ \end{rotate}};
\end{tikzpicture}
\hspace{0.2 cm}
\begin{tikzpicture}[scale=0.22]
\permutationNomGros{2,4,1,6,3,8,5,9,7}
        \draw (0,12.5) node {};
        \draw (12.5,3) node {};
	\draw (2.5,4.5) -- (2.5,2.2);
	\draw (5.5,3.5) -- (2.2,3.5);
	\draw (4.5,6.5) -- (4.5,3.2);
	\draw (7.5,5.5) -- (4.2,5.5);
	\draw (6.5,8.5) -- (6.5,5.2);
	\draw (9.5,7.5) -- (6.2,7.5);
	\draw (8.5,9.5) -- (8.5,7.2);
\end{tikzpicture}
\hspace{-0.4 cm}
\begin{tikzpicture}[scale=0.22]
\permutation{8,10,6,9,4,7,2,5,1,3}
        \draw (0,12.5) node {};
        \draw (12.5,3) node {};
	\draw (4.5,9.5) -- (2.2,9.5);
	\draw (3.5,6.5) -- (3.5,9.8);
	\draw (6.5,7.5) -- (3.2,7.5);
	\draw (5.5,4.5) -- (5.5,7.8);
	\draw (8.5,5.5) -- (5.2,5.5);
	\draw (7.5,2.5) -- (7.5,5.8);
	\draw (10.5,3.5) -- (7.2,3.5);
	\draw (9.5,1.5) -- (9.5,3.8);
	\draw (1.1,-0.5) node {\scriptsize $8$};
	\draw (2.3,-0.5) node {\scriptsize $10$};
	\draw (3.5,-0.5) node {\scriptsize $6$};
	\draw (4.5,-0.5) node {\scriptsize $9$};
	\draw (5.5,-0.5) node {\scriptsize $4$};
	\draw (6.5,-0.5) node {\scriptsize $7$};
	\draw (7.5,-0.5) node {\scriptsize $2$};
	\draw (8.5,-0.5) node {\scriptsize $5$};
	\draw (9.5,-0.5) node {\scriptsize $1$};
	\draw (10.5,-0.5) node {\scriptsize $3$};
\end{tikzpicture}
\begin{tikzpicture}[scale=0.22]
\permutation{4,1,6,3,8,5,9,7,10,2}
        \draw (0,12.5) node {};
        \draw (12.5,3) node {};
        \draw (7.5,10.5) node {\tiny{$M$}};
        \draw (10.5,10.5) node {\tiny{$A$}};
	\draw (1.5,4.5) -- (5.2,4.5);
	\draw (2.5,1.5) -- (2.5,4.8);
	\draw (3.5,6.5) -- (7.2,6.5);
	\draw (4.5,3.5) -- (4.5,6.8);
	\draw (5.5,8.5) -- (9.2,8.5);
	\draw (6.5,5.5) -- (6.5,8.8);
	\draw (8.5,7.5) -- (8.5,10.8);
	\draw (10.5,2.5) -- (2.2,2.5);
	\draw (1.5,-0.5) node {\scriptsize $4$};
	\draw (2.5,-0.5) node {\scriptsize $1$};
	\draw (3.5,-0.5) node {\scriptsize $6$};
	\draw (4.5,-0.5) node {\scriptsize $3$};
	\draw (5.5,-0.5) node {\scriptsize $8$};
	\draw (6.5,-0.5) node {\scriptsize $5$};
	\draw (7.5,-0.5) node {\scriptsize $9$};
	\draw (8.5,-0.5) node {\scriptsize $7$};
	\draw (9.7,-0.5) node {\scriptsize $10$};
	\draw (10.9,-0.5) node {\scriptsize $2$};
\end{tikzpicture}
\caption{The infinite oscillating sequence $\omega$, an increasing \epi $\xi$ of size $9$, a decreasing \epi of size
  $10$, and an increasing quasi-\epi of size $10$ 
  (obtained from $\xi$ by addition of a maximal element
  or equivalently by taking the inverse of the explicit quasi-\epi of size $10$ given in Subsection~\ref{ssec:quasi-oscillations}), 
  with a pin representation for each.}
\label{fig:episgrands}
\end{center}
\end{figure}

There are two increasing 
\epis of any size greater than or equal to $3$, that can also be given explicitly. 
For even size, they are 
\begin{align*}
& 2 \, 4 \, 1 \, 6 \, 3 \ldots (2k+2) \, (2k-1) \ldots (2n)\, (2n-3)\, (2n-1)  \text{ and } \\
& 3 \, 1\,  5 \, 2 \, 7\,  4 \ldots (2k+1)\,  (2k-2) \ldots (2n)\, (2n-2);
\end{align*}
for odd size,
\begin{align*}
& 2 \, 4 \, 1\,  6\,  3 \ldots (2k+2)\,  (2k-1) \ldots (2n)\, (2n-3)\, (2n+1)\, (2n-1) \text{ and } \\
& 3 \, 1 \, 5 \, 2 \, 7 \, 4 \ldots (2k+1)\,  (2k-2) \ldots (2n+1)\, (2n-2)\, (2n). 
\end{align*}
A similar statement holds for decreasing oscillations. 
As noticed in~\cite[Lemma 2.23]{BBR09}, every increasing (resp.~decreasing) oscillation 
is a pin-permutation. 
Moreover, by definition all oscillations of size at least $4$ are simple. 
Finally, notice also that permutations $1$, $2\, 4\, 1\, 3$ and $3\, 1\, 4\, 2$ 
are both increasing and decreasing \epis, and are the only ones with this property. 
The middle two diagrams of Figure~\ref{fig:episgrands} show some examples of oscillations. 

\medskip

In Appendix~\ref{sec:pinwords}, we will describe explicitly the set of pin words of all pin-per\-mu\-ta\-tions. 
In Subsection~\ref{subsec:pin_word_simple}, this requires some knowledge about the pin words of oscillations 
(w.l.o.g., only increasing oscillations) -- see in particular Lemmas~\ref{lem:f1f3} to~\ref{lem:fdouble}. 
In these lemmas, we have to distinguish cases according to the active knights of the increasing oscillations. 
For this reason, we review in the sequel some results from~\cite{BBR09} about the active knights of increasing oscillations. 

\begin{figure}[hbtp]
\begin{center}
\begin{tikzpicture}
\begin{scope}[scale=0.2]
\permutationNomKnight{1}{}
\end{scope}
\begin{scope}[scale=0.25,xshift=60]
\permutationNomKnight{2,1}{1/2/2/1}
\end{scope}
\begin{scope}[scale=0.25,xshift=150]
\permutationNomKnight{2,3,1}{1/2/3/1,2/3/3/1,1/2/2/3}
\end{scope}
\begin{scope}[scale=0.25,xshift=270]
\permutationNomKnight{3,1,2}{1/3/2/1,1/3/3/2,2/1/3/2}
\end{scope}
\begin{scope}[scale=0.25,xshift=390]
\permutationNomKnight{2,4,1,3}{1/2/2/4,1/2/3/1,2/4/4/3,3/1/4/3}
\end{scope}
\begin{scope}[scale=0.25,xshift=540]
\permutationNomKnight{3,1,4,2}{1/3/2/1,1/3/3/4,2/1/4/2,3/4/4/2}
\end{scope}
\begin{scope}[scale=0.25,xshift=700]
\permutationNomKnight{3,1,5,2,7,4,8,6}{1/3/2/1,7/8/8/6}
\end{scope}
\begin{scope}[scale=0.25,xshift=970]
\permutationNomKnight{3,1,5,2,7,4,9,6,8}{1/3/2/1,7/9/9/8}
\end{scope}
\end{tikzpicture}
\caption{The increasing oscillations of size less than $5$ and
  two increasing oscillations respectively of size $8$ with type $(V,V)$
  and $9$ with type $(V,H)$. Active knights are marked by edges between their two active points.}
\label{fig:lecture_epi_quadrant_1ou3}
\end{center}
\end{figure}
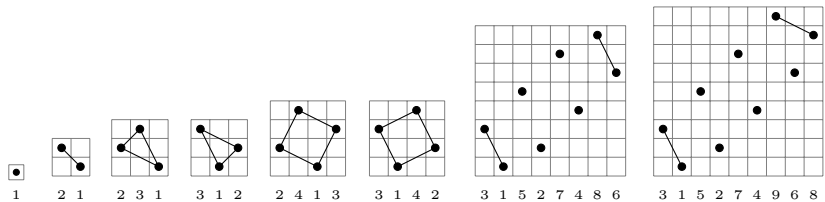

Figure~\ref{fig:lecture_epi_quadrant_1ou3} shows the active knights 
of the increasing oscillations of size up to $4$. 
Lemma 4.6 of~\cite{BBR09} describes the active knights
of simple pin-permutations, and in particular those of the increasing oscillations of size at least $4$. 
It follows from this lemma that an increasing
oscillation of size at least $5$ has exactly two active knights. They are
located at both ends of the main diagonal and they consist of two points
in relative order $21$  (see Figure~\ref{fig:lecture_epi_quadrant_1ou3}).
These active knights are either in horizontal ($H$) position
\begin{tikzpicture}[scale=.15] \draw [help lines] (0,0) grid (3,2);
\fill (0.5,1.5) circle (6pt); \fill (2.5,0.5) circle (6pt);
\end{tikzpicture} or in vertical ($V$) position
\begin{tikzpicture}[scale=.15] \draw [help lines] (0,0) grid (2,3);
\fill (0.5,2.5) circle (6pt); \fill (1.5,0.5) circle (6pt);
\end{tikzpicture}. Therefore there are four types of increasing oscillations of size at least $5$:
$(x,y)$ with $x, y \in \{H,V\}$, where $x$ is the type of the lower left active
knight and $y$ for the upper right. This definition can be extended to
increasing oscillations of size $4$, considering their two active knights
in relative order $21$ (see Figure~\ref{fig:lecture_epi_quadrant_1ou3}).
Note that an even size oscillation has
type $(H,H)$ or $(V,V)$ and an odd size one $(H,V)$ or $(V,H)$.  

\subsection{Quasi-oscillations}
\label{ssec:quasi-oscillations}

We recall the definition of quasi-\epis from~\cite{BBR09}. 

\begin{defi}\label{defn:quasiepi}
  An {\em increasing quasi-\epi} of size $n\geq 6$ is obtained from an
  increasing \epi $\xi$ of size $n-1$ by the addition of either a
  minimal element at the beginning of $\xi$ or a maximal element at
  the end of $\xi$, followed by the move of an element of $\xi$
  according to the rules of Table~\ref{table:def_qepis}\footnote{
    The first row of Table~\ref{table:def_qepis} reads as follows:
    If a maximal element is added to
    $\xi$, with $\xi \in S_{n-1}$ starting (resp. ending) with a pattern $231$
    (resp. $132$), then the corresponding increasing quasi-\epi
    $\beta$ is obtained by moving the left-most point of $\xi$ so that it becomes 
    the right-most (in $\beta$), and the main substitution point is
    the largest point of $\xi$ (see the rightmost diagram of Figure~\ref{fig:episgrands}). 
}.
\begin{table}[ht]
\centering\small
\begin{tabular}{|c|c|c|c|c|c|}
\hline
Element & Pattern & Pattern  & Element & \ldots which & Main subs-\\
inserted &  $\xi_1\xi_2\xi_3$ &    {\small $\xi_{n{-}3}\xi_{n{-}2}\xi_{n{-}1}$}  & to move \ldots & becomes & titution point \\
\hline
max & $231$ & $132$ & left-most & right-most & largest\\
\hline
max & $231$ & $312$ & left-most & right-most & right-most \\
\hline
max & $213$ & $132$ & smallest & largest & largest \\
\hline
max & $213$ & $312$ & smallest & largest & right-most \\
\hline
min & $231$ & $132$ & largest & smallest & left-most \\
\hline
min & $231$ & $312$ & right-most & left-most & left-most \\
\hline
min & $213$ & $132$ & largest & smallest & smallest \\
\hline
min & $213$ & $312$ & right-most & left-most & smallest \\
\hline
\end{tabular}
\caption{Building quasi-\epis from \epis, and defining their main substitution points.}
\label{table:def_qepis}
\end{table}

We define the \emph{auxiliary point} ($A$) 
to be the point added to $\xi$, and the {\em main substitution 
point} ($M$) to be an extremal point of $\xi$ according to Table~\ref{table:def_qepis}.

Furthermore, for $n= 4$ or $5$, there are two increasing quasi-\epis
of size $n$: $2\, 4\, 1\, 3$, $3\, 1\, 4\, 2$, $2\, 5\, 3 \, 1\, 4$
and $4\, 1\, 3 \, 5\, 2$.
Each of them has two possible choices for
its pair of auxiliary and main substitution points.
See Figure~\ref{fig:quasiepi} for more details.
Finally, a {\em decreasing quasi-\epi} is the reverse of an increasing quasi-\epi.
\end{defi}

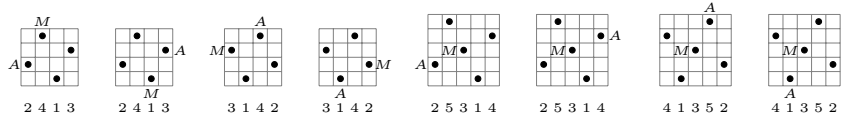
\begin{figure}[ht]
\begin{center}
\begin{tikzpicture}[scale=.19]
	\useasboundingbox (-0.5,0) rectangle (6,6);
        \permutationNom{2,4,1,3}
        \draw (0,5.5) node {};
        \draw (2.5,5.5) node {\tiny{$M$}};
        \draw (0.5,2.5) node {\tiny{$A$}};
        \end{tikzpicture}
\begin{tikzpicture}[scale=.19]
        \useasboundingbox (0,0) rectangle (6.5,6);
        \permutationNom{2,4,1,3}
        \draw (0,5.5) node {};
        \draw (5.5,3.5) node {\tiny{$A$}};
        \draw (3.5,0.5) node {\tiny{$M$}};
        \end{tikzpicture}
\begin{tikzpicture}[scale=.19]
	\useasboundingbox (-0.5,0) rectangle (6,6);
        \permutationNom{3,1,4,2}
        \draw (0,5.5) node {};
        \draw (0.5,3.5) node {\tiny{$M$}};
        \draw (3.5,5.5) node {\tiny{$A$}};
        \end{tikzpicture}
\begin{tikzpicture}[scale=.19]
	\useasboundingbox (0,0) rectangle (6.5,6);
        \permutationNom{3,1,4,2}
        \draw (0,5.5) node {};
        \draw (5.5,2.5) node {\tiny{$M$}};
        \draw (2.5,0.5) node {\tiny{$A$}};
        \end{tikzpicture}
\begin{tikzpicture}[scale=.19]
	\useasboundingbox (-0.5,0) rectangle (7,7);
        \permutationNom{2,5,3,1,4}
        \draw (0,6.5) node {};
        \draw (2.5,3.5) node {\tiny{$M$}};
        \draw (0.5,2.5) node {\tiny{$A$}};
\end{tikzpicture}
\begin{tikzpicture}[scale=.19]
	\useasboundingbox (0,0) rectangle (7.5,7);
        \permutationNom{2,5,3,1,4}
        \draw (0,6.5) node {};
        \draw (2.5,3.5) node {\tiny{$M$}};
        \draw (6.5,4.5) node {\tiny{$A$}};
\end{tikzpicture}
\begin{tikzpicture}[scale=.19]
	\useasboundingbox (-0.5,0) rectangle (7,7);
        \permutationNom{4,1,3,5,2}
        \draw (0,6.5) node {};
        \draw (2.5,3.5) node {\tiny{$M$}};
        \draw (4.5,6.5) node {\tiny{$A$}};
\end{tikzpicture}
\begin{tikzpicture}[scale=.19]
	\useasboundingbox (0,0) rectangle (7,7.5);
        \permutationNom{4,1,3,5,2}
        \draw (0,6.5) node {};
        \draw (2.5,3.5) node {\tiny{$M$}};
        \draw (2.5,0.5) node {\tiny{$A$}};
\end{tikzpicture} 

\caption{The diagrams of the increasing quasi-\epis of size $4$ and $5$, 
where auxiliary ($A$) and main ($M$) substitution points are marked.}
\label{fig:quasiepi}
\end{center}
\end{figure}

In particular, it follows from Definition~\ref{defn:quasiepi} that 
the auxiliary point of increasing quasi-oscillations of size at least $6$ is uniquely determined, 
whereas there are two possible choices of auxiliary point in increasing quasi-oscillations of size $4$ and $5$. 

As noticed in~\cite{BBR09} there are four increasing (resp. decreasing)
quasi-\epis of size $n$ for any $n\geq 6$, two of size $4$ ($2\, 4\, 1\,
3$ and $3\, 1\, 4\, 2$) and two of size $5$ ($2\, 5\, 3\, 1\,
4$ and $4\, 1\, 3\, 5\, 2$). 
Moreover, every quasi-oscillation is a simple pin-permutation. 

\medskip

Like oscillations, the quasi-oscillations of size at least $6$ may also be defined explicitly. 
Namely, one increasing quasi-oscillation is
{\small 
\begin{align*}
& 4\, 1\, 6\, 3 \ldots (2k+2)\, (2k-1) \ldots (2n-2)\, (2n-5)\, \underbrace{(2n-1)}_{M}\, (2n-3)\, \underbrace{(2n)}_{A}\, 2  \text{ {\normalsize for even size}} \\
& 4\, 1\, 6\, 3 \ldots (2k+2)\, (2k-1) \ldots (2n-5)\, (2n)\, (2n-3)\, \underbrace{(2n-1)}_{M}\, \underbrace{(2n+1)}_{A}\, 2 \text{ {\normalsize for odd size,}}
\end{align*}
}
where $M$ (resp. $A$) indicates the main substitution point (resp. the auxiliary point). 
The other three increasing quasi-oscillations are obtained applying some symmetries 
to the diagram of the above permutation $\sigma$, namely reflexion according to its two diagonals. 
In other words, the four increasing quasi-oscillations are $\sigma$, 
its so-called \emph{reverse-complement} $\sigma^{rc}$, 
and their inverses $\sigma^{-1}$ and $(\sigma^{rc})^{-1}$. 
The definition of the auxiliary and main substitution points follows along the application of these symmetries. 

\medskip

It should be noticed that each quasi-oscillation of size $4$ or $5$ is both increasing and decreasing.
However, once its auxiliary point is chosen among the four possibilities,
then its nature (increasing or decreasing) is determined without ambiguity,
and so is its main substitution point.
Moreover, knowing the (unordered) pair of points which are the auxiliary and main substitution points,
we can deduce which one is the auxiliary point without ambiguity.

\medskip

We conclude this paragraph about quasi-oscillations with a remark 
on the number of their active knights which involve (one of) their auxiliary point(s). 
This information will be useful in the proof of Lemma~\ref{lem:w_at_least_3} (p.\pageref{lem:w_at_least_3}). 

\begin{rem}
The auxiliary point of an increasing quasi-oscillation of size $n$ 
(or any of its auxiliary points, in case $n=4$ or $5$) 
belongs to exactly one active knight if $n  \neq 4$, 
and to exactly two active knights if $n=4$. 
\label{rem:knights_quasi-oscillations}
\end{rem}

\begin{proof}
Consider first increasing quasi-oscillations of size greater than $5$. 
From Lemma 4.6 of~\cite{BBR09} (see also the last diagram of Figure~\ref{fig:episgrands}),
the main substitution point belongs to exactly two active knights --
one formed with the auxiliary point and one formed with the point separating it from the auxiliary point --
and there are no other active knights. 

Consider now an increasing quasi-oscillation of size $4$ or $5$ (see Figure~\ref{fig:quasiepi})
where an auxiliary point $x$ is chosen. 
We may also apply to Lemma 4.6 of~\cite{BBR09} to count its active knights that involve $x$. 
Namely, an increasing quasi-oscillation of size $5$ has exactly $4$ active knights,
all of them contain the main substitution point (which is uniquely determined, regardless of the choice of $x$), 
and exactly one of them contains the auxiliary point $x$. 
Finally, an increasing quasi-oscillation of size $4$ has exactly $4$ active knights and each of its points (including the auxiliary point $x$)
belongs to exactly two active knights.
\end{proof}

We refer the reader to~\cite{BBR09}  for further properties of
oscillations and quasi-oscillations.

\section{Pin words of  pin-permutations}\label{sec:pinwords}

Our goal here is to describe the set $P(\pi)$ of pin
words that encode a pin-permutation $\pi$, 
following the recursive characterization of the decomposition trees of pin-per\-mu\-ta\-tions
that is given by Equation~\eqref{eq:pin_perm_trees} (p.\pageref{eq:pin_perm_trees}).

As outlined in Subsection~\ref{ssec:idees_construction_automates}, 
the characterization of $P(\pi)$ we provide is naturally divided into several cases,
depending on which term of Equation~\eqref{eq:pin_perm_trees} $\pi$ belongs to. First, we study the
non-recursive cases, then the recursive cases with a linear root and
finally the recursive cases with a prime root.
We start with a preliminary study of the ways children of
decomposition trees with linear root can be read in a pin
representation. These first results will be useful both in the
non-recursive and the recursive cases.

\begin{rem}\label{rem:ominus=oplus_transposed}
In the study that follows, we never examine the case of decomposition trees with a linear root labeled by $\ominus$. 
Indeed, permutations with decomposition trees of this form are the reverse of permutations whose decomposition trees have a linear root labeled by $\oplus$, 
and every argument and result on the $\oplus$ case can therefore be transposed to the $\ominus$ case. 
A similar symmetry holds for decomposition trees with prime roots labeled by increasing (resp. decreasing) 
quasi-oscillations and two children that are not leaves. 
\end{rem}

\subsection{Reading of children of a linear node}

\begin{defi}
  Let $\pi$ be a pin-permutation and $p=(p_1, \ldots , p_n)$ be a
  pin representation of $\pi$. 
  We say that $p$ \emph{reads} the points of $\pi$ in the order $p_1, \ldots , p_n$. 
  For any set $D$ of points of
  $\pi$, if $k$ is the number of maximal factors $p_i, p_{i+1} ,
  \ldots, p_{i+j}$ of $p$ that contain only points of $D$, we say that
  $D$ is \emph{read in $k$ \fois by $p$}.
  If $C$ is a set of points of $\pi$
  disjoint from $D$, we say that $D$ is \emph{read entirely before} $C$ if 
  every pin belonging to $D$ appears in $p$ before the first pin belonging to $C$.
\label{defn:read}
\end{defi}

Let $\pi$ be a pin-permutation whose decomposition tree $T$ has a linear
root. W.l.o.g., assume that $T=$ \tikz[sibling
  distance=7mm,baseline=-8pt, level distance=5mm] \node
[linear](X){$\oplus$} child {node {$T_1$}} child {node {$T_2$}} child
{node {$\ldots$}} child {node {$T_r$}}; and let $p=(p_1,\ldots,p_n)$ be
one of its pin representations. In the sequel, we denote by $i_0$ the
index of the child which contains $p_1$.

\begin{lem}\label{lem:lem1}
  Let $1 \leq i,j \leq r$ such that either $i < j < i_0$ or $i_0 < j <
  i$. Then $T_j$ is read by $p$ entirely before $T_i$.
\end{lem}

\begin{proof}
  Let $\ell = \min \{ \ell', p_{\ell'}\in T_i\}$. Let ${\mathcal B}_{p_{1},\ldots,p_{\ell}}$ be the
  bounding box of $\{p_1,\ldots,p_{\ell}\}$. As $p_1 \in T_{i_0}$ and $p_\ell \in T_i$, $T_j
  \subseteq {\mathcal B}_{p_1, \ldots,p_{\ell}}$ (see Figure~\ref{fig:bblem1}), hence it is
  entirely read before $p_\ell$ in $p$. Indeed, for all $k \geq
  2$, $p_k$ lies outside the bounding box of $\{p_1,\ldots ,p_{k-1}\}$.
\end{proof}

The previous lemma gives the possible orders in which children are read. Now we
characterize the children $T_i$ which may be read in several \fois.  When
this is the case, we will prove that the decomposition tree is
of a specific shape. This can indeed be deduced from the two following lemmas.
\begin{lem}\label{lem:onlyOneBlockUnfinishedReading}
  For every $k \in \{1, \ldots, n\}$ there is at most one
  child whose reading has started and is not finished after
  $(p_1,p_2,\ldots,p_k)$.
\end{lem}

\begin{proof}
  Suppose that pins $p_1,\ldots,p_k$ have already been read and
  that there are two children $T_i$ and $T_m$ with $i<m$ whose readings
  have started and are not finished. By Lemma~\ref{lem:lem1} there
  exists at most one child $T_j$ with $j < i_0$ and at most one child
  $T_j$ with $j > i_0$ whose readings have started and are not
  finished. Therefore $i \leq i_0$ and $m \geq i_0$. Note that $i = \min \{\ell \mid \exists h \in \{1,
  \ldots, k\}, p_h \in T_{\ell}\}$. The same goes for $m$
  changing the minimum into a maximum. If the reading of $T_i$ is not
  finished, since $T_i$ is $\oplus$-indecomposable, there must exist a
  pin $p_q$ in zone \tikz\draw[pattern=north east lines](0,0)
  rectangle (0.5,0.25); (see Figure~\ref{fig:bblem2}).  Such a pin is
  on the side of the bounding box ${\mathcal B}_{p_1, \ldots,p_k}$ of
  $\{p_1, \ldots,p_k\}$, and the same remark goes for $T_m$.
  But from Lemma~\ref{lem:[7]2.17} (p.\pageref{lem:[7]2.17})
  there is at most one pin lying on the sides of a bounding box,
  and this contradiction concludes the proof. 
\end{proof}

\begin{figure}[ht]
 \begin{minipage}[b]{.32\linewidth}
\begin{center}
 \begin{tikzpicture}[scale=.3]
\draw[help lines] (0,0) -- (-1,0);
\draw[help lines] (0,0) -- (0,-1);
 \draw (0,0) rectangle node {$T_i$} (3,3);
\draw[fill] (2.5,2.5) circle  (2pt);
\draw[help lines] (3,3) rectangle +(0.5,0.5);
\draw[help lines] (3.5,3.5) rectangle +(0.5,0.5);
 \draw (4,4) rectangle node {$T_{j}$} +(3,3);
\draw[help lines] (7,7) rectangle +(0.5,0.5);
\draw[help lines] (7.5,7.5) rectangle +(0.5,0.5);
 \draw (8,8) rectangle node {$T_{i_0}$} +(3,3);
\draw[fill] (8.5,8.5) circle  (2pt);
\draw[help lines] (11,11) -- (12,11);
\draw[help lines] (11,11) -- (11,12);
\draw [very thick] (2.5,2.5) rectangle (9,9);
\draw (7,1.5) node {${\mathcal B}_{p_{1},\ldots,p_{\ell}}$};
\draw (1,4.5) node (pl) {$p_{\ell}$};
\draw (6,11) node (p1) {$p_{1}$};
\draw [->, thick] (pl) -- (2.85,2.75);
\draw [->, thick] (p1) -- (8.5,8.5);

\end{tikzpicture}
\caption{Proof of \newline Lemma~\ref{lem:lem1}.}
\label{fig:bblem1}
\end{center}
\end{minipage}
 \begin{minipage}[b]{.32\linewidth}
\begin{center}
 \begin{tikzpicture}[scale=.3]
\draw[help lines] (0,0) -- (-1,0);
\draw[help lines] (0,0) -- (0,-1);
\draw (0,0) rectangle node {$T_i$} (3,3);
\draw[help lines] (3,3) rectangle +(1,1);
\draw[help lines] (4,4) rectangle +(1,1);
\draw[help lines] (5,5) rectangle +(1,1);
\draw[help lines] (6,6) rectangle +(1,1);
\draw[help lines] (7,7) rectangle +(1,1);
 \draw (8,8) rectangle node[shift={(.2,.2)}] {$T_{m}$} +(3,3);
\draw[help lines] (11,11) -- (12,11);
\draw[help lines] (11,11) -- (11,12);
\draw [very thick] (2.5,2.5) rectangle (9,9);
\draw (7,1.5) node {${\mathcal B}_{p_{1},\ldots,p_{k}}$};
\draw [pattern=north east lines] (0,2.5) rectangle (2.5,3);
\draw [pattern=north east lines] (2.5,0) rectangle (3,2.5);
\draw [pattern=north east lines] (9,8) rectangle +(2,1);
\draw [pattern=north east lines] (8,9) rectangle +(1,2);

\end{tikzpicture}
\caption{Proof of \newline Lemmas~\ref{lem:onlyOneBlockUnfinishedReading} and \ref{lem:oneBlockReadInSeveralTimes}.}
\label{fig:bblem2}
\end{center}
\end{minipage}
 \begin{minipage}[b]{.32\linewidth}
\begin{center}
\begin{tikzpicture}[scale=0.3,inner sep=0]
 \draw[very thick] (-2,-2) rectangle node[below=40pt]{$T_{i_0}$} (7,7) ;
\draw[thick] (2,2) rectangle node {\tiny $p_1, \ldots, p_\ell$}(6,6);
\draw[pattern=north east lines] (-4,-2) -- (-2,-2) -- (-2,-3);
\draw[pattern=north east lines] (9,7) -- (7,7) -- (7,9);
\draw (1.5,8.5) node (b) {\tiny $p_{\ell+1}, \ldots, p_{m-1}$};
\draw[thick,->] (b) -- (7,8);

\draw (1.5,6.5) node[circle,draw] {\tiny $1$};
\draw (0.5,6.5) node[circle,draw] {\tiny $3$};

\draw (6.5,1.5) node[circle,draw] {\tiny $2$};
\draw (6.5,0.5) node[circle,draw] {\tiny $4$};
\draw (-3.5,2.5) node (pm) {$p_m$};
\draw[->] (pm) -- (1.25,6);
\draw[->] (pm) -- (6,1.1);

\draw[help lines] (0,6) -- (7,6);
\draw[help lines] (6,0) -- (6,7);
\draw[help lines] (2,0) -- (2,7);
\draw[help lines] (1,0) -- (1,7);
\draw[help lines] (0,0) -- (0,7);
\draw[help lines] (0,0) -- (7,0);
\draw[help lines] (0,1) -- (7,1);
\draw[help lines] (0,2) -- (7,2);

\end{tikzpicture}
\caption{$T_{i_0}$ is read in \newline two \fois (Lemma~\ref{lem:readInTwoTimes}).}\label{fig:pm}
\end{center}
\end{minipage}
\end{figure}

\begin{lem}\label{lem:oneBlockReadInSeveralTimes}
Every child $T_i$ is read in one piece by $p$, except perhaps $T_{i_0}$.
\end{lem}
\begin{proof}
  Consider a child $T_i$ with $i\neq i_0$ which is read in more than
  one piece by $p$. Consider the pin $p_{k+1}$ which is the first pin
  outside $T_i$ after $p$ has started reading $T_i$. As $p_1$ is in
  $T_{i_0}$, $p_1$ is outside $T_i$ and the bounding box of $\{p_1, p_2,
  \ldots, p_{k-1}, p_{k}\}$ allows to define a zone \tikz\draw[pattern=north
  east lines](0,0) rectangle (0.5,0.25);
  in $T_i$ as shown in the bottom left part of Figure~\ref{fig:bblem2}.
  Since $T_i$ is $\oplus$-indecomposable, there is
  at least one pin in this zone. This pin is on the side of the
  bounding box of $\{p_1, p_2, \ldots, p_{k}\}$ so it is $p_{k+1}$ by
  Lemma~\ref{lem:[7]2.17} (p.\pageref{lem:[7]2.17}).
  Thus $p_{k+1} \in T_i$ which provides the desired contradiction.
\end{proof}

When a child may be read in several \fois, the
decomposition tree of the whole permutation $\pi$ has a special shape given
in the following lemma.

\begin{lem}\label{lem:readInTwoTimes}
The only permutations $\pi$ whose decomposition trees have a root $\oplus$
in which a child may be read in several \fois are those
whose decomposition trees have one of the shapes given in Figure~\ref{fig:severaltime}
where $\xi^+$ is an increasing oscillation of size at least $4$.

A given permutation $\pi$ may match several shapes of Figure~\ref{fig:severaltime}.
However if a child is read in more than one piece,
then it is necessarily the first child to be read (denoted $T_{i_0}$)
and it is read in two \fois;
in addition, there is exactly one shape of Figure~\ref{fig:severaltime} such that
the first part of $T_{i_0}$ to be read is $S$ and
the second part is the remaining leaves of $T_{i_0}$ with only the point $x$ read in between.
\end{lem}
\begin{figure}[ht]
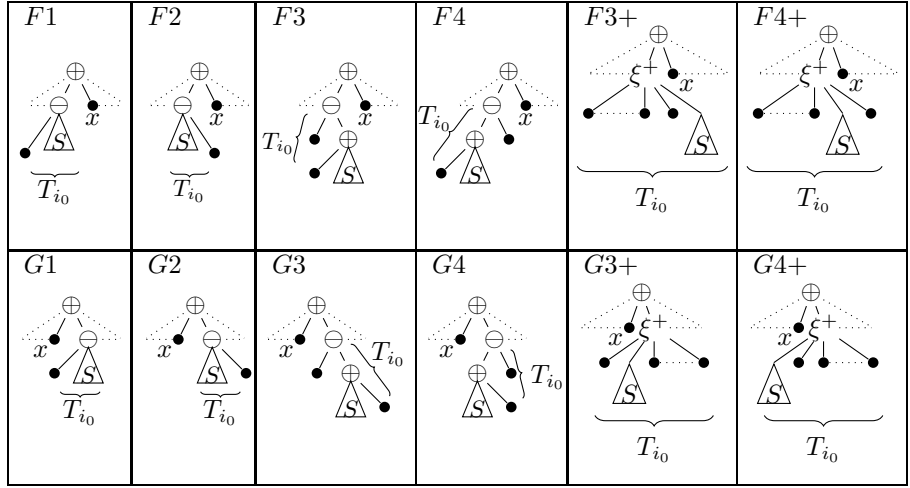

\centering
\begin{tabular}{|p{.1\textwidth}|p{.1\textwidth}|p{.14\textwidth}|p{.13\textwidth}|p{.15\textwidth}|p{.15\textwidth}|}
\hline
$F1$& $F2$ & $F3$ & $F4$ & $F3+$ & $F4+$\\
\begin{tikzArbre}{.3}
\useasboundingbox (-2.3,0) rectangle (4,-7.5);
\node[linear] {$\oplus$}
 	child [dotted] {node[leaf,fill=none] (T1) {}}
 	child {node[linear] (-) {$\ominus$}
		child[level distance=60pt] { node[leaf,anchor=center]  (gauche) {}}
		child [level distance=1pt,xshift=0pt,yshift=-8pt] {node[draw,shape=isosceles triangle,
shape border rotate=90,anchor=north] {\small $S$} edge from parent[draw=none]}
		child [sibling distance=25pt,level distance=60pt] {node[leaf,fill=none,anchor=west] (droite) {} edge from parent [draw=none]}
 	}
	child  {node[leaf] (x) {}}
	child [dotted] {node[leaf,fill=none] (Tk) {} };
\draw[dotted]  (T1) -- (-);
\draw[dotted]  (x) -- (Tk);
\draw (x) node [below=3pt] {$x$};
\draw[decorate,decoration={brace,amplitude=2pt,raise=7pt},anchor=center] (droite) -- node[below
= 10pt]{$T_{i_0}$} (gauche);
\end{tikzArbre}
&
\begin{tikzArbre}{.3}
\useasboundingbox (-2.3,0) rectangle (4,-7.5);
\node[linear] {$\oplus$}
 	child [dotted] {node (T1) {}}
 	child {node[linear] (-) {$\ominus$}
		child [sibling distance=18pt,level distance=60pt] {node  (gauche) {} edge from parent [draw=none]}
		child [level distance=2pt,xshift=0pt,yshift=-7pt] {node[draw,shape=isosceles triangle,
shape border rotate=90,anchor=north] {\small $S$} edge from parent[draw=none]}
		child[sibling distance=40pt,level distance=60pt] { node[leaf] (droite) {} }
 	}
	child  {node[leaf] (x) {}}
	child [dotted] {node (Tk) {} };
\draw[dotted]  (T1) -- (-);
\draw[dotted]  (x) -- (Tk);
\draw (x) node [below=3pt] {$x$};
\draw[decorate,decoration={brace, amplitude=2pt,raise=7pt}] (droite) -- node[below
= 10pt]{\small $T_{i_0}$} (gauche);
\end{tikzArbre}
&
\begin{tikzArbre}{.3}
\useasboundingbox (-3.4,0) rectangle (4,-7.5);
\node[linear] {$\oplus$}
 	child [dotted] {node (T1) {}}
 	child {node[linear] (-) {$\ominus$}
		child {node [leaf] {}}
		child {node[linear] {$\oplus$}
			child { node[leaf]  (gauche) {}}
			child [level distance=1pt,xshift=0pt,yshift=-7pt] {node[draw,shape=isosceles triangle,
	shape border rotate=90,anchor=north] {\small $S$} edge from parent[draw=none]}
			child [level distance=60pt,sibling distance=18pt] {node (droite) {} edge from parent [draw=none]}
 		}
	}
	child  {node[leaf](x) {}}
	child [dotted] {node (Tk) {} };
\draw[dotted]  (T1) -- (-);
\draw[dotted]  (x) -- (Tk);
\draw (x) node [below=3pt] {$x$};
\draw[decorate,decoration={brace, amplitude=2pt,raise =8pt}] (gauche) -- node[left
= 10pt]{\small $T_{i_0}$} (-);
\end{tikzArbre}
&
\begin{tikzArbre}{.3}
\useasboundingbox (-3.4,0) rectangle (4,-7.5);
\node [linear]{$\oplus$}
 	child [dotted] {node (T1) {}}
 	child {node[linear] (-) {$\ominus$}
		child {node[linear] {$\oplus$}
			child {node[leaf]  (gauche) {}}
			child [level distance=1pt,xshift=0pt,yshift=-7pt] {node[draw,shape=isosceles triangle,
	shape border rotate=90,anchor=north] {\small $S$} edge from parent[draw=none]}
			child [sibling distance=18pt] {node (droite) {} edge from parent [draw=none]}
 		}
		child {node[leaf]{}}
	}
	child  {node[leaf] (x) {}}
	child [dotted] {node (Tk) {} };
\draw[dotted]  (T1) -- (-);
\draw[dotted]  (x) -- (Tk);
\draw (x) node [below=3pt] {$x$};
\draw[decorate,decoration={brace, amplitude=2pt,raise =5pt}] (gauche) -- node[above left
= 6pt]{$T_{i_0}$} (-);
\end{tikzArbre}
&
\begin{tikzArbre}{.25}
\useasboundingbox (-4,0) rectangle (4,-11);
\node[linear] {$\oplus$}
	child [level distance=60pt,dotted] {node (T1) {}}
	child[missing]
	child [level distance=60pt] {node (xi) {$\xi^{+}$}
		child [level distance=60pt,sibling distance=35pt]{node (gauche) {} edge from parent [draw=none]}
		child[level distance=60pt] { node[leaf] (f1){}}
		child[missing]
		child[level distance=60pt] { node[leaf] (f2){}}
		child[level distance=60pt] {node[leaf] (fk) {}}
		child[child anchor=north,level distance=60pt] {node[draw,shape=isosceles triangle,
shape border rotate=90,anchor=north] {$S$}}
		child[level distance=60pt] {node (droite) {}edge from parent [draw=none]}
	}
	child[level distance=60pt]  {node[leaf] (x){}}
	child[level distance=60pt,missing]
	child[level distance=60pt,dotted] {node (Tk) {} };
\draw[dotted]  (T1) -- (xi);
\draw[dotted]  (x) -- (Tk);
\draw[dotted]  (f1) -- (f2);
\draw (x) node [below right=3pt] {$x$};
\draw[decorate,decoration={brace, amplitude=4pt,raise =20pt}] (droite) -- node[below
= 28pt]{$T_{i_0}$} (gauche);
\end{tikzArbre}
&
\begin{tikzArbre}{.25}
\useasboundingbox (-4,0) rectangle (4,-11);
\node[linear] {$\oplus$}
	child[level distance=60pt,dotted] {node (T1) {}}
	child[level distance=60pt,missing]
	child[level distance=60pt] {node (xi) {$\xi^{+}$}
		child[level distance=60pt,sibling distance=35pt] {node (gauche) {} edge from parent [draw=none]}
		child[level distance=60pt] { node[leaf] (f1){}}
		child[level distance=60pt,missing]
		child [level distance=60pt]{ node[leaf] (f2){}}
		child[level distance=60pt,child anchor=north] {node[draw,shape=isosceles triangle,
shape border rotate=90,anchor=north] {$S$}}
		child {node[leaf] (fk) {}}
		child[sibling distance=35pt] {node (droite) {}edge from parent [draw=none]}
	}
	child[level distance=60pt]  {node[leaf] (x){}}
	child[level distance=60pt,missing]
	child[level distance=60pt,dotted] {node (Tk) {} };
\draw[dotted]  (T1) -- (xi);
\draw[dotted]  (x) -- (Tk);
\draw[dotted]  (f1) -- (f2);
\draw (x) node [below right=3pt] {$x$};
\draw[decorate,decoration={brace, amplitude=4pt,raise =20pt}] (droite) -- node[below
= 28pt]{$T_{i_0}$} (gauche);
\end{tikzArbre}\\
\hline
$G1$ & $G2$ & $G3$ & $G4$ & $G3+$ & $G4+$\\
\begin{tikzArbre}{.3}
\useasboundingbox (-2.1,0) rectangle (4,-7.5);
\node[baseline,linear] {$\oplus$}
	child [dotted] {node (T1) {}}
	child  {node[leaf] (x) {}}
	child {node[linear] (-) {$\ominus$}
		child {node[leaf] (gauche) {}}
		child [level distance=1pt,xshift=0pt,yshift=-7pt] {node[draw,shape=isosceles triangle,
shape border rotate=90,anchor=north] {\small $S$} edge from parent[draw=none]}
		child [sibling distance=17pt] {node (droite) {} edge from parent [draw=none]}
		}
	child [dotted] {node (Tk) {} };
\draw[dotted]  (T1) -- (x);
\draw[dotted]  (-) -- (Tk);
\draw (x) node [below left=3pt] {$x$};
\draw[decorate,decoration={brace, amplitude=2pt,raise =7pt}] (droite) -- node[below
= 10pt]{$T_{i_0}$} (gauche);
\end{tikzArbre}
&
\begin{tikzArbre}{.3}
\useasboundingbox (-2.1,0) rectangle (4,-7.5);
\node[baseline,linear] {$\oplus$}
	child [dotted] {node (T1) {}}
	child  {node[leaf] (x) {}}
	child {node[linear] (-) {$\ominus$}
		child [sibling distance=17pt] {node (gauche) {} edge from parent [draw=none]}
		child [level distance=1pt,xshift=0pt,yshift=-7pt] {node[draw,shape=isosceles triangle,
shape border rotate=90,anchor=north] {\small $S$} edge from parent[draw=none]}
		child {node[leaf] (droite) {}}
		}
	child [dotted] {node (Tk) {} };
\draw[dotted]  (T1) -- (x);
\draw[dotted]  (-) -- (Tk);
\draw (x) node [below left=3pt] {$x$};
\draw[decorate,decoration={brace, amplitude=2pt,raise =7pt}] (droite) -- node[below
= 10pt]{$T_{i_0}$} (gauche);
\end{tikzArbre}
&
\begin{tikzArbre}{.3}
\useasboundingbox (-2,1) rectangle (4,-7.5);
\node[baseline,linear] {$\oplus$}
	child [dotted] {node (T1) {}}
	child  { node[leaf] (x) {}}
	child {node[linear] (-) {$\ominus$}
		child {node[leaf] {}}
		child {node[linear] {$\oplus$}
			child [sibling distance=17pt] {node (gauche) {} edge from parent [draw=none]}
			child [level distance=1pt,xshift=0pt,yshift=-7pt] {node[draw,shape=isosceles triangle,
	shape border rotate=90,anchor=north] {\small $S$} edge from parent[draw=none]}
			child { node[leaf] (droite) {}}
			}
	}
	child [dotted] {node (Tk) {} };
\draw[dotted]  (T1) -- (x);
\draw[dotted]  (-) -- (Tk);
\draw (x) node [below left=3pt] {$x$};
\draw[decorate,decoration={brace, amplitude=2pt,raise =4pt}] (-) -- node[above right
= 5pt]{$T_{i_0}$} (droite);
\end{tikzArbre}
&
\begin{tikzArbre}{.3}
\useasboundingbox (-2,1) rectangle (4,-7.5);
\node[baseline,linear] {$\oplus$}
	child [dotted] {node (T1) {}}
	child  {node[leaf] (x) {}}
	child {node[linear] (-) {$\ominus$}
		child {node[linear] {$\oplus$}
			child [sibling distance=17pt] {node (gauche) {} edge from parent [draw=none]}
			child [level distance=1pt,xshift=0pt,yshift=-7pt] {node[draw,shape=isosceles triangle,
	shape border rotate=90,anchor=north] {\small $S$} edge from parent[draw=none]}
			child {node[leaf] (droite) {}}
			}
		child { node[leaf] {}}
	}
	child [dotted] {node (Tk) {} };
\draw[dotted]  (T1) -- (x);
\draw[dotted]  (-) -- (Tk);
\draw (x) node [below left=3pt] {$x$};
\draw[decorate,decoration={brace, amplitude=2pt,raise =5pt}] (-) -- node[right
= 10pt]{$T_{i_0}$} (droite);
\end{tikzArbre}
&
\begin{tikzArbre}{.22}
\useasboundingbox (-3.5,0) rectangle (4,-11);
\node[baseline,linear] {$\oplus$}
	child [level distance=60pt,dotted] {node (T1) {}}
	child[missing]
	child [level distance=60pt] {node[leaf] (x){}}
	child [level distance=60pt]{node (xi) {$\xi^{+}$}
		child[level distance=60pt,sibling distance=35pt] {node (gauche) {} edge from parent [draw=none]}
		child[level distance=60pt] {node[leaf] (f1){}}
		child[level distance=60pt,child anchor=north] {node[draw,shape=isosceles triangle,
shape border rotate=90,anchor=north] {$S$}}
		child[level distance=60pt] {node[leaf] (f2){}}
		child[level distance=60pt,missing]
		child[level distance=60pt] {node[leaf] (fk) {}}
		child[level distance=60pt,sibling distance=35pt] {node (droite) {}edge from parent [draw=none]}
	}
	child[missing]
	child [level distance=60pt,dotted] {node (Tk) {} };
\draw[dotted]  (T1) -- (x);
\draw[dotted]  (xi) -- (Tk);
\draw[dotted]  (f2) -- (fk);
\draw (x) node [below left=3pt] {$x$};
\draw[decorate,decoration={brace, amplitude=4pt,raise =20pt}] (droite) -- node[below
= 28pt]{$T_{i_0}$} (gauche);
\end{tikzArbre}
&
\begin{tikzArbre}{.22}
\useasboundingbox (-3.5,0) rectangle (4,-11);
\node[baseline,linear,level distance=60pt] {$\oplus$}
	child [level distance=60pt,dotted] {node (T1) {}}
	child[missing]
	child [level distance=60pt] { node[leaf] (x){}}
	child [level distance=60pt] {node (xi) {$\xi^{+}$}
		child[level distance=60pt,sibling distance=35pt] {node (gauche) {} edge from parent [draw=none]}
		child[level distance=60pt,child anchor=north] {node[draw,shape=isosceles triangle,
shape border rotate=90,anchor=north] {$S$}}
		child[level distance=60pt] {node[leaf] (f1){}}
		child[level distance=60pt] { node[leaf] (f2){}}
		child[level distance=60pt,missing]
		child [level distance=60pt] {node[leaf] (fk) {}}
		child[level distance=60pt,sibling distance=35pt] {node (droite) {}edge from parent [draw=none]}
	}
	child[missing]
	child [level distance=60pt,dotted] {node (Tk) {} };
\draw[dotted]  (T1) -- (x);
\draw[dotted]  (xi) -- (Tk);
\draw[dotted]  (f2) -- (fk);
\draw (x) node [below left=3pt] {$x$};
\draw[decorate,decoration={brace, amplitude=4pt,raise =20pt}] (droite) -- node[below
= 28pt]{$T_{i_0}$} (gauche);
\end{tikzArbre}\\
\hline
\end{tabular}
\caption{Decomposition tree of $\pi$ when $T_{i_0}$ may be read in several \fois.
}
\label{fig:severaltime}
\end{figure}
In Figure~\ref{fig:severaltime} and in the sequel, we draw the attention of the reader to the difference between trees of the shape
$\begin{tikzpicture}[baseline=-10pt,sibling
distance=40pt,level distance=10pt]
\begin{scope}[scale=0.5]
\node[simple] {\textsc{r}} child[child anchor=north]
{node[draw,shape=isosceles triangle, shape border
rotate=90,anchor=north, inner sep=0,isosceles triangle apex angle=90] {\small $T$}} child [level distance =20pt] {[fill] circle (3 pt) node(y1)
{}}; \end{scope}
\end{tikzpicture}$
\  and
$\begin{tikzpicture}[scale=.5,baseline=-10pt, inner sep=0,sibling distance=25pt]
\node {\textsc{r}} child [missing] child
    [level distance=1pt,xshift=0pt,yshift=-7pt]
    {node[draw,shape=isosceles triangle, shape border
      rotate=90,anchor=north,isosceles triangle apex angle=90] {\small $T$} edge from
      parent[draw=none]} child [level distance =20pt]{[fill] circle (3pt) node (x){}};
  \end{tikzpicture}$\ : in the first case the root \textsc{r} has exactly $2$
  children, in the second one it has at least two children, $T$ being a
  forest.

\begin{proof}
  Let $\pi$ be a pin-permutation whose decomposition tree has a root $\oplus$.
  Let $p=(p_1,p_2,\ldots,p_n)$ be a pin representation of $\pi$ that reads one child in several \fois.
  Lemma~\ref{lem:oneBlockReadInSeveralTimes} ensures that there is only one such child, which is necessarily $T_{i_0}$.
  Denote $p_1,\ldots,p_\ell$ the first part of the reading of $T_{i_0}$.
  Then $p_{\ell+1},\ldots,p_{m-1}$ belong to other children until $p_m \in T_{i_0}$.

  As $T_{i_0}$ is a child of the root $\oplus$, each pin $p_i$ with $i
  \in \{ \ell+1,\ell+2,\ldots,m-1 \}$ lies in one of the zones
  \tikz\draw[pattern=north east lines](0,0) rectangle (0.5,0.25); as
  shown in Figure~\ref{fig:pm}. But if both zones contain at least one
  pin $p_i$ with $i \in \{ \ell+1,\ell+2,\ldots,m-1 \}$, the bounding
  box of $\{p_1,\ldots,p_{m-1}\}$ contains $T_{i_0}$ and thus
  $p_m$ cannot be outside the bounding box of $\{p_1,\ldots,p_{m-1}\}$. Hence all pins $p_i$
  with $i \in \{ \ell+1,\ell+2,\ldots,m-1 \}$ are in the same zone.

Assume w.l.o.g. that $\{p_{\ell+1},\ldots,p_{m-1}\}$ are in the
upper right zone of Figure~\ref{fig:pm} (otherwise, in the proof that follows,
cases $F1, \dots, F4+$ of Figure~\ref{fig:cases} are replaced by cases $G1, \dots, G4+$).
If $p_m$ respects
the independence condition, it must lie in the lower left corner of
the bounding box of $\{p_1,\ldots,p_\ell\}$ and every future pin of
$T_{i_0}$ lies in the same corner leading to a $\oplus$-decomposable
child $T_{i_0}$ which contradicts our hypothesis. Thus $p_m$ must be a separating
pin and $m=\ell+2$.

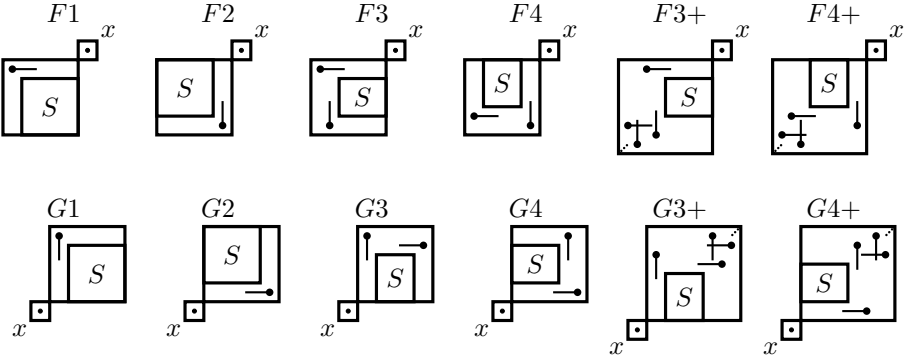
\begin{figure}[ht]
\begin{tabular}{cccccc}

$F1$ & $F2$ & $F3$ & $F4$ & $F3+$ & $F4+$\\
\begin{tikzpicture}[scale=.25]
\useasboundingbox (0,-1) rectangle (6.5,5.5);
\draw[very thick] (0,0) rectangle (4,4);
\draw[very thick] (1,0) rectangle node {$S$} (4,3) ;

\draw [very thick] (4,4) rectangle (5,5);
\draw[fill] (4.5,4.5) circle (3pt) node[above right=2pt] {$x$};
\pinL{0.5}{3.5};
\end{tikzpicture}
&
\begin{tikzpicture}[scale=.25]
\useasboundingbox (0,-1) rectangle (6.5,5.5);
\draw[very thick] (0,0) rectangle (4,4);
\draw[very thick] (0,1) rectangle node {$S$} (3,4) ;

\draw [very thick] (4,4) rectangle (5,5);
\draw[fill] (4.5,4.5) circle (3pt) node[above right=2pt] {$x$};
\pinD{3.5}{0.5};
\end{tikzpicture}
&
\begin{tikzpicture}[scale=.25]
\useasboundingbox (0,-1) rectangle (6.5,5.5);
\draw[very thick] (0,0) rectangle (4,4);
\draw[very thick] (1.5,1) rectangle node {$S$} (4,3) ;
\pinD{1}{0.5};
\pinL{0.5}{3.5};

\draw [very thick] (4,4) rectangle (5,5);
\draw[fill] (4.5,4.5) circle (3pt) node[above right=2pt] {$x$};
\end{tikzpicture}
&
\begin{tikzpicture}[scale=.25]
\useasboundingbox (0,-1) rectangle (6.5,5.5);
\draw[very thick] (0,0) rectangle (4,4);
\draw[very thick] (1,1.5) rectangle node {$S$} (3,4) ;
\pinL{0.5}{1};
\pinD{3.5}{0.5};

\draw [very thick] (4,4) rectangle (5,5);
\draw[fill] (4.5,4.5) circle (3pt) node[above right=2pt] {$x$};
\end{tikzpicture}
&
\begin{tikzpicture}[scale=.25]
\useasboundingbox (0,0) rectangle (6.5,6.5);
\draw[very thick] (0,0) rectangle (5,5);
\draw[very thick] (2.5,2) rectangle node {$S$} (5,4) ;
\pinD{1}{0.5};
\pinL{0.5}{1.5};
\pinD{2}{1};
\pinL{1.5}{4.5};
\draw (0.15,0.15) circle (1pt);
\draw (0.3,0.3) circle (1pt);
\draw (0.45,0.45) circle (1pt);

\draw [very thick] (5,5) rectangle (6,6);
\draw[fill] (5.5,5.5) circle (3pt) node[above right=2pt] {$x$};
\end{tikzpicture}
&
\begin{tikzpicture}[scale=.25]
\useasboundingbox (0,0) rectangle (6.5,6.5);
\draw[very thick] (0,0) rectangle (5,5);
\draw[very thick] (2,2.5) rectangle node {$S$} (4,5) ;
\pinL{0.5}{1};
\pinD{1.5}{0.5};
\pinL{1}{2};
\pinD{4.5}{1.5};
\draw (0.15,0.15) circle (1pt);
\draw (0.3,0.3) circle (1pt);
\draw (0.45,0.45) circle (1pt);

\draw [very thick] (5,5) rectangle (6,6);
\draw[fill] (5.5,5.5) circle (3pt) node[above right=2pt] {$x$};
\end{tikzpicture}\\
&&&&&\\
$G1$ & $G2$ & $G3$ & $G4$ & $G3+$ & $G4+$\\
\begin{tikzpicture}[scale=.25]
\useasboundingbox (0,0) rectangle (-6.5,-6.5);
\draw[very thick] (0,0) rectangle (-4,-4);
\draw[very thick] (0,-1) rectangle node {$S$} (-3,-4) ;

\draw [very thick] (-4,-4) rectangle (-5,-5);
\draw[fill] (-4.5,-4.5) circle (3pt) node[below left=2pt] {$x$};
\pinU{-3.5}{-0.5};
\end{tikzpicture}
&
\begin{tikzpicture}[scale=.25]
\useasboundingbox (0,0) rectangle (-6.5,-6.5);
\draw[very thick] (0,0) rectangle (-4,-4);
\draw[very thick] (-1,0) rectangle node {$S$} (-4,-3) ;

\draw [very thick] (-4,-4) rectangle (-5,-5);
\draw[fill] (-4.5,-4.5) circle (3pt) node[below left=2pt] {$x$};
\pinR{-0.5}{-3.5};
\end{tikzpicture}
&
\begin{tikzpicture}[scale=.25]
\useasboundingbox (0,0) rectangle (-6.5,-6.5);
\draw[very thick] (0,0) rectangle (-4,-4);
\draw[very thick] (-1,-1.5) rectangle node {$S$} (-3,-4) ;
\pinR{-0.5}{-1};
\pinU{-3.5}{-0.5};

\draw [very thick] (-4,-4) rectangle (-5,-5);
\draw[fill] (-4.5,-4.5) circle (3pt) node[below left=2pt] {$x$};
\end{tikzpicture}
&
\begin{tikzpicture}[scale=.25]
\useasboundingbox (0,0) rectangle (-6.5,-6.5);
\draw[very thick] (0,0) rectangle (-4,-4);
\draw[very thick] (-1.5,-1) rectangle node {$S$} (-4,-3) ;
\pinU{-1}{-0.5};
\pinR{-0.5}{-3.5};

\draw [very thick] (-4,-4) rectangle (-5,-5);
\draw[fill] (-4.5,-4.5) circle (3pt) node[below left=2pt] {$x$};
\end{tikzpicture}
&
\begin{tikzpicture}[scale=.25]
\useasboundingbox (0,0) rectangle (-6.5,-6.5);
\draw[very thick] (0,0) rectangle (-5,-5);
\draw[very thick] (-2,-2.5) rectangle node {$S$} (-4,-5) ;
\pinR{-0.5}{-1};
\pinU{-1.5}{-0.5};
\pinR{-1}{-2};
\pinU{-4.5}{-1.5};
\draw (-0.15,-0.15) circle (1pt);
\draw (-0.3,-0.3) circle (1pt);
\draw (-0.45,-0.45) circle (1pt);

\draw [very thick] (-5,-5) rectangle (-6,-6);
\draw[fill] (-5.5,-5.5) circle (3pt) node[below left=2pt] {$x$};
\end{tikzpicture}
&
\begin{tikzpicture}[scale=.25]
\useasboundingbox (0,0) rectangle (-6.5,-6.5);
\draw[very thick] (0,0) rectangle (-5,-5);
\draw[very thick] (-2.5,-2) rectangle node {$S$} (-5,-4) ;
\pinU{-1}{-0.5};
\pinR{-0.5}{-1.5};
\pinU{-2}{-1};
\pinR{-1.5}{-4.5};
\draw (-0.15,-0.15) circle (1pt);
\draw (-0.3,-0.3) circle (1pt);
\draw (-0.45,-0.45) circle (1pt);

\draw [very thick] (-5,-5) rectangle (-6,-6);
\draw[fill] (-5.5,-5.5) circle (3pt) node[below left=2pt] {$x$};
\end{tikzpicture}\\
\end{tabular}
\caption{Diagram of $T_{i_0}$ and $x$ if $T_{i_0}$ is read in two \fois, the first part being~$S$.}\label{fig:cases}
\end{figure}

As at most one point can lie on the sides of a bounding box, there are
only four possible positions for $p_m$ as depicted in Figure~\ref{fig:pm}. If there is no pin separating $p_m$ from $\{p_1,\ldots,
p_{m-1}\}$ then $p_m$ is either in position \tikz\draw (0,0)
node[circle,draw,inner sep=0] {\small $1$}; (case $F1$ on Figure~\ref{fig:cases}) or \tikz\draw (0,0) node[circle,draw,inner sep=0]
{\small $2$}; (case $F2$); moreover pins $\{p_1,p_2,\ldots,p_\ell,p_m\}$ form
a block and thus represent $T_{i_0}$ (because $T_{i_0}$ is
$\oplus$-indecomposable). Otherwise there is exactly one pin $p_{m+1}$
separating $p_m$ from the bounding box of $\{p_1,\ldots,p_\ell\}$,
thus $p_m$ is either in position \tikz\draw (0,0)
node[circle,draw,inner sep=0] {\small $3$}; (cases $F3$ and $F3+$) or
\tikz\draw (0,0) node[circle,draw,inner sep=0] {\small $4$}; (cases
$F4$ and $F4+$). Suppose that it is in position \tikz\draw (0,0)
node[circle,draw,inner sep=0] {\small $4$}; then $p_{m+1}$ is a left
pin separating $p_m$ from the preceding ones. There are again two
different cases: if $p_{m+2}$ respects the independence condition
(case $F4$), then $p_{m+1}$ ends $T_{i_0}$ (since $T_{i_0}$ is
$\oplus$-indecomposable). If $p_{m+2}$ respects the separation
condition then it can only separate $p_{m+1}$ and $\{p_1, \ldots,
p_m\}$ from below (case $F4+$). This process can be repeated
alternating between left and down pins until the following pin
$p_{m+k+1}$ is an independent pin, ending the child $T_{i_0}$.

Thus we have proved that $T_{i_0}$ is read in exactly two \fois,
$p_1,\ldots p_{\ell}$ for the first part and $p_m, \ldots, p_{m+k}$ with
$m=\ell+2$ for the second part. And from
Lemma~\ref{lem:oneBlockReadInSeveralTimes} the pin $p_{\ell+1}$ is
by itself a child $T_i$ of the root.
It is then straightforward to check from Figure~\ref{fig:cases} that $\pi$ has a decomposition tree
of one shape given in Figure~\ref{fig:severaltime} with $S = \{p_1, \dots, p_\ell\}$ and $x = p_{\ell+1}$.
\end{proof}

Note that in the proof of Lemma~\ref{lem:readInTwoTimes}, the order in which the
points corresponding to the leaves of $T_{i_0}\setminus S$ are read is uniquely
determined, leading to the following remark:

\begin{rem}
  If a child $T_{i_0}$ is read in two \fois with the first part fixed,
  then the second part consists of all remaining points of $T_{i_0}$ and
  the order in which they are read is uniquely determined.
 \label{rem:inTwoTimes}
\end{rem}

We now start the description of the set of pin words encoding any pin-permutation, by case study on Equation~\eqref{eq:pin_perm_trees} (p.\pageref{eq:pin_perm_trees}).

\subsection{Non-recursive cases}
\label{subsec:pin_word_simple}

\paragraph*{Permutation of size $1$}
The permutation $\pi=1$ (whose decomposition tree
is a leaf $\tikz \fill[draw] (0,0) circle (2pt);$) has exactly four pin words -- namely, $P(\pi) = \{1,2,3,4\}$.

\paragraph*{Simple permutations}
The only pin-permutations whose decomposition trees have a prime root
and are non-recursive are those whose decomposition trees are of the form
$\begin{tikzpicture}[baseline=-17pt,sibling distance=15pt,level
  distance=15pt,inner sep=1pt] \node[simple, inner sep=1pt] (X){$\pi$} child {[fill] circle (2pt)} child
  {[fill] circle (2pt) node (xx1) {}} child [missing] child {[fill] circle (2pt) node (xx2) {}};
  \draw [dotted] (xx1) -- (xx2);
  \end{tikzpicture}$ , \emph{i.e.}, the simple pin-permutations.
The following theorem describes properties of their pin words.

\begin{theo}\label{thm:nbpinwords}
A simple permutation has at most $48$ pin words, which are all strict or quasi-strict.
\end{theo}

\begin{proof}
Let $\pi$ be a simple permutation.
Then any pin representation $p$ of $\pi$ is proper
(see Remark~\ref{rem:simplepin} p.\pageref{rem:simplepin}) and $|\pi| \geq 3$,
so $p_3$ is a separating pin and $p$ is associated to $6$ pin words
by Remark~\ref{rem:nb_pin_words} (p.\pageref{rem:nb_pin_words}).

Moreover by Lemmas~4.3 and 4.6 of~\cite{BBR09} there are at most
$8$ possible beginnings $(p_1, p_2)$ of a pin representation of $\pi$.
Furthermore, each of these beginnings gives at most one pin representation $p$ of $\pi$.
Indeed $p_{k+1}$ has to be the only point separating $p_k$ from the previous points, since $p_{i+1}$ separates $p_i$ from previous points for all $i$.
So $\pi$ has at most $8$ pin representations and at most $48$ pin words.
Finally the first statement of Remark~\ref{rem:proper_strict} (p.\pageref{rem:proper_strict})
ensures that they are all strict or quasi-strict. 
\end{proof}


\paragraph*{Permutations whose decomposition trees have a linear root} 
W.l.o.g., since we focus on the non-recursive case,
$\pi = \begin{tikzpicture}[sibling distance=15pt,level
distance=15pt,baseline=-17pt,inner sep=0pt]
\node[linear] {$\oplus$}
	child {node (T1) {$\xi_1$}}
	child {node (T2) {$\xi_2$}}
	child[missing]
	child {node (Tk) {$\xi_r$} };
\draw[dotted]  (T2) -- (Tk);
\end{tikzpicture}$
where $\xi_i$ are increasing oscillations. 
Lemma~\ref{lem:cas_de_base_oplus_lecture_en_deux_fois} is
a direct consequence of Lemma~\ref{lem:readInTwoTimes}.

\begin{lem}
  Let $p = (p_1,p_2, \ldots, p_n)$ be a pin representation of
  $\pi$.  The only child $\xi_i$ which may be read in several \fois
  is the child $\xi_{i_0}$ to which $p_1$ belongs.  Moreover if $p$
  reads $\xi_{i_0}$ in several \fois, it is read in two \fois, the
  second child $\xi_i$ read by $p$ is either $\xi_{i_0-1}$ or
  $\xi_{i_0+1}$ and is a  leaf, denoted $x$. Finally,
  the set $E = \xi_{i_0} \cup \{x\}$ is read in one piece by $p$.
\label{lem:cas_de_base_oplus_lecture_en_deux_fois}
\end{lem}

Lemma~\ref{lem:cas_de_base_oplus_lecture_en_deux_fois} together with Lemma~\ref{lem:lem1} leads to the
following.

\begin{csq}
  Every pin representation $p$ of $\pi$ begins by entirely reading
  two consecutive children of the root, say $\xi_i$ and $\xi_{i+1}$,
  then $p$ reads in one piece each of the others $\xi_j$.  Moreover the
  children $\xi_j$ for $j <i$ are read in decreasing order ($\xi_{i-1},
  \xi_{i-2}, \ldots, \xi_1$) and the children $\xi_j$ for $j > i+1$ are read
  in increasing order ($\xi_{i+2}, \xi_{i+3}, \ldots, \xi_r$).
\label{csq:cas_de_base_oplus_lecture_en_deux_fois}
\end{csq}

Consequence~\ref{csq:cas_de_base_oplus_lecture_en_deux_fois} implies
that the restriction of $p$ to each child $\xi_j$ where $j < i$
(resp. $j > i+1$) is a pin representation of $\xi_j$ whose origin lies
in quadrant $1$ (resp. $3$) with respect to the bounding box of the
set of points of $\xi_j$.  Indeed $p_1$ and $p_2$ are in $\xi_i$ or
$\xi_{i+1}$ thus they lie in quadrant $1$
(resp. $3$) with respect to $\xi_j$. Since only $p_2$ may separate
$p_0$ from $p_1$, $p_0$ is also in quadrant $1$ (resp. $3$). 
This is the reason for introducing the functions $\fell$ in Subsection~\ref{ssec:idees_construction_automates}. 
Recall that for any increasing
oscillation $\xi$, we denote by $\fell(\xi)$ the set of pin words that
encode $\xi$ and whose origin lies in quadrant $\quadrantell = 1$ or $3$ with
respect to the points of $\xi$.

To characterize the pin words that encode a permutation $\pi =
\oplus[\xi_1,\xi_2,$ $ \ldots , \xi_r]$ where every $\xi_i$ is an
increasing oscillation, Consequence~\ref{csq:cas_de_base_oplus_lecture_en_deux_fois}
leads us naturally to introduce the shuffle product\footnote{The shuffle product is sometimes called \emph{merge} 
in the permutation patterns literature.} of sequences.
From the above discussion, Theorem~\ref{thm:pinwords_cas_lineaire_non_recursif} then follows,
providing the desired characterization. 

\begin{defi} \label{def:shuffle} Let ${\mathsf A} = ({\mathsf A}_1,{\mathsf A}_2, \ldots, {\mathsf A}_q)$ and
  ${\mathsf B} = ({\mathsf B}_1,{\mathsf B}_2, \ldots, {\mathsf B}_s)$ be two sequences of sets of words.  The {\em shuffle
    product} $ {\mathsf A} \, \shuffle\, {\mathsf B}$ of ${\mathsf A}$ and ${\mathsf B}$ is defined as
 \begin{eqnarray*}
   {\mathsf A}  \shuffle  {\mathsf B} \hspace{-1.2em} &= \big\{ c= c_1  \cdot \ldots \cdot c_{q+s} \mid \,
   \exists \, I=\{i_1,\ldots, i_q\},
   J=\{j_1,\ldots, j_s\} \text{ with } I \cap J=\emptyset , \\
   & i_1< \ldots < i_q,\ j_1< \ldots < j_s,
    \, and \, c_{i_k} \in {\mathsf A}_k \, \forall \, 1 \leq k \leq q, \, 
    c_{j_k} \in {\mathsf B}_k \, \forall \, 1 \leq k \leq s \big\}.
\end{eqnarray*}
\end{defi}

For example, letting ${\mathsf A} = \left( \{x\}, \{aay\}, \{aa\}\right)$ 
and ${\mathsf B}=\left( \{b\},\{b,xy\}\right)$, 
the shuffle of ${\mathsf A}$ and ${\mathsf B}$ is 
\begin{align*}
 {\mathsf A} \shuffle {\mathsf B} = \{ 
 & x \cdot aay \cdot aa \cdot b \cdot b , \quad x \cdot aay \cdot aa \cdot b \cdot xy , \hspace*{1.8cm} {\small \text{for }\{1,2,3\}\uplus\{4,5\}}\\
 & x \cdot aay \cdot b \cdot aa \cdot b , \quad x \cdot aay \cdot b \cdot aa \cdot xy , \hspace*{1.8cm} {\small \text{for }\{1,2,4\}\uplus\{3,5\}}\\
 & x \cdot aay \cdot b \cdot b \cdot aa , \quad x \cdot aay \cdot b \cdot xy \cdot aa , \hspace*{1.8cm} {\small \text{for }\{1,2,5\}\uplus\{3,4\}}\\
 & x \cdot b \cdot aay \cdot aa \cdot b , \quad x \cdot b \cdot aay \cdot aa \cdot xy , \hspace*{1.8cm} {\small \text{for }\{1,3,4\}\uplus\{2,5\}}\\
 & x \cdot b \cdot aay \cdot b \cdot aa , \quad x \cdot b \cdot aay \cdot xy \cdot aa , \hspace*{1.8cm} {\small \text{for }\{1,3,5\}\uplus\{2,4\}}\\
 & x \cdot b \cdot b \cdot aay \cdot aa , \quad x \cdot b \cdot xy \cdot aay \cdot aa , \hspace*{1.8cm} {\small \text{for }\{1,4,5\}\uplus\{2,3\}}\\
 & b \cdot x \cdot aay \cdot aa \cdot b , \quad b \cdot x \cdot aay \cdot aa \cdot xy , \hspace*{1.8cm} {\small \text{for }\{2,3,4\}\uplus\{1,5\}}\\
 & b \cdot x \cdot aay \cdot b \cdot aa , \quad b \cdot x \cdot aay \cdot xy \cdot aa , \hspace*{1.8cm} {\small \text{for }\{2,3,5\}\uplus\{1,4\}}\\
 & b \cdot x \cdot b \cdot aay \cdot aa , \quad b \cdot x \cdot xy \cdot aay \cdot aa , \hspace*{1.8cm} {\small \text{for }\{2,4,5\}\uplus\{1,3\}}\\
 & b \cdot b \cdot x \cdot aay \cdot aa , \quad b \cdot xy \cdot x \cdot aay \cdot aa \} \hspace*{1.8cm} {\small \text{for }\{3,4,5\}\uplus\{1,2\}}
\end{align*}
where every line in the above corresponds to 
the words of ${\mathsf A} \shuffle {\mathsf B}$ associated with the bipartition of $\{1,2,3,4,5\}$ into $I \uplus J$ which is indicated on the right.

\begin{theo}\label{thm:pinwords_cas_lineaire_non_recursif}
  The set $P(\pi)$ of pin words of a permutation
  $\pi = \oplus[\xi_1, \ldots ,$ $ \xi_r]$ where every $\xi_i$ is
  an increasing oscillation is:
\small{$$P(\pi) = \bigcup_{1\leq i \leq r-1} P(\oplus[\xi_i,\xi_{i+1}]) \cdot
\Big( (P^{(1)}(\xi_{i-1}) , \ldots , P^{(1)}(\xi_1) ) \shuffle
(P^{(3)}(\xi_{i+2}) , \ldots , P^{(3)}(\xi_r) ) \Big)\text{.}
$$}
\end{theo}

Lemmas~\ref{lem:f1f3} and \ref{lem:fdouble} below give explicit expressions
for $P^{(1)}(\xi)$, $P^{(3)}(\xi)$ and $P(\oplus[\xi_i,\xi_j])$
for every increasing oscillations $\xi$, $\xi_i$ and
$\xi_j$, hence with Theorem~\ref{thm:pinwords_cas_lineaire_non_recursif}
an explicit expression for $P(\pi)$.  Note that 
similar results can be obtained for permutations with root $\ominus$
and decreasing oscillations using Remark~\ref{rem:ominus=oplus_transposed} (p.\pageref{rem:ominus=oplus_transposed}). 

In the following lemmas, we distinguish several cases according to the type $(x,y)$, 
with $x,y \in \{H,V\}$, of increasing oscillations -- see Subsection~\ref{ssec:oscillations} 
for the definition of these types. 
It starts with Lemma~\ref{lem:f1f3}, 
whose proof immediately follows from a comprehensive study 
of the different cases illustrated in Figure~\ref{fig:lecture_epi_quadrant_1ou3} (p.\pageref{fig:lecture_epi_quadrant_1ou3}). 

\begin{lem}\label{lem:f1f3}
Let $\xi$ be an increasing oscillation of size  $n \geq 5$.

If $n$ is even, let $n=2p+2$, then $$\small{P^{(1)}(\xi) =
\begin{cases}
3L(DL)^{p}\text{ if $\xi$ has type } (H,H) \\
3D(LD)^{p}\text{ if $\xi$ has type } (V,V) \\
\end{cases}}
\small{P^{(3)}(\xi) = \begin{cases}
1R(UR)^{p}\text{ if $\xi$ has type } (H,H) \\
1U(RU)^{p}\text{ if $\xi$ has type } (V,V)\text{.} \\
\end{cases}}$$

If $n$ is odd, let $n=2p+1$, then $$\small{P^{(1)}(\xi) = \begin{cases}
  3(DL)^{p}\text{ if $\xi$ has  type } (H,V) \\
  3(LD)^{p}\text{ if $\xi$ has type } (V,H) \\
\end{cases}}
\small{P^{(3)}(\xi) = \begin{cases}
1(RU)^{p}\text{ if $\xi$ has type } (H,V) \\
1(UR)^{p}\text{ if $\xi$ has type } (V,H)\text{.} \\
\end{cases}}$$

For the increasing oscillations of size less than $4$,
the values of $P^{(1)}$ and $P^{(3)}$ are:
$$\footnotesize{\begin{array}{llll}
  P^{(1)}(1)= 3 & P^{(3)}(1)= 1 & P^{(1)}(21)= \{3D, 3L\} & P^{(3)}(21)=
  \{1R,1U\} \\
  P^{(1)}(231)=3DL & P^{(3)}(231)=
  1RU &P^{(1)}(312) = 3LD & P^{(3)}(312)= 1UR \\
  P^{(1)}(2413)=3LDL & P^{(3)}(2413)=1RUR & P^{(1)}(3142)=3DLD &
  P^{(3)}(3142)=1URU \\
\end{array}}$$
\label{lem:lecture_epi_quadrant_1ou3}
\end{lem}

\begin{rem}\label{rem:P1P3}
  If the increasing oscillation $\xi$ is of size $2$ then
  $\fell(\xi)$ contains two words, otherwise it is a
  singleton. 
  Moreover, for the map $\phi$ studied in Section~\ref{sec:strict pw} (see Definition~\ref{def:phi} p.\pageref{def:phi}), and for any increasing oscillation $\xi$, we have $\phi(P^{(3)}(\xi)) \subseteq \{U,R\}^{\star}$ and $\phi(P^{(1)}(\xi)) \subseteq \{L,D\}^{\star}$.
\end{rem}

We are further interested in describing
the set of pin words of $\oplus[\xi_i,\xi_j]$
for any increasing oscillations $\xi_i$ and $\xi_j$.
This is achieved in Lemma~\ref{lem:fdouble}.
For this purpose, we first describe the set $P(\xi)$ of pin words of any increasing oscillation $\xi$,
and the set $\fmix(\xi_i,\xi_j)$ of pin words of $\oplus[\xi_i,\xi_j]$ such that one of the two oscillations is read in two \fois.

\begin{lem}\label{lem:lecture_epi}
  Let $\Qminus$ (resp. $\SHminus$, resp. $\SVminus$) be the set of pin words of the
  permutation $21$ that are quasi-strict (resp. that are strict and
  end with $R$ or $L$, resp. with $U$ or $D$): $\Qminus =
  \{12,14,22,24,32,34,42,44\}$, $\SHminus = \{1R,2R,3L,4L\}$ and $\SVminus =
  \{1U,2D,3D,4U\}$.  Define similarly $\Qplus$ (resp. $\SHplus$,
  resp. $\SVplus$) for the permutation $12$.

Let $\xi$ be an increasing oscillation of size  $n \geq 5$.

If $n$ is even, let $n=2p+2$, then $$P(\xi) = \begin{cases}
(\Qminus+\SHminus)\cdot (DL)^{p} + (\Qminus+\SHminus)\cdot(UR)^{p}\text{ if $\xi$ has type } (H,H) \\
(\Qminus+\SVminus)\cdot (LD)^{p} + (\Qminus+\SVminus)\cdot(RU)^{p}\text{ if $\xi$ has type } (V,V)\text{.} \\
\end{cases}$$

If $n$ is odd, let $n=2p+1$, then $$P(\xi) = \begin{cases}
(\Qminus+\SVminus)\cdot L(DL)^{p-1} + (\Qminus+\SHminus)\cdot U(RU)^{p-1}\text{ if $\xi$ has type } (H,V) \\
(\Qminus+\SHminus)\cdot D(LD)^{p-1} + (\Qminus+\SVminus)\cdot R(UR)^{p-1}\text{ if $\xi$ has type } (V,H)\text{.} \\
\end{cases}$$

For the increasing oscillations of size less than $5$,
we have:

$\begin{array}{l}
  P(1)= \{1,2,3,4\} \qquad P(21)= \Qminus+\SHminus+\SVminus \\
  P(231)=(\Qminus+\SHminus)\cdot U + (\Qminus+\SVminus)\cdot L + (\Qplus+\SHplus+\SVplus) \cdot 4\\
  P(312)= (\Qminus+\SHminus)\cdot D + (\Qminus+\SVminus) \cdot R+ (\Qplus+\SHplus+\SVplus) \cdot 2  \\
  P(2413)=(\Qminus+\SHminus)\cdot (UR+DL) + (\Qplus+\SVplus)\cdot (RD+LU) \\
  P(3142)=(\Qplus+\SHplus)\cdot (UL+DR) + (\Qminus+\SVminus)\cdot (RU+LD) \\
\end{array}$

\medskip

In particular, $|P(\xi)| \leq 48$ for any increasing oscillation $\xi$ and if $|\xi| \neq 3$, $P(\xi)$ contains only strict and quasi-strict pin words.
\end{lem}

\begin{proof}
By comprehensive examination of the cases illustrated in Figure~\ref{fig:lecture_epi_quadrant_1ou3}.
\end{proof}

\begin{defi}\label{def:fdouble}
For any pair of increasing oscillations $(\xi_i,\xi_j)$, we denote
by $\fmix(\xi_i,\xi_j)$ the set of pin words encoding a pin
representation of $\oplus[\xi_i,\xi_j]$ that reads one of the two
oscillations in two \fois.
\end{defi}

\begin{lem}
Let $\xi_i$ and $\xi_j$ be two increasing oscillations.

If none of these two oscillations is of size $1$, or if both of them
are of size $1$, then $\fmix(\xi_i,\xi_j)$ is empty.

Otherwise, assume w.l.o.g.~that $\xi_i =1$, and set $|\xi_j|=2p+q+1$
with $q \in \{0,1\}$. If $|\xi_j|\geq 4$, then
$$\footnotesize{\fmix(\xi_i,\xi_j) = \begin{cases}
(13+23+33+43+1D+4D)\cdot(RU)^{p}R^q\text{ if $\xi_j$ has type } (H,H) \text{ or } (H,V) \\
(13+23+33+43+1L+2L)\cdot(UR)^{p}U^q\text{ if $\xi_j$ has type } (V,V) \text{ or } (V,H)\text{.} \\
\end{cases}}$$
If $|\xi_j|=3$, {\it i.e.}, $\xi_j = 231$ or $312$, we have
$$\begin{array}{l}
\fmix(1,231) = P(12) \cdot 3R +(13+23+33+43+1D+4D)\cdot RU\text{,}\\
\fmix(1,312) = P(12) \cdot 3U +(13+23+33+43+1L+2L)\cdot UR\text{.}
\end{array}$$
If $|\xi_j|=2$, {\it i.e.}, $\xi_j = 21$, we have
$$\fmix(1,21) = (13+23+33+43+1D+4D)\cdot R + (13+23+33+43+1L+2L) \cdot U \text{.}$$
In particular, $|\fmix(\xi_i,\xi_j)| \leq 22$ for any increasing oscillations $\xi_i$ and $ \xi_j$.
\label{lem:lecture_epi_mix}
\end{lem}
\begin{proof}
  From Lemma~\ref{lem:cas_de_base_oplus_lecture_en_deux_fois} (p.\pageref{lem:cas_de_base_oplus_lecture_en_deux_fois}) when one
  of the oscillations $\xi_i$ or $\xi_j$ is read in two \fois, then the other one 
  has size $1$. W.l.o.g.~assume that $|\xi_i|=1$. Then
  from Lemma~\ref{lem:readInTwoTimes} (p.\pageref{lem:readInTwoTimes})
  the decomposition tree of $\oplus[\xi_i,\xi_j]$ has one
  of the shapes of Figure~\ref{fig:severaltime} (p.\pageref{fig:severaltime}), with $T_{i_0}$
  corresponding to $\xi_j$ and $x$ corresponding to $\xi_i$.

If $\xi_j$ has size $2$, then $\xi_j=21$ and $\oplus[\xi_i,\xi_j]$ maps only to configurations $G1$ and $G2$ (see Figures~\ref{fig:severaltime} and~\ref{fig:cases}).
Therefore there are two pin representations of $\oplus[\xi_i,\xi_j]$ where $\xi_j$ is read in two \fois.
The twelve corresponding pin words (see Remark~\ref{rem:nb_pin_words} p.\pageref{rem:nb_pin_words})
are those given in Lemma~\ref{lem:lecture_epi_mix}.

If $\xi_j$ has size $3$, then $\xi_j=231$ (resp. $312$) and $\oplus[\xi_i,\xi_j]$ maps only to configurations $G2$ and $G4$ (resp. $G1$ and $G3$), and we conclude similarly.

Otherwise, $|\xi_j|\geq 4$.
Since $\xi_j$ is an increasing oscillation, $\oplus[\xi_i,\xi_j]$ maps only to configuration $G3+$ or to configuration $G4+$, with $|S|=1$.
These two cases are exclusive, and the configuration ($G3+$ or $G4+$) to which $\oplus[\xi_i,\xi_j]$ maps is determined by the type ($V$ or $H$) of the lower left active knight of $\xi_j$.
Lemma~\ref{lem:readInTwoTimes} and Remark~\ref{rem:inTwoTimes} (p.\pageref{rem:inTwoTimes})
ensure that there is exactly one pin representation for $\oplus[\xi_i,\xi_j]$ that reads $\xi_j$ in two \fois.
The six corresponding pin words are those given in Lemma~\ref{lem:lecture_epi_mix}. 
\end{proof}

From the expressions of $\fint, \fmix, P^{(1)}$ and $P^{(3)}$ we can deduce the
explicit expression of $P(\oplus[\xi_i,\xi_j])$ making use of the following result:

\begin{lem}\label{lem:fdouble}
  For any pair of increasing oscillations $(\xi_i,\xi_j)$:
  \begin{itemize}
  \item  If $|\xi_i| > 1$ and $|\xi_j|>1$, $P(\oplus[\xi_i,\xi_j]) = \big( \fint (\xi_j) \cdot P^{(1)} (\xi_i) \big) \bigcup \big(
 \fint (\xi_i) \cdot P^{(3)} (\xi_j) \big) 
$
 \item  If $|\xi_i| = 1$ and $|\xi_j| = 1$, $P(\oplus[\xi_i,\xi_j]) = P(12)$
 \item Otherwise assume w.l.o.g. that $|\xi_i| = 1$ and $|\xi_j| =2p+q$ with $q \in \{0,1\}$:
 $$\hspace*{-0.8cm}P(\oplus[\xi_i,\xi_j]) = \fmix (\xi_i,\xi_j) \bigcup \big( P(\xi_j) \cdot 3 \big) \bigcup \left( \{1,2,3,4\}\cdot P^{(3)}(\xi_j) \right) \bigcup P^{sep}(\xi_j)$$
 $$\text{with } P^{sep}(\xi_j) = \begin{cases}
(2+3)\cdot(UR)^{p}U^q\text{ if $\xi_j$ has type } (H,H) \text{ or } (H,V) \\
(3+4)\cdot(RU)^{p}R^q\text{ if $\xi_j$ has type } (V,V) \text{ or } (V,H)\text{.}
\end{cases}$$
 \end{itemize}
In particular, $|P(\oplus[\xi_i,\xi_j])| \leq 192$ for any oscillations $\xi_i$ and $ \xi_j$.
\end{lem}

\begin{proof}
For the first item, Lemma~\ref{lem:lecture_epi_mix} ensures that the
pin words of $\oplus[\xi_i,\xi_j]$ encode pin representations
reading both $\xi_i$ and $\xi_j$ in one piece. The two terms of
the union are obtained according to which oscillation (among
$\xi_i$ and $\xi_j$) is read first.  The second statement follows
directly from the fact that $\oplus[\xi_i,\xi_j] = 12$ in this
case.  From Lemma~\ref{lem:lecture_epi_mix}, the situation of the
third statement is the one where there are pin words encoding pin
representations reading $\xi_j$ in two \fois. The four terms of the
union account for four different kinds of pin words. Namely, $\fmix
(\xi_i,\xi_j)$ is the set of pin words reading $\xi_j$ in two
\fois, $P(\xi_j) \cdot 3$ is the set of pin words reading first
$\xi_j$ and then $\xi_i$, $\{1,2,3,4\}\cdot P^{(3)}(\xi_j)$ is
the set of pin words reading first $\xi_i$ and then $\xi_j$
starting with an independent pin, and $P^{sep}(\xi_j) $ is the set
of pin words reading first $\xi_i$ and then $\xi_j$ starting
with a separating pin.  Notice that in this last situation, the first
pin of $\xi_j$ may be separating only because $|\xi_i| =1$, so
that this case does not need to appear in the first item.
\end{proof}

This completes the explicit description of all the sets of pin words
appearing in Theorem~\ref{thm:pinwords_cas_lineaire_non_recursif}.

\subsection{Recursive case: decomposition trees with a linear root}

We now focus on pin-permutations whose decomposition trees have
a linear root $\oplus$ and a child $T_{i_0}$ which is not an increasing
oscillation. From~\cite[Lemma 3.7]{BBR09} $T_{i_0}$ is then the first
child read by any pin representation.
Lemma~\ref{lem:readInTwoTimes} (p.\pageref{lem:readInTwoTimes}) gives a characterization of
permutations in which $T_{i_0}$ may be read in several \fois. Moreover
from Remark~\ref{rem:inTwoTimes} if $T_{i_0}$ is read in two \fois the
first part $S$ being fixed, then the order of the points of the
remaining part is uniquely determined. Nevertheless, since some
permutations may satisfy several conditions $F1$ to $G4+$ of Lemma~\ref{lem:readInTwoTimes}, the first part $S$ to be read is not
uniquely determined. For example every permutation satisfying $F3$
also satisfies $F1$, and some permutations satisfy both $F1$ and
$G2$ (see Figure~\ref{fig:cases} p.\pageref{fig:cases}).
In Figure~\ref{fig:H} we classify the permutations according to
the conditions they satisfy.

Let $\setH$ be the set of permutations in which $T_{i_0}$ may be
read in several \fois. Then any permutation of $\setH$ satisfies exactly
one of the conditions $\horse{i}{j}$ of Figure~\ref{fig:H}.
We say that a permutation satisfies condition $\horse{i}{j}$ when 
its diagram has the corresponding shape in Figure~\ref{fig:H} (up to symmetry)
and does not satisfy any condition that appears above $\horse{i}{j}$ in Figure~\ref{fig:H}.
For example a permutation in $\horse{1}{2}$ cannot be in $\horse{2}{2}$.
One can check by a comprehensive verification that there is no other combination (up to
symmetry) of the conditions $F1$ to $G4+$ of Lemma~\ref{lem:readInTwoTimes}.  Moreover as $T_{i_0}$ is not an increasing
oscillation, the sets $S$ and $T$ that appear on Figure~\ref{fig:H} are such that $|S|\geq 2$ and $|T|\geq 1$.
\begin{figure}
\begin{tabular}{m{1.2cm}m{2cm}m{2.3cm}m{5cm}}
~ & \small{Diagram} & \hspace*{-0.5cm}\small{Decomposition tree} & \small{Automaton (see Appendix~\ref{sec:buildingAutomata})} \\
$\underset{(F4+G4)}{2H3}$
&
\begin{tikzDiagramme}{.2}
\draw[boite] (0,0) rectangle  (1,1);
\draw[fill] (0.5,0.5) circle (3pt) node[below=2pt] {\tiny $y$};
\draw[boite] (1,1) rectangle (6,6);
\draw[boite] (1,2) rectangle (5,6);
\draw[boite] (2,3) rectangle node {$T$} (4,5);
\draw[fill] (1.5,2.5) circle (3pt) node[right=2pt] {\tiny $a$};
\draw[fill] (4.5,5.5) circle (3pt) node[above=2pt] {\tiny $b$};
\draw[fill] (5.5,1.5) circle (3pt) node[right=2pt] {\tiny $c$};
\draw[boite] (6,6) rectangle (7,7);
\draw[fill] (6.5,6.5) circle (3pt) node[right=2pt] {\tiny $x$};
\draw[decorate,decoration={brace, amplitude=7pt}] (1,2) -- node[left=3pt]
 {\small $S$} (1,6);
\end{tikzDiagramme}
&
\begin{tikzArbre}{.4}
\begin{scope}[yshift=120]
  \node [linear]{$\oplus$} child[dotted] {node (a) {}} child { node[leaf](b) {}} child {node[linear] {$\ominus$}{ child {node[linear] {$\oplus$}
        child {node[leaf]{} node[below=3pt]{\small $a$}} child
        [level distance=1pt,xshift=0pt,yshift=-7pt]
        {node[draw,shape=isosceles triangle, shape border
          rotate=90,anchor=north] {\small $T$} edge from
          parent[draw=none]} child {node[leaf]{}
          node[below=3pt]{\small $b$}} } child {node[leaf]{}
        node[below=3pt]{\small $c$}} }} child {node[leaf]{}
    node(c) {}} child [dotted] {node (d) {}} ; \draw[dotted, thick]
  (a) -- (b); \draw[dotted, thick] (c) -- (d); \draw (c) node [below
  right=2pt] {$x$}; \draw (b) node [below left=2pt] {$y$};
\end{scope}
\end{tikzArbre}
&
\begin{tikzpicture}[scale=.2, shorten >= 2pt]
\useasboundingbox (-0.5,-0.5) rectangle (23,9);
\begin{scope}[scale=.7]
\patateTriangleUnlabelled{0}{0};
\patateTriangleUnlabelled{2}{0};
\patateTriangleUnlabelled{6}{0};
\patateTriangleUnlabelled{0}{2};
\patateTriangleUnlabelled{2}{2};
\patateTriangleUnlabelled{6}{2};
\patateTriangleUnlabelled{0}{6};
\patateTriangleUnlabelled{2}{6};
\patateTriangleUnlabelled{6}{6};
\begin{scope}
\draw (0.2,7.85) rectangle  +(1.6,0.3);
\inputState{0.35}{8};
\outputState{1.65}{8};
\end{scope}
\begin{scope}[shift={(2,0)}]
\draw (0.2,7.85) rectangle  +(1.6,0.3);
\inputState{0.35}{8};
\outputState{1.65}{8};
\end{scope}
\begin{scope}[shift={(6,0)}]
\draw (0.2,7.85) rectangle +(1.6,0.3);
\draw [pattern=north east lines] (0.2,7.85) rectangle +(0.3,0.3);
\inputState{0.35}{8};
\outputState{1.65}{8};
\end{scope}
\inputState{8.3}{8};
\begin{scope}[shift={(0,0)}]
\draw [rounded corners, fill=black!10,opacity=.5] (0,5) -- (0,6) -- (1,5);
\outputState{0.15}{5.7};
\end{scope}
\begin{scope}[shift={(2,0)}]
\draw [rounded corners, fill=black!10,opacity=.5] (0,5) -- (0,6) -- (1,5);
\outputState{0.15}{5.7};
\end{scope}
\begin{scope}[shift={(8,0)}]
\draw (-0.15,0.2) rectangle +(0.3,1.6);
\inputState{0}{0.35};
\outputState{0}{1.65};
\end{scope}
\begin{scope}[shift={(8,2)}]
\draw (-0.15,0.2) rectangle +(0.3,1.6);
\inputState{0}{0.35};
\outputState{0}{1.65};
\end{scope}
\begin{scope}[shift={(8,6)}]
\draw (-0.15,0.2) rectangle +(0.3,1.6);
\inputState{0}{0.35};
\outputState{0}{1.65};
\draw [pattern=north east lines] (-0.15,0.2) rectangle +(0.3,0.3);
\end{scope}

\draw [fill=black!20,opacity=.2, very thick] (2,0) circle (.5cm);
\draw [fill=black!20,opacity=.2, very thick] (0,2) circle (.5cm);
\draw [fill=black!20,opacity=.2, very thick] (2,2) circle (.5cm);
\draw [fill=black!20,opacity=.2, very thick] (6,2) circle (.5cm);
\draw [fill=black!20,opacity=.2, very thick] (8,2) circle (.5cm);
\draw [fill=black!20,opacity=.2, very thick] (0,6) circle (.5cm);
\draw [fill=black!20,opacity=.2, very thick] (2,6) circle (.5cm);
\draw [fill=black!20,opacity=.2, very thick] (8,6) circle (.5cm);
\draw [fill=black!20,opacity=.2, very thick] (0,8) circle (.5cm);
\draw [fill=black!20,opacity=.2, very thick] (2,8) circle (.5cm);
\draw [fill=black!20,opacity=.2, very thick] (4,8) circle (.5cm);
\draw [fill=black!20,opacity=.2, very thick] (6,8) circle (.5cm);
\draw [fill=black!20,opacity=.2, very thick] (8,8) circle (.5cm);
\draw [fill=black!20,opacity=.2, very thick] (4,0) circle (.5cm);
\draw [fill=black!20,opacity=.2, very thick] (4,2) circle (.5cm);
\draw [fill=black!20,opacity=.2, very thick] (4,6) circle (.5cm);
\draw [fill=black!20,opacity=.2, very thick] (8,0) circle (.5cm);
\draw [fill=black!20,opacity=.2, very thick] (8,4) circle (.5cm);
\draw [fill=black!20,opacity=.2, very thick] (6,0) circle (.5cm);
\draw (8,8) node (exit) {};
\draw (8,7) node (exit87) {};
\draw (7,8) node (exit78) {};
\draw [drop shadow,fill=black!10] (20,8) node[below=0.5,xshift=20pt] {${\mathcal A}_{\rho}$} ellipse  (12 and 3) ;
\draw [fill=black] (20,7) node[right]  {\small $q_{T\cup b}$} circle (10pt);
\draw [->] (8,6) .. controls +(0,-5) and +(0,-10) .. node[below] {\tiny $LDRU$} (20,7);
\draw [fill=black] (20,9) node[right=0pt]  {\small $q_{T \cup a}$} circle (10pt);
\draw [->] (8,6) .. controls +(0,-2) and +(-3,-10) .. node[above=-2pt] {\tiny $DRU$} (18,8);
\draw [->] (6,8) .. controls +(0,3) and +(0,3) .. node[above=-2pt] {\tiny $RDL$} (18,8);
\draw [->] (6,8) .. controls +(0,+4) and +(0,+6) .. node[above=-2pt] {\tiny $URDL$} (20,9);
\draw [fill=black] (18,8) node[left]  {\small $q_{S}$} circle (10pt);
\end{scope}
\end{tikzpicture}
\\
$\underset{(F1+G2)}{2H2}$ &
\begin{tikzDiagramme}{.2}
\draw[boite] (0,0) rectangle  (1,1);
\draw[fill] (0.5,0.5) circle (3pt) node[below=2pt] {\tiny $y$};
\draw[boite] (1,1) rectangle (6,6);
\draw[boite] (2,2) rectangle node {$T$} (5,5);
\draw[fill] (1.5,5.5) circle (3pt) node[left=2pt] {\tiny $a$};
\draw[fill] (5.5,1.5) circle (3pt) node[right=2pt] {\tiny $b$};
\draw[boite] (6,6) rectangle (7,7);
\draw[fill] (6.5,6.5) circle (3pt) node[right=2pt] {\tiny $x$};
\end{tikzDiagramme}
&
\begin{tikzArbre}{.4}
\begin{scope}[yshift=120]
  \node[linear] {$\oplus$} child[dotted] {node (a) {}} child {node[leaf]{} node[below=3pt]{\small $y$} (b) {}} child {node[linear] {$\ominus$}{
      child {node[leaf]{} node[below=3pt]{\small $a$}} child
      [level distance=1pt,xshift=0pt,yshift=-7pt]
      {node[draw,shape=isosceles triangle, shape border
        rotate=90,anchor=north] {\small $T$} edge from
        parent[draw=none]} child {node[leaf]{}
        node[below=3pt]{\small $b$}} } } child {node[leaf]{}
    node[below=3pt]{\small $x$} (c) {}} child [dotted] {node (d) {}} ;
  \draw[dotted, thick] (a) -- (b); \draw[dotted, thick] (c) -- (d);
\end{scope}
\end{tikzArbre}
&
\begin{tikzpicture}[scale=.2, shorten >= 2pt]
\useasboundingbox (-0.5,-0.5) rectangle (23,9);
\begin{scope}[scale=.7]
\patateTriangleUnlabelled{0}{0};
\patateTriangleUnlabelled{2}{0};
\patateTriangleUnlabelled{6}{0};
\patateTriangleUnlabelled{0}{2};
\patateTriangleUnlabelled{2}{2};
\patateTriangleUnlabelled{6}{2};
\patateTriangleUnlabelled{0}{6};
\patateTriangleUnlabelled{2}{6};
\patateTriangleUnlabelled{6}{6};
\begin{scope}
\draw (0.2,7.85) rectangle  +(1.6,0.3);
\inputState{0.35}{8};
\outputState{1.65}{8};
\end{scope}
\begin{scope}[shift={(2,0)}]
\draw (0.2,7.85) rectangle  +(1.6,0.3);
\inputState{0.35}{8};
\outputState{1.65}{8};
\end{scope}
\begin{scope}[shift={(6,0)}]
\draw (0.2,7.85) rectangle +(1.6,0.3);
\draw [pattern=north east lines] (0.2,7.85) rectangle +(0.3,0.3);
\inputState{0.35}{8};
\outputState{1.65}{8};
\end{scope}
\inputState{8.3}{8};
\begin{scope}[shift={(0,0)}]
\draw [rounded corners, fill=black!10,opacity=.5] (0,5) -- (0,6) -- (1,5);
\outputState{0.15}{5.7};
\end{scope}
\begin{scope}[shift={(2,0)}]
\draw [rounded corners, fill=black!10,opacity=.5] (0,5) -- (0,6) -- (1,5);
\outputState{0.15}{5.7};
\end{scope}
\begin{scope}[shift={(8,0)}]
\draw (-0.15,0.2) rectangle +(0.3,1.6);
\inputState{0}{0.35};
\outputState{0}{1.65};
\end{scope}
\begin{scope}[shift={(8,2)}]
\draw (-0.15,0.2) rectangle +(0.3,1.6);
\inputState{0}{0.35};
\outputState{0}{1.65};
\end{scope}
\begin{scope}[shift={(8,6)}]
\draw (-0.15,0.2) rectangle +(0.3,1.6);
\inputState{0}{0.35};
\outputState{0}{1.65};
\draw [pattern=north east lines] (-0.15,0.2) rectangle +(0.3,0.3);
\end{scope}

\draw [fill=black!20,opacity=.2, very thick] (2,0) circle (.5cm);
\draw [fill=black!20,opacity=.2, very thick] (0,2) circle (.5cm);
\draw [fill=black!20,opacity=.2, very thick] (2,2) circle (.5cm);
\draw [fill=black!20,opacity=.2, very thick] (6,2) circle (.5cm);
\draw [fill=black!20,opacity=.2, very thick] (8,2) circle (.5cm);
\draw [fill=black!20,opacity=.2, very thick] (0,6) circle (.5cm);
\draw [fill=black!20,opacity=.2, very thick] (2,6) circle (.5cm);
\draw [fill=black!20,opacity=.2, very thick] (8,6) circle (.5cm);
\draw [fill=black!20,opacity=.2, very thick] (0,8) circle (.5cm);
\draw [fill=black!20,opacity=.2, very thick] (2,8) circle (.5cm);
\draw [fill=black!20,opacity=.2, very thick] (4,8) circle (.5cm);
\draw [fill=black!20,opacity=.2, very thick] (6,8) circle (.5cm);
\draw [fill=black!20,opacity=.2, very thick] (8,8) circle (.5cm);
\draw [fill=black!20,opacity=.2, very thick] (4,0) circle (.5cm);
\draw [fill=black!20,opacity=.2, very thick] (4,2) circle (.5cm);
\draw [fill=black!20,opacity=.2, very thick] (4,6) circle (.5cm);
\draw [fill=black!20,opacity=.2, very thick] (8,0) circle (.5cm);
\draw [fill=black!20,opacity=.2, very thick] (8,4) circle (.5cm);
\draw [fill=black!20,opacity=.2, very thick] (6,0) circle (.5cm);
\draw (8,8) node (exit) {};
\draw (8,7) node (exit87) {};
\draw (7,8) node (exit78) {};
\draw [drop shadow,fill=black!10] (20,8) node[below=0.5,xshift=20pt] {${\mathcal A}_{\rho}$} ellipse  (12 and 3) ;
\draw [fill=black] (20,7) node[right]  {\small $q_{T\cup b}$} circle (10pt);
\draw [->] (8,6) .. controls +(0,-5) and +(0,-10) .. node[below] {\tiny $LUR$} (20,7);
\draw [fill=black] (20,9) node[left=0pt]  {\small $q_{T \cup a}$} circle (10pt);
\draw [->] (8,6) .. controls +(0,-4) and +(-3,-10) .. node[above=-2pt] {\tiny $DRU$} (20,9);
\draw [->] (6,8) .. controls +(0,3) and +(0,3) .. (10,6) .. controls +(0,-3) and +(-2,-1) .. node[above=-2pt] {\tiny $ULD$} (20,7);
\draw [->] (6,8) .. controls +(0,+3) and +(0,+6) .. node[above=-2pt] {\tiny $RDL$} (20,9);
\end{scope}
\end{tikzpicture}
\\
$\underset{(F4+G2)}{2H2\star}$ &
\begin{tikzDiagramme}{.2}
\draw[boite] (0,0) rectangle  (1,1);
\draw[fill] (0.5,0.5) circle (3pt) node[below=2pt] {\tiny $y$};
\draw[boite] (1,1) rectangle (6,6);
\draw[boite] (1,2) rectangle (5,6);
\draw[boite] (2,3) rectangle node {$S$} (5,6);
\draw[fill] (1.5,2.5) circle (3pt) node[right=2pt] {\tiny $a$};
\draw[fill] (5.5,1.5) circle (3pt) node[right=2pt] {\tiny $b$};
\draw[boite] (6,6) rectangle (7,7);
\draw[fill] (6.5,6.5) circle (3pt) node[right=2pt] {\tiny $x$};
\draw[decorate,decoration={brace, amplitude=7pt}] (1,2) -- node[left=3pt]
{\small $S'$} (1,6);
\end{tikzDiagramme}
&
\begin{tikzArbre}{.4}
\begin{scope}[yshift=120]
  \node[linear] {$\oplus$} child[dotted] {node (a) {}} child {node[leaf]{} node[below=3pt]{\small $y$} (b) {}} child {node[linear] {$\ominus$}{
      child {node[linear] {$\oplus$} child {node[leaf]{}
          node[below=3pt]{\small $a$}} child [level
        distance=1pt,xshift=0pt,yshift=-7pt]
        {node[draw,shape=isosceles triangle, shape border
          rotate=90,anchor=north] {\small $S$} edge from
          parent[draw=none]} child [missing]{node[leaf]{}} }
      child {node[leaf]{} node[below=3pt]{\small $b$}} }} child
  {node[leaf]{} node[below=3pt]{\small $x$} (c) {}} child
  [dotted] {node (d) {}} ; \draw[dotted, thick] (a) -- (b);
  \draw[dotted, thick] (c) -- (d);
\end{scope}
\end{tikzArbre}
&
\begin{tikzpicture}[scale=.2, shorten >= 2pt]
\useasboundingbox (-0.5,-0.5) rectangle (23,9);
\begin{scope}[scale=.7]
\patateTriangleUnlabelled{0}{0};
\patateTriangleUnlabelled{2}{0};
\patateTriangleUnlabelled{6}{0};
\patateTriangleUnlabelled{0}{2};
\patateTriangleUnlabelled{2}{2};
\patateTriangleUnlabelled{6}{2};
\patateTriangleUnlabelled{0}{6};
\patateTriangleUnlabelled{2}{6};
\patateTriangleUnlabelled{6}{6};
\begin{scope}
\draw (0.2,7.85) rectangle  +(1.6,0.3);
\inputState{0.35}{8};
\outputState{1.65}{8};
\end{scope}
\begin{scope}[shift={(2,0)}]
\draw (0.2,7.85) rectangle  +(1.6,0.3);
\inputState{0.35}{8};
\outputState{1.65}{8};
\end{scope}
\begin{scope}[shift={(6,0)}]
\draw (0.2,7.85) rectangle +(1.6,0.3);
\draw [pattern=north east lines] (0.2,7.85) rectangle +(0.3,0.3);
\inputState{0.35}{8};
\outputState{1.65}{8};
\end{scope}
\inputState{8.3}{8};
\begin{scope}[shift={(0,0)}]
\draw [rounded corners, fill=black!10,opacity=.5] (0,5) -- (0,6) -- (1,5);
\outputState{0.15}{5.7};
\end{scope}
\begin{scope}[shift={(2,0)}]
\draw [rounded corners, fill=black!10,opacity=.5] (0,5) -- (0,6) -- (1,5);
\outputState{0.15}{5.7};
\end{scope}
\begin{scope}[shift={(8,0)}]
\draw (-0.15,0.2) rectangle +(0.3,1.6);
\inputState{0}{0.35};
\outputState{0}{1.65};
\end{scope}
\begin{scope}[shift={(8,2)}]
\draw (-0.15,0.2) rectangle +(0.3,1.6);
\inputState{0}{0.35};
\outputState{0}{1.65};
\end{scope}
\begin{scope}[shift={(8,6)}]
\draw (-0.15,0.2) rectangle +(0.3,1.6);
\inputState{0}{0.35};
\outputState{0}{1.65};
\draw [pattern=north east lines] (-0.15,0.2) rectangle +(0.3,0.3);
\end{scope}

\draw [fill=black!20,opacity=.2, very thick] (2,0) circle (.5cm);
\draw [fill=black!20,opacity=.2, very thick] (0,2) circle (.5cm);
\draw [fill=black!20,opacity=.2, very thick] (2,2) circle (.5cm);
\draw [fill=black!20,opacity=.2, very thick] (6,2) circle (.5cm);
\draw [fill=black!20,opacity=.2, very thick] (8,2) circle (.5cm);
\draw [fill=black!20,opacity=.2, very thick] (0,6) circle (.5cm);
\draw [fill=black!20,opacity=.2, very thick] (2,6) circle (.5cm);
\draw [fill=black!20,opacity=.2, very thick] (8,6) circle (.5cm);
\draw [fill=black!20,opacity=.2, very thick] (0,8) circle (.5cm);
\draw [fill=black!20,opacity=.2, very thick] (2,8) circle (.5cm);
\draw [fill=black!20,opacity=.2, very thick] (4,8) circle (.5cm);
\draw [fill=black!20,opacity=.2, very thick] (6,8) circle (.5cm);
\draw [fill=black!20,opacity=.2, very thick] (8,8) circle (.5cm);
\draw [fill=black!20,opacity=.2, very thick] (4,0) circle (.5cm);
\draw [fill=black!20,opacity=.2, very thick] (4,2) circle (.5cm);
\draw [fill=black!20,opacity=.2, very thick] (4,6) circle (.5cm);
\draw [fill=black!20,opacity=.2, very thick] (8,0) circle (.5cm);
\draw [fill=black!20,opacity=.2, very thick] (8,4) circle (.5cm);
\draw [fill=black!20,opacity=.2, very thick] (6,0) circle (.5cm);
\draw (8,8) node (exit) {};
\draw (8,7) node (exit87) {};
\draw (7,8) node (exit78) {};
\draw [drop shadow,fill=black!10] (20,8) node[below=0.5,xshift=20pt] {${\mathcal A}_{\rho}$} ellipse  (12 and 3) ;
\draw [fill=black] (22,8) node[above right]  {\small $q_{S}$} circle (10pt);
\draw [fill=black] (20,8) node[left]  {\small $q_{S'}$} circle (10pt);
\draw [->] (8,6) .. controls +(0,-4) and +(0,-10) .. node[below] {\tiny $LDRU$} (22,8);
\draw [->] (8,6) .. controls +(0,-4) and +(0,-8) .. node[above=-2pt] {\tiny $DRU$} (20,8);
\draw [->] (6,8) .. controls +(0,+3) and +(0,+8) .. node[above=-2pt] {\tiny $RDL$} (20,8);
\end{scope}
\end{tikzpicture}
\\
$\underset{(F1+G1)}{2H1}$ &
\begin{tikzDiagramme}{.2}
\draw[boite] (0,0) rectangle  (1,1);
\draw[fill] (0.5,0.5) circle (3pt) node[below=2pt] {\tiny $y$};
\draw[boite] (1,1) rectangle (6,6);
\draw[boite] (2,1) rectangle node {$S$} (6,5);
\draw[fill] (1.5,5.5) circle (3pt) node[left=2pt] {\tiny $a$};
\draw[boite] (6,6) rectangle (7,7);
\draw[fill] (6.5,6.5) circle (3pt) node[right=2pt] {\tiny $x$};
\end{tikzDiagramme}
&
\begin{tikzArbre}{.4}
\begin{scope}[yshift=120]
  \node [linear]{$\oplus$} child[dotted] {node (a) {}} child {node[leaf]{} node[below=3pt]{\small $y$} (b) {}} child {node[linear] {$\ominus$}{
      child {node[leaf]{} node[below=3pt]{\small $a$}} child
      [level distance=1pt,xshift=0pt,yshift=-7pt]
      {node[draw,shape=isosceles triangle, shape border
        rotate=90,anchor=north] {\small $S$} edge from
        parent[draw=none]} child [missing] {node[leaf]{}} } }
  child {node[leaf]{} node[below=3pt]{\small $x$} (c) {}} child
  [dotted] {node (d) {}} ; \draw[dotted, thick] (a) -- (b);
  \draw[dotted, thick] (c) -- (d);
\end{scope}
\end{tikzArbre}
&
\begin{tikzpicture}[scale=.2, shorten >= 2pt]
\useasboundingbox (-0.5,-0.5) rectangle (23,9);
\begin{scope}[scale=.7]
\patateTriangleUnlabelled{0}{0};
\patateTriangleUnlabelled{2}{0};
\patateTriangleUnlabelled{6}{0};
\patateTriangleUnlabelled{0}{2};
\patateTriangleUnlabelled{2}{2};
\patateTriangleUnlabelled{6}{2};
\patateTriangleUnlabelled{0}{6};
\patateTriangleUnlabelled{2}{6};
\patateTriangleUnlabelled{6}{6};
\begin{scope}
\draw (0.2,7.85) rectangle  +(1.6,0.3);
\inputState{0.35}{8};
\outputState{1.65}{8};
\end{scope}
\begin{scope}[shift={(2,0)}]
\draw (0.2,7.85) rectangle  +(1.6,0.3);
\inputState{0.35}{8};
\outputState{1.65}{8};
\end{scope}
\begin{scope}[shift={(6,0)}]
\draw (0.2,7.85) rectangle +(1.6,0.3);
\draw [pattern=north east lines] (0.2,7.85) rectangle +(0.3,0.3);
\inputState{0.35}{8};
\outputState{1.65}{8};
\end{scope}
\inputState{8.3}{8};
\begin{scope}[shift={(0,0)}]
\draw [rounded corners, fill=black!10,opacity=.5] (0,5) -- (0,6) -- (1,5);
\outputState{0.15}{5.7};
\end{scope}
\begin{scope}[shift={(2,0)}]
\draw [rounded corners, fill=black!10,opacity=.5] (0,5) -- (0,6) -- (1,5);
\outputState{0.15}{5.7};
\end{scope}
\begin{scope}[shift={(8,0)}]
\draw (-0.15,0.2) rectangle +(0.3,1.6);
\inputState{0}{0.35};
\outputState{0}{1.65};
\end{scope}
\begin{scope}[shift={(8,2)}]
\draw (-0.15,0.2) rectangle +(0.3,1.6);
\inputState{0}{0.35};
\outputState{0}{1.65};
\end{scope}
\begin{scope}[shift={(8,6)}]
\draw (-0.15,0.2) rectangle +(0.3,1.6);
\inputState{0}{0.35};
\outputState{0}{1.65};
\draw [pattern=north east lines] (-0.15,0.2) rectangle +(0.3,0.3);
\end{scope}

\draw [fill=black!20,opacity=.2, very thick] (2,0) circle (.5cm);
\draw [fill=black!20,opacity=.2, very thick] (0,2) circle (.5cm);
\draw [fill=black!20,opacity=.2, very thick] (2,2) circle (.5cm);
\draw [fill=black!20,opacity=.2, very thick] (6,2) circle (.5cm);
\draw [fill=black!20,opacity=.2, very thick] (8,2) circle (.5cm);
\draw [fill=black!20,opacity=.2, very thick] (0,6) circle (.5cm);
\draw [fill=black!20,opacity=.2, very thick] (2,6) circle (.5cm);
\draw [fill=black!20,opacity=.2, very thick] (8,6) circle (.5cm);
\draw [fill=black!20,opacity=.2, very thick] (0,8) circle (.5cm);
\draw [fill=black!20,opacity=.2, very thick] (2,8) circle (.5cm);
\draw [fill=black!20,opacity=.2, very thick] (4,8) circle (.5cm);
\draw [fill=black!20,opacity=.2, very thick] (6,8) circle (.5cm);
\draw [fill=black!20,opacity=.2, very thick] (8,8) circle (.5cm);
\draw [fill=black!20,opacity=.2, very thick] (4,0) circle (.5cm);
\draw [fill=black!20,opacity=.2, very thick] (4,2) circle (.5cm);
\draw [fill=black!20,opacity=.2, very thick] (4,6) circle (.5cm);
\draw [fill=black!20,opacity=.2, very thick] (8,0) circle (.5cm);
\draw [fill=black!20,opacity=.2, very thick] (8,4) circle (.5cm);
\draw [fill=black!20,opacity=.2, very thick] (6,0) circle (.5cm);
\draw (8,8) node (exit) {};
\draw (8,7) node (exit87) {};
\draw (7,8) node (exit78) {};
\draw [drop shadow,fill=black!10] (20,8) node[below=0.5,xshift=20pt] {${\mathcal A}_{\rho}$} ellipse  (12 and 3) ;
\draw [fill=black] (20,8) node[above right]  {\small $q_{S}$} circle (10pt);
\draw [->] (8,6) .. controls +(0,-4) and +(0,-10) .. node[below] {\tiny $LUR$} (20,8);
\draw [->] (6,8) .. controls +(0,4) and +(0,8) .. node[below=-2pt] {\tiny $ULD$} (20,8);
\end{scope}
\end{tikzpicture}
\\
$\underset{(F1+F2)}{1H2}$ &
\begin{tikzDiagramme}{.2}
\useasboundingbox (0,0) rectangle (7,7);
\draw[boite] (1,1) rectangle (6,6);
\draw[boite] (2,2) rectangle node {$T$} (5,5);
\draw[fill] (1.5,5.5) circle (3pt) node[left=2pt] {\tiny $a$};
\draw[fill] (5.5,1.5) circle (3pt) node[right=2pt] {\tiny $b$};
\draw[boite] (6,6) rectangle (7,7);
\draw[fill] (6.5,6.5) circle (3pt) node[right=2pt] {\tiny $x$};
\end{tikzDiagramme}
&
\begin{tikzArbre}{.4}
 \begin{scope}[yshift=120]
 \node[linear] {$\oplus$} child[dotted] {node (a) {}} child {node [linear](b)
    {$\ominus$}{ child {node[leaf]{} node[below=3pt]{\small
          $a$}} child [level distance=1pt,xshift=0pt,yshift=-7pt]
      {node[draw,shape=isosceles triangle, shape border
        rotate=90,anchor=north] {\small $T$} edge from
        parent[draw=none]} child {node[leaf]{}
        node[below=3pt]{\small $b$}} } } child {node[leaf]{}
    node[below=3pt]{\small $x$} (c) {}} child [dotted] {node (d) {}} ;
  \draw[dotted, thick] (a) -- (b); \draw[dotted, thick] (c) -- (d);
\end{scope}
\end{tikzArbre}
&
\begin{tikzpicture}[scale=.2, shorten >= 2pt]
\begin{scope}[scale=.7]
\patateTriangleUnlabelled{0}{0};
\patateTriangleUnlabelled{2}{0};
\patateTriangleUnlabelled{6}{0};
\patateTriangleUnlabelled{0}{2};
\patateTriangleUnlabelled{2}{2};
\patateTriangleUnlabelled{6}{2};
\patateTriangleUnlabelled{0}{6};
\patateTriangleUnlabelled{2}{6};
\patateTriangleUnlabelled{6}{6};
\begin{scope}
\draw (0.2,7.85) rectangle  +(1.6,0.3);
\inputState{0.35}{8};
\outputState{1.65}{8};
\end{scope}
\begin{scope}[shift={(2,0)}]
\draw (0.2,7.85) rectangle  +(1.6,0.3);
\inputState{0.35}{8};
\outputState{1.65}{8};
\end{scope}
\begin{scope}[shift={(6,0)}]
\draw (0.2,7.85) rectangle +(1.6,0.3);
\draw [pattern=north east lines] (0.2,7.85) rectangle +(0.3,0.3);
\inputState{0.35}{8};
\outputState{1.65}{8};
\end{scope}
\inputState{8.3}{8};
\begin{scope}[shift={(0,0)}]
\draw [rounded corners, fill=black!10,opacity=.5] (0,5) -- (0,6) -- (1,5);
\outputState{0.15}{5.7};
\end{scope}
\begin{scope}[shift={(2,0)}]
\draw [rounded corners, fill=black!10,opacity=.5] (0,5) -- (0,6) -- (1,5);
\outputState{0.15}{5.7};
\end{scope}
\begin{scope}[shift={(8,0)}]
\draw (-0.15,0.2) rectangle +(0.3,1.6);
\inputState{0}{0.35};
\outputState{0}{1.65};
\end{scope}
\begin{scope}[shift={(8,2)}]
\draw (-0.15,0.2) rectangle +(0.3,1.6);
\inputState{0}{0.35};
\outputState{0}{1.65};
\end{scope}
\begin{scope}[shift={(8,6)}]
\draw (-0.15,0.2) rectangle +(0.3,1.6);
\inputState{0}{0.35};
\outputState{0}{1.65};
\draw [pattern=north east lines] (-0.15,0.2) rectangle +(0.3,0.3);
\end{scope}

\draw [fill=black!20,opacity=.2, very thick] (2,0) circle (.5cm);
\draw [fill=black!20,opacity=.2, very thick] (0,2) circle (.5cm);
\draw [fill=black!20,opacity=.2, very thick] (2,2) circle (.5cm);
\draw [fill=black!20,opacity=.2, very thick] (6,2) circle (.5cm);
\draw [fill=black!20,opacity=.2, very thick] (8,2) circle (.5cm);
\draw [fill=black!20,opacity=.2, very thick] (0,6) circle (.5cm);
\draw [fill=black!20,opacity=.2, very thick] (2,6) circle (.5cm);
\draw [fill=black!20,opacity=.2, very thick] (8,6) circle (.5cm);
\draw [fill=black!20,opacity=.2, very thick] (0,8) circle (.5cm);
\draw [fill=black!20,opacity=.2, very thick] (2,8) circle (.5cm);
\draw [fill=black!20,opacity=.2, very thick] (4,8) circle (.5cm);
\draw [fill=black!20,opacity=.2, very thick] (6,8) circle (.5cm);
\draw [fill=black!20,opacity=.2, very thick] (8,8) circle (.5cm);
\draw [fill=black!20,opacity=.2, very thick] (4,0) circle (.5cm);
\draw [fill=black!20,opacity=.2, very thick] (4,2) circle (.5cm);
\draw [fill=black!20,opacity=.2, very thick] (4,6) circle (.5cm);
\draw [fill=black!20,opacity=.2, very thick] (8,0) circle (.5cm);
\draw [fill=black!20,opacity=.2, very thick] (8,4) circle (.5cm);
\draw [fill=black!20,opacity=.2, very thick] (6,0) circle (.5cm);
\draw (8,8) node (exit) {};
\draw (8,7) node (exit87) {};
\draw (7,8) node (exit78) {};
\draw [drop shadow,fill=black!10] (20,8) node[below=0.5,xshift=20pt] {${\mathcal A}_{\rho}$} ellipse  (12 and 3) ;
\draw [fill=black] (20,7) node[right]  {\small $q_{T\cup a}$} circle (10pt);
\draw [->] (8,6) .. controls +(0,-5) and +(0,-10) .. node[below] {\tiny $DRU$} (20,7);
\draw [fill=black] (20,9) node[left=0pt]  {\small $q_{T \cup b}$} circle (10pt);
\draw [->] (8,6) .. controls +(0,-3) and +(-2,-10) .. node[above=-2pt] {\tiny $LUR$} (20,9);
\end{scope}
\end{tikzpicture}
\\
$\underset{(F4)}{1H2\star}$ &
\begin{tikzDiagramme}{.2}
\useasboundingbox (0,0) rectangle (7,7);
\draw[boite] (1,1) rectangle (6,6);
\draw[boite] (1,2) rectangle (5,6);
\draw[boite] (2,3) rectangle node {$S$} (5,6);
\draw[fill] (1.5,2.5) circle (3pt) node[right=2pt] {\tiny $a$};
\draw[fill] (5.5,1.5) circle (3pt) node[right=2pt] {\tiny $b$};
\draw[boite] (6,6) rectangle (7,7);
\draw[fill] (6.5,6.5) circle (3pt) node[right=2pt] {\tiny $x$};
\draw[decorate,decoration={brace, amplitude=7pt}] (1,2) -- node[left=3pt] {\small $S'$} (1,6);
\end{tikzDiagramme}
&
\begin{tikzArbre}{.4}
\begin{scope}[yshift=120]
 \node [linear]{$\oplus$}
	child[dotted] {node (a) {}}
	child {node [linear](b) {$\ominus$}{
		child {node [linear]{$\oplus$}
			child {node[leaf]{} node[below=3pt]{\small $a$}}
child [level distance=1pt,xshift=0pt,yshift=-7pt] {node[draw,shape=isosceles triangle,
shape border rotate=90,anchor=north] {\small $S$} edge from parent[draw=none]}
			child [missing]{node[leaf]{}}
		}
		child {node[leaf]{} node[below=3pt]{\small $b$}}
}}
	child {node[leaf]{} node[below=3pt]{\small $x$} (c) {}}
	child [dotted] {node (d) {}}
;
\draw[dotted, thick] (a) -- (b);
\draw[dotted, thick] (c) -- (d);
\end{scope}
\end{tikzArbre}
&
\begin{tikzpicture}[scale=.2, shorten >= 2pt]
\begin{scope}[scale=.7]
\patateTriangleUnlabelled{0}{0};
\patateTriangleUnlabelled{2}{0};
\patateTriangleUnlabelled{6}{0};
\patateTriangleUnlabelled{0}{2};
\patateTriangleUnlabelled{2}{2};
\patateTriangleUnlabelled{6}{2};
\patateTriangleUnlabelled{0}{6};
\patateTriangleUnlabelled{2}{6};
\patateTriangleUnlabelled{6}{6};
\begin{scope}
\draw (0.2,7.85) rectangle  +(1.6,0.3);
\inputState{0.35}{8};
\outputState{1.65}{8};
\end{scope}
\begin{scope}[shift={(2,0)}]
\draw (0.2,7.85) rectangle  +(1.6,0.3);
\inputState{0.35}{8};
\outputState{1.65}{8};
\end{scope}
\begin{scope}[shift={(6,0)}]
\draw (0.2,7.85) rectangle +(1.6,0.3);
\draw [pattern=north east lines] (0.2,7.85) rectangle +(0.3,0.3);
\inputState{0.35}{8};
\outputState{1.65}{8};
\end{scope}
\inputState{8.3}{8};
\begin{scope}[shift={(0,0)}]
\draw [rounded corners, fill=black!10,opacity=.5] (0,5) -- (0,6) -- (1,5);
\outputState{0.15}{5.7};
\end{scope}
\begin{scope}[shift={(2,0)}]
\draw [rounded corners, fill=black!10,opacity=.5] (0,5) -- (0,6) -- (1,5);
\outputState{0.15}{5.7};
\end{scope}
\begin{scope}[shift={(8,0)}]
\draw (-0.15,0.2) rectangle +(0.3,1.6);
\inputState{0}{0.35};
\outputState{0}{1.65};
\end{scope}
\begin{scope}[shift={(8,2)}]
\draw (-0.15,0.2) rectangle +(0.3,1.6);
\inputState{0}{0.35};
\outputState{0}{1.65};
\end{scope}
\begin{scope}[shift={(8,6)}]
\draw (-0.15,0.2) rectangle +(0.3,1.6);
\inputState{0}{0.35};
\outputState{0}{1.65};
\draw [pattern=north east lines] (-0.15,0.2) rectangle +(0.3,0.3);
\end{scope}

\draw [fill=black!20,opacity=.2, very thick] (2,0) circle (.5cm);
\draw [fill=black!20,opacity=.2, very thick] (0,2) circle (.5cm);
\draw [fill=black!20,opacity=.2, very thick] (2,2) circle (.5cm);
\draw [fill=black!20,opacity=.2, very thick] (6,2) circle (.5cm);
\draw [fill=black!20,opacity=.2, very thick] (8,2) circle (.5cm);
\draw [fill=black!20,opacity=.2, very thick] (0,6) circle (.5cm);
\draw [fill=black!20,opacity=.2, very thick] (2,6) circle (.5cm);
\draw [fill=black!20,opacity=.2, very thick] (8,6) circle (.5cm);
\draw [fill=black!20,opacity=.2, very thick] (0,8) circle (.5cm);
\draw [fill=black!20,opacity=.2, very thick] (2,8) circle (.5cm);
\draw [fill=black!20,opacity=.2, very thick] (4,8) circle (.5cm);
\draw [fill=black!20,opacity=.2, very thick] (6,8) circle (.5cm);
\draw [fill=black!20,opacity=.2, very thick] (8,8) circle (.5cm);
\draw [fill=black!20,opacity=.2, very thick] (4,0) circle (.5cm);
\draw [fill=black!20,opacity=.2, very thick] (4,2) circle (.5cm);
\draw [fill=black!20,opacity=.2, very thick] (4,6) circle (.5cm);
\draw [fill=black!20,opacity=.2, very thick] (8,0) circle (.5cm);
\draw [fill=black!20,opacity=.2, very thick] (8,4) circle (.5cm);
\draw [fill=black!20,opacity=.2, very thick] (6,0) circle (.5cm);
\draw (8,8) node (exit) {};
\draw (8,7) node (exit87) {};
\draw (7,8) node (exit78) {};
\draw [drop shadow,fill=black!10] (20,8) node[below=0.5,xshift=20pt] {${\mathcal A}_{\rho}$} ellipse  (12 and 3) ;
\draw [fill=black] (20,8) node[above right]  {\small $q_{S}$} circle (10pt);
\draw [->] (8,6) .. controls +(0,-5) and +(0,-10) .. node[below] {\tiny $LDRU$} (20,8);
\draw [fill=black] (18,8) node[above left=0pt]  {\small $q_{S'}$} circle (10pt);
\draw [->] (8,6) .. controls +(0,-3) and +(0,-10) .. node[above=0pt] {\tiny $DRU$} (18,8);
\end{scope}
\end{tikzpicture}
\\
$\underset{(F1)}{1H1}$ &
\begin{tikzDiagramme}{.2}
\useasboundingbox (0,0) rectangle (7,7);
\draw[boite] (1,1) rectangle (6,6);
\draw[boite] (2,1) rectangle node {$S$} (6,5);
\draw[fill] (1.5,5.5) circle (3pt) node[left=2pt] {\tiny $a$};
\draw[boite] (6,6) rectangle (7,7);
\draw[fill] (6.5,6.5) circle (3pt) node[right=2pt] {\tiny $x$};
\end{tikzDiagramme}
&
\begin{tikzArbre}{.4}
\begin{scope}[yshift=120]
 \node [linear] {$\oplus$}
	child[dotted] {node (a) {}}
	child {node [linear] (b) {$\ominus$}{
		child {node[leaf]{} node[below=3pt]{\small $a$}}
child [level distance=1pt,xshift=0pt,yshift=-7pt] {node[draw,shape=isosceles triangle,
shape border rotate=90,anchor=north] {\small $S$} edge from parent[draw=none]}
			child [missing] {node[leaf]{}}
		}
}
	child {node[leaf]{}  node[below=3pt]{\small $x$} (c) {}}
	child [dotted] {node (d) {}}
;
\draw[dotted, thick] (a) -- (b);
\draw[dotted, thick] (c) -- (d);
\end{scope}
\end{tikzArbre}
&
\begin{tikzpicture}[scale=.2, shorten >= 2pt]
\begin{scope}[scale=.7]
\patateTriangleUnlabelled{0}{0};
\patateTriangleUnlabelled{2}{0};
\patateTriangleUnlabelled{6}{0};
\patateTriangleUnlabelled{0}{2};
\patateTriangleUnlabelled{2}{2};
\patateTriangleUnlabelled{6}{2};
\patateTriangleUnlabelled{0}{6};
\patateTriangleUnlabelled{2}{6};
\patateTriangleUnlabelled{6}{6};
\begin{scope}
\draw (0.2,7.85) rectangle  +(1.6,0.3);
\inputState{0.35}{8};
\outputState{1.65}{8};
\end{scope}
\begin{scope}[shift={(2,0)}]
\draw (0.2,7.85) rectangle  +(1.6,0.3);
\inputState{0.35}{8};
\outputState{1.65}{8};
\end{scope}
\begin{scope}[shift={(6,0)}]
\draw (0.2,7.85) rectangle +(1.6,0.3);
\draw [pattern=north east lines] (0.2,7.85) rectangle +(0.3,0.3);
\inputState{0.35}{8};
\outputState{1.65}{8};
\end{scope}
\inputState{8.3}{8};
\begin{scope}[shift={(0,0)}]
\draw [rounded corners, fill=black!10,opacity=.5] (0,5) -- (0,6) -- (1,5);
\outputState{0.15}{5.7};
\end{scope}
\begin{scope}[shift={(2,0)}]
\draw [rounded corners, fill=black!10,opacity=.5] (0,5) -- (0,6) -- (1,5);
\outputState{0.15}{5.7};
\end{scope}
\begin{scope}[shift={(8,0)}]
\draw (-0.15,0.2) rectangle +(0.3,1.6);
\inputState{0}{0.35};
\outputState{0}{1.65};
\end{scope}
\begin{scope}[shift={(8,2)}]
\draw (-0.15,0.2) rectangle +(0.3,1.6);
\inputState{0}{0.35};
\outputState{0}{1.65};
\end{scope}
\begin{scope}[shift={(8,6)}]
\draw (-0.15,0.2) rectangle +(0.3,1.6);
\inputState{0}{0.35};
\outputState{0}{1.65};
\draw [pattern=north east lines] (-0.15,0.2) rectangle +(0.3,0.3);
\end{scope}

\draw [fill=black!20,opacity=.2, very thick] (2,0) circle (.5cm);
\draw [fill=black!20,opacity=.2, very thick] (0,2) circle (.5cm);
\draw [fill=black!20,opacity=.2, very thick] (2,2) circle (.5cm);
\draw [fill=black!20,opacity=.2, very thick] (6,2) circle (.5cm);
\draw [fill=black!20,opacity=.2, very thick] (8,2) circle (.5cm);
\draw [fill=black!20,opacity=.2, very thick] (0,6) circle (.5cm);
\draw [fill=black!20,opacity=.2, very thick] (2,6) circle (.5cm);
\draw [fill=black!20,opacity=.2, very thick] (8,6) circle (.5cm);
\draw [fill=black!20,opacity=.2, very thick] (0,8) circle (.5cm);
\draw [fill=black!20,opacity=.2, very thick] (2,8) circle (.5cm);
\draw [fill=black!20,opacity=.2, very thick] (4,8) circle (.5cm);
\draw [fill=black!20,opacity=.2, very thick] (6,8) circle (.5cm);
\draw [fill=black!20,opacity=.2, very thick] (8,8) circle (.5cm);
\draw [fill=black!20,opacity=.2, very thick] (4,0) circle (.5cm);
\draw [fill=black!20,opacity=.2, very thick] (4,2) circle (.5cm);
\draw [fill=black!20,opacity=.2, very thick] (4,6) circle (.5cm);
\draw [fill=black!20,opacity=.2, very thick] (8,0) circle (.5cm);
\draw [fill=black!20,opacity=.2, very thick] (8,4) circle (.5cm);
\draw [fill=black!20,opacity=.2, very thick] (6,0) circle (.5cm);
\draw (8,8) node (exit) {};
\draw (8,7) node (exit87) {};
\draw (7,8) node (exit78) {};
\draw [drop shadow,fill=black!10] (20,8) node[below=0.5,xshift=20pt] {${\mathcal A}_{\rho}$} ellipse  (12 and 3) ;
\draw [fill=black] (20,8) node[above right]  {\small $q_{S}$} circle (10pt);
\draw [->] (8,6) .. controls +(0,-2) and +(0,-10) .. node[below] {\tiny $LUR$} (20,8);
\end{scope}
\end{tikzpicture}
\\
$\underset{(F4+)}{1H1+}$ &
\begin{tikzDiagramme}{.2}
\useasboundingbox (0,0) rectangle (7,7);
\draw[boite] (1,1) rectangle (6,6);
\draw[boite] (3,4) rectangle node {$S$} (5,6);
\draw[fill] (2,3.5) circle (3pt);
\draw (2,3.5) -- +(1,0);
\draw[fill] (5.5,3) circle (3pt);
\draw[fill] (2.5,1.5) circle (3pt);
\draw (2.5,1.5) -- +(0,1);
\draw[fill] (1.5,2) circle (3pt);
\draw (1.5,2) -- +(1,0);
\draw[boite] (6,6) rectangle (7,7);
\draw[fill] (6.5,6.5) circle (3pt) node[right=2pt] {\tiny $x$};
\end{tikzDiagramme}
&
\begin{tikzArbre}{.4}
\begin{scope}[yshift=120]
 \node [linear] {$\oplus$}
	child[dotted] {node (a) {}}
	child {node (b) {$\xi^{+}$}
			child {node[leaf]{} node (aa) {}}
			child {node[leaf]{} node (ab) {}}
			child  {node[draw,shape=isosceles triangle,
				shape border rotate=90,anchor=north] {\small $S$} }
			child  {node[leaf]{}}
			child  [missing] {node[leaf]{}}
		}
	child {node[leaf]{} node[below=3pt]{\small $x$} (c) {}}
	child [dotted] {node (d) {}}
;
\draw[dotted, thick] (a) -- (b);
\draw[dotted, thick] (aa) -- (ab);
\draw[dotted, thick] (c) -- (d);
\end{scope}
\end{tikzArbre}
&
\begin{tikzpicture}[scale=.2, shorten >= 2pt]
\begin{scope}[scale=.7]
\patateTriangleUnlabelled{0}{0};
\patateTriangleUnlabelled{2}{0};
\patateTriangleUnlabelled{6}{0};
\patateTriangleUnlabelled{0}{2};
\patateTriangleUnlabelled{2}{2};
\patateTriangleUnlabelled{6}{2};
\patateTriangleUnlabelled{0}{6};
\patateTriangleUnlabelled{2}{6};
\patateTriangleUnlabelled{6}{6};
\begin{scope}
\draw (0.2,7.85) rectangle  +(1.6,0.3);
\inputState{0.35}{8};
\outputState{1.65}{8};
\end{scope}
\begin{scope}[shift={(2,0)}]
\draw (0.2,7.85) rectangle  +(1.6,0.3);
\inputState{0.35}{8};
\outputState{1.65}{8};
\end{scope}
\begin{scope}[shift={(6,0)}]
\draw (0.2,7.85) rectangle +(1.6,0.3);
\draw [pattern=north east lines] (0.2,7.85) rectangle +(0.3,0.3);
\inputState{0.35}{8};
\outputState{1.65}{8};
\end{scope}
\inputState{8.3}{8};
\begin{scope}[shift={(0,0)}]
\draw [rounded corners, fill=black!10,opacity=.5] (0,5) -- (0,6) -- (1,5);
\outputState{0.15}{5.7};
\end{scope}
\begin{scope}[shift={(2,0)}]
\draw [rounded corners, fill=black!10,opacity=.5] (0,5) -- (0,6) -- (1,5);
\outputState{0.15}{5.7};
\end{scope}
\begin{scope}[shift={(8,0)}]
\draw (-0.15,0.2) rectangle +(0.3,1.6);
\inputState{0}{0.35};
\outputState{0}{1.65};
\end{scope}
\begin{scope}[shift={(8,2)}]
\draw (-0.15,0.2) rectangle +(0.3,1.6);
\inputState{0}{0.35};
\outputState{0}{1.65};
\end{scope}
\begin{scope}[shift={(8,6)}]
\draw (-0.15,0.2) rectangle +(0.3,1.6);
\inputState{0}{0.35};
\outputState{0}{1.65};
\draw [pattern=north east lines] (-0.15,0.2) rectangle +(0.3,0.3);
\end{scope}

\draw [fill=black!20,opacity=.2, very thick] (2,0) circle (.5cm);
\draw [fill=black!20,opacity=.2, very thick] (0,2) circle (.5cm);
\draw [fill=black!20,opacity=.2, very thick] (2,2) circle (.5cm);
\draw [fill=black!20,opacity=.2, very thick] (6,2) circle (.5cm);
\draw [fill=black!20,opacity=.2, very thick] (8,2) circle (.5cm);
\draw [fill=black!20,opacity=.2, very thick] (0,6) circle (.5cm);
\draw [fill=black!20,opacity=.2, very thick] (2,6) circle (.5cm);
\draw [fill=black!20,opacity=.2, very thick] (8,6) circle (.5cm);
\draw [fill=black!20,opacity=.2, very thick] (0,8) circle (.5cm);
\draw [fill=black!20,opacity=.2, very thick] (2,8) circle (.5cm);
\draw [fill=black!20,opacity=.2, very thick] (4,8) circle (.5cm);
\draw [fill=black!20,opacity=.2, very thick] (6,8) circle (.5cm);
\draw [fill=black!20,opacity=.2, very thick] (8,8) circle (.5cm);
\draw [fill=black!20,opacity=.2, very thick] (4,0) circle (.5cm);
\draw [fill=black!20,opacity=.2, very thick] (4,2) circle (.5cm);
\draw [fill=black!20,opacity=.2, very thick] (4,6) circle (.5cm);
\draw [fill=black!20,opacity=.2, very thick] (8,0) circle (.5cm);
\draw [fill=black!20,opacity=.2, very thick] (8,4) circle (.5cm);
\draw [fill=black!20,opacity=.2, very thick] (6,0) circle (.5cm);
\draw (8,8) node (exit) {};
\draw (8,7) node (exit87) {};
\draw (7,8) node (exit78) {};
\draw [drop shadow,fill=black!10] (20,8) node[below=0.5,xshift=20pt] {${\mathcal A}_{\rho}$} ellipse  (12 and 3) ;
\draw [fill=black] (20,8) node[above right]  {\small $q_{S}$} circle (10pt);
\draw [->] (8,6) .. controls +(0,-2) and +(0,-10) .. node[below] {\tiny $LD..LDRU$} (20,8);
\end{scope}
\end{tikzpicture}
\\
\end{tabular}\\
\caption{The set $\setH$ and conditions $\horse{i}{j}$:
$\pi \in \setH$ if and only if $\pi$ satisfies one of the conditions $\horse{i}{j}$ shown above up to symmetry,
that form a partition of $\setH$.}  \label{fig:H}
\end{figure}

Similarly to $P(\pi)$ denoting the set of pin words encoding $\pi$, 
we denote by $P(T)$ the set of pin words that encode the permutation whose decomposition tree is $T$.

\begin{theo}\label{thm:linearRoot}
Let $\pi = \ $\begin{tikzpicture}[sibling distance=15pt,level
distance=15pt,baseline=-17pt,inner sep=0pt]
\node[linear] {$\oplus$}
	child {node (T1) {$\xi_1$}}
	child[missing]
	child {node (Ta) {$\xi_{\ell}$}}
	child[child anchor=north] {node[draw,shape=isosceles triangle,
shape border rotate=90,anchor=north,inner sep=0,isosceles triangle apex angle=90] {$T_{i_0}$}}
	child {node (Ta1) {$\xi_{\ell+2}$}}
	child[missing]
	child {node (Tk) {$\xi_r$} };
\draw[dotted]  (T1) -- (Ta);
\draw[dotted]  (Ta1) -- (Tk);
\end{tikzpicture} be a $\oplus$-decomposable permutation
where $T_{i_0}$ is the only child that is not an increasing oscillation.

For every  $i$ such that $1 \leq i \leq \ell$ and $j$ such that $\ell +2 \leq j \leq r$, set
$$ \F i = \big( P^{(1)}(\xi_{i}), \ldots ,
P^{(1)}(\xi_1) \big) \text{ and }\G j = \big(P^{(3)}(\xi_{j}), \ldots , P^{(3)}(\xi_r) \big)\text{.}$$

We describe below the set $P(\pi)$ of pin words encoding $\pi$.
When $\pi \in \setH$, these sets are given only when the diagram of $\pi$ is one of those shown in Figure~\ref{fig:H}.
When the diagram of $\pi$ is one of their symmetries, $P(\pi)$ is modified accordingly.

$\bullet$ If $\pi\notin \setH$ (\emph{i.e.}, if $\pi$ does not satisfy any condition shown up to symmetry on Figure~\ref{fig:H}) then $P(\pi) =
P_0= P(T_{i_0})\cdot \, \F\ell \shuffle \G{\ell+2} $.

$\bullet$ If $\pi$ satisfies condition $(1H1)$ then
$P(\pi) = P_0 \cup P_1$, with
$$P_1 = \underset{x \bigcup T_{i_0}}{\underbrace{P(S)
\cdot \underset{x}{\underbrace{1}} \cdot
\underset{a}{\underbrace{L}} }} \cdot \, \F \ell \shuffle \G {\ell+3}.$$

$\bullet$ If $\pi$ satisfies condition $(1H1+)$ then
$P(\pi) = P_0 \cup P_1$, with
$$P_1 = \underset{x \bigcup T_{i_0}}{\underbrace{P(S) \cdot
    \underset{x}{\underbrace{1}} \cdot w}} \cdot \, \F \ell \shuffle
    \G {\ell+3},$$ where $w$ is the unique word encoding the unique
    reading of the remaining leaves of $T_{i_0}$. Notice that $w$ is obtained from
    the unique word of $P^{(1)}(\xi)$ (see Remark~\ref{rem:P1P3} p.\pageref{rem:P1P3})
    by deleting its first letter.

$\bullet$ If $\pi$ satisfies condition $(1H2\star)$ then
$P(\pi) = P_0 \cup P_1 \cup P_2$, with
$$\footnotesize{P_1 = \underset{x \bigcup T_{i_0}}{\underbrace{P(S)
    \cdot \underset{x}{\underbrace{1}} \cdot
    \underset{b}{\underbrace{D}} \cdot \underset{a}{\underbrace{L}} }}
\cdot \, \F \ell \shuffle \G {\ell+3} \text{ and }
P_2 = \underset{x \bigcup T_{i_0}}{\underbrace{P(S')
    \cdot \underset{x}{\underbrace{1}} \cdot
    \underset{b}{\underbrace{D}} }} \cdot \, \F \ell \shuffle \G {\ell+3}.}$$

$\bullet$ If $\pi$ satisfies condition $(1H2)$ then
$P(\pi) = P_0 \cup P_1 \cup P_2$, with
$$\footnotesize{P_1 = \underset{x \bigcup T_{i_0}}{\underbrace{P(T\cup a)
    \cdot \underset{x}{\underbrace{1}} \cdot
    \underset{b}{\underbrace{D}} }} \cdot \, \F \ell \shuffle \G {\ell+3}
    \text{ and } P_2 = \underset{x \bigcup T_{i_0}}{\underbrace{P(T\cup b)
    \cdot \underset{x}{\underbrace{1}} \cdot
    \underset{a}{\underbrace{L}} }} \cdot \, \F \ell \shuffle \G {\ell+3}.}$$

$\bullet$ If $\pi$ satisfies condition $(2H1)$ then
$P(\pi) = P_0 \cup P_1 \cup P_2$, with
$$\footnotesize{P_1 = \underset{x \bigcup T_{i_0}}{\underbrace{P(S)
      \cdot \underset{x}{\underbrace{1}} \cdot
      \underset{a}{\underbrace{L}} }} \cdot \, \F \ell \shuffle
  \G {\ell+3} \text{ and }
P_2 = \underset{y \bigcup T_{i_0}}{\underbrace{P(S)
    \cdot \underset{y}{\underbrace{3}} \cdot
    \underset{a}{\underbrace{U}} }} \cdot \, \F {\ell-1} \shuffle \G {\ell+2}.}$$

$\bullet$ If $\pi$ satisfies condition $(2H2\star)$ then $P(\pi)
  = P_0 \cup P_1 \cup P_2 \cup P_3$, with
$$\footnotesize{P_1 = \underset{x \bigcup T_{i_0}}{\underbrace{P(S)
    \cdot \underset{x}{\underbrace{1}} \cdot
    \underset{b}{\underbrace{D}} \cdot \underset{a}{\underbrace{L}} }}
\cdot \, \F \ell \shuffle
\G {\ell+3}, \, \, P_2 = \underset{x \bigcup T_{i_0}}{\underbrace{P(S')
    \cdot \underset{x}{\underbrace{1}} \cdot
    \underset{b}{\underbrace{D}} }} \cdot \, \F {\ell} \shuffle \G
    {\ell+3}}$$
$$\footnotesize{\text{and } P_3 = \underset{y \bigcup T_{i_0}}{\underbrace{P(S')
\cdot \underset{y}{\underbrace{3}} \cdot \underset{b}{\underbrace{R}}
}} \cdot \, \F {\ell-1} \shuffle \G {\ell+2}.}$$

$\bullet$ If $\pi$ satisfies condition $(2H2)$ then
$P(\pi) = P_0 \cup P_1 \cup P_2 \cup P_3 \cup P_4$, with
$$\footnotesize{P_1 = \underset{x \bigcup T_{i_0}}{\underbrace{P(T\cup
    a) \cdot \underset{x}{\underbrace{1}} \cdot
    \underset{b}{\underbrace{D}} }} \cdot \, \F {\ell} \shuffle \G
    {\ell+3}, \, \,
P_2 = \underset{x \bigcup T_{i_0}}{\underbrace{P(T\cup
    b) \cdot \underset{x}{\underbrace{1}} \cdot
    \underset{a}{\underbrace{L}} }} \cdot \, \F {\ell} \shuffle \G {\ell+3},}$$
$$\footnotesize{P_3 = \underset{y \bigcup T_{i_0}}{\underbrace{P(T\cup a) \cdot
      \underset{y}{\underbrace{3}} \cdot \underset{b}{\underbrace{R}}
      }} \cdot \, \F {\ell-1} \shuffle \G {\ell+2}, \, P_4 =
      \underset{y \bigcup T_{i_0}}{\underbrace{P(T\cup b) \cdot
      \underset{y}{\underbrace{3}} \cdot \underset{a}{\underbrace{U}}
      }} \cdot \, \F {\ell-1} \shuffle \G {\ell+2}.}$$

$\bullet$ If $\pi$ satisfies condition $(2H3)$ then
$P(\pi) = P_0 \cup P_1 \cup P_2 \cup P_3 \cup P_4$, with
$$\footnotesize{P_1 = \underset{x \bigcup T_{i_0}}{\underbrace{P(S)
    \cdot \underset{x}{\underbrace{1}} \cdot
    \underset{c}{\underbrace{D}} }} \cdot \, \F {\ell} \shuffle \G
    {\ell+3}, \, \,
P_2 = \underset{x \bigcup T_{i_0}}{\underbrace{P(T\cup
    b) \cdot \underset{x}{\underbrace{1}} \cdot
    \underset{c}{\underbrace{D}} \cdot \underset{a}{\underbrace{L}} }}
\cdot \, \F {\ell} \shuffle \G {\ell+3},}$$
$$\footnotesize{P_3 = \underset{y \bigcup T_{i_0}}{\underbrace{P(T\cup
      a) \cdot \underset{y}{\underbrace{3}} \cdot
      \underset{c}{\underbrace{R}} \cdot \underset{b}{\underbrace{U}}
    }} \cdot \, \F {\ell-1} \shuffle \G {\ell+2}, \,
P_4 = \underset{y \bigcup T_{i_0}}{\underbrace{P(S)
      \cdot \underset{y}{\underbrace{3}} \cdot
      \underset{c}{\underbrace{R}} }} \cdot \, \F {\ell-1} \shuffle
\G {\ell+2}.}$$
\end{theo}

\begin{proof}
  For each item, it is easy to check that the given pin words are pin words encoding
  $\pi$. Conversely, we prove that a pin word encoding
  $\pi$ is necessarily in the set claimed to be $P(\pi)$.
  First of all, by~\cite[Lemma
  3.7]{BBR09}, every pin representation of $\pi$ starts in the only
  child that is not an increasing oscillation, \emph{i.e.}, with $T_{i_0}$.

  Let us start with the first point of Theorem~\ref{thm:linearRoot}.
  In this case, by definition of $\setH$, we know that $T_{i_0}$ is read in one
  piece. By Lemma~\ref{lem:oneBlockReadInSeveralTimes}
  (p.\pageref{lem:oneBlockReadInSeveralTimes}), the other
  children are also read in one piece, and Lemma~\ref{lem:lem1} ensures
  that the children closest to $T_{i_0}$ are read first. As there is no
  relative order between children $\xi_{\ell+2}$ to $\xi_r$ and children
  $\xi_{\ell}$ to $\xi_1$, this leads to the shuffling operation between
  pin words corresponding to these children, with an external origin
  placed in quadrant $3$ (resp.~$1$) with respect to their bounding box. 

  In the other cases of Theorem~\ref{thm:linearRoot},
  by Lemma~\ref{lem:readInTwoTimes}, every pin representation of $\pi$
  either reads $T_{i_0}$ in one piece or in two \fois. In case
  $T_{i_0}$ is read in one piece, the pin representation is as before
  encoded by pin words of $P_0$. If it is read in two \fois,
  Lemma~\ref{lem:readInTwoTimes} and its proof and Remark~\ref{rem:inTwoTimes}
  ensure that the corresponding pin words are those described.

  Consider for example $\pi$ satisfying condition $\horse{1}{1}$. Then $\pi$ satisfies
  condition $F1$ of Lemma~\ref{lem:readInTwoTimes}, and only this one
  by definition of $\horse{1}{1}$. If $T_{i_0}$ is read in two \fois, then
  Lemma~\ref{lem:readInTwoTimes} ensures that
  $S$ is the first part of $T_{i_0} \cup \{x\}$ to be read, followed by $x$
  and finally $a$. The corresponding pin words are indeed those described in $P_1$.

  Taking the other example of condition $\horse{2}{3}$, $P_1$ corresponds to
  condition $F2$ with $S = T \cup a \cup b$, $P_2$ corresponds to
  condition $F4$ with $S = T \cup b$, $P_3$ corresponds to condition
  $G4$ with $S = T \cup a$ and $P_4$ corresponds to condition $G2$
  with $S = T \cup a \cup b$.
\end{proof}

\begin{rem}
If $\pi$ is a $\ominus$-decomposable permutation, a similar description of $P(\pi)$ can be obtained from Remark~\ref{rem:ominus=oplus_transposed} (p.\pageref{rem:ominus=oplus_transposed}).
\end{rem}

\subsection{Recursive case: decomposition trees with a prime root}

We now turn to the study of the recursive case where
the decomposition tree has a root which is a simple permutation $\alpha$.
We start with the case where $\pi = $\begin{tikzpicture}[sibling
distance=10pt,level distance=10pt,baseline=-15pt] \node[simple] {$\alpha$} child
{[fill] circle (2 pt) node(x1) {}} child[missing] child {[fill] circle
(2 pt) node(xk) {}} child[child anchor=north]
{node[draw,shape=isosceles triangle, shape border
rotate=90,anchor=north,inner sep=0, isosceles triangle apex angle=90] {$T$}} child {[fill] circle (2 pt) node(y1)
{}} child[missing] child {[fill] circle (2 pt) node(yk) {}};
\draw[dotted] (x1) -- (xk);
\draw[dotted] (y1) -- (yk);
\end{tikzpicture} for a tree $T$ that is not a leaf.

We begin with the characterization of the possible ways a pin representation of $\pi$ may read $T$, introducing first a condition that will be useful in the sequel.

\begin{defi}
\label{def:c}
For a permutation $\pi = \alpha[1,\ldots,1,T,1,\ldots,1]$ with $\alpha=\alpha_1\ldots \alpha_k$, we define condition $(\mathcal{C})$ as follows:
$$(\mathcal{C})
\begin{cases}
  \bullet \, \alpha \text{ is an increasing -- resp. decreasing -- quasi-oscillation (see p.\pageref{defn:quasiepi});}\\
  \bullet \, T \text{ expands an auxiliary point of } \alpha\text{;} \\
  \bullet \, \text{the shape of T is \begin{tikzpicture}[baseline=-10pt,inner sep=0pt,scale=.5] \node {$\oplus$} child [missing] child
    [level distance=1pt,xshift=0pt,yshift=-7pt]
    {node[draw,shape=isosceles triangle, shape border
      rotate=90,anchor=north,isosceles triangle apex angle=90] {\small $T'$} edge from
      parent[draw=none]} child[level distance=30pt] {[fill] circle (3pt) node (x){}};
      \end{tikzpicture}
-- resp.
\begin{tikzpicture}[baseline=-10pt,inner sep=0pt,scale=.5] \node {$\ominus$} child [missing]
    child [level distance=1pt,xshift=0pt,yshift=-7pt]
    {node[draw,shape=isosceles triangle, shape border
      rotate=90,anchor=north,isosceles triangle apex angle=90] {\small $T'$} edge from
      parent[draw=none]} child[level distance=30pt]  {[fill] circle (3pt) node (x){}};
\end{tikzpicture}-- if the auxiliary point is $\alpha_k$}\\
\text{or $\alpha_{k-1}$ and
\begin{tikzpicture}[baseline=-10pt,inner sep=0pt,scale=.5] \node {$\oplus$}
        child[level distance=30pt]  {[fill] circle (3pt) node (x){}} child [level
        distance=1pt,xshift=0pt,yshift=-7pt]
        {node[draw,shape=isosceles triangle, shape border
          rotate=90,anchor=north,isosceles triangle apex angle=90] {\small $T'$} edge from
          parent[draw=none]} child [missing];
\end{tikzpicture}
-- resp.
\begin{tikzpicture}[baseline=-10pt,inner sep=0pt,scale=.5]
  \node {$\ominus$} child[level distance=30pt]  {[fill] circle (3pt) node (x){}} child
  [level distance=1pt,xshift=0pt,yshift=-7pt]
  {node[draw,shape=isosceles triangle, shape border
    rotate=90,anchor=north,isosceles triangle apex angle=90] {\small $T'$} edge from parent[draw=none]}
  child [missing];
\end{tikzpicture}
-- if the auxiliary point is $\alpha_1$ or $\alpha_2$.}
\end{cases}
$$
\end{defi}

\begin{lem}\label{lem:lecture_simple}
Let $p=(p_1,\ldots,p_n)$ be a pin representation associated to $\pi = \alpha[1,\ldots, 1, $ $T, 1, \ldots, 1]$.
Then one of the following statements holds:
\begin{itemize}
\item[$(1)$] $p_1 \in T$ and $T$ is read in one piece by $p$;
\item[$(2)$] $T = \{p_1,\ldots,p_i\} \bigcup \{ p_n \}$ with $i \neq n-1$, and $\pi$ satisfies condition $({\mathcal C})$;
\item[$(3)$] $p_1 \notin T$, $T=\{p_2,p_n\}$, and $\pi$ satisfies condition $({\mathcal C})$.
\end{itemize}
Moreover if $(3)$ is satisfied then $p$ is a proper pin representation uniquely determined by $\alpha$ and its auxiliary point;
it is up to symmetry the one depicted in the first diagram of Figure~\ref{fig:pin-rep_quasi-epi_1bloc}.
If $(2)$ is satisfied, defining $T'$ as in condition $(\mathcal{C})$, then $T' = \{p_1,\dots,p_i\}$ and
$(p_{i+1},\ldots,p_n)$ is uniquely determined by $\alpha$ and its auxiliary point,
as shown in the second diagram  of Figure~\ref{fig:pin-rep_quasi-epi_1bloc} up to symmetry.
\end{lem}

\begin{figure}[htbp]
\begin{minipage}[b]{.65\linewidth}
\begin{center}
\begin{tikzpicture}
 \begin{scope}[scale=.3]
\draw (7,2) rectangle (9,4);
\draw (9.5,2) node {$T$};
\draw (7.5,2.5) [fill] circle (.2);
\draw (7.5,1.6) node {\footnotesize $p_{2}$};
\draw (8.5,3.5) [fill] circle (.2);
\draw (9.8,3.5) node {\footnotesize $p_{n}$};
\draw (8.5,3.5) -- (2,3.5);
\draw (6.6,-0.7) node {\footnotesize $p_{3}$};
\draw (6,1.5) [fill] circle (.2);
\draw (6,2.1) node {\footnotesize $p_{1}$};

\draw (6.5,0) [fill] circle (.2);
\draw (6.5,0) -- (6.5,2);

\draw (5,0.5) [fill] circle (.2);
 \draw (5,0.5) -- (7,0.5);

\draw (5.5,-1) [fill] circle (.2);
\draw (5.5,-1) -- (5.5,1);

\draw (4,-0.5) [fill] circle (.2);
\draw (4,-0.5) -- (6,-0.5);

\draw (4.5,-2) [fill] circle (.2);
\draw (4.5,-2) -- (4.5,0);

\draw (3,-2) node {$\cdot$};
\draw (3.3,-1.7) node {$\cdot$};
\draw (3.6,-1.4) node {$\cdot$};

\draw (2,-2.5) [fill] circle (.2);
\draw (2,-2.5) -- (5,-2.5);
\draw (0.7,-2.5) node {\footnotesize $p_{n-2}$};

\draw (2.5,4) [fill] circle (.2);
\draw (2.5,4) -- (2.5,-3);
\draw (2.5,4.5) node {\footnotesize $p_{n-1}$};
 \end{scope}
\begin{scope}[xshift=3.7cm,scale=.3]
\draw (7,2) rectangle (11,6);
\draw (11.5,3.4) node {$T$};
\draw[thick] (7,2) rectangle (10,5);
\draw (8.5,3.4) node {$T'$};
\draw (8.5,1.4) node {$\mathcal B'$};
\draw (10.5,5.5) [fill] circle (.2);
\draw (11.8,5.5) node {\footnotesize $p_{n}$};
\draw (10.5,5.5) -- (2,5.5);

\draw (6,1.5) [fill] circle (.2);
\draw (6,2.5) node {\footnotesize $p_{i+1}$};
\draw (6.5,0) [fill] circle (.2);
\draw (6.5,0) -- (6.5,2);

\draw (5,0.5) [fill] circle (.2);
\draw (5,0.5) -- (7,0.5);

\draw (5.5,-1) [fill] circle (.2);
\draw (5.5,-1) -- (5.5,1);

\draw (4,-0.5) [fill] circle (.2);
\draw (4,-0.5) -- (6,-0.5);

\draw (4.5,-2) [fill] circle (.2);
\draw (4.5,-2) -- (4.5,0);

\draw (3,-2) node {$\cdot$};
\draw (3.3,-1.7) node {$\cdot$};
\draw (3.6,-1.4) node {$\cdot$};

\draw (2,-2.5) [fill] circle (.2);
\draw (2,-2.5) -- (5,-2.5);
\draw (0.7,-2.5) node {\footnotesize $p_{n-2}$};

\draw (2.5,6) [fill] circle (.2);
\draw (2.5,6) -- (2.5,-3);
\draw (2.5,6.5) node {\footnotesize $p_{n-1}$};
\end{scope}
\end{tikzpicture}
\caption{Diagram of $\pi$ when one child $T$ is not a leaf, is read in two \fois and  $p_1 \notin T$ or $p_1 \in T$.}
\label{fig:pin-rep_quasi-epi_1bloc}
\end{center}
\end{minipage}
\hspace*{0.2cm}
\begin{minipage}[b]{.3\linewidth}
\begin{center}
\tiny
\begin{tikzpicture}[scale=.4]
\useasboundingbox (0,-1) rectangle (5,5);
\draw [help lines] (2,2) grid (5,5);
\draw (0.5,0.5) [fill] circle (0.2);
\draw [fill] (3.5,3.5) circle (0.2);
\node at (-0.5,0.5) {\normalsize $p_1$};
\node at (4.5,3.5) {\normalsize $p_2$};
\node at (2,-0.7) {\normalsize $\mathcal B$};
\draw (2.5,2.5) node [inner sep = 1pt,draw,circle] {\tiny $3$};
\draw (4.5,2.5) node [inner sep = 1pt,draw,circle] {\tiny $4$};
\draw (2.5,4.5) node [inner sep = 1pt,draw,circle] {\tiny $2$};
\draw (4.5,4.5) node [inner sep = 1pt,draw,circle] {\tiny $1$};
\draw [thick] (0,0) rectangle (4,4);
\end{tikzpicture}
\caption{Possible \newline positions for $p_n$.}\label{fig:pn}
\end{center}
\end{minipage}
\end{figure}

\begin{proof}
  If $p_1 \not\in T$, then by Lemma 3.12{\it (ii)} of~\cite{BBR09}, $T
  = \{p_2,p_n\}$.  Up to symmetry assume that $\{p_1,p_2\}$ is
  an increasing subsequence of $\pi$.  As $\{p_2,p_n\}$ forms a block, $p_n$ is
  in one of the $4$ positions shown in Figure~\ref{fig:pn}. But
  position \tikz\node[inner sep = 1pt,circle,draw] {\tiny $3$}; is
  forbidden because it is inside of the bounding box $\mathcal B$ of
  $\{p_1,p_2\}$.  Positions \tikz\node[inner sep = 1pt,circle,draw]
  {\tiny $2$}; and \tikz\node[inner sep = 1pt,circle,draw] {\tiny
    $4$}; lie on the side of the bounding box $\mathcal B$.
  Thus, if $p_n$ lies in one of these positions,
  it must be read immediately after $p_1$ and $p_2$ and thus must
  be $p_3$ from Lemma~\ref{lem:[7]2.17} (p.\pageref{lem:[7]2.17}). But $n > 3$ ($\alpha$ is
  simple so $|\alpha | \geq 4$) so that these positions are also
  forbidden. Hence $p_n$ lies in position \tikz\node[inner sep = 1pt,circle,draw] {\tiny
    $1$}; and $T = 12$.

As $\alpha$ is a simple pin-permutation, $p_3$ respects the
separation condition. 
But if $p_3$ lies above or on the right of the
bounding box $\mathcal B$ then $p_n$ will be on the side of the bounding box of
$\{p_1,p_2,p_3\}$, hence $p_n=p_4$. But in that case, $\alpha$ has
only $3$ children, hence $|\alpha|=3$, contradicting the fact that $\alpha$ is simple.

By symmetry we can assume that $p_3$ lies below $\mathcal B$ (see
the first diagram of Figure~\ref{fig:pin-rep_quasi-epi_1bloc}).
The same argument goes for every pin $p_i$ with
$i = 3, \ldots, n-2$ and these pins form an alternating sequence of
left and down pins. As $p_n$ separates $p_{n-1}$ from all other pins,
$p_{n-1}$ must be an up or right pin (depending on the parity of
$n$). Then $\alpha$ is a quasi-oscillation in which the point expanded
by $T$ is an auxiliary point and $T = 12$ or $T=21$ depending on the
nature of $\alpha$ -- increasing or decreasing.
Consequently, $\pi$ satisfies condition $({\mathcal C})$.
Notice that given $\alpha$ and its auxiliary point,
once we know that $p_1 \notin T$ then $p$ is uniquely determined.

Suppose now that $p_1 \in T$ but $T$ is not read in one piece. By Lemma
3.11{\it (i)} of~\cite{BBR09}, it is read in two \fois, the second
part being $p_n$. Then $T = \{p_1,\ldots,p_i\} \bigcup \{ p_n \}$ with
$n \not= i+1$. Thus $p_n$ does not lie on the sides of the bounding
box $\mathcal B'$ of $\{p_1,\ldots,p_i\}$. Up to symmetry, we can assume that
$p_n$ lies in quadrant $1$ with respect to $\mathcal B'$ (see 
Figure~\ref{fig:pin-rep_quasi-epi_1bloc}).
Therefore, $T= \begin{tikzpicture}[baseline=-10pt,inner sep=0pt,scale=.5] \node {$\oplus$} child [missing] child
    [level distance=1pt,xshift=0pt,yshift=-7pt]
    {node[draw,shape=isosceles triangle, shape border
      rotate=90,anchor=north,isosceles triangle apex angle=90] {\small $T'$} edge from
      parent[draw=none]} child[level distance=30pt] {[fill] circle (3pt) node (x){}};
      \end{tikzpicture} $,
$T'$ being the sub-forest of $T$ whose leaves are the points in $\mathcal B'$, \emph{i.e.}, are $p_1,\ldots,p_i$.
As $T$ is a block, no pin must lie on the sides of the bounding box of $\{p_1,\ldots,p_i,p_n\}$.
Moreover $p_{i+1}$ does not lie in quadrant $1$ with respect to $T$,
otherwise $p_n$ would lie inside the bounding box of $\{p_1,\ldots,p_{i+1}\}$.
If $p_{i+1}$ lies in quadrant $2$ or $4$, $p_n$ would lie on the side of the bounding box of $\{p_1,\ldots,p_{i+1}\}$ and thus must be $p_{i+2}$.
This is in contradiction with $\alpha$ being simple.
Thus $p_{i+1}$ lies in quadrant $3$ with respect to $T$.
The same goes for $p_j$ with $j \in \{ i+1,\ldots,n-2 \}$.
Because $\alpha$ is simple,
we therefore deduce that all these pins form an
alternating sequence of left and down pins until $p_{n-1}$ which must
be an up or right pin depending on the parity of $n$.
Thus $\alpha$ is a quasi-oscillation in which the point expanded by $T$ is an auxiliary point. 
Moreover, $\pi$ satisfies condition (${\mathcal C}$) with $T' = \{p_1,\dots,p_i\}$, as can be seen (up to symmetry) on the right part of Figure~\ref{fig:pin-rep_quasi-epi_1bloc}.
Notice that given $\alpha$ and its auxiliary point,
once we know that $T = \{p_1,\ldots,p_i\} \bigcup \{ p_n \}$ with $i \neq n-1$
then $(p_{i+1},\ldots,p_n)$ is uniquely determined.

Finally if we are not in one of the two cases discussed above,
then $p_1 \in T$ and $T$ is read in one piece by $p$, concluding the proof.
\end{proof}

With the description of the pin representations of $\pi$ in
Lemma~\ref{lem:lecture_simple}, we
are able to give in Theorem~\ref{thm:conditionc} below an explicit
description of the set $P(\pi)$ of pin words that encode $\pi$. 
The statement of Theorem~\ref{thm:conditionc} makes use of the notation $Q_x(\alpha)$, 
which has appeared in Subsection~\ref{ssec:idees_construction_automates}, and that we define below. 

\begin{defi} \label{def:Qxalpha}
  For every simple pin-permutation $\alpha$, with an active point $x$
  (see p.\pageref{def:active}) marked, we define $Q_x(\alpha)$ as the set
  of strict pin words obtained by deleting the first letter of a
  quasi-strict pin word of $\alpha$ whose first point read in $\alpha$
  is $x$.
\end{defi}

Notice that $|u| = |\alpha| -1$ for all $u \in Q_x(\alpha)$.

\begin{rem} \label{rem:tailleQx}
  To each pin representation of $\alpha$ whose first point read is $x$
  corresponds exactly one word of $Q_x(\alpha)$. Indeed the
  quasi-strict pin words associated to a pin representation differ
  only in their first letter (see Figure~\ref{fig:origine} and Remark~\ref{rem:nb_pin_words} p.\pageref{rem:nb_pin_words}).
\end{rem}

\begin{theo}\label{thm:conditionc}
Let $\pi$ be a pin-permutation, whose decomposition tree has a
prime root $\alpha$, with exactly one child $T$ that is not a leaf.
Then, denoting by $x$ the point of $\alpha$ expanded by $T$, the
following holds:
\begin{itemize}
\item If $\pi$ does not satisfy condition $(\mathcal{C})$, then $P(\pi) = P(T)
  \cdot Q_x(\alpha)$.
 \item If $\pi$ satisfies condition $(\mathcal{C})$, we define $T'$ as in $(\mathcal{C})$, and we distinguish two sub-cases
   according to the number of leaves $|T|$ of $T$:
 \begin{itemize}
 \item[$(a)$] if $|T| \geq 3$, let $w$ be the
   unique word encoding the unique reading of the remaining leaves in
   $\pi$ after $T'$ is read when $T$ is read in two \fois. Then
   $P(\pi) = P(T) \cdot Q_x(\alpha)\, \cup \, P(T') \cdot w$.
  \item[$(b)$] if $|T| = 2$, let $P_{\{1,n\}}(\pi)$  be the set of pin words encoding the unique pin representation $p$ of $\pi$
  such that $T = \{p_1,p_n\}$. Define similarly $P_{\{2,n\}}(\pi)$ for the case $T = \{p_2,p_n\}$. Then
  $P(\pi) = P(T) \cdot Q_x(\alpha) \cup P_{\{1,n\}}(\pi) \cup P_{\{2,n\}}(\pi)$.
 \end{itemize}
\end{itemize}
\end{theo}

\begin{proof}
In each case, it is easy to check that the given pin words are pin
words encoding $\pi$.  Conversely, we prove that a pin word encoding
$\pi$ is necessarily in the set claimed to be $P(\pi)$.  Let $u = u_1
\dots u_n \in P(\pi)$ and $p=(p_1,\dots,p_n)$ be the associated pin
representation.
Then $p$ satisfies one statement of Lemma~\ref{lem:lecture_simple}.

If $p$ satisfies statement $(1)$ of Lemma~\ref{lem:lecture_simple} then,
setting $k = |T|$, $(p_k, \dots, p_n)$ is a pin representation of
$\alpha$ beginning with $x$.  Moreover as $T$ is a block of $\pi$,
$p_{k+1}$ is an independent pin, so that $u_{k+1}$ is a numeral. 
Thus for all
$\quadrantell$ in $\{1,2,3,4\}$, $\quadrantell \, u_{k+1} \dots u_n$ is a
pin word encoding $\alpha$ and starting with two numerals.
As $\alpha$ is simple, its pin
words are strict or quasi-strict,
hence $\quadrantell \, u_{k+1} \dots u_n$ is quasi-strict. 
Therefore $u_{k+1} \dots u_n \in Q_x(\alpha)$.
Moreover $(p_1, \dots, p_k)$ is a pin representation of
$T$. Hence $u \in P(T) \cdot Q_x(\alpha)$ which is included in the
set claimed to be $P(\pi)$,
regardless of whether $\pi$ satisfies condition $(\mathcal{C})$ or not.

If $p$ satisfies statement $(2)$ of Lemma~\ref{lem:lecture_simple} then
$\pi$ satisfies condition $(\mathcal{C})$.  If $|T|=2$ then $T=\{p_1,p_n\}$ and $u
\in P_{\{1,n\}}(\pi)$.  Notice that the uniqueness of the pin
representation such that $T=\{p_1,p_n\}$ follows from
Lemma~\ref{lem:lecture_simple}.  Indeed in this case
$(p_{i+1},\dots,p_n)$ is uniquely determined, $i=1$ and $p_1$ is
the only remaining point.  If $|T|\geq 3$ then from
Lemma~\ref{lem:lecture_simple}, $T'=\{p_1,\dots,p_{k-1}\}$ with $k=|T|$
thus the prefix of length $k-1$ of $u$ is in $P(T')$.
From Lemma~\ref{lem:lecture_simple}, $(p_k,\dots,p_n)$ is uniquely determined.
Moreover, as $k \geq 3$,
Remark~\ref{rem:nb_pin_words} (p.\pageref{rem:nb_pin_words}) ensures that
the letters encoding these points are uniquely determined.
This allows to define uniquely the word $w$ encoding $(p_k,\dots,p_n)$, yielding $u \in P(T') \cdot w$. 

If $p$ satisfies statement $(3)$ of Lemma~\ref{lem:lecture_simple} then
$\pi$ satisfies condition $(\mathcal{C})$, $|T|=2$ and $u \in P_{\{2,n\}}(\pi)$.
Notice that the uniqueness of the pin representation such that
$T=\{p_2,p_n\}$ follows from Lemma~\ref{lem:lecture_simple}.
\end{proof}

To make the set $P(\pi)$ of pin words in the statement of Theorem~\ref{thm:conditionc} explicit
(up to the recursive parts $P(T)$ and $P(T')$), 
we conclude the study of the case $\pi = \alpha[1,\ldots,1,T,1,\ldots,1]$ by stating some properties
of the (sets of) words $Q_x(\alpha)$,
$w$, $P_{\{1,n\}}(\pi)$ and $P_{\{2,n\}}(\pi)$ that appear in Theorem~\ref{thm:conditionc}.

\begin{rem}\label{rem:compute_Q_x(alpha)}
The set $Q_x(\alpha) \subseteq P(\alpha)$ can be determined in linear time w.r.t.~$|\alpha|$.
Indeed as $\alpha$ is simple it is sufficient to examine the proper pin
representations of $\alpha$ which start with an active knight containing $x$.
By Lemma~\ref{lem:[7]2.17} (p.\pageref{lem:[7]2.17}), these are entirely determined by their first two points.
Since $\alpha$ is simple these two points are in knight position.
Consequently, there are at most $8$ proper pin representations of $\alpha$ starting with $x$,
and associated pin words are obtained in linear time using
Remark~\ref{rem:nb_pin_words} (p.\pageref{rem:nb_pin_words}).
\end{rem}

\begin{lem}\label{lem:w_at_least_3}
In Theorem~\ref{thm:conditionc}, when $\pi$ satisfies condition $(\mathcal{C})$
and $|T| \geq 3$, the word $w$ is a strict pin word of length
at least $4$ encoding $\alpha$.
Denoting by $w'$ the suffix of length $2$ of $w$, then $P(T') \cdot \phi^{-1}(w') \subseteq P(T)$.
Moreover there exist a word $\bar{w}$ of $Q_x(\alpha)$ and a letter $Z$ such that
$w= \bar{w} \cdot Z$ and no word of $\phi(Q_x(\alpha))$ contains $Z$.
Finally when $|\alpha| \geq 5$ then $Q_x(\alpha)$ contains only $\bar{w}$.
Otherwise $|\alpha| = 4$ and $Q_x(\alpha)$ contains two words.
\end{lem}

\begin{proof}
Assume that $\pi$ satisfies condition $(\mathcal{C})$ and $|T| \geq 3$.
Define $T'$ as in condition $(\mathcal{C})$ and let $i = |T'|$. Notice that $i \geq 2$.
By definition of $w$ there exists a pin representation $p=(p_1,\dots,p_n)$ of $\pi$ such that $T' = \{p_1, \dots, p_i\}$,
$T$ is read in two \fois and any corresponding pin word $u = u_1 \dots u_n$ satisfies $u_{i+1} \dots u_n = w$.
Then $p$ satisfies statement $(2)$ of Lemma~\ref{lem:lecture_simple}, thus $T = \{p_1, \dots, p_i\} \cup \{p_n\}$
as in the second diagram of Figure~\ref{fig:pin-rep_quasi-epi_1bloc} and $\{p_{i+2}, \dots, p_n\}$ are separating pin.
As $i \geq 2$ and $T'$ is a block of $\pi$, $p_{i+1}$ is an independent pin encoded with a numeral.
So $w = u_{i+1} \dots u_n$ is a strict pin word.

Moreover as $T = \{p_1, \dots, p_i\} \cup \{p_n\}$,
$(p_{i+1}, \dots, p_{n})$ is a pin representation of $\alpha$ ending with $x$
thus $w$ is a pin word encoding $\alpha$ and $|w| = |\alpha| \geq 4$.
Likewise $(p_i, \dots, p_{n-1})$ is a pin representation of $\alpha$ beginning with $x$.

Denoting by $w'$ the suffix of length $2$ of $w$, then
from Lemma~\ref{lem:quadrant} (p.\pageref{lem:quadrant}) $\phi^{-1}(w')$
is a numeral indicating the quadrant in which $p_n$ lies with respect
to $T'$. And as $T = T' \cup \{p_n\}$, for all $u'$ in $P(T')$, $u'\cdot \phi^{-1}(w')$
belongs to $P(T)$.

Moreover letting $\bar{w}$ be the prefix of $w$ of length $|w|-1$,
for all $\quadrantell$ in $\{1,2,3,4\}$, $\quadrantell \, \bar{w}$ is a quasi-strict pin word of $\alpha$.
Therefore $\bar{w} \in Q_x(\alpha)$.  Denoting $Z$ the last letter of $w$, $Z$ encodes $p_n$.
Moreover $\bar{w}$ encodes $p_{i+1}, \dots, p_{n-1}$ and the position of $p_{i+1}, \dots, p_n$
is the same (up to symmetry) as the one shown on Figure~\ref{fig:pin-rep_quasi-epi_1bloc}.
On this figure, it is immediate to check that $\phi(\bar{w})$ does not contain $Z$.
To prove that this holds not only for $\bar{w}$ but also for all words of $Q_x(\alpha)$,
we first study the cardinality of $Q_x(\alpha)$.

\smallskip

From Remark~\ref{rem:tailleQx}, to each pin representation of $\alpha$ whose first point read is $x$ corresponds exactly one word of $Q_x(\alpha)$.
Recall from Remark~\ref{rem:compute_Q_x(alpha)} that a pin representation of $\alpha$ is determined by its first two points, which form an active knight. 
So we just have to compute the number of active knights of $\alpha$ to which $x$ belongs,
remembering that $\alpha$ is an increasing quasi-oscillation and $x$ is an auxiliary point of $\alpha$. 
This question has been addressed in Remark~\ref{rem:knights_quasi-oscillations} (p.\pageref{rem:knights_quasi-oscillations}). 
It follows that if $\alpha$ is an increasing quasi-oscillation of size greater than $4$
and $x$ is an auxiliary point of $\alpha$, then $|Q_x(\alpha)|=1$ as $x$ belongs to only one active knight;
and when $|\alpha|=4$, $|Q_x(\alpha)|=2$ as $x$ belongs to two active knights.

\smallskip

To conclude the proof, recall that the word $\bar{w}$ defined earlier belongs to $Q_x(\alpha)$
and is such that $\phi(\bar{w})$ does not contain $Z$.
When $|\alpha| \neq 4$, we have $|Q_x(\alpha)|=1$ so that $Q_x(\alpha) = \{\bar{w}\}$ and we conclude that no word of $\phi(Q_x(\alpha))$ contains $Z$.
When $|\alpha| = 4$, $|Q_x(\alpha)|=2$ and there is only one word $\bar{w}'$ different from $\bar{w}$ in $Q_x(\alpha)$,
which may be computed from Figure~\ref{fig:quasiepi} (p.\pageref{fig:quasiepi}).
We then check by comprehensive verification (of the four cases of size $4$ on Figure~\ref{fig:quasiepi}) that $\phi(\bar{w}')$ does not contain $Z$.
Details are left to the reader.
\end{proof}

We are furthermore able to describe $w$ explicitly in Remark~\ref{rem:w_explicit} below, and we record its expression here for future use in our work.

\begin{defi}\label{def:w_alpha}
To each quasi-oscillation $\alpha$ of which an auxiliary point $A$ is marked,
we associate a word $w^A_\alpha$ defined below.
Denoting by $M$ the main substitution point of $\alpha$ corresponding to $A$  and by  $K_{A,M}$ the active knight formed by $A$ and $M$ then:

When  $\alpha$ is increasing and $K_{A,M}$ is of type $H$ (resp. $V$), \\
\indent \indent  if $A$ is in the top right corner of $\alpha$, we set\\
\indent \indent \indent \indent  $w^A_\alpha = (DL)^{p-2}DRU$ (resp. $w^A_\alpha =(LD)^{p-2}LUR$) if  $|\alpha| = 2p$ \\
\indent \indent \indent \indent  $w^A_\alpha = (DL)^{p-2}UR$ (resp. $w^A_\alpha =(LD)^{p-2}RU$) if $|\alpha| = 2p-1$;

\indent \indent if $A$ is in the bottom left corner of $\alpha$, we set \\
\indent \indent \indent \indent  $w^A_\alpha = (UR)^{p-2}ULD$ (resp. $w^A_\alpha =(RU)^{p-2}RDL$) if  $|\alpha| = 2p$\\
\indent \indent \indent \indent   $w^A_\alpha = (UR)^{p-2}DL$ (resp. $w^A_\alpha =(RU)^{p-2}LD$) if  $|\alpha| = 2p-1$;

When $\alpha$ is decreasing, $w^A_\alpha$ is obtained by symmetry exchanging left and right.
\end{defi}

Notice that for quasi-oscillations that are both increasing and decreasing the choice of $A$ determines their nature, so that  $w^A_\alpha$ is properly defined.

\begin{rem}\label{rem:w_explicit}
If $A$ is in the top right corner of $\alpha$
(see Figure~\ref{fig:pin-rep_quasi-epi_1bloc} p.\pageref{fig:pin-rep_quasi-epi_1bloc} or Figure~\ref{fig:episgrands} p.\pageref{fig:episgrands}),
then $w = 3\cdot w^A_\alpha$.
If $A$ is in the bottom left (resp. top left, bottom right) corner of $\alpha$ then $w = 1\cdot w^A_\alpha$
(resp.  $w= 4\cdot w^A_\alpha$, $w = 2\cdot w^A_\alpha$).
\end{rem}

\begin{rem}\label{rem:w_2}
In Theorem~\ref{thm:conditionc}, if $\pi$ satisfies condition $(\mathcal{C})$
and $|T| =2$, then $ P_{\{1,n\}}(\pi) \cup
P_{\{2,n\}}(\pi) = \{u \in P(\pi) \mid u \text{
  strict or quasi-strict}\}$, denoted $P_{\textsc{sqs}}(\pi)$.
This set corresponds to two proper pin
representations, so it contains $12$ pin words
(see Remark~\ref{rem:nb_pin_words} p.\pageref{rem:nb_pin_words}). 
Moreover, with the notations of Lemma~\ref{lem:lecture_epi} (p.\pageref{lem:lecture_epi}),
and $K_{A,M}$ as in Definition~\ref{def:w_alpha},
we have an explicit description of $P_{\textsc{sqs}}(\pi)$:

When $K_{A,M}$ is of type $H$ (resp. $V$), \\
\indent \indent if $\alpha$ is increasing (see Figure~\ref{fig:pin-rep_quasi-epi_1bloc} p.\pageref{fig:pin-rep_quasi-epi_1bloc}), then \\
\indent \indent \indent $P_{\textsc{sqs}}(\pi) = (\Qplus+\SHplus) \cdot w^A_\alpha$
(resp. $P_{\textsc{sqs}}(\pi) = (\Qplus+\SVplus) \cdot w^A_\alpha$) \\
\indent \indent and if $\alpha$ is decreasing, then \\
\indent \indent \indent $P_{\textsc{sqs}}(\pi) = (\Qminus+\SHminus) \cdot w^A_\alpha$
(resp. $P_{\textsc{sqs}}(\pi) = (\Qminus+\SVminus) \cdot w^A_\alpha$).
\end{rem}

\medskip

This concludes the study of the case $\pi =
\alpha[1,\ldots,1,T,1,\ldots,1]$.  It now remains to deal with the
case where more than one child of $\alpha$ is not a leaf. From
Theorem~3.1 of~\cite{BBR09} (see also Equation~\eqref{eq:pin_perm_trees} p.\pageref{eq:pin_perm_trees}), in this case $\alpha$ is an increasing
(resp. decreasing) quasi-oscillation having exactly two children
that are not leaves, and these are completely determined.

\begin{theo}\label{thm:primeRoot_cas_special}
  Let $\pi = $ \begin{tikzpicture}[level distance=17pt,sibling
  distance=6pt,baseline=-10pt,inner sep=0] \node[simple,inner sep=0] (X)
  {$\beta^{+}$} child {[fill] circle (2pt)} child [missing] child {[fill] circle
    (2pt) node (xx1){}} child [missing] child [sibling distance=0pt] {node (xx2){} edge from parent [draw=none] }child[thick, dotted]  {node[thin, shape=isosceles triangle, solid, draw, shape border rotate=90,anchor=apex, minimum
    height=5mm,inner sep=1pt,isosceles triangle apex angle=110]   {$T$}}
  child [sibling distance=0pt] {node (xx3){} edge from parent [draw=none] } child [missing] child[dash pattern=on 3pt off 2pt on 1pt off
  2pt] {node (xx4) {$12$}} child [missing] child {[fill] circle (2pt)};
  \draw [dotted] (xx1) -- (xx2);
  \draw [dotted] (xx3) -- (xx4);
  \end{tikzpicture}
  where $\beta^+$ is an increasing quasi-oscillation, the
  permutation $12$ expands an auxiliary point of $\beta^{+}$ and $T$ (of size at least $2$) expands the
  corresponding main substitution point of $\beta^{+}$. Then $P(\pi) = P(T) \cdot w$ where $w$
  is the unique word encoding the unique reading of the remaining leaves in $\pi$ after $T$ is read.
\end{theo}
\vspace{-0.3cm}
~\\
\begin{minipage}{.6\textwidth}
\begin{proof}
Let $p=(p_1,\ldots,p_n)$ be a pin representation associated to $\pi$.
According to Section 3.4 (and more precisely Figure 10) of~\cite{BBR09}, the configuration depicted on Figure~\ref{fig:pin-rep_quasi-epi_2blocs}
is the only possible configuration up to symmetry for a pin-permutation whose root is a simple permutation with two non trivial children.
Thus the sequence $(p_{k+1},\ldots,p_n)$ is uniquely determined in $\pi$.
Moreover $k+1 \geq 3$, so that the suffix encoding $(p_{k+1},\ldots,p_n)$ in a pin word of $p$ is a word
$w$ uniquely determined from Remark~\ref{rem:nb_pin_words} (p.\pageref{rem:nb_pin_words}).
\end{proof}
\end{minipage}
\quad
\begin{minipage}{.35\textwidth}
\begin{tikzpicture}[scale=.3]
\draw (6,1) rectangle (8,3);
\draw (7,2) node {$T$};

\draw (9,4) rectangle (11,6);
\draw (11.7,4) node {$T'$};

\draw (10.5,5.5) [fill] circle (.2);
\draw (11.8,5.5) node {\footnotesize $p_{n}$};
\draw (10.5,5.5) -- (2,5.5);

\draw (9.5,4.5) [fill] circle (.2);
\draw (9.5,3.6) node {\footnotesize $p_{k+1}$};

\draw (8.5,0) [fill] circle (.2);
\draw (8.8,-0.8) node {\footnotesize $p_{k+2}$};
\draw (8.5,0) -- (8.5,2);

\draw (5,0.5) [fill] circle (.2);
\draw (5,0.5) -- (9,0.5);

\draw (5.5,-1) [fill] circle (.2);
\draw (5.5,-1) -- (5.5,1);

\draw (4,-0.5) [fill] circle (.2);
\draw (4,-0.5) -- (6,-0.5);

\draw (4.5,-2) [fill] circle (.2);
\draw (4.5,-2) -- (4.5,0);

\draw (3,-2) node {$\cdot$};
\draw (3.3,-1.7) node {$\cdot$};
\draw (3.6,-1.4) node {$\cdot$};

\draw (2,-2.5) [fill] circle (.2);
\draw (2,-2.5) -- (5,-2.5);
\draw (0.7,-2.5) node {\footnotesize $p_{n-2}$};

\draw (2.5,6) [fill] circle (.2);
\draw (2.5,6) -- (2.5,-3);
\draw (2.5,6.5) node {\footnotesize $p_{n-1}$};
\end{tikzpicture}
  
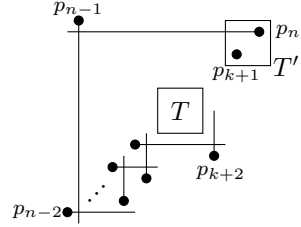
\captionof{figure}{Diagram of $\pi$ if two children are not leaves}
\label{fig:pin-rep_quasi-epi_2blocs}
\end{minipage}

\begin{rem} \label{rem:w_cas_2_non_feuilles}
The word $w$ in the statement of Theorem~\ref{thm:primeRoot_cas_special} is a strict pin word
uniquely determined by $\beta^{+}$ and the two points expanded in $\beta^{+}$.

More precisely, taking the notations of 
Definition~\ref{def:w_alpha}
(with $\beta^+$ instead of $\alpha$), $w = 1 \cdot w^A_{\beta^+}$ (resp. $w = 3 \cdot w^A_{\beta^+}$)
when $A$ is in the top right (resp. bottom left) corner of $\beta^+$
(see Figure~\ref{fig:pin-rep_quasi-epi_2blocs}).
\end{rem}

For decomposition trees whose root is a decreasing quasi-oscillation, we obtain from Remark~\ref{rem:ominus=oplus_transposed} (p.\pageref{rem:ominus=oplus_transposed})
a description of $P(\pi)$ similar to the one of Theorem~\ref{thm:primeRoot_cas_special}.

\section[Building deterministic automata $\mathcal A_{\pi}$]{Building deterministic automata $\mathcal A_{\pi}$ accepting the languages $\overleftarrow{{\mathcal L}_{\pi}}$} 
\label{sec:buildingAutomata}

Appendix~\ref{sec:pinwords} gives a recursive description 
of the set $P(\pi)$ of pin words encoding $\pi$, 
for any pin-permutation $\pi$. 
As explained in Subsection~\ref{ssec:idees_construction_automates}, 
we next use this precise knowledge about $P(\pi)$ 
to build deterministic automata $\mathcal A_{\pi}$ 
recognizing the languages $\overleftarrow{{\mathcal L}_{\pi}}$, 
for any pin-permutation $\pi$. 
Recall from Subsection~\ref{ssec:idees_construction_automates}
(p.\pageref{ssec:idees_construction_automates}) that 
\begin{align*}
\overleftarrow{{\mathcal L}_{\pi}}& = \bigcup_{ u \in P(\pi) \atop u = u^{(1)}u^{(2)}\ldots u^{(j)}} A^\star
\overleftarrow{\phi(u^{(j)})} A^\star \ldots A^\star
\overleftarrow{\phi(u^{(2)})} A^\star \overleftarrow{\phi(u^{(1)})}
A^\star
\end{align*}
where $A = \{U,D,L,R\}$, $\phi$ is the map introduced in
Definition~\ref{def:phi} (p.\pageref{def:phi}) and for every pin word $u$, by $u=u^{(1)}u^{(2)}\ldots u^{(j)}$ we mean 
(here and everywhere after) that $u^{(1)}u^{(2)}\ldots u^{(j)}$ is the strong
numeral-led factor decomposition of $u$.

To build these automata $\mathcal A_{\pi}$ recognizing $\overleftarrow{{\mathcal L}_{\pi}}$, 
we proceed again recursively, distinguishing several cases following Equation~\eqref{eq:pin_perm_trees} p.\pageref{eq:pin_perm_trees}, 
\emph{i.e.}, according to the shape of the decomposition tree of $\pi$. 
We also present an alternative construction of $\mathcal A_{\pi}$
whose complexity is optimized; but instead of
$\overleftarrow{{\mathcal L}_{\pi}}$, the automaton recognizes a
language $\mathcal{L}'_{\pi}$ such that $\mathcal{L}'_{\pi}\cap {\mathcal M} =
\overleftarrow{\mathcal{L}_{\pi}} \cap {\mathcal M}$. 
In both constructions, the automata $\mathcal A_{\pi}$ are deterministic, 
and we have explained in Subsection~\ref{ssec:structure_algo} that 
this is the key to control the complexity of our algorithm. 
Moreover, in addition to being deterministic, 
both automata are complete and have a unique final state 
without outgoing transitions except for a loop labeled by all letters of $A$. 
These properties of $\mathcal A_{\pi}$ are inherited 
from the smaller automata used in its construction.

Our construction of automata accepting words $v \in \overleftarrow{\mathcal{L}_{\pi}}$ relies on a
greedy principle: at each step we find the first occurrence of
$\overleftarrow{\phi(u^{(\ell)})}$ that appears in the
suffix of the word $v$ that has not yet been read by the automaton.
This is facilitated by the fact that
in $\overleftarrow{\mathcal{L}_{\pi}}$ the
factors $\overleftarrow{\phi(u^{(\ell)})}$ are separated by
$A^{\star}$.

The reason why we consider \emph{reversed} words is in order to preserve determinism. Indeed
intuitively the possible beginnings of pin words encoding a
permutation may be numerous, whereas all these words end with very
similar shuffle products as it appears in
Theorems~\ref{thm:pinwords_cas_lineaire_non_recursif}
and~\ref{thm:linearRoot}
(p.\pageref{thm:pinwords_cas_lineaire_non_recursif} and
\pageref{thm:linearRoot}).

The description of our construction of the automata 
$\mathcal A_{\pi}$ is organized as follows. 
In Subsection~\ref{sec:ac}, we present generic
constructions of automata that will be used several times.  In
Subsections~\ref{pw:non-recursive} to~\ref{sec:decomp_simple}, we
construct recursively the automata $\mathcal A_{\pi}$ that recognize
the languages $\overleftarrow{{\mathcal L}_{\pi}}$ for any
pin-permutation $\pi$, distinguishing cases according to the
decomposition tree of $\pi$ -- see Equation~\eqref{eq:pin_perm_trees}
p.\pageref{eq:pin_perm_trees}.  In these constructions, some states of
the automata must be marked, and this is detailed in
Subsection~\ref{subsection:marquage}.  We conclude with
Subsection~\ref{ssec:complexity} that analyzes the complexity of
building $\mathcal A_{\pi}$.

\subsection{Generic constructions of deterministic automata}\label{sec:ac}
We present some generic constructions that are used in the next subsections.
We refer the reader to~\cite{HU79} for more details about automata.

\paragraph*{Aho-Corasick algorithm} Let $X$ be a finite set of
words over a finite alphabet $A$. The Aho-Corasick algorithm~\cite{AC75}
builds a deterministic automaton that recognizes $A^{\star}X$ in
linear time and space w.r.t.~the sum $\| X \|$ of the lengths of the
words of $X$. The first step of the algorithm consists in constructing
a tree-automaton whose states are labeled by the prefixes of the words
of $X$. The initial state is the empty word $\varepsilon$. For any
word $u$ and any letter $a$ there is a transition labeled by $a$ from
state $u$ to state $ua$ if $ua$ is a prefix of a word of $X$. At this
step the final states are the leaves of the tree. The second step
consists in adding transitions in the automaton according to a
breadth-first traversal of the tree-automaton to obtain a complete
automaton. For any state $u$ and any letter $a$, the transition from
$u$ labeled by $a$ goes to the state corresponding to the longest
suffix of $ua$ that is also a prefix of a word of $X$. The set of
final states is the set of states corresponding to words having a
suffix in $X$. These states correspond to a leaf or an internal node -- when
there is a factor relation between two words of $X$ -- of the original
tree-automaton. The ones corresponding to internal nodes
are marked on the fly during the construction of
the missing transitions.

\begin{rem}\label{rem:Aho_Corasick_usuel}
Notice that all transitions labeled with a letter of $A$ that does not
appear in any word of $X$ go to the initial state.
Moreover the reading of any word $u$ by the automaton leads to the state labeled with
the longest suffix of $u$ that is also a prefix of a word of $X$.
\end{rem}

\paragraph*{A variant for first occurrences} An adaptation of
the Aho-Corasick algorithm allows us to build in linear time and space
w.r.t.~$\| X \|$ a deterministic automaton, denoted $\AC{X}$,
recognizing the set of words ending with a {\em first} occurrence of a
word of $X$ (which is strictly included in $A^{\star}X$).
First we perform the first step of the Aho-Corasick
algorithm on $X$, obtaining a tree automaton.
We modify the second step as follows:
in the breadth-first traversal, we stop the exploration of a branch and delete its
descendants as soon as a final state is reached. Moreover we do not
build the outgoing transitions
from the final states, nor the loops on the final states.
This ensures that the language recognized is the set of words
ending with a first occurrence of a word of $X$.  Finally we merge the
final states into a unique final state $f$ to obtain $\AC{X}$.
Moreover if we add a loop labeled by all letters of $A$ on $f$ we
obtain an automaton ${\mathcal {AC}}^{\circlearrowleft}(X)$ that
recognizes the set $A^{\star}XA^{\star}$ of words having a factor in $X$.

\begin{rem}
The main difference between our variant and the construction of
Aho-Corasick is that we \emph{stop} as soon as a \emph{first}
occurrence of a word of $X$ is read. This ensures that $\AC{X}$ has a
unique final state without any outgoing transition.
\end{rem}

This variant for first occurrences satisfies properties
analogous to Remark~\ref{rem:Aho_Corasick_usuel}:

\begin{lem}\label{lem:trans_initial}
In $\AC{X}$, all transitions labeled with a letter that does not
appear in any word of $X$ go to the initial state.
Moreover let $u$ be a word without any factor in $X$ except maybe as a suffix.
Then the reading of $u$ by $\AC{X}$ leads to the state labeled with
the longest suffix of $u$ that is also a prefix of a word of $X$.
\end{lem}

\begin{proof}
Let $\mathcal A$ be the usual Aho-Corasick automaton on $X$.
Then $\AC{X}$ (before the merge of all final states in $f$) is a subautomaton of $\mathcal A$, and
therefore the first assertion is a direct consequence of Remark~\ref{rem:Aho_Corasick_usuel}.
Let $u$ be a word without any factor in $X$ except maybe as a suffix.
Then the path of the reading of $u$ by $\mathcal A$ does not visit any final state,
except maybe the last state reached.
Thus all this path is included in $\AC{X}$ and we conclude using Remark~\ref{rem:Aho_Corasick_usuel}.
\end{proof}

\paragraph*{A variant for a partition $X_1, X_2$} When the set $X$ is
partitioned into two subsets $X_1$ and $X_2$ such that no word of
$X_1$ (resp. $X_2$) is a factor of a word of $X_2$ (resp. $X_1$)
\footnote{This is a simple condition that allows us to define without ambiguity
words ending with a first occurrence either in $X_1$ or in $X_2$.},
we adapt the previous construction and build a deterministic automaton
$\AC{X_1, X_2}$ which recognizes the same language as $\AC{X}$.  But
instead of merging all final states into a unique final state, we
build two final states $f_1$ and $f_2$ corresponding to the first
occurrence of a word of $X_1$ (resp. $X_2$).  This construction is
linear in time and space w.r.t.~$\| X_1 \| + \| X_2 \|$.  In what
follows we will use this construction only when $X_1$ and $X_2$ are
languages on disjoint alphabets, so that the factor independence
condition is trivially satisfied.

\paragraph*{Concatenation} Building an automaton
${\mathcal A}_1 \cdot {\mathcal A}_2$ recognizing the concatenation
${\mathcal L}_1 \cdot {\mathcal L}_2$ of two languages respectively
recognized by the deterministic automata ${\mathcal A}_1$ and
${\mathcal A}_2$ is easy when ${\mathcal A}_1$ has a unique final state without outgoing transitions.
Indeed it is enough to merge the
final state of ${\mathcal A}_1$ and the initial state of ${\mathcal
  A}_2$ into a unique state that is not initial (resp. not final),
except when the initial state of ${\mathcal A}_1$ (resp. ${\mathcal
  A}_2$) is final. Note that the resulting
automaton is deterministic and of size at most $|{\mathcal
  A}_1|+|{\mathcal A}_2|$, where the {\em size} $|{\mathcal A}|$ is
the number of states of any automaton ${\mathcal A}$.
This construction is done in constant time.

When the final state of ${\mathcal A}_1$ has no outgoing transitions except for a loop labeled by all letters of $A$ and when ${\mathcal L}_2$ is of type $\Astar \cdot {\mathcal L}$ for an arbitrary language  ${\mathcal L}$, we can do the same construction to obtain an automaton recognizing the concatenation ${\mathcal L}_1 \cdot {\mathcal L}_2 = {\mathcal L}_1 \cdot \Astar \cdot {\mathcal L}$. We just have to delete the loop on the final state of ${\mathcal A}_1$ before merging states.

In particular, according to this construction,
the automaton obtained concatenating ${\mathcal {AC}}^{\circlearrowleft}(X)$ and ${\mathcal A}$
is $\AC{X} \cdot {\mathcal A}$. Therefore,
even though $\AC{X}$ recognizes a language strictly included in $\Astar X$,
$\AC{X} \cdot {\mathcal A}$ recognizes $\Astar X \Astar {\mathcal L}$
when ${\mathcal A}$ recognizes a language $\Astar {\mathcal L}$.

\paragraph*{Union} We say that an automaton is {\em almost complete} if
for any letter $a$, all  non final states have an outgoing transition labeled by $a$ (notice that the only difference with complete automaton is that final states are allowed to miss some transitions).
Let ${\mathcal A}_1$ and ${\mathcal A}_2$ be two
deterministic automata that are 
almost complete\footnote{
Notice that the automata $\AC{X}$, ${\mathcal {AC}}^{\circlearrowleft}(X)$ and $\AC{X_1, X_2}$ satisfy these conditions.
}.
We define the automaton ${\mathcal U}({\mathcal
  A}_1,{\mathcal A}_2)$ as follows. We perform the Cartesian product
of ${\mathcal A}_1$ and ${\mathcal A}_2$ beginning from the pair of
initial states (see~\cite{HU79}). However we stop exploring a path when it enters a final
state of ${\mathcal A}_1$ or ${\mathcal A}_2$.  Therefore in
${\mathcal U}({\mathcal A}_1,{\mathcal A}_2)$ there is no outgoing
transitions from any state $(q_1,q_2)$ such that $q_1$ or $q_2$ is
final. Moreover these states are merged into a unique final state of
${\mathcal U}({\mathcal A}_1,{\mathcal A}_2)$.  Let ${\mathcal L}_1$
(resp. ${\mathcal L}_2$, ${\mathcal L}$) be the language recognized by
${\mathcal A}_1$ (resp. ${\mathcal A}_2$, ${\mathcal U}({\mathcal
  A}_1,{\mathcal A}_2)$). Then ${\mathcal L}$ is the set of words of
${\mathcal L}_1 \cup {\mathcal L}_2$ truncated after their first
factor in ${\mathcal L}_1 \cup {\mathcal L}_2$. The language
$({\mathcal L}_1 \cup {\mathcal L}_2)A^\star = {\mathcal L} A^\star$
    is recognized
by the automaton ${\mathcal U}^{\circlearrowleft}({\mathcal A}_1,{\mathcal A}_2)$ with an
additional loop labeled by all letters of $A$ on the final state. Notice that ${\mathcal U}({\mathcal A}_1,{\mathcal A}_2)$ (resp. ${\mathcal U}^{\circlearrowleft}({\mathcal A}_1,{\mathcal A}_2)$) is deterministic,
almost complete (resp. complete) and has a unique final state without outgoing transitions (resp. whose only outgoing transition is a loop labeled by all letters of $A$). The complexity in time and space of these constructions is in ${\mathcal O}(|{\mathcal A}_1| \cdot |{\mathcal A}_2|)$.

\subsection{Pin-permutation of size $1$ and simple pin-permutations} \label{pw:non-recursive}

\paragraph*{Pin-permutation of size $1$} 
When $\pi = 1$, we have seen in Subsection~\ref{ssec:idees_construction_automates} 
that $\overleftarrow{{\mathcal L}_{\pi}} = A^\star{\mathcal M}_2 A^\star$ 
is recognized by the automaton ${\mathcal A}_{\pi}$ of Figure~\ref{fig:automate1} (p.\pageref{fig:automate1}). 

\paragraph*{Simple pin-permutations}
\label{par:simple}
In this paragraph, for a simple permutation $\pi$ whose set of pin words $P(\pi)$ is given, we build the automaton ${\mathcal A}_{\pi}$. 
The computation of $P(\pi)$ from $\pi$ is discussed in
Subsection~\ref{ssec:finding_pin-perm} as a sub-procedure of the algorithm described in
Section~\ref{sec:polynomial}. The study that follows is based on the upcoming lemma.

\begin{lem}\label{lem:sqs_L(u)}
For every permutation $\pi$ (not necessarily simple), we have
$$\bigcup\limits_{u \in
  P(\pi) \atop u \text{ strict or quasi-strict}}\text{\hspace{-0.9cm}}
\mathcal{L}(u) = \Astar \cdot E_{\pi}^{\textsc{s}} \cdot \Astar \ \cup \ \Astar \cdot \mathcal{M}_2 \cdot \Astar \cdot E_{\pi}^{\textsc{qs}} \cdot \Astar
$$
where $E_{\pi}^{\textsc{s}} =\{\phi(u) \mid u \in P(\pi), u \text{ is
  strict}\}$ and $E_{\pi}^{\textsc{qs}} = \{\phi(u^{(2)}) \mid
u=u^{(1)}u^{(2)} \in P(\pi), u \textrm{ is}$
$\textrm{qua\-si-strict}\}$.
\end{lem}

\begin{proof}
By definition of ${\mathcal L}(u)$ (see Definition~\ref{def:lu} p.\pageref{def:lu}),
$$\bigcup\limits_{u \in P(\pi) \atop u \text{ strict or quasi-strict}}\text{\hspace{-1cm}} {\mathcal L}(u)= \Big(\bigcup\limits_{u \in P(\pi) \atop u \text{ strict}}
A^\star\phi(u)A^\star \Big) \, \bigcup \,  \Big(\text{\hspace{-0.5cm}}
\bigcup\limits_{u=u^{(1)}u^{(2)} \in P(\pi) \atop u \text{
    quasi-strict}} \hspace{-0.5cm}
A^\star\phi(u^{(1)})A^\star\phi(u^{(2)})A^\star \Big)\text{.}$$

Moreover, as can be seen on Figure~\ref{fig:origine} (p.\pageref{fig:origine}),
for every quasi-strict pin word $u = u^{(1)}u^{(2)} \in
P(\pi)$, the words $hu^{(2)}$ also belong to $ P(\pi)$ for all
$h\in\{1,2,3,4\}$.  This allows to write
$$\bigcup\limits_{u \in P(\pi) \atop u \text{ strict or quasi-strict}}\text{\hspace{-1cm}} {\mathcal L}(u)=  \Big( \bigcup\limits_{u \in P(\pi) \atop u \text{ strict}}
A^\star \phi(u) A^\star \Big) \, \bigcup \,
\Big( \text{\hspace{-0.2cm}}\bigcup\limits_{h \in \{1,2,3,4\}} \hspace{-0.5cm} A^\star\phi(h)\,
\, \cdot \text{\hspace{-0.6cm}} \bigcup\limits_{u=u^{(1)}u^{(2)} \in P(\pi) \atop u \text{
    quasi-strict}} \hspace{-0.5cm} A^\star\phi(u^{(2)})A^\star
\, \Big)  \text{.}$$ 
Hence
$$\bigcup\limits_{u \in P(\pi) \atop u \text{ strict or
    quasi-strict}}\text{\hspace{-1cm}} \mathcal{L}(u) = \Astar \cdot
E_{\pi}^{\textsc{s}} \cdot \Astar \ \cup \ \Astar \cdot \mathcal{M}_2
\cdot \Astar \cdot E_{\pi}^{\textsc{qs}} \cdot \Astar \text{,}
$$
concluding the proof.
\end{proof}

\begin{lem}\label{lem:ascsquadratique} 
For every permutation $\pi$ whose sets of strict and quasi-strict pin
words (or equivalently $E_{\pi}^{\textsc{s}}$ and $E_{\pi}^{\textsc{qs}}$) are given,
one can build in time and space ${\mathcal O}
\left(|E_{\pi}^{\textsc{s}}| \cdot |E_{\pi}^{\textsc{qs}}| \cdot
|\pi|^2 \right)$ 
a deterministic complete automaton $\mathcal{A}_{\pi}^{\textsc{sqs}}$ having a unique
final state without outgoing transitions except for a loop labeled by
all letters of $A$ that recognizes the language $ \bigcup\limits_{u
  \in P(\pi) \atop u \text{ strict or quasi-strict}}
\text{\hspace{-0.8cm}} \overleftarrow{\mathcal{L}(u)}$. 
\end{lem}

\begin{proof}
From Lemma~\ref{lem:sqs_L(u)}, we have
$$\bigcup\limits_{u \in P(\pi) \atop u
 \text{ strict or quasi-strict}}
\text{\hspace{-1cm}} \overleftarrow{\mathcal{L}(u)} = \overleftarrow{A^\star
  E_{\pi}^{\textsc{s}}A^\star} \bigcup \overleftarrow{A^\star
  \mathcal{M}_2 A^\star E_{\pi}^{\textsc{qs}}A^\star} = \left( A^\star
\overleftarrow{E_{\pi}^{\textsc{s}}} \bigcup (A^\star
\overleftarrow{E_{\pi}^{\textsc{qs}}} \cdot A^\star \mathcal{M}_2)
\right) A^\star\text{.}
$$

Recall that $\AC{\overleftarrow{E_{\pi}^{\textsc{s}}}}$,
$\AC{\overleftarrow{E_{\pi}^{\textsc{qs}}}}$ and $\AC{\M_2}$ are automata recognizing respectively the
set of words ending with a first occurrence of a word of $\overleftarrow{E_{\pi}^{\textsc{s}}}$,
$\overleftarrow{E_{\pi}^{\textsc{qs}}}$ and $\M_2$ and obtained using the construction given in Subsection~\ref{sec:ac}.
The sizes of the first two automata are respectively ${\mathcal O} \left(|E_{\pi}^{\textsc{s}}| \cdot |\pi| \right)$
and ${\mathcal O} \left(|E_{\pi}^{\textsc{qs}}| \cdot |\pi| \right)$, and the size of the third one is constant.
Indeed for all $w$ in $E_{\pi}^{\textsc{s}}$, $|w| = |\pi|+1$ and for all $w$ in $E_{\pi}^{\textsc{qs}}$,
$|w| = |\pi|$ so that $\|E_{\pi}^{\textsc{s}}\| = |E_{\pi}^{\textsc{s}}| \cdot (|\pi|+1)$ and
$\|E_{\pi}^{\textsc{qs}}\| = |E_{\pi}^{\textsc{s}}| \cdot |\pi|$.

Then the deterministic automaton ${\mathcal A}_{\pi}^{\textsc{sqs}}$ is obtained as the union
$\linebreak {\mathcal U}^{\circlearrowleft} (\AC{\overleftarrow{E_{\pi}^{\textsc{s}}}} \ ,\
{\AC{\overleftarrow{E_{\pi}^{\textsc{qs}}}} \cdot \AC{\M_2}} )$
in time and space
${\mathcal O} \left(|E_{\pi}^{\textsc{s}}| \cdot | E_{\pi}^{\textsc{qs}}| \cdot |\pi|^2 \right)$.
\end{proof}

Lemma~\ref{lem:ascsquadratique} is used mostly in two special cases where \emph{all} the pin words of $\pi$
are strict or quasi-strict, and that we identify explicitly in the following remark.

\begin{rem}\label{rem:ascquadratique}
By Theorem~\ref{thm:nbpinwords} (p.\pageref{thm:nbpinwords})
the pin words encoding a simple permutation are either strict or quasi-strict and there are at most 48 of them.
Therefore when $\pi$ is a simple permutation, we take
$\mathcal{A}_{\pi}=\mathcal{A}_{\pi}^{\textsc{sqs}}$ (from Lemma~\ref{lem:ascsquadratique}) and the
time and space complexity of the construction of $\mathcal{A}_{\pi}$ is
quadratic w.r.t.~$|\pi|$, as soon as the pin words of $\pi$ are given. 

Notice also that when $\pi=12$,
$\mathcal{A}_{\pi}=\mathcal{A}_{\pi}^{\textsc{sqs}}$ and the time and
space complexity of the construction is ${\mathcal O}(1)$.
\end{rem}

The above construction follows the general scheme announced at the
beginning of the section, but it is not optimized. We actually can
provide a more specific construction of $\mathcal{A}_{\pi}$ whose complexity is
\emph{linear} when the permutation $\pi$ is simple. This construction relies on the
same idea as the one we give in~\cite{BBPR10}.

\begin{lem} \label{lem:sqs_linear_complexity}
For every permutation $\pi$ whose sets of strict and quasi-strict pin
words (or equivalently $E_{\pi}^{\textsc{s}}$ and $E_{\pi}^{\textsc{qs}}$) are given,
one can build in time and space ${\mathcal
  O} \left((|E_{\pi}^{\textsc{s}}|+|E_{\pi}^{\textsc{qs}}|) \cdot
|\pi| \right)$ 
a deterministic complete automaton $\mathcal{A}_{\pi}^{\textsc{sqs}}$ having a unique
final state without outgoing transitions except for a loop labeled by
all letters of $A$ that recognizes a language ${\mathcal L}'$ such that $${\mathcal
  L}'\cap {\mathcal M} = \text{\hspace{-1cm}}\bigcup\limits_{u \in P(\pi) \atop u \text{
    strict or quasi-strict}} \text{\hspace{-0.8cm}} \overleftarrow{\mathcal{L}(u)} \cap
{\mathcal M}.$$
\end{lem}

\begin{proof}
Define $E_{\pi}= E_{\pi}^{\textsc{s}} \ \cup \ \mathcal{M}_2
\cdot E_{\pi}^{\textsc{qs}}$ and $\mathcal{L}' = \Astar
\cdot \overleftarrow{E_{\pi}} \cdot \Astar$.  Let us prove that
$${\mathcal L}'\cap {\mathcal M} = \text{\hspace{-1cm}}\bigcup\limits_{u \in P(\pi) \atop u
  \text{ strict or quasi-strict}}\text{\hspace{-0.8cm}} \overleftarrow{\mathcal{L}(u)} \cap
              {\mathcal M}\text{.}$$ This will be enough to conclude
              the proof of Lemma~\ref{lem:sqs_linear_complexity},
              setting $\mathcal{A}_{\pi}^{\textsc{sqs}}= {\mathcal
                {AC}}^{\circlearrowleft}(\overleftarrow{E_{\pi}})$.

From Lemma~\ref{lem:sqs_L(u)}, we have
$$\bigcup\limits_{u \in
  P(\pi) \atop u \text{ strict or quasi-strict}} \text{\hspace{-0.8cm}}
\mathcal{L}(u) = \Astar \cdot E_{\pi}^{\textsc{s}} \cdot \Astar \ \cup \ \Astar \cdot \mathcal{M}_2 \cdot \Astar \cdot E_{\pi}^{\textsc{qs}} \cdot \Astar \text{.}
$$
Since $(\Astar \cdot \M_2 \cdot \Astar \cdot E_{\pi}^{\textsc{qs}} \cdot \Astar) \cap \mathcal{M} = (\Astar \cdot \mathcal{M}_2 \cdot  E_{\pi}^{\textsc{qs}} \cdot \Astar) \cap  \mathcal{M}$ and $E_{\pi}= E_{\pi}^{\textsc{s}} \ \cup \ \mathcal{M}_2 \cdot E_{\pi}^{\textsc{qs}}$, we obtain
$$\bigcup\limits_{u \in
  P(\pi) \atop u \text{ strict or quasi-strict}}\text{\hspace{-0.8cm}}
\overleftarrow{\mathcal{L}(u)} \cap {\mathcal M} = \Astar \cdot \overleftarrow{E_{\pi}} \cdot \Astar \cap \mathcal{M}\text{,}$$
concluding the proof.
\end{proof}

\begin{rem}\label{rem:simple_linear}
When $\pi$ is a simple permutation, the automaton
$\mathcal{A}_{\pi}^{\textsc{sqs}}$ of Lemma~\ref{lem:sqs_linear_complexity} recognizes a language
$\mathcal{L}'_{\pi}$ such that $\mathcal{L}'_{\pi}\cap {\mathcal M} =
\overleftarrow{\mathcal{L}_{\pi}} \cap {\mathcal M}$.
In the optimized alternative construction of $\mathcal{A}_{\pi}$
mentioned in Subsection~\ref{ssec:idees_construction_automates} and at the beginning of Appendix~\ref{sec:buildingAutomata},
for a simple permutation $\pi$
we take $\mathcal{A}_{\pi} = \mathcal{A}_{\pi}^{\textsc{sqs}}$
from Lemma~\ref{lem:sqs_linear_complexity} and
the time and space complexity of building the automaton $\mathcal{A}_{\pi}$ is \emph{linear}
w.r.t.~$|\pi|$ as soon as $P(\pi)$ is given.
\end{rem}

\subsection{Pin-permutations with a linear root: non-recursive case} \label{ssec:non-rec_linear}

W.l.o.g. (see Remark~\ref{rem:ominus=oplus_transposed} p.\pageref{rem:ominus=oplus_transposed}),
the only non-recursive case with a linear root is the one
where $\pi = \oplus[\xi_1, \ldots, \xi_r]$, every $\xi_i$ being an
increasing oscillation.
Theorem~\ref{thm:pinwords_cas_lineaire_non_recursif}
(p.\pageref{thm:pinwords_cas_lineaire_non_recursif}) gives an
explicit description of the elements of $P(\pi)$.
These words are the
concatenation of a pin word belonging to some
$P(\oplus[\xi_i,\xi_{i+1}])$ with a pin word belonging to
the shuffle product $$ (P^{(1)}(\xi_{i-1}) , \ldots , P^{(1)}(\xi_1) )
\, \shuffle \, (P^{(3)}(\xi_{i+2}) , \ldots , P^{(3)}(\xi_r) ).$$

To shorten the notations in the following, let us define
$$ \PhiP{\quadrantell}{j} = \overleftarrow{\phi(P^{(\quadrantell)}(\xi_{j}))} \text{ for }\quadrantell =1 \text{ or }3 \text{ and } 1\leq j \leq r\text{.}$$
\label{pagedef:PhiP}

From Lemma~\ref{lem:f1f3} (p.\pageref{lem:f1f3}), for all $j$,
the pin words of $P^{(1)}(\xi_j)$ and $P^{(3)}(\xi_j)$
respectively are strict. Hence, the decomposition of $u
\in P(\pi)$ into \snl factors that is needed to describe
$\overleftarrow{{\mathcal L}_{\pi}}$ (see p.\pageref{expression L_pi}) 
is easily obtained and gives: \label{eq_ssec:non-rec_linear}
$$
\overleftarrow{{\mathcal L}_{\pi}} = \bigcup_{1\leq i \leq r-1}
\left(\left(A^{\star}\PhiP{1}{1},\ldots,A^{\star}\PhiP{1}{i-1}\right)
\shuffle
\left(A^{\star}\PhiP{3}{r},\ldots,A^{\star}\PhiP{3}{i+2}\right)\right)
\cdot \overleftarrow{{\mathcal L}_{\oplus[\xi_i,\xi_{i+1}]}}
$$
In the above equation, we have $\overleftarrow{{\mathcal L}_{\oplus[\xi_i,\xi_{i+1}]}}$
where we might have expected to find $A^{\star} \overleftarrow{\phi(P(\oplus[\xi_i,\xi_{i+1}]))} A^{\star}$.
The reason is that
the term $A^{\star} \overleftarrow{\phi(P(\oplus[\xi_i,\xi_{i+1}]))} A^{\star}$ is not
properly defined, since $P(\oplus[\xi_i,\xi_{i+1}])$ contains pin words that are not strict. 

The automaton ${\mathcal A}_{\pi}$ is then built by assembling smaller
automata, whose construction is explained below.

\paragraph*{Construction of ${\mathcal A}(\xi_i,\xi_j)$}
Since for all $i,j$, the languages
$\PhiP{1}{i}$ and
$\PhiP{3}{j}$ -- that contain at most two words -- are defined on disjoint
alphabets (see Remark~\ref{rem:P1P3} p.\pageref{rem:P1P3}), we can use
the construction given in Subsection~\ref{sec:ac} to build the
deterministic automata ${\mathcal
  A}(\xi_i,\xi_j)=\AC{\PhiP{1}{i},
  \PhiP{3}{j}}$.
Figure~\ref{fig:atomicAutomaton} shows a diagram of ${\mathcal
  A}(\xi_i,\xi_j)$ and defines states $s_{ij}$, $f^{(1)}_{ij}$ and
$f^{(3)}_{ij}$.

\begin{lem}\label{lem:complexite_aij}
For all $i,j$, the complexity in time and space of the construction of ${\mathcal A}(\xi_i,\xi_j)$ is $\O(|\xi_i|+|\xi_j|).$
\end{lem}

\begin{figure}[htbp]
\hspace{.02\linewidth}
  \begin{minipage}[t]{.60\linewidth}
    \begin{center}
\begin{tikzpicture}
\draw [drop shadow,rounded corners,thick, fill=black!10] (0,0) -- (2,0) -- (0,2) -- cycle;
\outputState{0.15}{1.65};
\outputState{1.65}{0.15};
\inputState{0.15}{0.15};
\draw  (-.4,-.5) node[ right ,text width=2.33cm] {Initial state $s_{ij}$
};
\draw [->] (-.2,-.2) -- (0,0.1);
\draw [->,shorten >= 7pt] (2.4,0.5) node[ right ,text width=3cm] {Final state $f^{(1)}_{ij}$} -- (1.65,0.15);
\node at (4,0.05) {accepting $\PhiP{1}{i}$};
\draw [->, shorten >= 7pt] (1.2,1.8) node[ right ,text width=3cm] {Final state $f^{(3)}_{ij}$} -- (0.15,1.65);
\node at (2.8,1.35) {accepting $\PhiP{3}{j}$};
\draw [dashed,->] (0.15,0.15) -- node[above=-4pt] {\footnotesize $\{L,D\}$} (1.55,0.15);
\draw [dashed,->] (0.15,0.15) -- node[right=-4pt] {\footnotesize $\{U,R\}$} (0.15,1.55);
\end{tikzpicture}
      \caption{Atomic automaton {${\mathcal A}(\xi_{i},\xi_{j})$} used \newline in the construction of $\mathcal{A}_{\pi}$.}
      \label{fig:atomicAutomaton}
    \end{center}
  \end{minipage}
  \begin{minipage}[t]{.3\linewidth}
    \begin{center}
\begin{tikzpicture}
\draw [drop shadow,rounded corners,thick, fill=black!10] (0,0) -- (2,0) -- (0,2) -- cycle;
\outputState{0.15}{1.65};
\outputState{1.65}{0.15};
\inputState{0.15}{0.15};
\draw (0.7,0.5) node {\footnotesize ${\mathcal A}(\xi_{i},\xi_{i+1})$};
\draw [fill=black!20,opacity=.2, very thick] (1.9,0.1) circle (.45cm);
\draw [fill=black!20,opacity=.2, very thick] (1.9,1.9) circle (.45cm);
\draw [fill=black!20,opacity=.2, very thick] (0.1,1.9) circle (.45cm);
\draw (1.85,0.2) rectangle  +(0.3,1.6);
\draw (2,1) node[right]  {\footnotesize ${\mathcal A}_{\xi_{i+1}}$};
\draw (1,2) node[above]  {\footnotesize ${\mathcal A}_{\xi_{i}}$};
\draw (0.2,1.85) rectangle +(1.6,0.3);
\draw (2,2) node  {\footnotesize $f_i$};
\inputState{0.35}{2}
\inputState{2}{0.35}
\outputState{2}{1.65}
\outputState{1.65}{2}
\draw [->] (2.35,2) .. controls +(0.6,0) and +(0.6,0.6) .. node[right] {$A$} (2.2,2.25);
\draw [->] (-.3,-.3) -- (0,0);
\end{tikzpicture}
      \caption{Automaton $\mathcal{A}^{\oplus}(\xi_i,\xi_{i+1})$.}
      \label{fig:ADoubleNonMixte}
    \end{center}
  \end{minipage}
\hspace{.02\linewidth}
\end{figure}

\paragraph*{Construction of ${\mathcal A}_{\xi_{i}}$}
By Lemma~\ref{lem:lecture_epi} (p.\pageref{lem:lecture_epi}), for any
$i$, $|P(\xi_i)| \leq 48$, the pin words in $P (\xi_{i})$ are explicit
and all of them are either strict or quasi-strict except when  $|\xi_{i}|=3$.

If $|\xi_{i}| \neq 3$, from Lemma~\ref{lem:ascsquadratique} it is possible to
build the deterministic automaton ${\mathcal A}_{\xi_{i}}$ in quadratic
time and space w.r.t.~$|\xi_{i}|$, and from Lemma~\ref{lem:sqs_linear_complexity}
the construction can be optimized to be linear in time and space
w.r.t.~$|\xi_{i}|$. 

If $|\xi_{i}|=3$, $P(\xi_{i})$ can be partitioned into two parts $P_{\textsc{sqs}}
\uplus P(12)\cdot \quadrantell$ where $\quadrantell = 4$ (resp. $2$) if $\xi_{i} = 231$ (resp. $312$) and $P_{\textsc{sqs}}$ is the set
of strict and quasi-strict pin words of $P(\xi_{i})$.  With
Lemma~\ref{lem:ascsquadratique} or~\ref{lem:sqs_linear_complexity} we build the automaton
$\mathcal{A}_{\xi_{i}}^{\textsc{sqs}}$ corresponding to $P_{\textsc{sqs}}$,
and the automaton corresponding to $P(12)\cdot \quadrantell$ is the concatenation of two basic
automata, $\AC{\overleftarrow{\phi(\quadrantell)}}$
(where $\phi(\quadrantell)$ for $\quadrantell = 2$ or $4$ is given p.\pageref{def:phi}) and
${\mathcal A}_{12}$ (see Remark~\ref{rem:ascquadratique} p.\pageref{rem:ascquadratique}).
Finally the automaton ${\mathcal A}_{\xi_{i}}$ is the union
$\U^{\circlearrowleft} (\mathcal{A}_{\xi_{i}}^{\textsc{sqs}}, \AC{\overleftarrow{\phi(\quadrantell)}}\cdot \mathcal{A}_{12})$.
As $|\xi_{i}| = 3$, ${\mathcal A}_{\xi_{i}}$ is built in constant time.

\begin{lem}\label{lem:complexite_a_xi_j}
For all $i$, building ${\mathcal A}_{\xi_{i}}$ is done in time and space $\O(|\xi_i|^2)$
with the classical construction and $\O(|\xi_i|)$ in the optimized version.
\end{lem}

\paragraph*{Construction of $\mathcal{A}^{\oplus}(\xi_i,\xi_{i+1})$}
We now explain how to build a deterministic automaton $\mathcal{A}^{\oplus}(\xi_i,\xi_{i+1})$ recognizing $\overleftarrow{{\mathcal
    L}_{\oplus[\xi_i,\xi_{i+1}]}}$.
Lemma~\ref{lem:fdouble} (p.\pageref{lem:fdouble}) describes the pin words of $P(\oplus[\xi_{i},\xi_{i+1}])$, proving the correctness of the following constructions.

If $|\xi_{i}| > 1$ and $|\xi_{i+1}| > 1$, we obtain
$\mathcal{A}^{\oplus}(\xi_i,\xi_{i+1})$ by gluing together automata
$\mathcal{A}(\xi_i,\xi_{i+1})$, $\mathcal{A}_{\xi_{i}}$ and
$\mathcal{A}_{\xi_{i+1}}$ as shown in
Figure~\ref{fig:ADoubleNonMixte}.  More precisely $f_{i(i+1)}^{(1)}$
(resp. $f_{i(i+1)}^{(3)}$) and the initial state of
$\mathcal{A}_{\xi_{i+1}}$ (resp. $\mathcal{A}_{\xi_{i}}$) are merged
into a unique state that is neither initial nor final. The final states of
$\mathcal{A}_{\xi_{i}}$ and $\mathcal{A}_{\xi_{i+1}}$ are also merged
into a unique final state $f_{i}$ having a loop labeled by all letters of~$A$.

If $|\xi_{i}| = 1$ and $|\xi_{i+1}| = 1$ then $\mathcal{A}^{\oplus}(\xi_i,\xi_{i+1}) = \mathcal{A}_{12}$ is built using Remark~\ref{rem:ascquadratique}.

Otherwise assume w.l.o.g. that $|\xi_{i}| = 1$ and $|\xi_{i+1}| >
1$. The set of pin words $P(\oplus[\xi_{i},\xi_{i+1}])$ can be partitioned
into two parts $P_{\textsc{sqs}}$ and $P'$, $P_{\textsc{sqs}}$ being the set of strict
and quasi-strict pin words.  If $|\xi_{i+1}| \not= 3$, then $P' =
P(\xi_{i+1}) \cdot 3$ and $\mathcal{A}^{\oplus}(\xi_i,\xi_{i+1})=\U^{\circlearrowleft}\left(\mathcal{A}^{\textsc{sqs}}_{\oplus[\xi_i,\xi_{i+1}]},\AC{\overleftarrow{\phi(3)}}\cdot {\mathcal
  A_{\xi_{i+1}}}\right) $. If $|\xi_{i+1}| = 3$, then $P' =
P(\xi_{i+1}) \cdot 3 \bigcup P(12)\cdot 3X$ where $X$ is a
direction, and we use again concatenation and union but performed on automata of constant size.

Note that in all cases ${\mathcal A}^{\oplus} (\xi_{i},\xi_{i+1})$ has a
unique final state -- that we denote $f_i$ -- without outgoing transitions except for the
loop labeled by all letters of $A$.

\begin{lem}\label{lem:complexite_double}
For all $i$, building the automaton $\mathcal{A}^{\oplus}(\xi_i,\xi_{i+1})$ is done in time and space
${\mathcal O}(|\xi_i|^4 + |\xi_{i+1}|^4)$ with the
classical construction and
$\O(|\xi_i|^2+|\xi_{i+1}|^2)$ in the optimized version.
\end{lem}
\begin{proof}
The complexity of the construction of $\mathcal{A}(\xi_i,\xi_{i+1})$
is linear w.r.t.~$|\xi_{i}|+|\xi_{i+1}|$ from
Lemma~\ref{lem:complexite_aij} and the one of $\mathcal{A}_{\xi_{i}}$
is quadratic -- or linear in the optimized version --
w.r.t.~$|\xi_{i}|$ (see Lemma~\ref{lem:complexite_a_xi_j}). 
This concludes the proof except when $|\xi_{i}| =1$ and $|\xi_{i+1}| \neq 1$ (or conversely). 
When $|\xi_{i}| =1$, from Lemma~\ref{lem:fdouble} (p.\pageref{lem:fdouble}),
we have an explicit description of $P(\oplus[\xi_{i},\xi_{i+1}])$ and $|P(\oplus[\xi_{i},\xi_{i+1}])| \leq 192$.
Hence, the complexity of the construction of $\mathcal{A}^{\textsc{sqs}}_{\oplus[\xi_i,\xi_{i+1}]}$ when $|\xi_i|=1$
is quadratic -- or linear in the optimized version -- w.r.t.~$|\xi_{i+1}|$,
using Lemmas~\ref{lem:ascsquadratique} and \ref{lem:sqs_linear_complexity}.
Then the result follows from the complexity of the union of two automata, which is proportional to the product
of the sizes of the automata.
\end{proof}

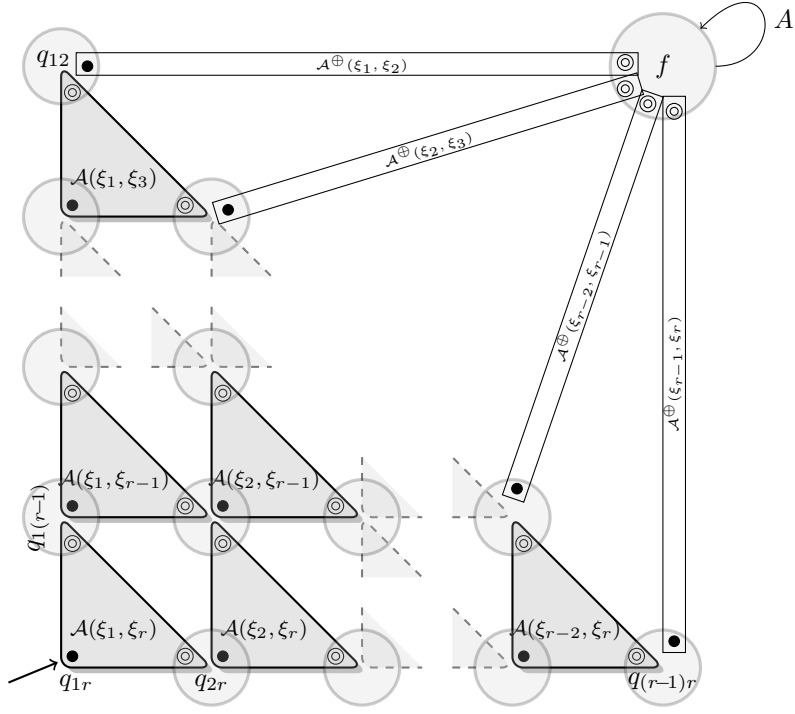
\begin{figure}[ht]
\begin{center}
\begin{tikzpicture}
\patateTriangle{0}{0}{1}{r};
\patateTriangle{2}{0}{2}{r};
\patateTriangle{6}{0}{r-2}{r};
\patateTriangle{0}{2}{1}{r-1};
\patateTriangle{2}{2}{2}{r-1};
\patateTriangle{0}{6}{1}{3};

\draw [opacity=.5,rounded corners,thick, fill=black!10,dashed] (0,5.2) -- (0,6) -- (0.8,5.2);
\draw [opacity=.5,rounded corners,thick, fill=black!10,dashed] (0,4.8) -- (0,4) -- (0.8,4);
\draw [opacity=.5,rounded corners,thick, fill=black!10,dashed] (2,4.8) -- (2,4) -- (2.8,4);
\draw [opacity=.5,rounded corners,thick, fill=black!10,dashed] (4,0.8) -- (4,0) -- (4.8,0);
\draw [opacity=.5,rounded corners,thick, fill=black!10,dashed] (4,2.8) -- (4,2) -- (4.8,2);
\draw [opacity=.5,rounded corners,thick, fill=black!10,dashed] (2,5.2) -- (2,6) -- (2.8,5.2);
\draw [opacity=.5,rounded corners,thick, fill=black!10,dashed] (5.2,2) -- (6,2) -- (5.2,2.8);
\draw [opacity=.5,rounded corners,thick, fill=black!10,dashed] (5.2,0) -- (6,0) -- (5.2,0.8);
\draw [opacity=.5,rounded corners,thick, fill=black!10,dashed] (1.2,4) -- (2,4) -- (1.2,4.8);
\draw [opacity=.5,rounded corners,thick, fill=black!10,dashed] (4,1.2) -- (4,2) -- (4.8,1.2);
\draw [fill=black!20,opacity=.2, very thick] (8,8) circle (.7cm);
\draw (8,8) node {$f$};
\draw [fill=black!20,opacity=.2, very thick] (2,2) circle (.5cm);
\draw [fill=black!20,opacity=.2, very thick] (0,2) circle (.5cm);
\draw [fill=black!20,opacity=.2, very thick] (0,4) circle (.5cm);
\draw [fill=black!20,opacity=.2, very thick] (0,6) circle (.5cm);
\draw [fill=black!20,opacity=.2, very thick] (0,8) circle (.5cm);
\draw [fill=black!20,opacity=.2, very thick] (2,0) circle (.5cm);
\draw [fill=black!20,opacity=.2, very thick] (2,4) circle (.5cm);
\draw [fill=black!20,opacity=.2, very thick] (2,6) circle (.5cm);
\draw [fill=black!20,opacity=.2, very thick] (4,0) circle (.5cm);
\draw [fill=black!20,opacity=.2, very thick] (4,2) circle (.5cm);
\draw [fill=black!20,opacity=.2, very thick] (6,0) circle (.5cm);
\draw [fill=black!20,opacity=.2, very thick] (6,2) circle (.5cm);
\draw [fill=black!20,opacity=.2, very thick] (8,0) circle (.5cm);
\draw (0.2,7.88) rectangle (7.67,8.18);
\draw [xshift=2.1cm,yshift=5.9cm,rotate=17] (0,0) rectangle node[rotate=17]{\tiny ${\mathcal A^{\oplus}}(\xi_{2},\xi_{3})$} +(5.9,0.3);
\inputState{2.22}{6.09};
\draw [xshift=6.15cm,yshift=2.2cm,rotate=71] (0,0) rectangle node[rotate=71]{\tiny ${\mathcal A^{\oplus}}(\xi_{r-2},\xi_{r-1})$} +(5.7,0.3);
\inputState{6.06}{2.38};
\draw [xshift=8.3cm,yshift=0.2cm,rotate=90] (0,0) rectangle node[rotate=90]{\tiny ${\mathcal A^{\oplus}}(\xi_{r-1},\xi_{r})$} +(7.4,0.3);
\inputState{8.15}{0.35};
\outputState{8.15}{7.4};
\outputState{7.8}{7.5};
\outputState{7.5}{7.7};
\outputState{7.5}{8.05};
\draw (4,8.03) node {\tiny ${\mathcal A^{\oplus}}(\xi_{1},\xi_{2})$};
\draw [->] (8.7,8) .. controls +(1,0) and +(1,1) .. node[right] {$A$} (8.5,8.5);
\draw (-0.3,2) node [rotate=90]  {$q_{1(r\!-\!1)}$};
\draw (2,-0.2) node {$q_{2r}$};
\draw (0.2,-0.2) node {$q_{1r}$};
\draw (-0.1,8.1) node {$q_{ 12}$};
\inputState{0.35}{8};
\draw (8,-0.2) node {$q_{(r\!-\!1)r}$};
\draw [thick,->,shorten >= 6pt] (-0.7,-0.2) -- (.15,.15);
\end{tikzpicture}

\caption{Automaton $\mathcal{A}_{\pi}$ for $\pi=\oplus[\xi_1, \ldots, \xi_r]$  where
  every $\xi_i$ is an increasing oscillation.}
\label{fig:automate_cas_oplus_non_recursif}
\end{center}
\end{figure}

\paragraph*{Assembling $\mathcal{A}_{\pi}$}
According to  the description of $\overleftarrow{{\mathcal L}_\pi}$ given
p.\pageref{eq_ssec:non-rec_linear}, the automata ${\mathcal
  A}(\xi_{i},\xi_{j})$ and ${\mathcal A}^{\oplus}
(\xi_{i},\xi_{i+1})$ can be glued together to finish the construction
of $\mathcal{A}_{\pi}$, as shown in
Figure~\ref{fig:automate_cas_oplus_non_recursif}. More precisely for any
$i,j$ such that $1\leq i<j\leq r$
\begin{itemize}
\item if $i+1 \neq j$, then $s_{ij}$, $f^{(1)}_{(i-1)j}$ and
  $f^{(3)}_{i(j+1)}$ are merged into a unique state $q_{ij}$ that is neither
  initial (except when $i=1$ and $j=r$) nor final,
\item if $i+1=j$, $f^{(1)}_{(i-1)j}$, $f^{(3)}_{i(j+1)}$ and the
  initial state of $\mathcal{A}^{\oplus}(\xi_i,\xi_{j}) = \mathcal{A}^{\oplus}(\xi_i,\xi_{i+1})$ are
  merged into a unique state $q_{ij}$ that is neither initial nor final,
\end{itemize}
and the final states $f_i$ of the automata $\mathcal{A}^{\oplus}(\xi_i,\xi_{i+1})$ are merged into a unique final state $f$
having a loop labeled by all letters of $A$. 
The states $q_{ij}$ defined above correspond to the
shuffle product construction.
To be more precise, taking the final state to be the merged state $q_{ij}$
and adding a loop labeled by all letters of $A$ on it,
the accepted language would be 
$\left(A^{\star}\PhiP{1}{1},\ldots,A^{\star}\PhiP{1}{i-1}\right)
\shuffle
\left(A^{\star}\PhiP{3}{r},\ldots,A^{\star}\PhiP{3}{j+1}\right) \linebreak \cdot \Astar$
as it will be proved in the following.
The automata $\mathcal{A}^{\oplus}(\xi_i,\xi_{i+1})$ in the second item above correspond to the concatenation
with $\overleftarrow{{\mathcal L}_{\oplus[\xi_i,\xi_{i+1}]}}$.

Note that if $r=2$, $\mathcal{A}_{\pi}$ is $\mathcal{A}^{\oplus}(\xi_1,\xi_{2})$.

\begin{lem}\label{lem:complexite_linear}
For any permutation $\pi$ such that  $\pi = \oplus[\xi_1, \ldots, \xi_r]$, every $\xi_i$ being an
increasing oscillation, the construction of the automaton $\mathcal{A}_{\pi}$
 is done in time and space
${\mathcal O}(|\pi|^4)$ with the
classical construction and
$\O(|\pi|^2)$ in the optimized version.
\end{lem}
\begin{proof}
Denote by $n$ the size of $\pi$, then $n=\sum_{i=1}^{r}|\xi_i|$.
By construction, taking into account the merge of states:
\begin{eqnarray*}
|\mathcal{A}_{\pi}|&\leq&\sum_{i=1}^{r-2 }\sum_{j=i+2}^r |{\mathcal
  A}(\xi_{i},\xi_{j})|+\sum_{i=1}^{r-1}|{\mathcal A}^{\oplus}(\xi_{i},\xi_{i+1})|
\end{eqnarray*}
and from Lemmas~\ref{lem:complexite_aij} and \ref{lem:complexite_double}, in the classical construction
\begin{eqnarray*}
|\mathcal{A}_{\pi}|&=&\O\left(\sum_{i=1}^{r-2 }\sum_{j=i+2}^r (|\xi_{i}| + |\xi_{j}|)+\sum_{i=1}^{r-1}(|\xi_{i}|^4+|\xi_{i+1}|^4)\right)=\O(n^4)
\end{eqnarray*}
and in the optimized version
\begin{eqnarray*}
|\mathcal{A}_{\pi}|&=&\O\left(\sum_{i=1}^{r-2 }\sum_{j=i+2}^r (|\xi_{i}| + |\xi_{j}|)+\sum_{i=1}^{r-1}(|\xi_{i}|^2+|\xi_{i+1}|^2)\right)=\O(n^2).
\end{eqnarray*}
We conclude the proof noticing that all these automata are built in linear time w.r.t.~their size.
\end{proof}

\paragraph*{Correctness of the construction}
We now prove that the automaton $\mathcal{A}_{\pi}$ given in
Figure~\ref{fig:automate_cas_oplus_non_recursif}  recognizes the
language $\overleftarrow{{\mathcal L}_{\pi}}$, for $\pi=\oplus[\xi_1, \ldots, \xi_r]$,
every $\xi_i$ being an increasing oscillation.

\begin{defi}\label{def:trace}
Let ${\mathcal A}$ be a deterministic complete automaton over the
alphabet $A$ whose set of states is $Q$ and let $u$ be a word in
$A^{\star}$.  We define $trace_{\mathcal A}(u)$ as the word of $Q^{|u|+1}$ that
consists in the states of the automaton that are visited when reading $u$ from the
initial state of ${\mathcal A}$, and for all $q \in Q$ we define
$q\cdot u$ to be the state of ${\mathcal A}$ reached when reading $u$ from $q$.
\end{defi}

Let $u$ be a word in $A^{\star}$. We define two parameters on $u$:
$$i(u) = \max\{ i \in \{ 0, r\} \mid u \in \Astar \cdot \PhiP{1}{1} \cdot \Astar \cdot \PhiP{1}{2} \ldots \Astar \cdot \PhiP{1}{i} \cdot \Astar  \}\text{ , and}$$
$$j(u) = \min\{ j \in \{ 1, r+1\} \mid u \in \Astar \cdot  \PhiP{3}{r}\cdot  \Astar \cdot \PhiP{3}{r-1} \ldots \Astar \cdot \PhiP{3}{j} \cdot \Astar  \}.$$

\begin{rem}\label{rem:which_shuffle}
Every word $u$ such that $i(u) \geq i$ and $j(u) \leq j$ belongs to 
$$\Big((A^{\star}\PhiP{1}{1},A^{\star}\PhiP{1}{2},\ldots,A^{\star}\PhiP{1}{i})
\shuffle
(A^{\star}\PhiP{3}{r},A^{\star}\PhiP{3}{r-1},\ldots,A^{\star}\PhiP{3}{j})\Big)\cdot A^{\star}.$$
\end{rem}

\begin{proof}
By definition, $u$ belongs to
$\Astar \cdot \PhiP{1}{1} \cdot \Astar \cdot \PhiP{1}{2} \ldots \Astar \cdot \PhiP{1}{i} \cdot \Astar$
and to $\Astar \cdot  \PhiP{3}{r}\cdot  \Astar \cdot \PhiP{3}{r-1} \ldots \Astar \cdot \PhiP{3}{j} \cdot \Astar$.
Moreover, by Remark~\ref{rem:P1P3} (p.\pageref{rem:P1P3}), all $\PhiP{1}{k}$ are words on the alphabet $\{L,D\}$ while all $\PhiP{3}{k}$ are words on $\{U,R\}$.
These alphabets being disjoint, we conclude that $u$ belongs to
$\Big((A^{\star}\PhiP{1}{1},A^{\star}\PhiP{1}{2},\ldots,A^{\star}\PhiP{1}{i})
\shuffle
(A^{\star}\PhiP{3}{r},A^{\star}\PhiP{3}{r-1},\ldots,A^{\star}\PhiP{3}{j})\Big)\cdot A^{\star}$.
\end{proof}

The following lemma characterizes the first visit of state $ q_{ij}$ in $\mathcal{A}_{\pi}$:
\begin{lem}\label{lem:etatDeLecture}
Let $Q$ be the set of states of $\mathcal{A}_{\pi}$ (see Figure~\ref{fig:automate_cas_oplus_non_recursif}),
$(i,j) \neq (1,r)$ and $u \in A^{\star}$. Then
$trace_{\mathcal{A}_{\pi}}(u) \in (Q \setminus q_{ij})^{\star} \cdot q_{ij}$ if and only if 
$u = vw$ with
either $w \in \PhiP{1}{i-1}$, $i(v) = i-2$ and $j(v) = j+1$;
or $w \in \PhiP{3}{j+1}$, $i(v) = i-1$ and $j(v) = j+2$.
\end{lem}

\begin{proof}
By induction on $r-j+i-1$, using ${\mathcal A}(\xi_i,\xi_j)=\AC{\PhiP{1}{i}, \PhiP{3}{j}}$.
Notice that $r-j+i-1$ is the number of automata ${\mathcal A}(\xi_k,\xi_\ell)$ that we need to go through before reaching $q_{ij}$.
\end{proof}

\begin{theo}\label{thm:5-9}
If $\pi = \oplus[\xi_1, \ldots, \xi_r]$ where every $\xi_i$ is an
increasing oscillation then automaton $\mathcal{A}_{\pi}$ given in
Figure~\ref{fig:automate_cas_oplus_non_recursif} recognizes the
language $\overleftarrow{{\mathcal L}_{\pi}}$.
\end{theo}

\begin{proof}
Assume that $r \geq 3$ (otherwise $r=2$, $\mathcal{A}_{\pi}$ is
$\mathcal{A}^{\oplus}(\xi_1,\xi_{2})$ and the result trivially holds).
We first prove that every word recognized by $\mathcal{A}_{\pi}$ is in
$\overleftarrow{{\mathcal L}_{\pi}}$. Let $u$ be a word recognized by
$\mathcal{A}_{\pi}$. Then $trace_{{\mathcal A}_{\pi}}(u)$ ends with
the final state $f$ of $\mathcal{A}_{\pi}$. As $f$ is accessible only
from some ${\mathcal A}^{\oplus}(\xi_{k},\xi_{k+1})$,
$trace_{{\mathcal A}_{\pi}}(u)$ contains some $q_{k(k+1)}$. Moreover
for all $(i,j) \neq (1,r)$, every path from the initial state $q_{1r}$
to $q_{ij}$ contains $q_{(i-1)j}$ or $q_{i(j+1)}$. Therefore
$trace_{{\mathcal A}_{\pi}}(u) \in q_{i_1j_1}Q^\star q_{i_2j_2}Q^\star
\dots q_{i_{r-1}j_{r-1}}Q^\star f$ where $(i_1,j_1) = (1,r)$,
$(i_{r-1},j_{r-1}) = (k,k+1)$ and for all $\ell$,
$(i_{\ell+1},j_{\ell+1}) \in \{(i_\ell+1,j_\ell),
(i_\ell,j_\ell-1)\}$. Hence by definition of ${\mathcal A}(\xi_i,\xi_j)$
and ${\mathcal A}^{\oplus}(\xi_{k},\xi_{k+1})$ and by the
expression of $\overleftarrow{{\mathcal L}_{\pi}}$ given
p.\pageref{eq_ssec:non-rec_linear}, $u \in \overleftarrow{{\mathcal
    L}_{\pi}}$.

Conversely, let $u \in \overleftarrow{{\mathcal L}_{\pi}}$.
We want to prove that $q_{1r}\cdot u = f$, $q_{1r}$ being the initial state of
$\mathcal{A}_{\pi}$ and $f$ its final state.
The expression of $\overleftarrow{{\mathcal L}_{\pi}}$ given p.\pageref{eq_ssec:non-rec_linear}
ensures that there exists $k$ such that $u = vw$ with $i(v) \geq k-1$,
$j(v) \leq k+2$ and $w \in \overleftarrow{{\mathcal L}_{\oplus[\xi_{k},\xi_{k+1}]}}$.
Let $u = v'w'$ with $v'= v_1 \dots v_s$ the shortest prefix of $u$ such that $j(v')-i(v') \leq 3$, and set $i = i(v')$.
Since $v'$ is of minimal length, $j(v') = i+3$ and there exists $v'' \in \Astar$ such that $v = v'v''$.
So $w' = v''w$ belongs to $\Astar \cdot \overleftarrow{{\mathcal L}_{\oplus[\xi_{k},\xi_{k+1}]}}=\overleftarrow{{\mathcal L}_{\oplus[\xi_{k},\xi_{k+1}]}}$.
Thus, using also Remark~\ref{rem:which_shuffle}, we have:

\hspace{-0.75cm}
\begin{tikzpicture}
\node at (-3,-0.8) {$u$};
\node at (-2.5,-0.8) {$=$};
\node at (-2.5,-2) {$=$};
\draw[|-|] (-2,-0.8) -- node[above] {$v$}  node[below] {\footnotesize $\in
(A^{\star}\PhiP{1}{1},\ldots,A^{\star}\PhiP{1}{k-1}) \shuffle (A^{\star}\PhiP{3}{r},\ldots,A^{\star}\PhiP{3}{k+2})
$} (7,-0.8);
\draw[|-|] (7,-0.8) -- node[above] {$w$}  node[below] {\footnotesize $\in \overleftarrow{{\mathcal L}_{\oplus[\xi_k,\xi_{k+1}]}}$} (9,-0.8);
\draw[|-|] (-2,-2) -- node[above] {$v'$}  node[below] {\footnotesize $ \in
(A^{\star}\PhiP{1}{1},\ldots,A^{\star}\PhiP{1}{i}) \shuffle (A^{\star}\PhiP{3}{r},\ldots,A^{\star}\PhiP{3}{i+3})
$} (5.3,-2);
\draw[|-|] (5.3,-2) -- node[above] {$w'$} (9,-2);
\end{tikzpicture}

Since $v'$ is of minimal length, $i(v_1 \dots v_{s-1}) < i(v')$ or $j(v_1 \dots v_{s-1}) > j(v')$.
Thus $v' = \bar{v}\bar{w}$ with
either $\bar{w} \in \PhiP{1}{i(v')}$, $i(\bar{v}) = i(v')-1$ and $j(\bar{v}) = j(v')$;
or $\bar{w} \in \PhiP{3}{j(v')}$, $i(\bar{v}) = i(v')$ and $j(\bar{v}) = j(v')+1$.
By Lemma~\ref{lem:etatDeLecture}, $trace_{{\mathcal A}_\pi}(v')$ ends with $q_{i(v')+1,j(v')-1}$.

Therefore $u = v' w'$ with $q_{1r}\cdot v' = q_{i+1,i+2}$ and $w' \in \overleftarrow{{\mathcal L}_{\oplus[\xi_{k},\xi_{k+1}]}}$.

If $i=k-1$ then $q_{1r}\cdot u = (q_{1r}\cdot v')\cdot w' =
  q_{k,k+1}\cdot w' = f$ as $w'$ belongs to the language
  $\overleftarrow{{\mathcal L}_{\oplus[\xi_{k},\xi_{k+1}]}}$
  recognized by the automaton ${\mathcal
    A}^{\oplus}(\xi_{k},\xi_{k+1})$.

If $i \leq k-3$. By definition $i(u) \geq k-1$ and $i(v')=i$, and as
  $u=v'w'$, $w' \in \Astar \cdot \PhiP{1}{i+1} \cdot \Astar \cdot \PhiP{1}{i+2} \cdot
  \Astar \ldots \Astar \cdot \PhiP{1}{k-1} \cdot \Astar$. Therefore as $i \leq k-3$,
  $w' \in \Astar \cdot \PhiP{1}{i+1} \cdot \Astar \cdot \PhiP{1}{i+2} \cdot \Astar$ and $w'$
  belongs to the language $\overleftarrow{{\mathcal
      L}_{\oplus[\xi_{i+1},\xi_{i+2}]}}$ recognized by the automaton
  ${\mathcal A}^{\oplus}(\xi_{i+1},\xi_{i+2})$. Finally as $u = v'w'$ and
  $trace_{{\mathcal A}_\pi}(v')$ ends with $q_{i+1,i+2}$,
  $trace_{{\mathcal A}_\pi}(u)$ ends with $f$.

If $i \geq k+1$ then $j(v') \geq k+4$ and
by symmetry of $i(u)$ and $j(u)$ the proof is similar to the previous case.

If $i = k-2$, as $v=v'v''$ and $i(v) \geq k-1$ and $i(v')=i$ then
$v'' \in \Astar \cdot \PhiP{1}{i+1} \cdot \Astar$. Moreover $w \in
\overleftarrow{{\mathcal L}_{\oplus[\xi_{i+2},\xi_{i+3}]}} $ so $w' =
v''w \in \Astar \cdot \PhiP{1}{i+1} \cdot \overleftarrow{{\mathcal
    L}_{\oplus[\xi_{i+2},\xi_{i+3}]}}$. Therefore $w' \in
\overleftarrow{{\mathcal L}_{\oplus[\xi_{i+1},\xi_{i+2},\xi_{i+3}]}}$
and by Theorem~\ref{thm:lienLangagesMotif}
(p.\pageref{thm:lienLangagesMotif}), $w' \in \overleftarrow{{\mathcal
    L}_{\oplus[\xi_{i+1},\xi_{i+2}}}]$. Hence $w'$ is recognized by
${\mathcal A}^{\oplus}(\xi_{i+1},\xi_{i+2})$ so $q_{1r}\cdot u = f$
(since $q_{1r}\cdot v' = q_{i+1,i+2}$).

If $i = k$, by symmetry of $i(u)$ and $j(u)$ the proof is similar to the previous case, 
concluding  the proof of Theorem~\ref{thm:5-9}.
\end{proof}

\begin{rem}\label{rem:versionopt1}
With the optimized construction of ${\mathcal A}_{\pi}$, we prove
similarly that ${\mathcal A}_{\pi}$ recognizes a language $\mathcal{L}'_{\pi}$ 
such that   $\mathcal{L}'_{\pi}\cap {\mathcal M} =
\overleftarrow{\mathcal{L}_{\pi}} \cap {\mathcal M}$.
\end{rem}

We end this subsection with a remark which will be  useful in Subsection~\ref{subsection:marquage}.
\begin{rem}\label{rem:pattern}
Let $\pi^{(1)} = \oplus[\xi_2, \ldots, \xi_r]$ be the pattern of $\pi$
obtained by deletion of the elements of $\xi_1$.
If $r \geq 3$ then ${\mathcal A}_{\pi^{(1)}}$ is obtained by taking  $q_{2r}$ (see
Figure~\ref{fig:automate_cas_oplus_non_recursif}) as initial state and by considering only
the states of ${\mathcal A}_{\pi}$ that are accessible from $q_{2r}$.
Thus in ${\mathcal A}_{\pi}$ the language recognized from $q_{2r}$ is $\overleftarrow{{\mathcal L}_{\pi^{(1)}}}$.
If $r=2$ then $\pi^{(1)} = \xi_2$, ${\mathcal A}_{\pi^{(1)}}$ is also a
part of $\mathcal{A}^{\oplus}(\xi_1,\xi_{2}) = {\mathcal A}_{\pi}$ and
$\overleftarrow{{\mathcal L}_{\pi^{(1)}}}$ is the language recognized from
the bottom right  state of Figure~\ref{fig:ADoubleNonMixte} (p.\pageref{fig:ADoubleNonMixte}).
The same property holds with the pattern $\pi^{(r)} = \oplus[\xi_1, \ldots, \xi_{r-1}]$,
the state $q_{1(r-1)}$ and the top left state of Figure~\ref{fig:ADoubleNonMixte}.
\end{rem}

\subsection{Pin-permutations with a linear root: recursive case} \label{ssec:rec_linear}

Suppose w.l.o.g. that the decomposition tree of $\pi$ is
 
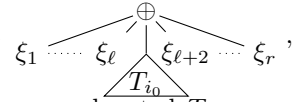
\begin{tikzpicture}[sibling distance=15pt,level
distance=16pt,baseline=-12pt]
\node[linear] {$\oplus$}
	child {node (T1) {$\xi_1$}}
	child[missing]
	child {node (Ta) {$\xi_{\ell}$}}
	child[child anchor=north] {node[draw,shape=isosceles triangle,
shape border rotate=90,anchor=north,inner sep=0,isosceles triangle apex angle=90] {$T_{i_0}$}}
	child {node (Ta1) {$\xi_{\ell+2}$}}
	child[missing]
	child {node (Tk) {$\xi_r$}};
\draw[dotted]  (T1) -- (Ta);
\draw[dotted]  (Ta1) -- (Tk);
\end{tikzpicture}, {\it i.e.}, the root has label $\oplus$ and
all of its children but one -- denoted $T_{i_0}$ -- are increasing oscillations.
Then the  automaton $\mathcal{A}(T_{i_0})=\mathcal{A}_\rho$ associated to the
permutation $\rho$ whose decomposition tree is $T_{i_0}$ is
recursively obtained.
We explain how to build $\mathcal{A}_\pi$ from $\mathcal{A}_\rho$.

\paragraph*{If $\pi \notin \setH$, \emph{i.e.} if $\pi$ does not satisfy any condition of
  Figure~\ref{fig:H} (p.\pageref{fig:H})} Then Theorem~\ref{thm:linearRoot}
(p.\pageref{thm:linearRoot}) ensures that $P(\pi) = P(\rho)\cdot
\F\ell \shuffle \G{\ell+2}$ with
$$ \F \ell = \big( P^{(1)}(\xi_{\ell}), \ldots ,
P^{(1)}(\xi_1) \big) \text{ and }\G{\ell+2} = \big(P^{(3)}(\xi_{\ell+2}),
 \ldots , P^{(3)}(\xi_r) \big).$$
This characterization translates into the following expression for $\overleftarrow{{\mathcal L}_{\pi}} $:
\begin{equation*}
\overleftarrow{{\mathcal L}_{\pi}} =
\left(\left(A^{\star}\PhiP{1}{1},\ldots,A^{\star}\PhiP{1}{\ell}\right)
\shuffle
\left(A^{\star}\PhiP{3}{r},\ldots,A^{\star}\PhiP{3}{\ell+2}\right)\right)
\cdot \overleftarrow{{\mathcal L}_{\rho}}
\end{equation*}
with the notations $\PhiP{h}{j}$ from p.\pageref{pagedef:PhiP}.

To deal with the shuffle product, we use again the automata ${\mathcal
  A}(\xi_i,\xi_j)$ whose initial and final states are
$s_{ij}$, $f^{(1)}_{ij}$ and $f^{(3)}_{ij}$ (see
Figure~\ref{fig:atomicAutomaton} p.\pageref{fig:atomicAutomaton}). We
furthermore introduce the deterministic automata ${\mathcal
  A}^{(1)}(\xi_i)=\AC{\PhiP{1}{i}}$ for $1\leq i \leq \ell$ and
${\mathcal A}^{(3)}(\xi_j)=\AC{\PhiP{3}{j}}$ for $\ell+2\leq j \leq r$
whose initial and final states are denoted respectively $s_i^{(1)}$,
$f_i^{(1)}$, $s_j^{(3)}$ and $f_j^{(3)}$.  The automaton ${\mathcal
  A}^{(1)}(\xi_i)$ (resp. ${\mathcal A}^{(3)}(\xi_j)$) corresponds to
the reading of parts of $\F{\ell}$ \Big(resp. $\G{\ell+2}$\Big)
in the shuffle product $\F{\ell} \shuffle \G{\ell+2}$
after all the parts of $\G{\ell+2}$ \Big(resp. $\F{\ell}$\Big) are read.

With these notations, the language $\overleftarrow{{\mathcal L}_{\pi}}$ associated to
$\pi$ is the one recognized by the automaton $\mathcal{A}_{\pi}$
given in Figure~\ref{fig:automate_cas_oplus_recursif} where the
following merges are performed:
\begin{itemize}
\item for any $i,j$ such that $1\leq i \leq \ell$ and $\ell +2 \leq
  j\leq r$, $s_{ij}$, $f^{(1)}_{(i-1)j}$ and $f^{(3)}_{i(j+1)}$
  are merged into a unique state $q_{ij}$ that is neither initial (except when
  $i=1$ and $j=r$) nor final,
\item for $1\leq i \leq \ell$, $s_i^{(1)}$, $f^{(1)}_{i-1}$ and
  $f^{(3)}_{i(\ell+2)}$ are merged into a unique state $q_{i}$ that is neither initial
  nor final,
\item for $\ell+2 \leq j \leq r$, $s_j^{(3)}$, $f^{(3)}_{j+1}$ and
  $f^{(1)}_{\ell j}$ are merged into a unique state $q_{j}$ that is neither initial
  nor final,
\item $f^{(3)}_{\ell+2}$, $f^{(1)}_{\ell}$ and the initial state of
  ${\mathcal A}_{\rho}$ are merged into a unique state $q_{\rho}$ that is neither initial
  nor final.
\end{itemize}

\begin{figure}[ht]
\begin{center}
\begin{tikzpicture}
\patateTriangle{0}{0}{1}{r};
\patateTriangle{2}{0}{2}{r};
\patateTriangle{6}{0}{\ell}{r};
\patateTriangle{0}{2}{1}{r-1};
\patateTriangle{2}{2}{2}{r-1};
\patateTriangle{6}{2}{\ell}{r-1};
\patateTriangle{0}{6}{1}{\ell+2};
\patateTriangle{2}{6}{2}{\ell+2};
\patateTriangle{6}{6}{\ell}{\ell+2};
\begin{scope}
\draw (0.2,7.85) rectangle node[above=3pt] {\footnotesize ${\mathcal A}^{(1)}(\xi_{1})$} +(1.6,0.3);
\inputState{0.35}{8};
\outputState{1.65}{8};
\end{scope}
\begin{scope}[shift={(2,0)}]
\draw (0.2,7.85) rectangle node[above=3pt] {\footnotesize ${\mathcal A}^{(1)}(\xi_{2})$} +(1.6,0.3);
\inputState{0.35}{8};
\outputState{1.65}{8};
\end{scope}
\draw [opacity=.5,dashed] (4.6,7.85) -- (4.2,7.85) -- (4.2,8.15) -- (4.6,8.15);
\draw [opacity=.5,dashed] (5.4,7.85) -- (5.8,7.85) -- (5.8,8.15) -- (5.4,8.15);
\begin{scope}[shift={(6,0)}]
\draw (0.2,7.85) rectangle node[above=3pt] {\footnotesize ${\mathcal A}^{(1)}(\xi_{\ell})$} +(1.6,0.3);
\draw [pattern=north east lines] (0.2,7.85) rectangle +(0.3,0.3);
\inputState{0.35}{8};
\outputState{1.65}{8};
\end{scope}
\draw (9,8) ellipse  (1cm and .25cm); 
\draw (9,8) node{ ${\mathcal A}_{\rho}$};
\inputState{8.3}{8};
\outputState{9.7}{8};
\draw [->] (9.8,8) .. controls +(1,-0.65) and +(1,0.65) .. node[right] {$A$} (9.78,8.08);

\begin{scope}[shift={(0,0)}]
\draw [rounded corners, fill=black!10,opacity=.5,dashed] (0,5.2) -- (0,6) -- (0.8,5.2);
\end{scope}
\begin{scope}[shift={(2,0)}]
\draw [rounded corners, fill=black!10,opacity=.5,dashed] (0,5.2) -- (0,6) -- (0.8,5.2);
\end{scope}
\begin{scope}[shift={(4,0)}]
\draw [rounded corners, fill=black!10,opacity=.5,dashed] (0,5.2) -- (0,6) -- (0.8,5.2);
\end{scope}
\begin{scope}[shift={(6,0)}]
\draw [rounded corners, fill=black!10,opacity=.5,dashed] (0,5.2) -- (0,6) -- (0.8,5.2);
\end{scope}
\begin{scope}[shift={(4,-4)}]
\draw [rounded corners, fill=black!10,opacity=.5,dashed] (0,5.2) -- (0,6) -- (0.8,5.2);
\end{scope}
\begin{scope}[shift={(4,0)}]
\draw [rounded corners, fill=black!10,opacity=.5,dashed] (0,.8) -- (0,0) -- (.8,0);
\end{scope}
\begin{scope}[shift={(4,2)}]
\draw [rounded corners, fill=black!10,opacity=.5,dashed] (0,.8) -- (0,0) -- (.8,0);
\end{scope}
\begin{scope}[shift={(4,6)}]
\draw [rounded corners, fill=black!10,opacity=.5,dashed] (0,.8) -- (0,0) -- (.8,0);
\end{scope}
\begin{scope}[shift={(4,8)}]
\draw [rounded corners, fill=black!10,opacity=.5,dashed] (0,-0.8) -- (0,0) -- (0.8,-0.8);
\end{scope}
\begin{scope}[shift={(0,4)}]
\draw [rounded corners, fill=black!10,opacity=.5,dashed] (0,.8) -- (0,0) -- (.8,0);
\end{scope}
\begin{scope}[shift={(2,4)}]
\draw [rounded corners, fill=black!10,opacity=.5,dashed] (0,.8) -- (0,0) -- (.8,0);
\end{scope}
\begin{scope}[shift={(6,4)}]
\draw [rounded corners, fill=black!10,opacity=.5,dashed] (0,.8) -- (0,0) -- (.8,0);
\end{scope}
\begin{scope}[shift={(6,0)}]
\draw [rounded corners, fill=black!10,opacity=.5,dashed] (-.8,0) -- (0,0) -- (-.8,.8);
\end{scope}
\begin{scope}[shift={(6,2)}]
\draw [rounded corners, fill=black!10,opacity=.5,dashed] (-.8,0) -- (0,0) -- (-.8,.8);
\end{scope}
\begin{scope}[shift={(6,4)}]
\draw [rounded corners, fill=black!10,opacity=.5,dashed] (-.8,0) -- (0,0) -- (-.8,.8);
\end{scope}
\begin{scope}[shift={(6,6)}]
\draw [rounded corners, fill=black!10,opacity=.5,dashed] (-.8,0) -- (0,0) -- (-.8,.8);
\end{scope}
\begin{scope}[shift={(2,4)}]
\draw [rounded corners, fill=black!10,opacity=.5,dashed] (-.8,0) -- (0,0) -- (-.8,.8);
\end{scope}
\begin{scope}[shift={(8,4)}]
\draw [rounded corners, fill=black!10,opacity=.5,dashed] (-.8,0) -- (0,0) -- (-.8,.8);
\end{scope}

\begin{scope}[shift={(8,0)}]
\draw (-0.15,0.2) rectangle +(0.3,1.6);
\inputState{0}{0.35};
\outputState{0}{1.65};
\draw (0,1) node[right=4pt] {\footnotesize ${\mathcal A}^{(3)}({\xi_{r}})$};
\end{scope}
\begin{scope}[shift={(8,2)}]
\draw (-0.15,0.2) rectangle +(0.3,1.6);
\inputState{0}{0.35};
\outputState{0}{1.65};
\draw (0,1) node[right=4pt] {\footnotesize ${\mathcal A}^{(3)}({\xi_{r-1}})$};
\end{scope}
\draw [opacity=.5,dashed] (7.85,4.6) -- (7.85,4.2) -- (8.15,4.2) -- (8.15,4.6);
\draw [opacity=.5,dashed] (7.85,5.4) -- (7.85,5.8) -- (8.15,5.8) -- (8.15,5.4);
\begin{scope}[shift={(8,6)}]
\draw (-0.15,0.2) rectangle +(0.3,1.6);
\inputState{0}{0.35};
\outputState{0}{1.65};
\draw (0,1) node[right=4pt] {\footnotesize ${\mathcal A}^{(3)}({\xi_{\ell+2}})$};
\draw [pattern=north east lines] (-0.15,0.2) rectangle +(0.3,0.3);
\end{scope}

\draw [fill=black!20,opacity=.2, very thick] (0,4) circle (.5cm);
\draw [fill=black!20,opacity=.2, very thick] (2,4) circle (.5cm);
\draw [fill=black!20,opacity=.2, very thick] (6,4) circle (.5cm);
\draw [fill=black!20,opacity=.2, very thick] (6,6) circle (.5cm);
\draw [fill=black!20,opacity=.2, very thick] (2,0) circle (.5cm);
\draw [fill=black!20,opacity=.2, very thick] (0,2) circle (.5cm);
\draw [fill=black!20,opacity=.2, very thick] (2,2) circle (.5cm);
\draw [fill=black!20,opacity=.2, very thick] (6,2) circle (.5cm);
\draw [fill=black!20,opacity=.2, very thick] (8,2) circle (.5cm);
\draw [fill=black!20,opacity=.2, very thick] (0,6) circle (.5cm);
\draw [fill=black!20,opacity=.2, very thick] (2,6) circle (.5cm);
\draw [fill=black!20,opacity=.2, very thick] (8,6) circle (.5cm);
\draw [fill=black!20,opacity=.2, very thick] (0,8) circle (.5cm);
\draw [fill=black!20,opacity=.2, very thick] (2,8) circle (.5cm);
\draw [fill=black!20,opacity=.2, very thick] (4,8) circle (.5cm);
\draw [fill=black!20,opacity=.2, very thick] (6,8) circle (.5cm);
\draw [fill=black!20,opacity=.2, very thick] (8,8) circle (.5cm);
\draw [fill=black!20,opacity=.2, very thick] (4,0) circle (.5cm);
\draw [fill=black!20,opacity=.2, very thick] (4,2) circle (.5cm);
\draw [fill=black!20,opacity=.2, very thick] (4,6) circle (.5cm);
\draw [fill=black!20,opacity=.2, very thick] (8,0) circle (.5cm);
\draw [fill=black!20,opacity=.2, very thick] (8,4) circle (.5cm);
\draw [fill=black!20,opacity=.2, very thick] (6,0) circle (.5cm);
\draw (-0.2,2) node [rotate=90]  {$q_{1(r\!-\!1)}$};
\draw (0.2,-0.2) node {$q_{1r}$};
\draw (2,-0.2) node {$q_{2r}$};
\draw (8.3,2) node {$q_{r\!-\!1}$};
\draw (8.1,-0.2) node {$q_{r}$};
\draw (-0.1,8.1) node {$q_{1}$};
\draw (2,8.2) node {$q_{2}$};
\draw (6,8.2) node {$q_{\ell}$};
\draw (8,8.23) node {$q_{\rho}$};
\draw (8.3,6) node {$q_{\ell\!+\!2}$};
\draw [thick,->,shorten >= 6pt] (-0.7,-0.2) -- (.15,.15);
\end{tikzpicture}
\caption{The automaton $\mathcal{A}_{\pi}$ for $\pi=\oplus[\xi_1,
  \ldots, \xi_{\ell}, \rho, \xi_{\ell+2}, \ldots, \xi_r]$, where every
  $\xi_i$ but $\rho$ is an increasing oscillation (in the case $\pi \notin \setH$).}
\label{fig:automate_cas_oplus_recursif}
\end{center}
\end{figure}
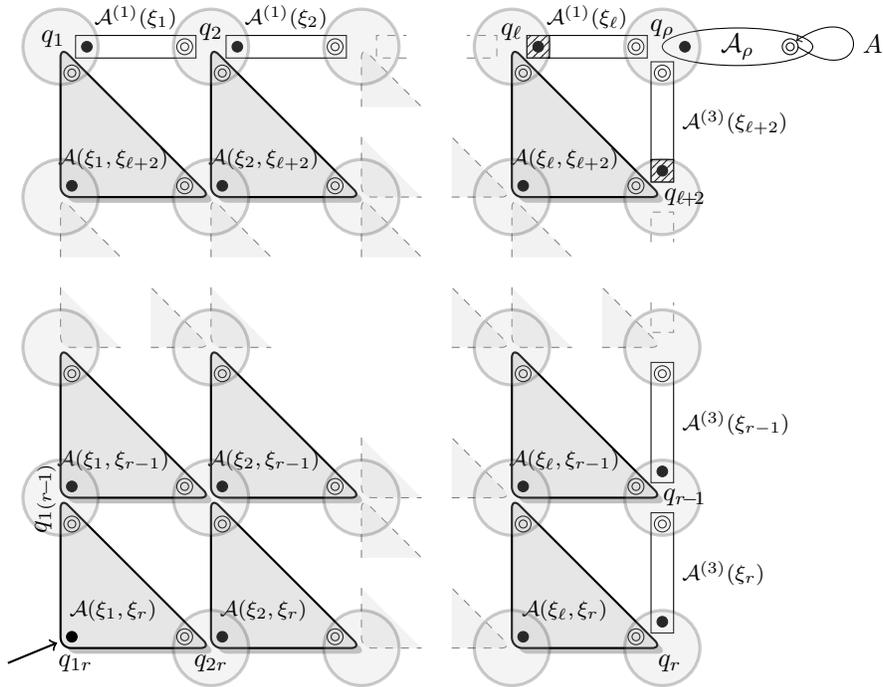

Note that if $\ell = 0$ (resp. $r = \ell +1$) \emph{i.e.}, if $T_{i_0}$ is the first (resp. last) child,
then only the automaton ${\mathcal A}_{\rho}$ and the automata ${\mathcal A}^{(3)}(\xi_j)$
(resp. ${\mathcal A}^{(1)}(\xi_i)$) appear in ${\mathcal A}_{\pi}$ whose initial state is then $q_r$ (resp. $q_1$).

The proof that the automaton $\mathcal{A}_{\pi}$ obtained by this construction recognizes $\overleftarrow{{\mathcal L}_{\pi}}$ is omitted.
However this construction is very similar to the non-recursive case of the previous section where the proofs are detailed.

\begin{lem} \label{lem:complexite_linear_rec}
For any pin-permutation $\pi = \oplus[\xi_1, \ldots,
  \xi_{\ell}, \rho, \xi_{\ell+2},\ldots, \xi_r]$ such that every $\xi_i$ but $\rho$ is
an increasing oscillation and $\pi \notin \setH$, the construction of the automaton
$\mathcal{A}_{\pi}$ (see Figure~\ref{fig:automate_cas_oplus_recursif})
is done in time and space
$\O\left((|\pi|-|\rho|)^2\right)$ plus the additional complexity
due to the construction of $\mathcal{A}_{\rho}$,
both in the classical and the optimized construction.
\end{lem}

\begin{proof}
First notice that $|\pi|-|\rho|=\sum_{i=1, i \neq \ell+1}^r|\xi_i|$. Moreover, taking into account the merge of states:
$$|\mathcal{A}_{\pi}|\leq
\sum_{i=1}^{\ell}\sum_{j=\ell+2}^r|\A{\xi_i,\xi_j}|+\sum_{i=1}^{\ell}
|{\mathcal A}^{(1)}(\xi_{i})| + \sum_{j=\ell+2}^r |{\mathcal
  A}^{(3)}(\xi_{j})|+|\mathcal{A}_{\rho}|.$$ From
Lemma~\ref{lem:complexite_aij} (p.\pageref{lem:complexite_aij}) and
the fact that $|P^{(1)}(\xi_{i})| \leq 2$ and $|P^{(3)}(\xi_{j})| \leq
2$ (see Remark~\ref{rem:P1P3} p.\pageref{rem:P1P3}), it follows that
$$|\mathcal{A}_{\pi}|-|\mathcal{A}_{\rho}|=\O\left( \sum_{i=1}^{\ell}
\sum_{j=\ell+2}^r(|\xi_i|+|\xi_j|)+\sum_{i=1 \atop i \neq
  \ell+1}^r|\xi_i| \right)=\O((|\pi|-|\rho|)^2),$$ concluding the
proof, since the time of the construction is linear w.r.t.~the size of the automaton.
\end{proof}

We end this paragraph with a remark which will be useful in Subsection~\ref{subsection:marquage}.

\begin{rem}\label{rem:pattern_rec}
If $\ell \neq 0$, let $\pi^{(1)}$ be the pattern of $\pi$ obtained by
deletion of the elements of $\xi_1$.  Then ${\mathcal A}_{\pi^{(1)}}$
is obtained by taking $q_{2r}$ (see
Figure~\ref{fig:automate_cas_oplus_recursif}) as initial state and by
considering only the states of ${\mathcal A}_{\pi}$ that are
accessible from $q_{2r}$.  Thus in ${\mathcal A}_{\pi}$ the language
recognized from $q_{2r}$ is $\overleftarrow{{\mathcal
    L}_{\pi^{(1)}}}$.  If $r \neq \ell +1$ the same property holds
with the pattern $\pi^{(r)}$ (obtained by deletion of the elements of
$\xi_r$) and the state $q_{1(r-1)}$.  We take the convention that
$q_{1(\ell+1)}=q_1$, $q_{(\ell+1)r}=q_r$ and $q_{(\ell+1)(\ell+1)}=
q_{\rho}$.
\end{rem}

\paragraph*{If $\pi \in \setH$, \emph{i.e.} if one of the conditions given in Figure~\ref{fig:H}
  (p.\pageref{fig:H}) holds for $\pi$}
Then
Theorem~\ref{thm:linearRoot} (p.\pageref{thm:linearRoot}) ensures that
$P(\pi)$ is the union of the set $P_0 = P(\rho)\cdot \F\ell \shuffle
\G{\ell+2}$ that we consider in the previous paragraph and some other
sets that are very similar, all ending with the same kind of shuffle
product. As the automaton ${\mathcal A}_{\pi}$ recognizes reversed words,
these similar ends lead to similar beginnings in the automaton. So
the automaton ${\mathcal A}_{\pi}$ has the same general structure as
the automaton ${\mathcal A}$ of
Figure~\ref{fig:automate_cas_oplus_recursif}  but
some transitions are added to account for the words in $P(\pi)$ not belonging to $P_0$.

More precisely ${\mathcal A}_{\pi}$ is obtained performing the following modifications
on the automaton ${\mathcal A}$ of Figure~\ref{fig:automate_cas_oplus_recursif}.
First we add new paths as depicted in the last column of
Figure~\ref{fig:H} (p.\pageref{fig:H}). 
These paths start in the shaded states $q_{\ell}$ or
$q_{\ell+2}$ of Figure~\ref{fig:automate_cas_oplus_recursif} and
arrive in marked states of $\mathcal{A}_{\rho}$.
If a path is labeled in Figure~\ref{fig:H} by a word $w$ with $s$ letters
we build $s-1$ new states and $s$ transitions such that the reading of $w$
from the shaded state leads to the corresponding marked state of $\mathcal{A}_{\rho}$.
These marked states may be seen as initial states of subautomata:
in Figure~\ref{fig:H}, for all $Y$, $q_Y$ is a state of $\mathcal{A}_{\rho}$ such that
the language recognized from $q_Y$ is $\overleftarrow{{\mathcal L}_{\sigma}}$,
where $\sigma$ is the permutation whose diagram is $Y$.
The way in which such states of ${\mathcal A}_{\rho}$ are marked is explained in Subsection~\ref{subsection:marquage}.

Moreover to keep the resulting automaton deterministic and complete
when adding these new paths we have to make some other changes.
Notice that state $q_{\ell}$ (resp. $q_{\ell+2}$) is the initial state
of ${\mathcal A}^{(1)}(\xi_{\ell})=\AC{\PhiP{1}{\ell}}$ (resp. ${\mathcal
  A}^{(3)}(\xi_{\ell +2})=\AC{\PhiP{3}{\ell +2}}$).
Remark~\ref{rem:P1P3}
(p.\pageref{rem:P1P3}) ensures that $\PhiP{1}{\ell} = \overleftarrow{\phi(P^{(1)}(\xi_{\ell}))}$ 
(resp. $\PhiP{3}{\ell+2} = \overleftarrow{\phi(P^{(3)}(\xi_{\ell +2}))}$) 
is defined on the alphabet $\{L,D\}$ (resp. $\{U,R\}$).
Therefore, from Lemma~\ref{lem:trans_initial} (p.\pageref{lem:trans_initial}),
transitions labeled by $U$ or $R$
(resp. $L$ or $D$) leaving $q_\ell$ (resp. $q_{\ell+2}$) are loops on
the initial state $q_{\ell}$ (resp. $q_{\ell+2}$) of
${\mathcal A}^{(1)}(\xi_\ell)$ (resp. ${\mathcal A}^{(3)}(\xi_{\ell+2})$).
Hence, as can be seen on see Figure~\ref{fig:H}, the new transitions leaving
shaded states are labeled by directions that correspond to loops in
${\mathcal A}$.  So we just have to delete
the loops and replace them by the new transitions in order to preserve
the determinism of the automaton.

Now we make some other changes to preserve completeness and ensure that even though we have deleted loops
on shaded states, all words that were recognized by the automaton
${\mathcal A}$ are still recognized by the modified automaton
${\mathcal A}_{\pi}$. As $q_\ell$ (resp. $q_{\ell+2}$) is not
reachable from $q_{\ell+2}$ (resp. $q_\ell$) we can handle separately
new states reachable from $q_\ell$ and new states reachable from
$q_{\ell+2}$. Let $q_0$ be $q_\ell$ (resp. $q_{\ell+2}$). Like in
the Aho-Corasick algorithm we label any new state $q$ reachable from $q_0$
by the shortest word labeling a path from $q_0$ to $q$. So these
labels begin with $U$ or $R$ (resp. $L$ or $D$) (see
Figure~\ref{fig:H}). Notice that the states in the part ${\mathcal
  A}^{(1)}(\xi_\ell)$ (resp. ${\mathcal A}^{(3)}(\xi_{\ell+2})$) of
${\mathcal A}$ are also labeled in such a way, but their labels differ
from the ones of the new states since they contain only letters $L$ or
$D$ (resp. $U$ or $R$).
By Lemma~\ref{lem:trans_initial} (p.\pageref{lem:trans_initial}),
we know that in ${\mathcal A}^{(1)}(\xi_\ell)$
(resp. ${\mathcal A}^{(3)}(\xi_{\ell+2})$) all transitions labeled by
$U$ or $R$ (resp. $L$ or $D$) go to $q_0$, therefore we replace them
by transitions going to the new state labeled by $U$ or $R$ (resp. $L$ or $D$)
if such a new state exists (otherwise we keep the transition going to $q_0$).
We complete the construction by adding missing
transitions from the states newly created: for any such state $q$, the
transition from $q$ labeled by $Z$ goes to the longest suffix of $q
\cdot Z$ that is a state of the automaton -- either a new state or
a state of ${\mathcal A}^{(1)}(\xi_\ell)$
(resp. ${\mathcal A}^{(3)}(\xi_{\ell+2})$).

The proof that the automaton $\mathcal{A}_{\pi}$ obtained by this
construction recognizes $\overleftarrow{{\mathcal L}_{\pi}}$ is
omitted to avoid the examination of the eight cases of
Figure~\ref{fig:H}.  However, it is similar to the proof of
Theorem~\ref{lem:A_pi_primitif_rec} (p.\pageref{lem:A_pi_primitif_rec}),
with some of the difficulties released (since labels on the new paths
are explicit in Figure~\ref{fig:H}, while they are not in the proof of Theorem~\ref{lem:A_pi_primitif_rec}).

\begin{lem}\label{lem:complexite_linear_rec_2}
The complexity of building $\mathcal{A}_{\pi}$ given in Lemma~\ref{lem:complexite_linear_rec}
(p.\pageref{lem:complexite_linear_rec}) still holds if $\pi \in \setH$.
\end{lem}

\begin{proof}
When $\pi \in \setH$, the construction of $\mathcal{A}_{\pi}$
is the same as in the case $\pi \notin \setH$, with some new paths added.
There are at most four new paths, $\O(|\rho|)$ new states in each path,
$\O(|\rho|)$ transitions from these new states, and the modification
of transitions in ${\mathcal A}^{(1)}(\xi_\ell)$ (resp. ${\mathcal
  A}^{(3)}(\xi_{\ell+2})$) is done in $\O(|\xi_\ell|)$
(resp. $\O(|\xi_{\ell+2}|)$).
So in the construction of $\mathcal{A}_{\pi}$ described above,
we have to add a time and space complexity
$\O(|\rho|+|\xi_\ell|+|\xi_{\ell+2}|)$ w.r.t. the case $\pi \notin \setH$.
As $|\xi_\ell|+|\xi_{\ell+2}| = \O\left((|\pi|-|\rho|)\right)$
and as the complexity of the 
construction of ${\mathcal A}_{\rho}$ is bigger than $\O(|\rho|)$,
this does not change the overall estimation of the complexity of the
construction of ${\mathcal A}_{\pi}$ given in
Lemma~\ref{lem:complexite_linear_rec}.
\end{proof}

\subsection{Pin-permutations with a prime root: recursive case}
\label{sec:decomp_simple}

\subsubsection{Exactly one child of the root is not a leaf.}

Let  $\pi = $\begin{tikzpicture}[sibling
distance=10pt,level distance=10pt,baseline=-15pt] \node[simple] {$\alpha$} child
{[fill] circle (2 pt) node(x1) {}} child[missing] child {[fill] circle
(2 pt) node(xk) {}} child[child anchor=north]
{node[draw,shape=isosceles triangle, shape border
rotate=90,anchor=north,inner sep=0, isosceles triangle apex angle=90] {$T$}} child {[fill] circle (2 pt) node(y1)
{}} child[missing] child {[fill] circle (2 pt) node(yk) {}};
\draw[dotted] (x1) -- (xk);
\draw[dotted] (y1) -- (yk);
\end{tikzpicture} where $\alpha$ is a simple permutation all of whose children but $T$ are
leaves. Denote by $\rho$ the permutation whose decomposition tree
is $T$, and by $x$ the point of $\alpha$ expanded by $T$.

Recall that $Q_x(\alpha)$ (see Definition~\ref{def:Qxalpha} p.\pageref{def:Qxalpha}) denotes the set of strict pin words obtained
by deleting the first letter of quasi-strict pin words of $\alpha$
whose first point read in $\alpha$ is $x$.

\paragraph*{If $\pi$ does not satisfy condition ($\mathcal C$) (see Definition~\ref{def:c}
p.\pageref{def:c})} \hspace{-0.35cm} Then from
Theorem~\ref{thm:conditionc} (p.\pageref{thm:conditionc}),
$P(\pi)=P(\rho) \cdot Q_x(\alpha)$.  So $\overleftarrow{{\mathcal
    L}_{\pi}} = \Astar \cdot \overleftarrow{\phi(Q_x(\alpha))} \cdot
\overleftarrow{{\mathcal L}_{\rho}}$ and since $\overleftarrow{{\mathcal L}_{\rho}}=\Astar \cdot \overleftarrow{{\mathcal L}_{\rho}}$  the automaton ${\mathcal
  A}_{\pi}$ recognizing $\overleftarrow{{\mathcal L}_{\pi}}$ is
obtained by the concatenation of
$\AC{\overleftarrow{\phi(Q_x(\alpha))}}$ with ${\mathcal A}_{\rho}$,
which is recursively obtained.

\paragraph*{If $\pi$ satisfies condition ($\mathcal C$) and $|T| \geq 3$} \hspace{-0.3cm} Then
by Theorem~\ref{thm:conditionc} (p.\pageref{thm:conditionc}) --
and using the notations of this theorem,
$P(\pi)$ contains $P(\rho) \cdot Q_x(\alpha)$
and some other words.  Defining $T'$ as in condition ($\mathcal C$),
$\rho'$ the permutation whose decomposition tree is $T'$, and $w$ the
unique word encoding the unique reading of the remaining leaves in
$\pi$ after $T'$ is read when $T$ is read in two \fois, these other
words are $P(\rho') \cdot w$.  Note that from
Lemma~\ref{lem:w_at_least_3} (p.\pageref{lem:w_at_least_3}) $w$ is a
strict pin word.  So $\overleftarrow{{\mathcal L}_{\pi}} = \Astar
\cdot \overleftarrow{\phi(Q_x(\alpha))} \cdot \overleftarrow{{\mathcal
    L}_{\rho}} \cup \Astar \cdot \overleftarrow{\phi(w)} \cdot
\overleftarrow{{\mathcal L}_{\rho'}}$.
The skeleton of ${\mathcal A}_{\pi}$ is the concatenation of the automaton
$\AC{\overleftarrow{\phi(Q_x(\alpha))}}$ with ${\mathcal A}_{\rho}$
and then as in the recursive case with a linear root,
we add some new transitions to account for the words in $P(\rho') \cdot w$.

Denoting $Z$ the last letter of $w$ (\emph{i.e.}, the first letter of
$\overleftarrow{\phi(w)}$), Lemma~\ref{lem:w_at_least_3} ensures that
no word of $\overleftarrow{\phi(Q_x(\alpha))}$ contains $Z$ and
therefore by Lemma~\ref{lem:trans_initial}
(p.\pageref{lem:trans_initial}) all the transitions labeled by $Z$ in 
the automaton $\AC{\overleftarrow{\phi(Q_x(\alpha))}}$ go to its initial state, denoted $q_0$.  We built an
automaton $\mathcal{A}$
by performing the following modifications on
$\AC{\overleftarrow{\phi(Q_x(\alpha))}}$: remove the loop labeled by
$Z$ on $q_0$ and add a path reading $\overleftarrow{\phi(w)}$ from
$q_0$ to a new final state $f'$.  Label all states $q$ of
$\mathcal{A}$ by the shortest word labeling a path from the initial
state $q_0$ to $q$.  Replace any transition labeled by $Z$ from a
state $q$ of $\AC{\overleftarrow{\phi(Q_x(\alpha))}}$ to $q_0$ by a
transition from $q$ to the new state labeled by $Z$.  Finally complete
the automaton with transitions from the states of the added path: for all such states $q$ but $f'$, the
transition from $q$ labeled by $a$ goes to the longest suffix of $q
\cdot a$ that is a state of the automaton -- either a new state or a pre-existing state.
Notice that the automaton ${\mathcal A}$ we obtain is almost complete and has exactly two final states,
without outgoing transitions:
$f$ -- the unique final state of $\AC{\overleftarrow{\phi(Q_x(\alpha))}}$ -- and $f'$.

The automaton ${\mathcal A}_{\pi}$ is then obtained from $\mathcal{A}$ and ${\mathcal A}_{\rho}$
by merging $f$ with the initial state $q_{T}$ of ${\mathcal A}_{\rho}$ and $f'$ with a marked state $q_{T'}$
(see Subsection~\ref{subsection:marquage}) of ${\mathcal A}_{\rho}$ which is
a state from which the recognized language is $\overleftarrow{{\mathcal L}_{\rho'}}$.
This construction is shown in Figure~\ref{fig:automate_primitif_rec}.

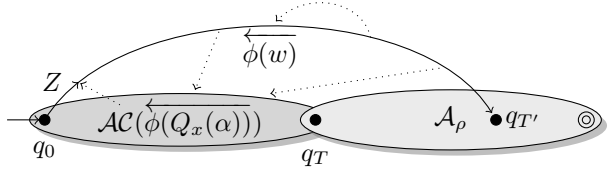
\begin{figure}[htbp]
\begin{center}
\begin{tikzpicture}
\useasboundingbox (-.5, -.1) rectangle (7.5,1.4);
\draw [drop shadow,fill=black!16] (1.8,0) node {$\AC{\overleftarrow{\phi(Q_x(\alpha))}}$} ellipse  (2cm and .35cm) ;
\draw [->] (-.5,0) -- (-.07,0);
\inputState{0}{0};
\draw (0,-.4) node {$q_{0}$};
\draw [drop shadow,fill=black!8] (5.4,0) node {$\mathcal{A}_\rho$} ellipse  (2cm and .4cm) ;
\inputState{6}{0};
\draw (6,0) node[right] {$q_{T'}$};
\draw [->,shorten >=3pt, inner sep=0pt,
        postaction={decorate, decoration={markings, mark=at position 0.1 with {\arrow{>}}}},
        postaction={decorate, decoration={markings, mark=at position 0.1 with \node (nZ){};}},
        postaction={decorate, decoration={markings, mark=at position 0.15 with \node (n1){};}},
        postaction={decorate, decoration={markings, mark=at position 0.5 with \node (n3){};}},
        postaction={decorate, decoration={markings, mark=at position 0.7 with \node (n4){};}},
        postaction={decorate, decoration={markings, mark=at position 0.85 with \node (n5){};}},
        postaction={decorate, decoration={markings, mark=at position 0.4 with \node (n2){};}}
        ] (0,0) .. controls +(1,1.65) and +(-1,1.65) .. node[below=2pt] {$\overleftarrow{\phi(w)}$} (6,0);
\draw[dotted,->] (1,0.2) -- (nZ);
\draw[dotted,->] (n2) -- (2,0.4);
\draw[dotted,->] (n5) -- (3,0.35);
\draw [dotted,->,shorten >= 2pt] (n4) .. controls +(-0.1,0.5) and +(0.5,0.5) .. (n3);
\draw (0.1,0.5) node {$Z$};
\inputState{3.6}{0};
\draw (3.6,-.5) node {$q_{T}$};
\outputState{7.2}{0};
\end{tikzpicture}
\end{center}
\caption{Automaton $\mathcal{A}_\pi$ for $\pi=\alpha[1,\ldots,1,\rho,1,\ldots,1]$.}
\label{fig:automate_primitif_rec}
\end{figure}

Notice that the automaton $\mathcal{A}$ obtained from $\AC{\overleftarrow{\phi(Q_x(\alpha))}}$
is somehow very similar to $\AC{\overleftarrow{\phi(Q_x(\alpha))}, \{\overleftarrow{\phi(w)}\}}$ but
because $\overleftarrow{\phi(w)}$ has a suffix in $\overleftarrow{\phi(Q_x(\alpha))}$ (from Lemma~\ref{lem:w_at_least_3}),
the sets of words $X_1 = \overleftarrow{\phi(Q_x(\alpha))}$ and $X_2 = \{\overleftarrow{\phi(w)}\}$
do not satisfy the independence condition required in our construction of $\AC{X_1,X_2}$.

\begin{lem}\label{lem:AC(Qx(alpha),Phi(w))}
The automaton $\mathcal{A}$ of the above construction recognizes the set of words ending with a first occurrence of a
word of $\overleftarrow{\phi(Q_x(\alpha))}$.  Moreover for any word  $u$ recognized by $\mathcal{A}$,
$\, q_0 \cdot u = f'$ if $\overleftarrow{\phi(w)}$ is a suffix of $u$, and $q_0 \cdot u = f$ otherwise.
\end{lem}

\begin{proof}
From Lemma~\ref{lem:w_at_least_3} (p.\pageref{lem:w_at_least_3}),
there exists a word $\bar{w} \in \overleftarrow{\phi(Q_x(\alpha))}$
and a letter $Z \in A$ such that $\overleftarrow{\phi(w)} =
Z\bar{w}$. Moreover no word of $\overleftarrow{\phi(Q_x(\alpha))}$
contains $Z$.

Therefore by construction, merging states $f$ and $f'$ of
$\mathcal{A}$ into a unique final state, we would obtain the automaton
$\AC{\overleftarrow{\phi(Q_x(\alpha))} \cup
  \{\overleftarrow{\phi(w)}\}}$.  Consequently, since
$\overleftarrow{\phi(w)}$ has a suffix in
$\overleftarrow{\phi(Q_x(\alpha))}$, the automaton $\mathcal{A}$
recognizes the set of words ending with a first occurrence of a word
of $\overleftarrow{\phi(Q_x(\alpha))}$.

Let $u$ be a word ending with its first occurrence of a word of $\overleftarrow{\phi(Q_x(\alpha))}$,
then $u$ does not have any factor in $\overleftarrow{\phi(Q_x(\alpha))} \cup \{\overleftarrow{\phi(w)}\}$
except as a suffix.
Lemma~\ref{lem:trans_initial} (p.\pageref{lem:trans_initial}) ensures that
$q_0 \cdot u$ is the state labeled with longest suffix of $u$ that is also a prefix of a word of
$\overleftarrow{\phi(Q_x(\alpha))} \cup \{\overleftarrow{\phi(w)}\}$, concluding the proof.
\end{proof}

Lemma~\ref{lem:AC(Qx(alpha),Phi(w))} allows us to prove the
correctness of the above construction of $\mathcal{A}_{\pi}$.  The
idea is the following: if $u \in \Astar \cdot
\overleftarrow{\phi(Q_x(\alpha))} \cdot \overleftarrow{{\mathcal
    L}_{\rho}}$ (resp. $\Astar \cdot \overleftarrow{\phi(w)}
\cdot\overleftarrow{{\mathcal L}_{\rho'}}$) and if
$trace_{\mathcal{A}_{\pi}}(u)$ contains $q_{T'}$ (resp. $q_T$) and not
$q_{T}$ (resp.  $q_{T'}$) before, then $u$ is still accepted by
$\mathcal{A}_{\pi}$ since $\overleftarrow{{\mathcal L}_{\rho}}
\subseteq \overleftarrow{{\mathcal L}_{\rho'}}$
(resp. $\overleftarrow{\phi(w)} \cdot \overleftarrow{{\mathcal
    L}_{\rho'}} \subseteq \overleftarrow{{\mathcal L}_{\rho}}$).  This
is formalized in the following theorem.

\begin{theo}\label{lem:A_pi_primitif_rec}
The automaton $\mathcal{A}_{\pi}$ obtained by the above construction recognizes $\overleftarrow{{\mathcal L}_{\pi}}$.
\end{theo}

\begin{proof}
Recall that $\overleftarrow{{\mathcal L}_{\pi}} = \Astar \cdot \overleftarrow{\phi(Q_x(\alpha))} \cdot \overleftarrow{{\mathcal L}_{\rho}} \cup \Astar \cdot \overleftarrow{\phi(w)} \cdot \overleftarrow{{\mathcal L}_{\rho'}}$.
The above construction ensures that every word accepted by $\mathcal{A}_{\pi}$ belongs to the language $\overleftarrow{{\mathcal L}_{\pi}}$.
Conversely let us prove that every word of $\overleftarrow{{\mathcal L}_{\pi}}$ is accepted by $\mathcal{A}_{\pi}$.

Let $u$ be a word of $\overleftarrow{{\mathcal L}_{\pi}}$.
From Lemma~\ref{lem:w_at_least_3} (p.\pageref{lem:w_at_least_3}), there is a word $\bar{w} \in \overleftarrow{\phi(Q_x(\alpha))}$ and a letter $Z \in A$ such that
$\overleftarrow{\phi(w)} = Z\bar{w}$.
Therefore $u$ has a factor in $\overleftarrow{\phi(Q_x(\alpha))}$.
Hence we can decompose $u$ uniquely as $u = u_1 u_2$ where $u_1$ is the prefix of $u$ ending with the first occurrence of a factor in $\overleftarrow{\phi(Q_x(\alpha))}$.
Consequently from Lemma~\ref{lem:AC(Qx(alpha),Phi(w))}, $q_0 \cdot u_1$ is either $q_T$ or $q_{T'}$, namely
$q_0 \cdot u_1 = q_{T'}$ if $\overleftarrow{\phi(w)}$ is a suffix of $u_1$ and $q_0 \cdot u_1 = q_T$ otherwise.

Moreover, since $u$ belongs to $\overleftarrow{{\mathcal L}_{\pi}}$,
and because by definition $\overleftarrow{{\mathcal L}_{\rho}}= \Astar \cdot \overleftarrow{{\mathcal L}_{\rho}}$ (and similarly for $\rho'$),
we deduce that $u_2$ belongs to $\overleftarrow{{\mathcal L}_{\rho}}$ or $\overleftarrow{{\mathcal L}_{\rho'}}$.
Let us finally notice that, since $\rho' \leq \rho$,
Theorem~\ref{thm:lienLangagesMotif} (p.\pageref{thm:lienLangagesMotif}) ensures that
$\overleftarrow{{\mathcal L}_{\rho}} \subseteq \overleftarrow{{\mathcal L}_{\rho'}}$
thus $u_2 \in \overleftarrow{{\mathcal L}_{\rho'}}$.

If $q_0 \cdot u_1 = q_{T'}$ then as $u_2 \in \overleftarrow{{\mathcal L}_{\rho'}}$,
$u$ is recognized by ${\mathcal A}_\pi$.
Assume on the contrary that $q_0 \cdot u_1 = q_T$.
Then $q_0 \cdot u = q_T \cdot u_2$ and by definition of $q_T$ it is enough to prove that $u_2 \in \overleftarrow{{\mathcal L}_\rho}$.

Assume first that  $u \notin \Astar \cdot \overleftarrow{\phi(w)} \cdot \overleftarrow{{\mathcal L}_{\rho'}}$.
Then 
since $u \in \overleftarrow{{\mathcal L}_{\pi}}$, we have
$u \in \Astar \cdot \overleftarrow{\phi(Q_x(\alpha))} \cdot \overleftarrow{{\mathcal L}_{\rho}}$.
Because $u_1$ ends with the first occurrence of a
factor of $\overleftarrow{\phi(Q_x(\alpha))}$,
we deduce that $u_2 \in \Astar \cdot
\overleftarrow{{\mathcal L}_{\rho}} = \overleftarrow{{\mathcal L}_{\rho}}$.

Otherwise $u \in \Astar \cdot \overleftarrow{\phi(w)} \cdot
\overleftarrow{{\mathcal L}_{\rho'}}$.
Recall that $u_1$ is the prefix of $u$ ending with the first occurrence of a
factor of $\overleftarrow{\phi(Q_x(\alpha))}$.
First (using also Lemma~\ref{lem:AC(Qx(alpha),Phi(w))} and $q_0 \cdot u_1 = q_T$), this implies
that $\overleftarrow{\phi(w)}$ is not a suffix of $u_1$.
And second, this also implies that $\overleftarrow{\phi(w)}$ is not a factor of $u_1$.
But by assumption $\overleftarrow{\phi(w)}$ is a factor of $u$.
We claim that the first occurrence of $\overleftarrow{\phi(w)}$ in $u$
starts after the end of $u_1$.
We have just proved that $\overleftarrow{\phi(w)}$ is not a factor of $u_1$.
Moreover, $\overleftarrow{\phi(w)} = Z\bar{w}$ starts with the letter $Z$,
and from Lemma~\ref{lem:w_at_least_3}
(p.\pageref{lem:w_at_least_3}) the $|\bar{w}|$ last letters of $u_1$
are different from $Z$
(recall that all words of $\overleftarrow{\phi(Q_x(\alpha))}$ have the same length $|\alpha| = |\bar{w}|$).
This proves our claim, and
consequently, $u_2 \in \Astar \cdot
\overleftarrow{\phi(w)} \cdot \overleftarrow{{\mathcal L}_{\rho'}}$.
Let $v \in \Astar$ and $v' \in \overleftarrow{{\mathcal L}_{\rho'}}$
such that $u_2 = v \cdot \overleftarrow{\phi(w)} \cdot v'$. 
From Lemma~\ref{lem:w_at_least_3} p.\pageref{lem:w_at_least_3},
denoting by $w'$ the suffix of length $2$ of $w$, for all $u'$ in
$P(\rho')$, $u'\cdot \phi^{-1}(w')$ belongs to $P(\rho)$.
Therefore $\overleftarrow{w'} \overleftarrow{{\mathcal L}_{\rho'}}
\subseteq \overleftarrow{{\mathcal L}_{\rho}}$.  But $v' \in
\overleftarrow{{\mathcal L}_{\rho'}}$ and $\overleftarrow{w'}$ is a
prefix of $\overleftarrow{\phi(w)}$, thus $u_2 = v \cdot
\overleftarrow{\phi(w)}\cdot v' \in \Astar \cdot \overleftarrow{w'}
\cdot \Astar \cdot \overleftarrow{{\mathcal L}_{\rho'}} \subseteq
\overleftarrow{{\mathcal L}_{\rho}}$, concluding the proof.
\end{proof}

\begin{rem}\label{rem:versionopt2}
With the optimized construction of ${\mathcal A}_{\pi}$, we prove
similarly that ${\mathcal A}_{\pi}$ recognize a language $\mathcal{L}'_{\pi}$ 
such that   $\mathcal{L}'_{\pi}\cap {\mathcal M} =
\overleftarrow{\mathcal{L}_{\pi}} \cap {\mathcal M}$.
\end{rem}

\paragraph*{If $\pi$ satisfies condition ($\mathcal C$) and $|T|=2$}\label{page:automaton_conditionC_T=2} \hspace{-0.4cm} Then the construction is no more recursive.
Permutation $\pi$ and its pin words are explicit.
More precisely from Theorem~\ref{thm:conditionc}
(p.\pageref{thm:conditionc}), $P(\pi) = P_{\{1,n\}}(\pi) \cup
P_{\{2,n\}}(\pi) \cup P(T) \cdot Q_x(\alpha)$.
Thus from Remark~\ref{rem:w_2} (p.\pageref{rem:w_2}),
$$\overleftarrow{{\mathcal L}_{\pi}} = \Big( \bigcup\limits_{u \in
  P(\pi) \atop u \text{ strict or quasi-strict}}
\overleftarrow{\mathcal{L}(u)} \Big) \quad \bigcup \quad \Astar
\cdot \overleftarrow{\phi(Q_x(\alpha))} \cdot \overleftarrow{{\mathcal
    L}_{\rho}} \text{.}$$ Therefore $\mathcal{A}_{\pi}$ is the automaton
${\mathcal U}^{\circlearrowleft}(\mathcal{A}_{\pi}^{\textsc{sqs}} \ ,\
{\mathcal {AC}}(\overleftarrow{\phi(Q_x(\alpha))}) \cdot \mathcal{A}_{\rho} )$.

\begin{lem} \label{lem:complexite_primitif_rec}
Let $\pi = \alpha[1, \ldots, 1, \rho, 1,\ldots, 1]$ where $\alpha$ is a simple permutation whose set $P(\alpha)$ of pin words is given.
Then the construction of the automaton
$\mathcal{A}_{\pi}$
is done in time and space
$\O\left(|\pi|-|\rho|\right)$ plus the additional time and space
due to the construction of $\mathcal{A}_{\rho}$, except when $\pi$ satisfies condition ($\mathcal{C}$) and $|T|=2$. In this latter case,
the complexity is $\O\left(|\pi|^3\right)$ with the classical construction and $\O\left(|\pi|^2\right)$ in the optimized version.
\end{lem}

\begin{proof}
Recall that $Q_x(\alpha)$  contains words of length $|\alpha|-1$.
Its cardinality is smaller than the one of $P(\alpha)$, hence smaller
than $48$ (see Theorem~\ref{thm:nbpinwords}
p.\pageref{thm:nbpinwords}). Moreover $Q_x(\alpha)$ can be determined
in linear time w.r.t.~$|\alpha|$ as described in Remark~\ref{rem:compute_Q_x(alpha)} (p.\pageref{rem:compute_Q_x(alpha)}).
Consequently, ${\mathcal{AC}} (\overleftarrow{\phi(Q_x(\alpha))})$ is built in time and space $\O\left(|\alpha|\right) = \O\left(|\pi|-|\rho|\right)$.

If $\pi$ does not satisfy condition ($\mathcal{C}$) then
$\mathcal{A}_{\pi} = {\mathcal{AC}}(\overleftarrow{\phi(Q_x(\alpha))}) \cdot \mathcal{A}_{\rho}$,
so that $|\mathcal{A}_{\pi}| = |{\mathcal{AC}}(\overleftarrow{\phi(Q_x(\alpha))})| + |\mathcal{A}_{\rho}|$
and the time complexity of this construction is $\O\left(|\pi|-|\rho|\right)$ plus the additional time to build $\mathcal{A}_{\rho}$.

If $\pi$ satisfies condition ($\mathcal{C}$) and $|T|\geq 3$, then $|w|=|\alpha|$ and
by Remark~\ref{rem:w_explicit} (p.\pageref{rem:w_explicit}),
$w$ is explicitly determined.
Consequently, so is the additional path labeled by $\overleftarrow{\phi(w)}$ added to the automaton (see Figure~\ref{fig:automate_primitif_rec}).
The modifications of the transitions between this path and ${\mathcal{AC}}(\overleftarrow{\phi(Q_x(\alpha))})$
are performed in linear time w.r.t.~the length of this path and $|{\mathcal{AC}}(\overleftarrow{\phi(Q_x(\alpha))})|$,
\emph{i.e.}, in  $\O(|\phi(w)|+|\alpha|) = \O(|\pi|-|\rho|)$.
We conclude that $\mathcal{A}_{\pi}$ is built in $\O\left(|\pi|-|\rho|\right)$ time and space plus the additional time and space to build $\mathcal{A}_{\rho}$.

If $\pi$ satisfies condition ($\mathcal{C}$) and $|T|=2$, then
$\mathcal{A}_{\pi} =  {\mathcal U}^{\circlearrowleft}(\mathcal{A}_{\pi}^{\textsc{sqs}},
{\mathcal{AC}} (\overleftarrow{\phi(Q_x(\alpha))}) \cdot \mathcal{A}_{\rho} )$. 
Recall that $P_{\textsc{sqs}}(\pi)$ is given in Remark~\ref{rem:w_2} (p.\pageref{rem:w_2}) and contains $12$ pin words.
Hence, with the classical construction (resp. in the optimized version), from Lemma~\ref{lem:ascsquadratique} (p.\pageref{lem:ascsquadratique})
(resp. Lemma~\ref{lem:sqs_linear_complexity} p.\pageref{lem:sqs_linear_complexity}) and Remark~\ref{rem:w_2},
we can build $\mathcal{A}_{\pi}^{\textsc{sqs}}$ in time and space $\O\left(|\pi|^2\right)$ (resp. $\O\left(|\pi|\right)$).
Moreover 
since $|\rho|=2$, $\mathcal{A}_{\rho}$ is obtained in constant time,
so that ${\mathcal{AC}} (\overleftarrow{\phi(Q_x(\alpha))}) \cdot \mathcal{A}_{\rho}$ is obtained in time and space $\O(|\pi|-|\rho|) = \O\left(|\pi|\right) $.
Finally, $\mathcal{A}_{\pi}$ is built in time and space $\O\left(|\pi|^3\right)$ (resp. $\O\left(|\pi|^2\right)$) with the classical (resp. optimized) construction. 
\end{proof}

\subsubsection{Two children are not leaves.} \label{sec:tcl}

Up to symmetry this means that $\pi = $ \begin{tikzpicture}[level
    distance=17pt,sibling distance=6pt,baseline=-10pt,inner sep=0]
  \node[simple,inner sep=0] (X) {$\beta^{+}$} child {[fill] circle
    (2pt)} child [missing] child {[fill] circle (2pt) node (xx1){}}
  child [missing] child [sibling distance=0pt] {node (xx2){} edge from
    parent [draw=none] }child[thick, dotted] {node[thin,
      shape=isosceles triangle, solid, draw, shape border
      rotate=90,anchor=apex, minimum height=5mm,inner
      sep=1pt,isosceles triangle apex angle=110] {$T$}} child [sibling
    distance=0pt] {node (xx3){} edge from parent [draw=none] } child
  [missing] child[dash pattern=on 3pt off 2pt on 1pt off 2pt] {node
    (xx4) {$12$}} child [missing] child {[fill] circle (2pt)}; \draw
  [dotted] (xx1) -- (xx2); \draw [dotted] (xx3) -- (xx4);
  \end{tikzpicture}, 
where $\beta^+$ is an increasing quasi-oscillation, the
permutation $12$ expands an auxiliary point of $\beta^{+}$ and $T$ 
expands the corresponding main substitution point of $\beta^{+}$.

Theorem~\ref{thm:primeRoot_cas_special}
(p.\pageref{thm:primeRoot_cas_special}) ensures that the pin words
encoding $\pi$ are of the form $v.w$ where $v \in P(T)$ and $w$ is a
strict pin word of length $|\beta^{+}|$ uniquely determined by $\beta^+$ and the two points expanded in $\beta^{+}$,
and known explicitly from Remark~\ref{rem:w_cas_2_non_feuilles} (p.\pageref{rem:w_cas_2_non_feuilles}).

Therefore
$\overleftarrow{\mathcal{L}_{\pi}}=
A^\star\overleftarrow{\phi(w)}\overleftarrow{\mathcal{L}_{\rho}}$
where $\rho$ is the permutation whose decomposition tree is $T$.
The automaton $\mathcal{A}_{\pi}$ recognizing
$\overleftarrow{\mathcal{L}_{\pi}}$ is the concatenation of
$\mathcal{AC}(\{\overleftarrow{\phi(w)}\})$ with $\mathcal{A}_{\rho}$, which is recursively obtained.

This construction is done in $\O\left(|\overleftarrow{\phi(w)}|\right) = \O\big(|\pi|-|\rho|\big)$ time
and space in addition to the time and space complexity of the construction of $\mathcal{A}_{\rho}$.

\subsection{Marking states}
\label{subsection:marquage}
In our constructions of Subsections~\ref{ssec:rec_linear} and
\ref{sec:decomp_simple} we need transitions going to initial states of
subautomata.  We could duplicate the corresponding subautomata. But
when these are recursively obtained an exponential blow-up can
occur. To keep a polynomial complexity we replace duplication by the
marking of these special states. The marking is made on the fly during
the construction and we explain how in this subsection.

The need of creating a transition going to a marked state (of a
subautomaton) happens only when building the automaton
$\mathcal{A}_{\pi}$
in Subsection~\ref{ssec:rec_linear} for a
permutation $\pi$ whose decomposition tree has a linear root and
satisfies a condition $\horse{i}{j}$ of Figure~\ref{fig:H} (p.\pageref{fig:H}),
or in Subsection~\ref{sec:decomp_simple}
for a permutation $\pi$ whose decomposition tree
has a prime root and satisfies condition ($\mathcal{C}$) (see
Definition~\ref{def:c} p.\pageref{def:c}) with $|T|\geq 3$.

In both cases we need to mark in the subautomaton $\mathcal{A}_{\rho}$
with $\rho \leq \pi$ some states $q_Y$ such that
the language recognized taking $q_Y$ as initial state is $\overleftarrow{{\mathcal L}_{\sigma}}$,
where $\sigma \leq \rho$ is the permutation whose diagram (or decomposition tree) is $Y$.

As it appears in Figure~\ref{fig:H} and in condition ($\mathcal{C}$), in
almost all such situations, the marked state belongs to a subautomaton
corresponding to a permutation $\rho$ whose decomposition tree $R$ has a
linear root.  There is only one situation where this root is prime:
when $\pi$ satisfies condition $(1H1+)$.
We first focus on this case.

\vspace{-0.5cm}

\paragraph*{Prime root}
Let $\theta$ be a permutation of decomposition tree
 $R = $\begin{tikzpicture}[sibling
    distance=10pt,level distance=13pt,baseline=-15pt] \node[simple]
  {$\xi^{+}$} child {[fill] circle (2 pt) node(x1) {}} child[missing]
  child {[fill] circle (2 pt) node(xk) {}} child[child anchor=north]
  {node[draw,shape=isosceles triangle, shape border
      rotate=90,anchor=north,inner sep=0, isosceles triangle apex
      angle=90] {$S$}} child {[fill] circle (2 pt) node(y1) {}}
  child[missing] child {[fill] circle (2 pt) node(yk) {}};
  \draw[dotted] (x1) -- (xk); \draw[dotted] (y1) -- (yk);
\end{tikzpicture} where $\xi^{+}$ is an increasing oscillation,
and let $\gamma$ be the permutation whose decomposition tree is $S$.
In the case where $\pi$ satisfies condition $(1H1+)$, we need to mark in the
automaton $\mathcal{A}_{\theta}$ the state $q$ such that when starting from $q$ the
language recognized is the one recognized by $\mathcal{A}_{\gamma}$. 
(Notice that w.r.t.~the previous paragraph, we have changed the notations $\rho$ to $\theta$ and $\sigma$ to $\gamma$
to avoid confusions with the notations used in Subsection~\ref{sec:decomp_simple}.)

The automaton $\mathcal{A}_{\theta}$ is obtained as described in
Subsection~\ref{sec:decomp_simple}, when exactly one child of the root is
not a leaf (indeed $|S| \geq 2$). The marking of state $q$ depends on
how the automaton $\mathcal{A}_{\theta}$ is built and in particular on
whether $\theta$ satisfies condition ($\mathcal{C}$) or not.

Recall that $\xi^{+}$ is an increasing oscillation. If $\xi^{+}$ has
a size at least $5$, it is not a quasi-oscillation, and $\theta$ does not
satisfy condition (${\mathcal C}$). Therefore $\mathcal{A}_{\theta}$ is the
concatenation of two automata the second of which is $\mathcal{A}_{\gamma}$
whose initial state can be readily marked.

If $\xi^{+}$ has size $4$, then $\xi^{+}=2\, 4\, 1\, 3$ or $3\, 1\,
4\, 2$ is a quasi-oscillation and $\theta$ may satisfy condition ($\mathcal{C}$).
If it is not the case, $\mathcal{A}_{\theta}$ is obtained as
above and so is the marking of state $q$.
If on the contrary $\theta$ satisfies condition ($\mathcal{C}$),
the construction of $\mathcal{A}_{\theta}$ depends on $|S|$.
If $|S| \geq 3$, $\mathcal{A}_{\theta}$ is again the
concatenation of two automata the second one being
$\mathcal{A}_{\gamma}$, but with some states and transitions added.
As these transitions are not reachable from the initial state of $\mathcal{A}_{\gamma}$,
we mark it as above.
If $|S|=2$, then $R$ has size $5$ and the construction
is not recursive anymore. We want to mark in $\mathcal{A}_{\theta}$ a
state $q$ corresponding to the initial state of
$\mathcal{A}_{\gamma}$. But
in the construction of $\mathcal{A}_{\theta}$ in Subsection~\ref{sec:decomp_simple},
we have built an automaton $\mathcal{A}'$
such that $\mathcal{A}_{\theta}= {\mathcal
  U}^{\circlearrowleft}(\mathcal{A}^{\textsc{sqs}}_{\theta},
\mathcal{A}'\cdot \mathcal{A}_{\gamma}).$ Therefore
$\mathcal{A}_{\theta}$ is a Cartesian product and the state $q$
has been replicated several times.  As $|S|=2$, $\mathcal{A}_{\gamma}$ has a constant size,
hence in this particular case we just duplicate it and mark its initial state instead of
marking a state inside $\mathcal{A}_{\theta}$.

\paragraph*{Linear root}
Consider now the case where the decomposition tree $R$ of the permutation $\rho$ has a linear root.
The need of a marked state in $\mathcal{A}_{\rho}$
happens only when  the leftmost (resp. rightmost) child of
$R$ is a leaf $z$.

In almost all cases, the marked state $q$ is such that the language
accepted starting from $q$ is the set of words encoding the readings
of all nodes of $R$ except the leaf $z$. There are at most two
such leaves and from Remarks~\ref{rem:pattern} and \ref{rem:pattern_rec}
(p.\pageref{rem:pattern} and \pageref{rem:pattern_rec}),
the corresponding marked states of $\mathcal{A}_{\rho}$
(which is built as described in Subsection~\ref{ssec:non-rec_linear} or \ref{ssec:rec_linear})
are $q_{1(r-1)}$ and $q_{2r}$ in Figure~\ref{fig:automate_cas_oplus_non_recursif}
(p.\pageref{fig:automate_cas_oplus_non_recursif}) or
\ref{fig:automate_cas_oplus_recursif} (p.\pageref{fig:automate_cas_oplus_recursif}) -- with $\rho$ instead of $\pi$.
There is however one exception, corresponding to the special case described in Remark~\ref{rem:pattern}:
when $R$ has exactly two children, which are $z$
and an increasing oscillations $\xi$.
In this special case the construction of $\mathcal{A}_{\rho}$ is not recursive anymore.
Instead of marking in $\mathcal{A}_{\rho}$ a state $q$ corresponding to the initial state of $\mathcal{A}_{\xi}$,
we just duplicate $\mathcal{A}_{\xi}$ and mark its initial state.
If $|\xi| < 4$ then $|\mathcal{A}_{\xi}| = \mathcal{O}(1)$.
Otherwise $\xi$ is a simple permutation and
$|\mathcal{A}_{\xi}|$ is quadratic (or linear in the optimized complexity)
w.r.t.~$|\xi|$.
In both cases $|\mathcal{A}_{\xi}|+|\mathcal{A}_{\rho}|$ has the same order as
$|\mathcal{A}_{\rho}|$ and since the construction is not recursive,
this does not change the overall complexity of the construction of $\mathcal{A}_{\pi}$.

The few cases where the marked stated $q$ is not as above
(\emph{i.e.}, is not such that the language
accepted starting from $q$ is the set of words encoding the readings
of all nodes of $R$ except $z$) correspond to
state $q_S$ of conditions ($2H2\star$) and ($1H2\star$) and
states $q_{T \cup a}$ and  $q_{T \cup b}$ of condition ($2H3$).
In these cases, $R$ has exactly two children: $z$ and a subtree $R'$ whose root is linear.
Then the leftmost (resp. rightmost) child of $R'$ is a leaf $z'$
and the marked state $q$ is such that the language
accepted starting from $q$ is the set of words encoding the readings
of all nodes of $R'$ except the leaf $z'$. 
We are in the same situation as above, 
so the states can be marked in the same way, 
except that now we have to mark states in
$\mathcal{A}_{\rho'}$ instead of $\mathcal{A}_{\rho}$,
where $\rho'$ is the permutation whose decomposition tree is $R'$
and $\mathcal{A}_{\rho'}$ is a subautomaton of $\mathcal{A}_{\rho}$ built recursively.

Notice that we never create transitions going to marked states
belonging to automata built more than two levels of recursion deeper.
Indeed in all conditions above the created transitions go to the
automaton build in the previous step of recursion, except for
conditions ($2H3$), ($2H2\star$) and ($1H2\star$) where two levels of
recursion are involved.

\subsection{Complexity analysis} \label{ssec:complexity}

\begin{theo}\label{thm:complexity-A_pi}
For every pin-permutation $\pi$ of size $n$, 
$\mathcal{A}_{\pi}$ is built in time and space $\mathcal{O}(n^2)$ in the optimized version and $\mathcal{O}(n^4)$ in the classical version.
\end{theo}

\begin{proof}
To build $\mathcal{A}_\pi$, we first need to decide which shape of Equation~\eqref{eq:pin_perm_trees}
is matched by the decomposition tree of $\pi$, and whether $\pi \in \setH$ or whether $\pi$ satisfies condition $\mathcal C$.
The reader familiar with matching problems will be convinced that this can be done in $\mathcal{O}(n)$ time.
In any case, a linear algorithm for this tree matching problem is detailed in Subsection~\ref{ssec:finding_pin-perm}
as a sub-procedure of the global algorithm of Section~\ref{sec:polynomial}.

Then Theorem~\ref{thm:complexity-A_pi} follows from the complexities of
the previous constructions that are summarized in Table~\ref{tab:complexity}
in which we denote by $\rho$ the permutation whose decomposition tree is $T$.

\begin{table}[ht]
\begin{center}
{\small 
\begin{tabular}{|l|l|l|l|}
  \hline
  pin-permutation of size $n$& Complexity & Optimized & Lemma \\
  \hline
  size 1 & $\mathcal{O}(1)$ & $\mathcal{O}(1)$ & \\
  \hline
  simple & $\mathcal{O}(n^2)$ & $\mathcal{O}(n)$ & \ref{rem:ascquadratique}, \ref{rem:simple_linear} \\
  \hline
  root $\oplus$ non-recursive & $\mathcal{O}(n^4)$ & $\mathcal{O}(n^2)$ & \ref{lem:complexite_linear} \\
  \hline
  root $\oplus$ recursive, & $\mathcal{O}(\,(n-|\rho|)^2\,) \ +$ &  $\mathcal{O}(\,(n-|\rho|)^2\,) \ +$ & \ref{lem:complexite_linear_rec}, \\
  one child $T$ is not an & contribution for $\mathcal{A}_{\rho}$ & contribution for  $\mathcal{A}_{\rho}$ & \ref{lem:complexite_linear_rec_2} \\
  increasing oscillation & & & \\
  \hline
  root is prime recursive, & & & \\
  ${\mathcal C}$ is satisfied, & $\mathcal{O}(n^3)$ & $\mathcal{O}(n^2)$ & \ref{lem:complexite_primitif_rec} \\
  and $T$ has size $2$ & & & \\
  \hline
  root is prime recursive & $\mathcal{O}(\,n-|\rho|\,) \ +$ &$\mathcal{O}(\,n-|\rho|\,) \ +$ & \ref{lem:complexite_primitif_rec}, \\
   (if not preceding case) & contribution for $\mathcal{A}_{\rho}$ & contribution for $\mathcal{A}_{\rho}$  & \S \ref{sec:tcl} \\
  \hline
\end{tabular}
}
\end{center}
\caption{Complexities of the automata constructions, in all possible cases.}
\label{tab:complexity}
\end{table}
 
In the optimized version (resp. in the classical version)
the complexity is at most in $\O(n^2)$ (resp. $\O(n^4)$) in the
non-recursive cases and at most in $\mathcal{O}((n-|\rho|)^2)$
plus the additional complexity of the construction of $\mathcal{A}_{\rho}$
in the recursive cases. Notice that no extra time is needed to mark
the states of the automaton, as they are marked when they are built.
Consequently in the optimized version (resp. in the classical version)
the automaton $\mathcal{A}_{\pi}$ can be built in time and space $\O(n^2)$
(resp. $\O(n^4)$), $n$ being the size of $|\pi|$.

Indeed let $K$ be the number of levels of recursion needed in the construction of $\mathcal{A}_{\pi}$.
Then we can set $\rho_1=\pi$ and define recursively permutations $\rho_i$ for $2 \leq i \leq K$,
$\rho_i$ being the permutation $\rho$ that appears recursively when building $\mathcal{A}_{\rho_{i-1}}$. 
From Table~\ref{tab:complexity}, we deduce that, in the optimized version, the time and space complexity for building $\mathcal{A}_{\pi}$ is:
$$\O \left((|\rho_1|-|\rho_{2}|)^2 \right) + |\mathcal{A}_{\rho_2}| = \ldots = \O \left(\sum_{i=1}^{K-1} (|\rho_i|-|\rho_{i+1}|)^2 + |\rho_K|^2\right)\text{.}$$
Since every $\rho_i$ is a pattern of $\pi$, we have $|\rho_i|-|\rho_{i+1}|\leq n$ and $|\rho_K|\leq n$. Hence, the time and space complexity for building $\mathcal{A}_{\pi}$ is:
$$\O \left( n \cdot \Big( \sum_{i=1}^{K-1} (|\rho_i|-|\rho_{i+1}|) + |\rho_K| \,\Big) \right) = \O \left( n \cdot |\rho_1| \right) =\O(n^2).$$
In the same way we get the complexity $\O(n^4)$ for the classical version.
\end{proof}

\end{appendices}

\bibliographystyle{plain}

\end{document}